\documentclass{iopjournal}

\usepackage{amsmath,amsfonts}
\usepackage{algorithmic}
\usepackage{algorithm}
\usepackage[caption=false,font=normalsize,labelfont=sf,textfont=sf]{subfig}
\usepackage{textcomp}
\usepackage{stfloats}
\usepackage{url}
\usepackage{verbatim}
\usepackage{graphicx}
\usepackage{cite}

\usepackage{amssymb}
\usepackage{mathdots}
\usepackage{mathrsfs}
\usepackage{enumitem}
\usepackage{bibunits}
\defaultbibliographystyle{IEEEtran}

\usepackage{xr}

\def\st{\text{subject to }}

\def\m #1{\boldsymbol{#1}}

\def\bee{\begin{equation}}
\def\ene{\end{equation}}

\def\beq{\begin{eqnarray}}
\def\enq{\end{eqnarray}}

\newtheorem{lem}{Lemma}

\newtheorem{cor}{Corollary}
\newtheorem{thm}{Theorem}
\newtheorem{prop}{Proposition}

\newenvironment{proof}{{\noindent\it \quad Proof:}\;}{\hfill $\square$\par}

\def\diag #1{\text{diag}#1}

\def\st {\text{ subject to }}

\begin{document}

\articletype{Paper} 

\begin{bibunit}[IEEEtran]

\title{Jointly Sparse Blind Deconvolution via Riemannian Optimization}

\author{Wenlong Wang$^1$\orcid{0009-0001-5855-8723}, Baiyang Guo$^2$\orcid{0009-0003-2399-0879}, Zai Yang$^{1,*}$\orcid{0000-0002-9502-5176}, Shixiang Chen$^3$\orcid{0000-0002-3261-0714} and Junpeng Shi$^2$\orcid{0000-0002-9910-0663}}

\affil{$^1$School of Mathematics and Statistics, Xi’an Jiaotong University, Xi’an, China}

\affil{$^2$College of Electronic Engineering, National University of Defense Technology, Hefei, China}

\affil{$^3$School of Mathematical Sciences,University of Science and Technology of China, Hefei, China}

\affil{$^*$Author to whom any correspondence should be addressed.}

\email{wang1813857265@stu.xjtu.edu.cn, guobaiyang@nudt.edu.cn, yangzai@xjtu.edu.cn, shxchen@ustc.edu.cn and shijunpeng20@nudt.edu.cn}


\begin{abstract}
Blind deconvolution has been widely applied in system identification and signal processing. While joint sparsity commonly arises in practical scenarios, effectively exploiting this structure to enhance recovery performance remains a challenging and largely open problem. In this paper, we propose a joint-sparsity-promoting optimization problem and develop a Riemannian optimization algorithm for its accurate and efficient solution. We further establish theoretical guarantees that characterize the non-asymptotic relationship between the estimation error and the sample complexity, showing that exploiting joint sparsity can significantly reduce the sample complexity required for successful recovery. Numerical experiments are provided that validate the theoretical results and demonstrate the effectiveness of the proposed approach.
\end{abstract}

\section{Introduction}
Blind deconvolution is a fundamental inverse problem in information science and signal processing~\cite{levin2009understanding,6680763}, with broad applications in system identification and signal recovery, including seismic imaging~\cite{kaaresen1998multichannel}, astronomical imaging~\cite{prato2013convergent}, neuroscience~\cite{benfenati2023neural,ekanadham2011blind}, fluorescence microscopy~\cite{zhang2025computational,1628876}, radar systems~\cite{vargas2023dual}, image restoration~\cite{lv2018convex}, and wireless communications~\cite{8242678}. Recovering both an unknown channel and unknown input signals from their convolved output without direct knowledge of either constituent is a cornerstone capability in modern information systems. However, blind deconvolution is fundamentally ill-posed: the convolutive structure introduces inherent identifiability ambiguities that prevent unique recovery without additional constraints. Multichannel observation scenarios arise in many practical applications, where each channel shares a common filter and receives a sparse input signal. Incorporating both multichannel measurement and signal sparsity constraints significantly mitigates the ill-posedness of the problem, enabling identifiable and well-posed recovery~\cite{8762219,li2018global,qu2020exact,shi2021manifold,8365828,akhavan2019multichannel,cao2024generalized}. This paper focuses on a particularly compelling scenario where the input signals across all channels share the same sparse support, referred to as jointly sparse blind deconvolution.

The jointly sparse blind deconvolution model arises naturally in radar sensing and fluorescence microscopy applications. In passive sensing systems, the unknown filter corresponds to the transmitted waveform~\cite{267014,757212}, whereas in active sensing systems it may represent sensor imperfections, such as unknown gains and phases~\cite{7779126,10535469,eldar2020sensor}. The joint sparsity assumption reflects the fact that multichannel measurements are acquired in the same propagation environment and therefore share a small number of sparse propagation paths~\cite{rossi2013spatial}. This model underlies a variety of important applications, including the resolution of overlapping echoes with unknown characteristics~\cite{267014,757212}, blind gain and phase calibration~\cite{7779126,10535469,eldar2020sensor}, integrated radar--communication systems~\cite{vargas2023dual,jacome2024multi}, and synthetic aperture radar autofocus~\cite{4799379}. Jointly sparse blind deconvolution also arises in certain fluorescence microscopy applications, such as fixed-cell microtubule imaging~\cite{min2013fluorescent}. In these settings, microtubules constitute sparse cellular structures and can be regarded as approximately stationary over short observation intervals~\cite{min2013fluorescent,vavrdova2019multicolour}. Meanwhile, the imaging system may suffer from non-ideal effects, including optical aberrations, defocus, and refractive-index mismatch~\cite{soulez2012blind,soulez2016superresolution}. Consequently, accurate image reconstruction requires the simultaneous estimation of both the underlying specimen structure and the unknown imaging system response.

For the jointly sparse blind deconvolution model, the observation in each channel is represented as the circulant convolution of an unknown filter and an input signal:
\begin{equation}
\label{circulant-convolution}
\begin{aligned}
\m{y}_l=\m{g}\circledast\m{x}_l\in\mathbb{R}^N,\;l=1,\cdots,L,
\end{aligned}
\end{equation}
where $N$ and $L$ denote the filter length and the number of channels (also referred to as the sample size), respectively, $\m{x}_l\in\mathbb{R}^N$ represents the input signal of the $l$-th channel, and $\m{g}\in\mathbb{R}^N$ denotes the filter. The symbol $\circledast$ represents the circulant convolution operator, which can be expressed equivalently as a circulant matrix multiplication:
\begin{equation}
\begin{aligned}
\m{g}\circledast\m{x}=\mathcal{C}\left(\m{g}\right)\m{x},
\end{aligned}
\end{equation}
where $\m{x}$ denotes an input signal and
\begin{equation}
\begin{aligned}
\mathcal{C}\left(\m{g}\right)=\begin{bmatrix}g_1&g_N&\cdots&g_2\\g_2&g_1&\cdots&g_3\\\vdots&\vdots&\ddots&\vdots\\g_N&g_{N-1}&\cdots&g_1\end{bmatrix}\in\mathbb{R}^{N\times N}.
\end{aligned}
\end{equation}
The circulant convolution formulation is widely adopted in practical scenarios where the filter satisfies periodic boundary conditions~\cite{cho2009fast} or serves as an approximation of linear convolution when the filter has compact support or decays rapidly~\cite{strohmer2002four,li2018global}. Throughout this paper, we assume that the input signals $\left\{\m{x}_{l}\right\}_{l=1}^L$ share a common support set $\Omega\subseteq\left\{1,2,\cdots,N\right\}$ with cardinality $\left|\Omega\right|=K$. Our objective is to recover both the filter $\m{g}$ and the inputs $\left\{\m{x}_{l}\right\}_{l=1}^L$ from the observations $\left\{\m{y}_{l}\right\}_{l=1}^L$ under the joint sparsity assumption.

The jointly sparse blind deconvolution problem reduces to the general multichannel sparse blind deconvolution problem when the input signals in different channels are sparse but do not necessarily share a common support. Various methods have been proposed to address multichannel sparse blind deconvolution with theoretical recovery guarantees~\cite{8762219,li2018global,qu2020exact,shi2021manifold}. Existing theoretical analyses primarily focus on two aspects: the sample complexity required for successful recovery and the development of optimization algorithms with provable convergence guarantees. This paper aims to further improve both the theoretical guarantees and the algorithmic performance by explicitly leveraging the joint sparsity structure.

\subsection{Contributions}
In this paper, we develop an algorithmic framework and theoretical analysis for jointly sparse blind deconvolution. The main contributions are summarized as follows.
\begin{enumerate}
\item
We formulate jointly sparse blind deconvolution as a Riemannian optimization problem on the unit sphere. To exploit the common support shared by multiple input signals, we introduce a smooth joint-sparsity-promoting objective function that effectively aggregates information across channels while remaining amenable to both geometric analysis and algorithmic development.
\item
We establish a geometric characterization of the proposed optimization landscape. Under a probabilistic model in which the nonzero entries of the input signals are independently drawn from a standard Gaussian distribution, we show that every approximate second-order stationary point lies within a neighborhood of the target solution, provided the number of channels is sufficiently large. This result yields a rigorous recovery guarantee for jointly sparse blind deconvolution via nonconvex optimization.
\item
Building on the above landscape analysis, we develop a Riemannian gradient descent algorithm with negative curvature search (RGD-NCS). We prove that the proposed algorithm computes an approximate second-order stationary point within a finite number of iterations. Combined with the landscape characterization, this result guarantees that the underlying filter can be recovered accurately with high probability. Furthermore, we extend both the optimization framework and the proposed algorithm to the complex-valued setting.
\end{enumerate}
Extensive numerical experiments are conducted to validate the proposed framework. The results demonstrate the effectiveness of the proposed approach and corroborate the theoretical prediction that exploiting joint sparsity reduces the sample complexity required for successful recovery. Finally, the application in array signal processing is presented to illustrate the usefulness of the proposed method.

\subsection{Relations to Prior Art}

The advantages of multichannel observations have been extensively studied in areas such as compressed sensing~\cite{4014378,eldar2009average} and spectral analysis~\cite{7484756,8533395}. In these settings, the measurements are known linear transformations of unknown input signals, and the objective is to recover the input signals by exploiting their joint sparsity. Related research has shown that, when the nonzero entries of the sparse input signals are generated from an appropriate random distribution, exact recovery can be achieved via convex optimization with high probability, provided that the sample size is sufficiently large. In contrast, the jointly sparse blind deconvolution problem considered in this paper is substantially more challenging because the linear transformation itself is unknown, leading to a highly ill-conditioned bilinear inverse problem. Moreover, the only structural assumption imposed on the transformation is that it admits a circulant matrix representation.

Early studies on sparse blind deconvolution are mainly focused on incorporating additional prior information about either the filter or the input signals~\cite{6680763,7508934,10985864,9830780,8365828,huang2018blind}. Some works assume that the filter admits a sparse representation under a known dictionary, thereby significantly reducing the number of unknown parameters and improving identifiability~\cite{6680763,7508934,10985864,9830780,huang2018blind}. Under these assumptions, convex optimization methods can successfully recover both the filter and input signals when the observation length is sufficiently large. Other studies impose stronger structural assumptions on the input signals. For example, the work in~\cite{8365828} assumes that the support of the input signals is restricted to the first few entries and that the matrix formed by the nonzero entries admits a low-dimensional representation in a known subspace. In contrast to these approaches, which rely on additional structural priors, we consider a fully blind setting in which both the input signals and their support are unknown, the nonzero entries are random, and no specific structural assumptions are imposed on the common support set.

The main challenges in sparse blind deconvolution stem from the bilinear nonconvexity and parameter ambiguities induced by the convolution structure. Similar challenges also arise in dictionary learning, which aims to recover sparse representations of observed data~\cite{sun2016complete,7755786,7165675,7378964,sun2025global,ruetz2026convergence,liang2025simple}. Since both the dictionary and the sparse coefficients are unknown, dictionary learning is likewise a bilinear inverse problem, differing from sparse blind deconvolution primarily in that the dictionary in the latter is constrained to have a circulant structure. Extensive research on dictionary learning has inspired reparameterization-based optimization methods for bilinear inverse problems. The key idea is to exploit the assumption that the filter is invertible, i.e., the circulant matrix $\mathcal{C}\left(\m{g}\right)$ is nonsingular. Under this assumption, the unknown input signals $\left\{\m{x}_{l}\right\}_{l=1}^L$ can be recovered from the observations $\left\{\m{y}_{l}\right\}_{l=1}^L$ via $\m{x}_l=\mathcal{C}\left(\m{g}\right)^{-1}\m{y}_l$ once the filter $\m{g}$ is known. This paper adopts the same assumption throughout. Consequently, the joint recovery problem can be reduced to the estimation of the filter alone. The work in~\cite{cosse2017note} reformulates sparse blind deconvolution as the problem of finding the sparsest vector in an affine subspace, where a linear structure is exploited to resolve the parameter ambiguities, thereby enabling the use of convex optimization techniques with theoretical recovery guarantees. Although the success of this approach relies on the restrictive Bernoulli model, in which the nonzero entries take values in $\pm1$, it provides an important insight: appropriately designed constraints can effectively alleviate parameter ambiguities. Subsequent work extended this framework to the Bernoulli--Gaussian model, where the Bernoulli distribution determines the support locations and the Gaussian distribution models the nonzero values~\cite{wang2016blind}. However, the resulting guarantees require the filter to be sufficiently close to a spiky signal. The work in~\cite{8762219} further relaxes these assumptions by imposing a unit-sphere constraint to mitigate the parameter ambiguities, employing a negative $\ell_4$-based objective to promote sparsity, and introducing preconditioning techniques to improve the optimization landscape. More recent studies have adopted $\ell_1$-based objective functions to further enhance sparsity promotion, including approaches based on the Huber function~\cite{qu2020exact} and the log-cosh function~\cite{shi2021manifold}. Our work builds upon these developments by employing a joint-sparsity-promoting objective function while retaining the unit-sphere constraint and preconditioning techniques. The key distinction is that the proposed objective is based on the $\ell_{2,1}$ norm (see details in the main context), which explicitly exploits the joint sparsity shared across multiple channels.

Existing theoretical analyses of sparse blind deconvolution primarily focus on characterizing the sample complexity required for successful recovery. The work in~\cite{li2018global} establishes recovery guarantees using a negative $\ell_4$-based objective function together with preconditioning techniques. Subsequent studies demonstrate that these sample complexity requirements can be improved through $\ell_1$-based objective functions~\cite{qu2020exact,shi2021manifold}. These results generally show that, when the sample size is sufficiently large relative to the filter length, the optimization landscape contains a substantial region in which the gradient points toward the target solution, and iterates initialized within this region remain there and converge to a point sufficiently close to the target solution. In contrast, the joint-sparsity-promoting objective function considered in this paper gives rise to a significantly more intricate optimization landscape containing numerous saddle points. Consequently, poor solutions cannot be excluded solely on the basis of first-order stationarity conditions. Similar challenges arise in a variety of nonconvex optimization problems, including low-rank matrix recovery~\cite{ge2017no}, phase retrieval~\cite{cai2023provable}, and related settings~\cite{carmon2018accelerated,chen2023local}. In these problems, a central theoretical objective is to establish the absence of spurious local minima, namely, to show that every local minimum is globally optimal. Our theoretical analysis follows a similar line of reasoning. Specifically, we prove that sufficiently accurate approximate second-order stationary points lie sufficiently close to the target solution. This result implies that undesirable first-order stationary points can be ruled out through second-order optimality conditions. Furthermore, our analysis reveals that the sample complexity required for successful recovery depends on both the filter length and the sparsity level, while exhibiting a substantially weaker dependence on the filter length than existing results~\cite{qu2020exact,shi2021manifold}. This finding further highlights the advantages of exploiting joint sparsity.

\subsection{Notations and Organization}
The notations used throughout this paper are summarized as follows. Boldface letters denote vectors and matrices. The sets of real and complex numbers are denoted by $\mathbb{R}$ and $\mathbb{C}$, respectively. For a vector $\m{x}$, its transpose, conjugate transpose, complex conjugate, $\ell_1$ norm, $\ell_2$ norm, $\ell_4$ norm, and $\ell_\infty$ norm are denoted as $\m{x}^\top$, $\m{x}^H$, $\overline{\m{x}}$, $\left\|\m{x}\right\|_1$, $\left\|\m{x}\right\|_2$, $\left\|\m{x}\right\|_4$, and $\left\|\m{x}\right\|_{\infty}$, respectively. For a matrix $\m{X}$, the notations $\m{X}^\top$, $\m{X}^H$, $\overline{\m{X}}$, $\m{X}^{-1}$, and $\sigma_{\text{min}}\left(\m{X}\right)$ denote its transpose, conjugate transpose, complex conjugate, inverse, and smallest singular value, respectively. For a symmetric matrix $\m{X}$, $\lambda_{\text{min}}\left(\m{X}\right)$ denotes its smallest eigenvalue. The $j$-th entry of the vector $\m{x}$ is denoted by $x_j$, and the $\left(i_1,i_2\right)$-th entry of the matrix $\m{X}$ is denoted by $X_{i_1,i_2}$. The $i$-th row and column of a matrix $\m{X}$ are denoted as $\m{X}_{i,:}$ and $\m{X}_{:,i}$, respectively. The $\ell_{2,1}$ norm of a matrix $\m{X}\in\mathbb{C}^{N\times L}$ is defined as $\left\|\m{X}\right\|_{2,1}=\sum_{n=1}^N\left\|\m{X}_{n,:}^\top\right\|_2$. For a vector $\m{x}$ and a support set $\Omega$, $\mathcal{P}_{\Omega}\left(\m{x}\right)$ denotes the projection operator that retains the subvector of $\m{x}$ indexed by $\Omega$. For symmetric matrices $\m{X}$ and $\m{Y}$, the notation $\m{X}\succeq\m{Y}$ indicates that $\m{X}-\m{Y}$ is positive semidefinite (PSD). The notation $\left|\cdot\right|$ denotes either the absolute value of a scalar or the cardinality of a set. The identity matrix is denoted as $\m{I}$. The set $\mathbb{S}^{N-1}=\left\{\m{h}\in\mathbb{R}^N:\;\left\|\m{h}\right\|_2=1\right\}$ denotes the unit sphere in $\mathbb{R}^N$. The vector $\m{e}_m$ denotes the $m$-th canonical basis vector, whose $m$-th entry is $1$ and all other entries are $0$. $\m{J}_n$ denotes the $n\times n$ reversal matrix with ones on the anti-diagonal and zeros elsewhere. The notation $\mathbb{E}\left[\cdot\right]$ denotes expectation, and $\mathbb{P}\left[\cdot\right]$ denotes probability. The diagonal matrix with vector $\m{x}$ on its diagonal is represented as $\diag\left(\m{x}\right)$. For positive integers $a$ and $b$, the notation $a\bmod b$ denotes the remainder when $a$ is divided by $b$. For a manifold $\mathcal{M}$ and a point $\m{x}\in\mathcal{M}$, the tangent space of $\mathcal{M}$ at $\m{x}$ is denoted by $\mathrm{T}_{\m{x}}\mathcal{M}$. The matrix ${\m{F}^{\left[N\right]}}\in\mathbb{C}^{N\times N}$ denotes the $N\times N$ unitary discrete Fourier transform (DFT) matrix whose $\left(i_1,i_2\right)$-th entry is given by $\frac{1}{\sqrt{N}}\exp\left\{j2\pi\frac{\left(i_1-1\right)\left(i_2-1\right)}{N}\right\}$. For a vector $\m{x}\in\mathbb{C}^N$, its DFT spectrum is denoted by $\widehat{\m{x}}$ and defined as $\widehat{\m{x}}=\sqrt{N}{\m{F}^{\left[N\right]}}^H\m{x}$.

The remainder of this paper is organized as follows. Section~\ref{Sec:Preliminaries} introduces the fundamental concepts of sparse blind deconvolution. Section~\ref{Sec:mean-result} presents the main results, including the optimization framework for jointly sparse blind deconvolution and the characterization of the optimization landscape under statistical assumptions on the input signals. Section~\ref{Sec:Algorithm} describes the proposed algorithm and establishes its convergence and reconstruction error guarantees. Section~\ref{Sec:extension} extends the proposed framework and algorithm to the complex-valued setting. Section~\ref{Sec:simulation} presents numerical experiments, and Section~\ref{Sec:conclusion} concludes the paper.

\section{Preliminaries}
\label{Sec:Preliminaries}

This section introduces several fundamental concepts in sparse blind deconvolution, including the properties of circulant convolution and circulant matrices, parameter identifiability, and optimization techniques relevant to algorithm design. These concepts provide the theoretical foundation for the developments presented in the subsequent sections.

\subsection{Properties of Circulant Convolution}
This subsection summarizes several fundamental properties of circulant convolution.
\begin{lem}
\label{lemma-circulant convolution}
For vectors $\m{u}$, $\m{v}$, and $\m{w}$ of length $N$, the following results hold true:
\begin{enumerate}[label=\arabic*)]
\item DFT diagonalization~\cite[Theorem 3.2.2]{davis1979circulant}:
\begin{equation}
\label{DFT diagonalization}
\begin{aligned}
\mathcal{C}\left(\m{u}\right)=\m{F}^{\left[N\right]}\diag\left(\widehat{\m{u}}\right){\m{F}^{\left[N\right]}}^H.
\end{aligned}
\end{equation}
\item Commutativity~\cite[Theorem 7]{gray2006toeplitz}:
\begin{equation}
\label{commutativity-2}
\begin{aligned}
\mathcal{C}\left(\m{u}\right)\mathcal{C}\left(\m{v}\right)=&\mathcal{C}\left(\m{v}\right)\mathcal{C}\left(\m{u}\right).
\end{aligned}
\end{equation}
\item Associativity~\cite[Appendix A.3]{bamieh2018discovering}:
\begin{equation}
\label{equ-associativity}
\begin{aligned}
\m{u}\circledast\m{v}\circledast\m{w}=\m{u}\circledast\left(\m{v}\circledast\m{w}\right).
\end{aligned}
\end{equation}
\item Inverse property~\cite[(2)]{shen2011determinants}: if a circulant matrix is invertible, then its inverse is also circulant.
\end{enumerate}
\end{lem}

We also introduce a useful matrix representation of circulant convolution:
\begin{equation}
\label{circulant-convolution-matrix}
\begin{aligned}
\m{u}\circledast\m{v}=\mathcal{C}\left(\m{u}\right)\m{v}=\begin{bmatrix}\m{u}^\top\m{\Gamma}_1\m{v}&\cdots&\m{u}^\top\m{\Gamma}_N\m{v}\end{bmatrix}^\top,
\end{aligned}
\end{equation}
where $\left\{\m{\Gamma}_n=\begin{bmatrix}\m{J}_n\\&\m{J}_{N-n}\end{bmatrix}\right\}_{n=1}^N$ is a collection of symmetric matrices. We next provide the properties of the special circulant matrix $\mathcal{C}\left(\m{e}_m\right)$ for $m=1,\cdots,N$.
\begin{lem}
\label{lem-transpose}
The matrix $\mathcal{C}\left(\m{e}_m\right)$ is orthogonal and satisfies
\begin{equation}
\label{transpose}
\begin{aligned}
\mathcal{C}\left(\m{e}_m\right)^\top=\mathcal{C}\left(\m{e}_{\left(\left(N+1-m\right)\bmod N\right)+1}\right).
\end{aligned}
\end{equation}
\end{lem}
\begin{proof}
Since
\begin{equation}
\begin{aligned}
\left\|\m{u}\right\|_2=\left\|\m{u}\circledast\m{e}_m\right\|_2=\left\|\mathcal{C}\left(\m{e}_m\right)\m{u}\right\|_2
\end{aligned}
\end{equation}
the matrix $\mathcal{C}\left(\m{e}_m\right)$ is orthogonal. Moreover, it is straightforward to verify that
\begin{equation}
\begin{aligned}
\m{e}_1\circledast\m{e}_1=\m{e}_1,
\end{aligned}
\end{equation}
and
\begin{equation}
\begin{aligned}
\m{e}_m\circledast\m{e}_{\left(\left(N+1-m\right)\bmod N\right)+1}=\m{e}_1.
\end{aligned}
\end{equation}
Therefore,
\begin{equation}
\begin{aligned}
\mathcal{C}\left(\m{e}_{\left(\left(N+1-m\right)\bmod N\right)+1}\right)
\mathcal{C}\left(\m{e}_m\right)=\mathcal{C}\left(\m{e}_1\right)=\m{I},
\end{aligned}
\end{equation}
which implies \eqref{transpose} and completes the proof.
\end{proof}

Lemma~\ref{lem-transpose} shows that the filter $\m{e}_{\left(\left(N+1-m\right)\bmod N\right)+1}$ serves as an inverse representation of the filter $\m{e}_m$.

Based on these properties, the remainder of this section presents several fundamental results and algorithmic frameworks for sparse blind deconvolution.

\subsection{Parameter Identifiability for Blind Deconvolution}
The blind deconvolution problem in \eqref{circulant-convolution} suffers from inherent parameter ambiguities induced by the circulant convolution model, including global scaling ambiguity and circulant shift ambiguity between the filter $\m{g}$ and the input signals $\left\{\m{x}_{l}\right\}_{l=1}^L$. Consequently, recovery of $\m{g}$ and $\left\{\m{x}_{l}\right\}_{l=1}^L$ can only be up to a scaling factor and a circulant shift. Specifically, the filter is considered to be accurately recovered if there exists an estimate $\widetilde{\m{g}}$ such that
\begin{equation}
\begin{aligned}
\widetilde{\m{g}}=s\m{g}\circledast\m{e}_m,
\end{aligned}
\end{equation}
or equivalently,
\begin{equation}
\label{equ-filter}
\begin{aligned}
\mathcal{C}\left(\widetilde{\m{g}}\right)=s\mathcal{C}\left(\m{g}\right)\mathcal{C}\left(\m{e}_m\right),
\end{aligned}
\end{equation}
where $s\neq0$ is an unknown scaling factor and $\m{e}_m$ corresponds to a circulant shift. Under this condition, it follows from \eqref{circulant-convolution}, \eqref{equ-filter}, and Lemma~\ref{lem-transpose} that the corresponding signal estimates satisfy, for $l=1,\cdots,L$,
\begin{equation}
\label{signal-est}
\begin{aligned}
\widetilde{\m{x}_l}=&\mathcal{C}\left(\widetilde{\m{g}}\right)^{-1}\m{y}_l\\
=&\left[s\mathcal{C}\left(\m{g}\right)\mathcal{C}\left(\m{e}_m\right)\right]^{-1}\mathcal{C}\left(\m{g}\right)\m{x}_l\\
=&s^{-1}\mathcal{C}\left(\m{e}_m\right)^\top\m{x}_l\\
=&s^{-1}\m{e}_{\left(\left(N+1-m\right)\bmod N\right)+1}\circledast\m{x}_l.
\end{aligned}
\end{equation}

\subsection{Multichannel Sparse Blind Deconvolution: Reparameterization and Preconditioning}
The main challenges in multichannel sparse blind deconvolution arise from the nonconvexity induced by the bilinear observation model and the unfavorable optimization landscape caused by an unknown ill-conditioned filter $\m{g}$. Without exploiting joint sparsity, existing approaches typically address these challenges through reparameterization to eliminate the bilinearity and preconditioning to improve the optimization landscape~\cite{8762219,li2018global,qu2020exact,shi2021manifold}. Specifically, these approaches reparameterize the input signals in terms of the filter and the observations, thereby reducing the joint recovery of the filter and input signals to the estimation of the filter alone. In particular, by the inverse property established in Lemma~\ref{lemma-circulant convolution}, there exists a vector $\m{h}\in\mathbb{R}^N$ such that
\begin{equation}
\begin{aligned}
\mathcal{C}\left(\m{h}\right)=\mathcal{C}\left(\widetilde{\m{g}}\right)^{-1}=s^{-1}\mathcal{C}\left(\m{e}_m\right)^\top\mathcal{C}\left(\m{g}\right)^{-1}=s^{-1}\mathcal{C}\left(\m{g}\right)^{-1}\mathcal{C}\left(\m{e}_m\right)^\top,
\end{aligned}
\end{equation}
where the last equality follows from the commutativity of circulant matrices in~\eqref{commutativity-2}. Furthermore, Lemma~\ref{lem-transpose} implies that
\begin{equation}
\label{filter-inverse}
\begin{aligned}
\m{h}=s^{-1}\mathcal{C}\left(\m{g}\right)^{-1}\m{e}_{\left(\left(N+1-m\right)\bmod N\right)+1},\;m=1,\cdots,N,
\end{aligned}
\end{equation}
serves as a valid inverse representation of the filter $\m{g}$. Consequently, once $\m{h}$ is estimated, the input signals can be recovered as
\begin{equation}
\begin{aligned}
\widetilde{\m{x}_l}=\mathcal{C}\left(\m{h}\right)\m{y}_l=\m{h}\circledast\m{y}_l,\;l=1,\cdots,L.
\end{aligned}
\end{equation}
To eliminate the scaling ambiguity and exclude the trivial solution $\m{h}=\mathbf{0}$, $\m{h}$ is constrained to lie on the unit sphere, leading to the optimization problem
\begin{equation}
\label{opt-MSBD}
\begin{aligned}
\min_{\m{h}\in\mathbb{S}^{N-1}}\;f\left(\begin{bmatrix}\mathcal{C}\left(\m{h}\right)\m{y}_1&\mathcal{C}\left(\m{h}\right)\m{y}_2&\cdots&\mathcal{C}\left(\m{h}\right)\m{y}_L\end{bmatrix}\right)=f\left(\begin{bmatrix}\mathcal{C}\left(\m{y}_1\right)\m{h}&\mathcal{C}\left(\m{y}_2\right)\m{h}&\cdots&\mathcal{C}\left(\m{y}_L\right)\m{h}\end{bmatrix}\right),
\end{aligned}
\end{equation}
where $f$ denotes a sparsity-promoting objective function, typically constructed as a sum of component-wise functions. Existing formulations can be broadly categorized into negative $\ell_4$-based objective $f_{-\ell_4}\left(\m{X}\right)=-\sum_{n=1}^N\sum_{l=1}^L\left|\m{X}_{n,l}\right|^4$ and $\ell_1$-based objective $f_{\ell_1}\left(\m{X}\right)=\sum_{n=1}^N\sum_{l=1}^L\left|\m{X}_{n,l}\right|$. Compared with negative $\ell_4$-based objectives, $\ell_1$-based objective functions typically exhibit a more favorable optimization landscape, as the landscape around the target solution is sharper, thereby enabling probability concentration under lower sample complexity requirements~\cite{qu2020exact,shi2021manifold}. However, when the filter $\m{g}$ is ill-conditioned, i.e., $\mathcal{C}\left(\m{g}\right)$ has a large condition number, the optimization landscape deteriorates, and the optimal solution no longer corresponds to the target. To address this issue, preconditioning is employed to improve the conditioning of the optimization problem, yielding
\begin{equation}
\label{opt-MSBD-manifold}
\begin{aligned}
\min_{\m{h}\in\mathbb{S}^{N-1}}\;f\left(\begin{bmatrix}\mathcal{C}\left(\m{y}_1\right)\m{R}\m{h}&\mathcal{C}\left(\m{y}_2\right)\m{R}\m{h}&\cdots&\mathcal{C}\left(\m{y}_L\right)\m{R}\m{h}\end{bmatrix}\right),
\end{aligned}
\end{equation}
where $\m{R}$ is a preconditioning matrix constructed from the observations $\left\{\m{y}_{l}\right\}_{l=1}^L$. Existing component-wise sparsity-promoting objective functions typically employ the preconditioning matrix~\cite{li2018global,qu2020exact,shi2021manifold}:
\begin{equation}
\label{preconditioning-matrix}
\begin{aligned}
\m{R}=\left[\frac{1}{L}\sum_{l=1}^L\mathcal{C}\left(\m{y}_l\right)^\top\mathcal{C}\left(\m{y}_l\right)\right]^{-\frac{1}{2}}.
\end{aligned}
\end{equation}
Substituting different sparsity-promoting objective functions into \eqref{opt-MSBD-manifold} yields different optimization problems. 
In practice, because the $\ell_1$ norm is nonsmooth, it is typically replaced by smooth approximations, such as the Huber function~\cite{qu2020exact} and the log-cosh function~\cite{shi2021manifold}. 

Despite their effectiveness, these formulations treat each signal independently and therefore cannot exploit the joint sparsity shared across multiple channels. Developing an optimization framework that effectively incorporates this structural prior is the primary focus of this paper. The next section presents the proposed formulation and establishes the corresponding theoretical guarantees.

\section{Jointly Sparse Blind Deconvolution}
\label{Sec:mean-result}

In this section, we propose an optimization framework that exploits the joint sparsity of the input signals to achieve simultaneous recovery of both the filter and the signals. We further investigate the geometric properties of the resulting optimization problem to facilitate algorithm design. Our analysis is based on the following assumptions: 1) the nonzero components $\mathcal{P}_{\Omega}\left(\m{x}_l\right)$, for $l=1,\cdots,L$, are i.i.d. Gaussian random vectors with zero mean and covariance matrix $\m{I}$; and 2) $\mathcal{C}\left(\m{g}\right)$ is invertible with condition number $\kappa$. We begin with a detailed formulation of the optimization problem in the form of \eqref{opt-MSBD-manifold}.

\subsection{Optimization Problem}
To clarify the roles of the sparsity-promoting objective function and the preconditioning technique, we reformulate \eqref{opt-MSBD-manifold} by introducing an auxiliary optimization variable $\m{Z}$, whose columns correspond to the estimated input signals across different channels:
\begin{equation}
\label{opt-MSBD-Z}
\begin{aligned}
\min_{\m{h}\in\mathbb{S}^{N-1},\m{Z}\in\mathbb{R}^{N\times L}}\;f\left(\m{Z}\right),\;\st\m{Z}_{:,l}=\left(\m{R}\m{h}\right)\circledast\m{y}_l,\;l=1,\cdots,L.
\end{aligned}
\end{equation}
A key component of this formulation is the design of a joint-sparsity-promoting objective function. A natural choice for promoting joint sparsity is the widely used mixed $\ell_{2,1}$ norm. However, its nonsmoothness poses significant challenges for both algorithm design and theoretical analysis. To facilitate the subsequent development of efficient algorithms and rigorous performance guarantees, we instead adopt a smooth approximation of the $\ell_{2,1}$ norm, leading to the optimization problem
\begin{equation}
\label{ini-opt-0}
\begin{aligned}
\min_{\m{h}\in\mathbb{S}^{N-1},\m{Z}\in\mathbb{R}^{N\times L}}\;&\sum_{n=1}^N\left(\frac{1}{L}\left\|\m{Z}_{n,:}\right\|_2^2+\varepsilon\right)^{\frac{1}{2}},\;\st\m{Z}_{:,l}=\left(\m{R}\m{h}\right)\circledast\m{y}_l,\;l=1,\cdots,L,
\end{aligned}
\end{equation}
where $\varepsilon>0$ is a smoothing parameter. From \eqref{circulant-convolution-matrix}, the $n$-th entry of $\left(\m{R}\m{h}\right)\circledast\m{y}_l$ can be expressed as $\m{y}_l^\top\m{\Gamma}_n\m{R}\m{h}$ for every $n=1,\cdots,N$ and $l=1,\cdots,L$. Consequently, \eqref{ini-opt-0} is equivalent to
\begin{equation}
\label{ini-opt}
\begin{aligned}
\min_{\m{h}\in\mathbb{S}^{N-1},\m{Z}\in\mathbb{R}^{N\times L}}\;&\sum_{n=1}^N\left(\frac{1}{L}\left\|\m{Z}_{n,:}\right\|_2^2+\varepsilon\right)^{\frac{1}{2}},\;\st{Z}_{n,l}=\m{y}_l^\top\m{\Gamma}_n\m{R}\m{h},\;n=1,\dots,N,\;l=1,\cdots,L.
\end{aligned}
\end{equation}
Eliminating the auxiliary variable $\m{Z}$ yields the following optimization problem on the sphere:
\begin{equation}
\label{opt-ini-general-0}
\begin{aligned}
\min_{\m{h}\in\mathbb{S}^{N-1}}\;\sum_{n=1}^N\left(\frac{1}{L}\sum_{l=1}^L\left|\m{y}_l^\top\m{\Gamma}_n\m{R}\m{h}\right|^2+\varepsilon\right)^{\frac{1}{2}}.
\end{aligned}
\end{equation}
We still employ the same preconditioning matrix $\m{R}$ as in~\eqref{preconditioning-matrix} and will justify this choice in Section~\ref{subsec: general}. Once an optimal solution $\widetilde{\m{h}}$ is obtained, an estimate of the inverse representation of the filter $\m{g}$ is given by
\begin{equation}
\begin{aligned}
\widetilde{\m{g}}_{\text{inv}}=\m{R}\widetilde{\m{h}}.
\end{aligned}
\end{equation}
Accordingly, an estimate of the filter $\m{g}$ is given by
\begin{equation}
\begin{aligned}
\widetilde{\m{g}}=\mathcal{C}\left(\widetilde{\m{g}}_{\text{inv}}\right)^{-1}\m{e}_1=\mathcal{C}\left(\m{R}\widetilde{\m{h}}\right)^{-1}\m{e}_1.
\end{aligned}
\end{equation}

The remainder of this section is devoted to analyzing the geometric properties of the optimization problem in \eqref{opt-ini-general-0} under the statistical assumptions imposed on the input signals. The primary challenge stems from the fact that the shared support structure plays a dual role. On the one hand, it provides additional information that can potentially improve recovery performance and reduce the sample complexity required for successful recovery. On the other hand, it introduces strong dependencies across channels, fundamentally altering the geometry of the resulting optimization problem. In particular, the joint-sparsity-promoting objective function is no longer separable across channels, thereby preventing the direct application of existing theoretical frameworks developed for sparse blind deconvolution. To address these challenges, we develop a new geometric analysis tailored to the proposed optimization problem. Our objective is to understand how the additional structural information introduced by joint sparsity affects the optimization landscape and, consequently, the recovery guarantees and sample complexity of blind deconvolution. The analysis proceeds in two stages. First, we consider a special case in which the filter $\m{g}$ is a unit impulse and show that the optimization problem \eqref{opt-ini-general-0} is simplified in this case and admits recovery of both the filter and the input signals from its approximate second-order stationary points. Second, we study the general setting in which the filter $\m{g}$ is an arbitrary invertible filter. We show that, with the preconditioning matrix defined in \eqref{preconditioning-matrix}, problem~\eqref{opt-ini-general-0} approaches the idealized special-case formulation as the sample size increases.

\subsection{Geometry in the Special Case}
\label{subsec:special}
We now consider the special case in which the unknown filter $\m{g}$ is the unit impulse $\m{e}_1$, namely,
\begin{equation}
\begin{aligned}
\mathcal{C}\left(\m{g}\right)=\m{I}.
\end{aligned}
\end{equation}
In this setting, the observations $\left\{\m{y}_l\right\}_{l=1}^L$ coincide with the input signals $\left\{\m{x}_l\right\}_{l=1}^L$, and every canonical basis vector $\m{e}_m$, $m=1,\cdots,N$, serves as a valid inverse representation of the filter $\m{g}$. In this case, preconditioning is not required, and the optimization problem in \eqref{ini-opt} reduces to
\begin{equation}
\label{opt-special-0}
\begin{aligned}
\min_{\m{h}\in\mathbb{S}^{N-1},\m{Z}\in\mathbb{R}^{N\times L}}\;&\sum_{n=1}^N\left(\frac{1}{L}\left\|\m{Z}_{n,:}\right\|_2^2+\varepsilon\right)^{\frac{1}{2}},\;\st{Z}_{n,l}=\m{h}^\top\m{\Gamma}_n\m{x}_l,\;n=1,\dots,N,l=1,\cdots,L,
\end{aligned}
\end{equation}
or equivalently,
\begin{equation}
\label{opt-special}
\begin{aligned}
\min_{\m{h}\in\mathbb{S}^{N-1}}\;\psi_\varepsilon^0\left(\m{h}\right)=&\sum_{n=1}^N\left(\frac{1}{L}\sum_{l=1}^L\left|\m{x}_l^\top\m{\Gamma}_n\m{h}\right|^2+\varepsilon\right)^{\frac{1}{2}}.
\end{aligned}
\end{equation}
The unit sphere becomes a Riemannian manifold when equipped with the Euclidean-induced Riemannian metric
\begin{equation}
\begin{aligned}
g_{\m{h}}\left(\m{x},\m{y}\right)=\m{x}^\top\m{y},
\end{aligned}
\end{equation}
where $\m{x}$ and $\m{y}$ are tangent vectors in the tangent space at $\m{h}\in\mathbb{S}^{N-1}$. Therefore, \eqref{opt-special} can be viewed as an optimization problem on a Riemannian manifold. Define
\begin{equation}
\begin{aligned}
\m{E}_{\Omega}=\frac{1}{L}\sum_{l=1}^L\m{x}_l\m{x}_l^\top-\m{\Sigma}_{\Omega},
\end{aligned}
\end{equation}
where $\m{\Sigma}_\Omega=\mathbb{E}\left[\m{x}_l\m{x}_l^\top\right]\succeq\m{0}$ is a diagonal matrix whose diagonal contains $0/1$ entries indicating the support set $\Omega$. Then, the objective function in \eqref{opt-special} can be expressed as
\begin{equation}
\label{obj-0-re}
\begin{aligned}
\psi_\varepsilon^0\left(\m{h}\right)=&\sum_{n=1}^N\left[\m{h}^\top\m{\Gamma}_n^\top\left(\frac{1}{L}\sum_{l=1}^L\m{x}_l\m{x}_l^\top\right)\m{\Gamma}_n\m{h}+\varepsilon\right]^{\frac{1}{2}}=\sum_{n=1}^N\left[\m{h}^\top\m{\Gamma}_n^\top\left(\m{\Sigma}_{\Omega}+\m{E}_{\Omega}\right)\m{\Gamma}_n\m{h}+\varepsilon\right]^{\frac{1}{2}}.
\end{aligned}
\end{equation}
The asymptotic optimization problem corresponding to $\m{E}_{\Omega}=\m{0}$ is given by
\begin{equation}
\label{opt-special-asy}
\begin{aligned}
\min_{\m{h}\in\mathbb{S}^{N-1}}\;\sum_{n=1}^N\left(\m{h}^\top\m{\Gamma}_n^\top\m{\Sigma}_{\Omega}\m{\Gamma}_n\m{h}+\varepsilon\right)^{\frac{1}{2}}.
\end{aligned}
\end{equation}
The following lemma shows that the desired solution can be exactly recovered in the asymptotic setting. To state the result, we first introduce the notion of a non-uniform support set. Specifically, the common support set $\Omega$ is said to be non-uniform if it does not exhibit any nontrivial cyclic symmetry; that is, there exists no nonzero shift $s\in\left\{1,2,\cdots,N-1\right\}$ such that
\begin{equation}
\begin{aligned}
\left(\left(\Omega+s-1\right)\bmod N\right)+1=\Omega.
\end{aligned}
\end{equation}
The following property establishes that the non-uniformity of the support set is both necessary and sufficient for exact recovery in the asymptotic setting.
\begin{prop}
\label{lem-support}
The global minimizers of problem~\eqref{opt-special-asy} are signed canonical basis vectors if and only if $\Omega$ is non-uniform.
\end{prop}
\begin{proof}
($\Leftarrow$) Since $\left\|\m{h}\right\|_2=1$ and $\m{\Gamma}_n^\top\m{\Sigma}_{\Omega}\m{\Gamma}_n$ is diagonal for each $n=1,\cdots,N$, it follows that
\begin{equation}
\label{inequ-obj}
\begin{aligned}
\sum_{n=1}^N\left(\m{h}^\top\m{\Gamma}_n^\top\m{\Sigma}_{\Omega}\m{\Gamma}_n\m{h}+\varepsilon\right)^{\frac{1}{2}}=&\sum_{n=1}^N\left[\sum_{m=1}^Nh_m^2\left(\m{e}_m^\top\m{\Gamma}_n^\top\m{\Sigma}_{\Omega}\m{\Gamma}_n\m{e}_m+\varepsilon\right)\right]^{\frac{1}{2}}\\
\geq&\sum_{n=1}^N\sum_{m=1}^Nh_m^2\left(\m{e}_m^\top\m{\Gamma}_n^\top\m{\Sigma}_{\Omega}\m{\Gamma}_n\m{e}_m+\varepsilon\right)^{\frac{1}{2}}\\
=&\sum_{m=1}^Nh_m^2\sum_{n=1}^N\left(\m{e}_m^\top\m{\Gamma}_n^\top\m{\Sigma}_{\Omega}\m{\Gamma}_n\m{e}_m+\varepsilon\right)^{\frac{1}{2}}\\
=&\sum_{m=1}^Nh_m^2\left[K\sqrt{1+\varepsilon}+\left(N-K\right)\sqrt{\varepsilon}\right]\\
=&K\sqrt{1+\varepsilon}+\left(N-K\right)\sqrt{\varepsilon}\\
=&\sum_{n=1}^N\left(\m{e}_m^\top\m{\Gamma}_n^\top\m{\Sigma}_{\Omega}\m{\Gamma}_n\m{e}_m+\varepsilon\right)^{\frac{1}{2}},
\end{aligned}
\end{equation}
where the inequality follows from the strict concavity of the square-root function. Moreover, equality in \eqref{inequ-obj} holds if and only if, for every $n=1,\cdots,N$, the quantity $\m{e}_m^\top\m{\Gamma}_n^\top\m{\Sigma}_{\Omega}\m{\Gamma}_n\m{e}_m$ takes the same value for all indices $m$ satisfying $h_m\neq0$. Suppose that $\m{h}\neq\pm\m{e}_m$ for every $m=1,\ldots,N$. Then there exist distinct indices $m$ and $\widetilde{m}$ such that $h_m\neq0$ and $h_{\widetilde{m}}\neq0$. If $\Omega$ is non-uniform, there exists at least one $n\in\{1,\cdots,N\}$ such that
\begin{equation}
\begin{aligned}
\m{e}_{m}^\top\m{\Gamma}_n^\top\m{\Sigma}_{\Omega}\m{\Gamma}_n\m{e}_{m}\neq\m{e}_{\widetilde{m}}^\top\m{\Gamma}_n^\top\m{\Sigma}_{\Omega}\m{\Gamma}_n\m{e}_{\widetilde{m}}.
\end{aligned}
\end{equation}
This contradicts the equality condition above. Therefore, equality in \eqref{inequ-obj} holds only if $\m{h}=\pm\m{e}_m$ for some $m=1,\ldots,N$. Consequently, every global minimizer of problem~\eqref{opt-special-asy} is a signed canonical basis vector.

$(\Rightarrow)$ Suppose that $\Omega$ is uniform. Then there exists a nonzero shift $s\in\left\{1,2,\cdots,N-1\right\}$ such that $\left((\Omega-1+s)\bmod N\right)+1=\Omega$. For any $m\in\left\{1,2,\cdots,N\right\}$, define its cyclically shifted index as $m_s=\left((m-1+s)\bmod N\right)+1$. We immediately have that $m_s\neq m$. Moreover, the cyclic invariance of $\Omega$ implies that $\m{e}_m^\top\m{\Gamma}_n^\top
\m{\Sigma}_{\Omega}\m{\Gamma}_n\m{e}_m=\m{e}_{m_s}^\top\m{\Gamma}_n^\top
\m{\Sigma}_{\Omega}\m{\Gamma}_n\m{e}_{m_s}$ for all $n=1,\cdots,N$. Consequently, the equality condition in \eqref{inequ-obj} is satisfied by every vector of the form $\m{h}=c_1\m{e}_m+c_2\m{e}_{m_s}$ with $c_1^2+c_2^2=1$. Hence, every such vector attains the global minimum of \eqref{opt-special-asy}, contradicting the assumption that all global minimizers are signed canonical basis vectors. Therefore, $\Omega$ must be non-uniform.
\end{proof}

Proposition~\ref{lem-support} shows that, under the non-uniform support assumption, problem \eqref{opt-special-asy} admits no additional global minimizers that are indistinguishable from the target solution. This assumption is mild and is satisfied in most practical scenarios. In contrast, a uniform support pattern imposes stringent algebraic constraints on both the filter length $N$ and the support set $\Omega$, making it highly unlikely to occur in practice. Moreover, it is shown in~\cite{wainwright2019high} that $\m{E}_{\Omega}$ converges to $\m{0}$ as the sample size $L$ increases. The following lemma provides a non-asymptotic tail bound for the sample covariance matrix of a Gaussian random vector.
\begin{lem}\cite[Example 6.2]{wainwright2019high}
\label{lem-tail}
For any $\delta>0$, the following bound holds with probability at least $1-2\exp\left\{-\frac{L\delta^2}{2}\right\}$:
\begin{equation}
\begin{aligned}
\left\|\m{E}_\Omega\right\|_2\leq2\left(\sqrt{\frac{K}{L}}+\delta\right)+\left(\sqrt{\frac{K}{L}}+\delta\right)^2.
\end{aligned}
\end{equation}
\end{lem}

Consequently, when $L$ is sufficiently large, problem~\eqref{opt-special} can be viewed as a perturbation of its asymptotic counterpart~\eqref{opt-special-asy}. We next establish rigorously that solving \eqref{opt-special} recovers the target solution by characterizing the geometric landscape of the objective function over the unit sphere under the non-uniform support assumption. Specifically, we investigate the first- and second-order geometric properties of the objective function in \eqref{opt-special} and show that, when the sample size $L$ is sufficiently large, every approximate second-order stationary point lies in a neighborhood of the target solution. 

According to~\cite[Proposition 3.53]{boumal2023introduction}, the Riemannian gradient of $\psi_\varepsilon^0$, denoted by $\text{grad}\;\psi_\varepsilon^0\left(\m{h}\right)$, is obtained by projecting the Euclidean gradient $\nabla\;\psi_\varepsilon^0\left(\m{h}\right)$ onto the tangent space:
\begin{equation}
\label{grad}
\begin{aligned}
\text{grad}\;\psi_\varepsilon^0\left(\m{h}\right)=&\left(\m{I}-\m{h}\m{h}^\top\right)\nabla\;\psi_\varepsilon^0\left(\m{h}\right)\\
=&\sum_{n=1}^N\left[\m{h}^\top\m{\Gamma}_n^\top\left(\m{\Sigma}_{\Omega}+\m{E}_{\Omega}\right)\m{\Gamma}_n\m{h}+\varepsilon\right]^{-\frac{1}{2}}\left(\m{I}-\m{h}\m{h}^\top\right)\m{\Gamma}_n^\top\left(\m{\Sigma}_{\Omega}+\m{E}_{\Omega}\right)\m{\Gamma}_n\m{h}.
\end{aligned}
\end{equation}
Consider an approximate first-order stationary point $\m{h}\in\mathbb{S}^{N-1}$ satisfying
\begin{equation}
\label{prop-first-order stationary}
\begin{aligned}
\sqrt{g_{\m{h}}\left(\text{grad}\;\psi_\varepsilon^0\left(\m{h}\right),\text{grad}\;\psi_\varepsilon^0\left(\m{h}\right)\right)}=\left\|\text{grad}\;\psi_\varepsilon^0\left(\m{h}\right)\right\|_2\leq\xi.
\end{aligned}
\end{equation}
To determine whether $\m{h}$ is also an approximate second-order stationary point, we further investigate the second-order geometry of $\psi_\varepsilon^0$ at $\m{h}$. Specifically, we examine whether the Riemannian Hessian of $\psi_\varepsilon^0$ on the unit sphere, denoted by $\text{Hess}\;\psi_\varepsilon^0\left(\m{h}\right)$, is approximately PSD; that is, whether $\lambda_{\text{min}}\left(\text{Hess}\;\psi_\varepsilon^0\left(\m{h}\right)\right)\geq-\xi$ holds for a sufficiently small positive constant $\xi$. Observe that if there exists a tangent vector $\m{d}\in\text{T}_{\m{h}}\mathbb{S}^{N-1}$ such that $\frac{\m{d}^\top\text{Hess}\;\psi_\varepsilon^0\left(\m{h}\right)\m{d}}{\left\|\m{d}\right\|_2^2}$ is sufficiently negative, then $\text{Hess}\;\psi_\varepsilon^0\left(\m{h}\right)$ must possess a sufficiently negative eigenvalue, implying that $\m{h}$ cannot be an approximate second-order stationary point. Motivated by this observation, we construct a tangent vector $\m{d}$ for the case $\left\|\m{h}\right\|_{\infty}\neq 1$ and show that $\frac{\m{d}^\top\text{Hess}\;\psi_\varepsilon^0\left(\m{h}\right)\m{d}}{\left\|\m{d}\right\|_2^2}$ is sufficiently negative whenever $\m{h}$ is sufficiently far from every canonical basis vector. Without loss of generality (WLOG), let $o$ denote the index of the largest-magnitude entry of $\m{h}$. We further assume that $h_o>0$. This assumption is trivial, since replacing $\m{h}$ with $-\m{h}$ yields an equivalent solution whenever $h_o<0$. Consequently,
\begin{equation}
\label{prop-max}
\begin{aligned}
h_o=\left\|\m{h}\right\|_\infty\geq\frac{1}{\sqrt{N}}.
\end{aligned}
\end{equation}
Since $\m{e}_o$ is the canonical basis vector closest to $\m{h}$, we consider the direction
\begin{equation}
\label{direction}
\begin{aligned}
\m{d}=\m{e}_o-\left(\m{h}^\top\m{e}_o\right)\m{h}=\m{e}_o-h_o\m{h}\in\text{T}_{\m{h}}\mathbb{S}^{N-1},
\end{aligned}
\end{equation}
which is collinear with the Riemannian logarithmic direction from $\m{h}$ to $\m{e}_o$. Its length is given by
\begin{equation}
\begin{aligned}
\left\|\m{d}\right\|_2=\sqrt{\m{d}^\top\m{d}}=\sqrt{1-h_o^2}\leq1.
\end{aligned}
\end{equation}
We next characterize the second-order behavior of $\psi_\varepsilon^0$ along the direction $\m{d}$ by deriving a bound for $\frac{\m{d}^\top\text{Hess}\;\psi_\varepsilon^0\left(\m{h}\right)\m{d}}{\left\|\m{d}\right\|_2^2}$ under the condition $h_o\neq 1$.

The following lemma establishes the concentration behavior of the quantity $\m{h}^\top\m{\Gamma}_n^\top\m{\Sigma}_{\Omega}\m{\Gamma}_n\m{h}$, which appears in the objective function of the asymptotic optimization problem \eqref{opt-special-asy}. This quantity serves as a key indicator of the proximity of $\m{h}$ to a canonical basis vector.
\begin{lem}
\label{lem-diff}
Let $\m{h}$ satisfy \eqref{prop-max}. If $\Omega$ is non-uniform, then
\begin{equation}
\begin{aligned}
\min_{n\in\chi_o}\m{h}^\top\m{\Gamma}_n^\top\m{\Sigma}_{\Omega}\m{\Gamma}_n\m{h}\leq\min\left\{Kh_o^2,1-\frac{1-h_o^2}{K}\right\},
\end{aligned}
\end{equation}
where
\begin{equation}
\begin{aligned}
\chi_o=\left\{n:\;\m{e}_{o}^\top\m{\Gamma}_n^\top\m{\Sigma}_{\Omega}\m{\Gamma}_n\m{e}_{o}=1\right\}=\left\{\big((i+o-2)\bmod N\big)+1:\; i\in\Omega\right\}
\end{aligned}
\end{equation}
denotes the set of indices $n$ for which the permutation operator $\m{\Gamma}_n$ maps the index $o$ into the support set $\Omega$.
\end{lem}
\begin{proof}
The proof exploits the non-uniformity of the support set, which ensures that the contributions of different components of $\m{h}$ can be distinguished through the objective function in problem~\eqref{opt-special-asy}. The detailed proof is deferred to Section~\ref{pf-lem-diff} of Supplementary Material.
\end{proof}


Lemma~\ref{lem-diff} implies that $\displaystyle\min_{n\in\chi_o}\m{h}^\top\m{\Gamma}_n^\top\m{\Sigma}_{\Omega}\m{\Gamma}_n\m{h}$ tends to increase as $h_o$ becomes larger. This quantity measures the extent to which $\m{h}$ is concentrated around its dominant coordinate. In particular, larger values indicate that $\m{h}$ is more strongly aligned with a canonical basis vector. Building upon Lemma~\ref{lem-diff}, we next derive an upper bound for $\frac{\m{d}^\top\text{Hess}\;\psi_\varepsilon^0\left(\m{h}\right)\m{d}}{\left\|\m{d}\right\|_2^2}$.
\begin{lem}
\label{lem-2st}
Let $\m{h}$ be an approximate first-order stationary point satisfying \eqref{prop-first-order stationary} and \eqref{prop-max}. Suppose that $\Omega$ is non-uniform. Then there exist positive constants $c_1$, $c_2$, and $c_3$ such that, whenever $\left\|\m{E}_\Omega\right\|_2\leq c_1K^{-2}\max^{-1}\left\{K^{2},N\right\}$, $\varepsilon\leq c_2K^{-6}N^{-2}$, and $\xi\leq c_3K^{-3}N^{-\frac{1}{2}}$, the following inequality holds for $h_o\in\left[\sqrt{\frac{1}{N}},1\right)$:
\begin{equation}
\begin{aligned}
\frac{\m{d}^\top\text{Hess}\;\psi_\varepsilon^0\left(\m{h}\right)\m{d}}{\left\|\m{d}\right\|_2^2}\leq&\left(1+c_1+c_2\right)^{-\frac{3}{2}}\frac{h_o^2}{\left\|\m{d}\right\|_2^2}\left(\min\left\{1-\frac{1-h_o^2}{K},Kh_o^2\right\}\right)^{-\frac{3}{2}}\min\left\{-\frac{1-h_o^2}{K},Kh_o^2-1\right\}\\
&+\frac{h_o^2}{\left\|\m{d}\right\|_2^2}\left(\min\left\{1-\frac{1-h_o^2}{K},Kh_o^2\right\}\right)^{-\frac{3}{2}}\left(\left\|\m{E}_\Omega\right\|_2+\varepsilon\right)\\
&+4\left\|\m{E}_\Omega\right\|_2\sum_{n\in\chi_{\m{h}}}\left\|\mathcal{P}_{\Omega}\left(\m{\Gamma}_n\m{h}\right)\right\|_2+\frac{20}{\left\|\m{d}\right\|_2^2}h_o^{-2}K\left\|\m{E}_\Omega\right\|_2+\frac{c_1+c_2+2c_2^{\frac{1}{2}}+c_3}{\left\|\m{d}\right\|_2^2}K^{-3},
\end{aligned}
\end{equation}
where
\begin{equation}
\begin{aligned}
\chi_{\m{h}}=&\left\{n:\;\m{h}^\top\m{\Gamma}_n^\top\m{\Sigma}_{\Omega}\m{\Gamma}_n\m{h}>0\right\}.
\end{aligned}
\end{equation}
\end{lem}
\begin{proof}
The main objective of the proof is to establish negative curvature of the Riemannian Hessian along the direction $\m{d}$. To this end, we decompose the Hessian into a dominant geometric component and a perturbation component. We then show that the dominant component is strictly negative, while the perturbation component can be controlled by $\|\m{E}_\Omega\|_2$, $\varepsilon$, and $\xi$. The detailed proof is provided in Section~\ref{pf-lem-2st} of Supplementary Material.
\end{proof}

Note that when the smoothing parameter $\varepsilon$, the first-order stationarity residual $\xi$, and $\left\|\m{E}_\Omega\right\|_2$ are sufficiently small, the positive constants $c_1$, $c_2$, and $c_3$ can be made arbitrarily small. By setting $c_1$, $c_2$, and $c_3$ to zero, Lemma~\ref{lem-2st} implies that the quantity $\frac{\m{d}^\top\text{Hess}\;\psi_\varepsilon^0\left(\m{h}\right)\m{d}}{\left\|\m{d}\right\|_2^2}$ is asymptotically upper bounded by $-\frac{h_o^2}{\left\|\m{d}\right\|_2^2}\left(1-\frac{1-h_o^2}{K}\right)^{-\frac{3}{2}}\frac{1-h_o^2}{K}$. This upper bound is strictly negative whenever $h_o<1$. Consequently, any $\m{h}$ that remains sufficiently far from all canonical basis vectors cannot be a second-order stationary point. Motivated by this observation, we further refine the upper bound on $\frac{\m{d}^\top\text{Hess}\;\psi_\varepsilon^0\left(\m{h}\right)\m{d}}{\left\|\m{d}\right\|_2^2}$ based on Lemma~\ref{lem-2st}. In particular, we consider three regimes corresponding to different ranges of $h_o$.


\textbf{Case 1 (Small $h_o$):}
We consider the regime $Kh_o^2\leq1-\frac{1-h_o^2}{K}$, which is equivalent to $h_o\leq\sqrt{\frac{1}{K+1}}$. For this regime, we establish the following proposition.
\begin{prop}
\label{thm-case1}
Let $\m{h}$ be an approximate first-order stationary point satisfying \eqref{prop-first-order stationary} and \eqref{prop-max}. Suppose that $\Omega$ is non-uniform and $h_o\in\left[\sqrt{\frac{1}{N}},\sqrt{\frac{1}{K+1}}\right]$. Then there exist positive constants $c_1$, $c_2$, $c_3$, $c_4$ such that, whenever $\left\|\m{E}_\Omega\right\|_2\leq c_1K^{-\frac{5}{2}}\max^{-1}\left\{N,KN^{\frac{1}{2}},K^2\right\}$, $\varepsilon\leq c_2K^{-6}N^{-2}$, and $\xi\leq c_3K^{-3}N^{-\frac{1}{2}}$, it holds that
\begin{equation}
\begin{aligned}
\frac{\m{d}^\top\text{Hess}\;\psi_\varepsilon^0\left(\m{h}\right)\m{d}}{\left\|\m{d}\right\|_2^2}\leq-c_4K^{-2}.
\end{aligned}
\end{equation}
\end{prop}
\begin{proof}
The key idea is to specialize the bound in Lemma~\ref{lem-2st} to the regime $h_o\in\left[\sqrt{\frac{1}{N}},\sqrt{\frac{1}{K+1}}\right]$ by imposing more explicit conditions on $\left\|\m{E}_\Omega\right\|_2$. This yields a concrete negative-curvature bound for the Riemannian Hessian along the direction $\m{d}$. The detailed proof is provided in Section~\ref{pf-thm-case1} of Supplementary Material.
\end{proof}

Proposition~\ref{thm-case1} implies that, when $\varepsilon$, $\xi$, and $\left\|\m{E}_\Omega\right\|_2$ are sufficiently small, no second-order stationary points exist far from any canonical basis vector.

\textbf{Case 2 (Moderate $h_o$):}
We consider the regime $\sqrt{1-\frac{1}{4K}}\geq h_o>\sqrt{\frac{1}{K+1}}$. For this regime, we establish the following result.
\begin{prop}
\label{thm-case2-1}
Let $\m{h}$ be an approximate first-order stationary point satisfying \eqref{prop-first-order stationary} and \eqref{prop-max}. Suppose that $\Omega$ is non-uniform and $h_o\in\left(\sqrt{\frac{1}{K+1}},\sqrt{1-\frac{1}{4K}}\right]$. Then there exist positive constants $c_1$, $c_2$, $c_3$, $c_4$ such that, whenever $\left\|\m{E}_\Omega\right\|_2\leq c_1K^{-\frac{5}{2}}\max^{-1}\left\{N,K^2\right\}$, $\varepsilon\leq c_2K^{-6}N^{-2}$, and $\xi\leq c_3K^{-3}N^{-\frac{1}{2}}$, it holds that
\begin{equation}
\begin{aligned}
\frac{\m{d}^\top\text{Hess}\;\psi_\varepsilon^0\left(\m{h}\right)\m{d}}{\left\|\m{d}\right\|_2^2}\leq-c_4K^{-2}.
\end{aligned}
\end{equation}
\end{prop}
\begin{proof}
The proof follows the same argument as that of Proposition~\ref{thm-case1}. The detailed proof is provided in Section~\ref{pf-thm-case2-1} of Supplementary Material.
\end{proof}

Proposition~\ref{thm-case2-1} implies that, when $\varepsilon$, $\xi$, and $\left\|\m{E}_\Omega\right\|_2$ are sufficiently small, no second-order stationary points exist if the point remains moderately far from every canonical basis vector.

\textbf{Case 3 (Large $h_o$):}
We consider the regime $\sqrt{1-\frac{1}{4K^2}}\geq h_o>\sqrt{1-\frac{1}{4K}}$. For this regime, we establish the following result.
\begin{prop}
\label{thm-case2-2}
Let $\m{h}$ be an approximate first-order stationary point satisfying \eqref{prop-first-order stationary} and \eqref{prop-max}. Suppose that $\Omega$ is non-uniform and $h_o\in\left(\sqrt{1-\frac{1}{4K}},\sqrt{1-\frac{1}{4K^2}}\right]$. Then there exist positive constants $c_1$, $c_2$, $c_3$, $c_4$ such that, whenever $\left\|\m{E}_\Omega\right\|_2\leq c_1K^{-\frac{5}{2}}\max^{-1}\left\{N,K^2\right\}$, $\varepsilon\leq c_2K^{-6}N^{-2}$, and $\xi\leq c_3K^{-3}N^{-\frac{1}{2}}$, it holds that
\begin{equation}
\begin{aligned}
\frac{\m{d}^\top\text{Hess}\;\psi_\varepsilon^0\left(\m{h}\right)\m{d}}{\left\|\m{d}\right\|_2^2}\leq-c_4K^{-3}.
\end{aligned}
\end{equation}
\end{prop}
\begin{proof}
The proof follows the same argument as that of Proposition~\ref{thm-case1}. The detailed proof is provided in Section~\ref{pf-thm-case2-2} of Supplementary Material.
\end{proof}

Proposition~\ref{thm-case2-2} implies that, when $\varepsilon$, $\xi$, and $\left\|\m{E}_\Omega\right\|_2$ are sufficiently small, no approximate second-order stationary points exist unless the point is sufficiently close to a canonical basis vector. Combining the analyses of the above three regimes yields the following result.
\begin{cor}
\label{cor-case1,2-1,2-2}
Let $\m{h}$ be an approximate first-order stationary point satisfying \eqref{prop-first-order stationary} and \eqref{prop-max}. Suppose that $\Omega$ is non-uniform and $h_o\in\left[\sqrt{\frac{1}{N}},\sqrt{1-\frac{1}{4K^2}}\right]$. Then there exist positive constants $c_1$, $c_2$, $c_3$, $c_4$ such that, whenever $\left\|\m{E}_\Omega\right\|_2\leq c_1K^{-\frac{5}{2}}\max^{-1}\left\{N,KN^{\frac{1}{2}},K^2\right\}$, $\varepsilon\leq c_2K^{-6}N^{-2}$, and $\xi\leq c_3K^{-3}N^{-\frac{1}{2}}$, it holds that
\begin{equation}
\begin{aligned}
\frac{\m{d}^\top\text{Hess}\;\psi_\varepsilon^0\left(\m{h}\right)\m{d}}{\left\|\m{d}\right\|_2^2}\leq-c_4K^{-3}.
\end{aligned}
\end{equation}
\end{cor}
\begin{proof}
The result follows directly from Propositions~\ref{thm-case1}, \ref{thm-case2-1}, and \ref{thm-case2-2}.
\end{proof}

Corollary~\ref{cor-case1,2-1,2-2} indicates that no approximate second-order stationary points exist unless the point is sufficiently close to a canonical basis vector. We next analyze the distance between an approximate first-order stationary point and its nearest canonical basis vector when the point is sufficiently close to a canonical basis vector, specifically when $h_o>\sqrt{1-\frac{1}{4K^2}}$. The following proposition characterizes this result.
\begin{prop}
\label{thm-case3}
Let $\m{h}$ be an approximate first-order stationary point satisfying \eqref{prop-first-order stationary} and \eqref{prop-max}. Suppose that $\Omega$ is non-uniform and $h_o\in\left(\sqrt{1-\frac{1}{4K^2}},1\right]$. There exist positive constants $c_1$, $c_2$, $c_3$, $c_4$ such that, whenever $\left\|\m{E}_\Omega\right\|_2\leq c_1K^{-2}$ and $\varepsilon\leq c_2K^{-2}$, it holds that
\begin{equation}
\begin{aligned}
\left\|\m{h}-\m{e}_o\right\|_2\leq&c_3K^{-1}\xi+c_4K^{-1}N\left\|\m{E}_\Omega\right\|_2.
\end{aligned}
\end{equation}
\end{prop}
\begin{proof}
The key idea is to establish a local error bound using only the approximate first-order stationarity condition when $\m{h}$ is already sufficiently close to a canonical basis vector. The detailed proof is provided in Section~\ref{pf-thm-case3} of Supplementary Material.
\end{proof}

By combining Corollary~\ref{cor-case1,2-1,2-2} and Proposition~\ref{thm-case3}, we obtain the following result.
\begin{thm}
\label{thm-2st}
Suppose that $\Omega$ is non-uniform and $\m{h}$ is an approximate first-order stationary point satisfying \eqref{prop-first-order stationary} and \eqref{prop-max}. There exist positive constants $c_1$, $c_2$, $c_3$, $c_4$, and $c_5$ such that, whenever $\left\|\m{E}_\Omega\right\|_2\leq c_1K^{-\frac{5}{2}}\max^{-1}\left\{N,KN^{\frac{1}{2}},K^2\right\}$, $\varepsilon\leq c_2K^{-6}N^{-2}$, $\xi\leq c_3K^{-3}N^{-\frac{1}{2}}$, and $\text{Hess}\;\psi_\varepsilon^0\left(\m{h}\right)\succeq -c_4K^{-3}\m{I}$, it holds that
\begin{equation}
\label{err-bound}
\begin{aligned}
\left\|\m{h}-\m{e}_o\right\|_2\leq&c_5K^{-1}\left(\xi+N\left\|\m{E}_\Omega\right\|_2\right).
\end{aligned}
\end{equation}
\end{thm}
\begin{proof}
Since $\Omega$ is non-uniform, $\m{h}$ is an approximate first-order stationary point satisfying \eqref{prop-first-order stationary} and \eqref{prop-max}, Corollary~\ref{cor-case1,2-1,2-2} implies that $h_o\in\left(\sqrt{1-\frac{1}{4K^2}},1\right]$ provided $\text{Hess}\;\psi_\varepsilon^0\left(\m{h}\right)\succeq -c_4K^{-3}\m{I}$. Applying Proposition~\ref{thm-case3} then yields the desired result.
\end{proof}

Theorem~\ref{thm-2st} establishes a recovery error bound in terms of $\left\|\m{E}_\Omega\right\|_2$. By combining this result with the concentration bound in Lemma~\ref{lem-tail}, we further characterize the dependence of the recovery error on the sample size $L$ in the following theorem.
\begin{thm}
\label{thm-err}
Suppose that $\Omega$ is non-uniform and $\m{h}$ is an approximate first-order stationary point satisfying \eqref{prop-first-order stationary} and \eqref{prop-max}. There exist positive constants $c_1$, $c_2$, $c_3$, $c_4$, and $c_5$ such that, whenever $\varepsilon\leq c_2K^{-6}N^{-2}$, $\xi\leq c_3K^{-3}N^{-\frac{1}{2}}$, and $\text{Hess}\;\psi_\varepsilon^0\left(\m{h}\right)\succeq -c_4K^{-3}\m{I}$, for
\begin{equation}
\label{impulse-cond-L}
\begin{aligned}
L\geq c_1\left(\sqrt{K}+\sqrt{\log L}\right)^2K^{5}\max\left\{N^2,K^2N,K^4\right\},
\end{aligned}
\end{equation}
it holds with probability at least $1-2L^{-8}$ that
\begin{equation}
\label{err-bound-L}
\begin{aligned}
\left\|\m{h}-\m{e}_o\right\|_2\leq&c_5K^{-1}\left[\xi+N\left(\sqrt{\frac{K}{L}}+\sqrt{\frac{\log L}{L}}\right)\right].
\end{aligned}
\end{equation}
\end{thm}
\begin{proof}
Setting $\delta=\sqrt{\frac{16\log L}{L}}$, we obtain from Lemma~\ref{lem-tail} that, with probability at least $1-2\exp\left\{-\frac{16\log L}{2}\right\}=1-2L^{-8}$,
\begin{equation}
\label{inequ-E}
\begin{aligned}
\left\|\m{E}_\Omega\right\|_2\leq2\left(\sqrt{\frac{K}{L}}+\sqrt{\frac{16\log L}{L}}\right)+\left(\sqrt{\frac{K}{L}}+\sqrt{\frac{16\log L}{L}}\right)^2.
\end{aligned}
\end{equation}
Moreover, there exists a positive constant $c$ such that, if
\begin{equation}
\begin{aligned}
L\geq 16c^{-2}\left(\sqrt{K}+4\sqrt{\log L}\right)^2\left(K^{\frac{5}{2}}\max\left\{N,KN^{\frac{1}{2}},K^2\right\}\right)^2,
\end{aligned}
\end{equation}
then
\begin{equation}
\begin{aligned}
\left\|\m{E}_\Omega\right\|_2\leq4\left(\sqrt{\frac{K}{L}}+\sqrt{\frac{16\log L}{L}}\right)\leq cK^{-\frac{5}{2}}\min\left\{N^{-1},K^{-1}N^{-\frac{1}{2}},K^{-2}\right\}.
\end{aligned}
\end{equation}
Applying Theorem~\ref{thm-2st} completes the proof.
\end{proof}

Theorem~\ref{thm-err} explicitly characterizes the recovery error attained by an approximate second-order stationary point. Note that the sample complexity requirement in \eqref{impulse-cond-L} is approximately of order $O\left(K^6N^2\right)$ when logarithmic factors are ignored and $N\geq K^2$. A related result was established in~\cite{shi2021manifold}. In contrast to the present work, \cite{shi2021manifold} does not assume joint sparsity of the input signals. Instead, each entry of $\left\{\m{x}_{l}\right\}_{l=1}^L$ is modeled as an independent Bernoulli--Gaussian random variable. Under this model, the required sample complexity is approximately of order $O\left(\theta^{-2}N^4\right)$, where $\theta$ denotes the Bernoulli parameter governing the sparsity ratio; see \cite[Theorem~2]{shi2021manifold}. Since $\theta$ can be interpreted as $\theta\approx\frac{K}{N}$, this condition becomes $O\left(K^{-2}N^6\right)$. Consequently, in the practically relevant regime where $K^2\leq N$, the sample complexity requirement established in this paper is less restrictive than that of~\cite{shi2021manifold}. Such a regime commonly arises in applications of interest. For example, in radar systems, the filter length $N$ is determined by the number of temporal measurement samples, whereas the sparsity level $K$ corresponds to the number of targets and is typically small~\cite{richards2010principles}. These results demonstrate that exploiting joint sparsity can substantially reduce the sample complexity required for successful recovery.

In the next subsection, we extend the preceding analysis to the case of general invertible filters. This setting is considerably more challenging because the target solutions no longer possess the simple structure of canonical basis vectors.

\subsection{Geometry in the General Case}
\label{subsec: general}

We now consider the general case, where the filter $\m{g}$ is arbitrary but satisfies the invertibility condition that $\mathcal{C}\left(\m{g}\right)$ is nonsingular. An effective strategy is to introduce a suitable preconditioning matrix $\m{R}$ so that solving problem~\eqref{opt-ini-general-0} becomes nearly equivalent to solving the corresponding problem in the special case \eqref{opt-special}~\cite{8762219,li2018global,qu2020exact,shi2021manifold}. Following the existing literature on sparse blind deconvolution, the preconditioning matrix is designed to compensate for the effect of the unknown filter on the input signals. Specifically, the constraints in the multichannel sparse blind deconvolution problem~\eqref{opt-MSBD-Z} can be rewritten as
\begin{equation}
\begin{aligned}
\left(\m{R}\m{h}\right)\circledast\m{y}_l=\left(\m{R}\m{h}\right)\circledast\left(\m{g}\circledast\m{x}_l\right)=\left(\m{R}\m{h}\right)\circledast\m{g}\circledast\m{x}_l=\left(\mathcal{C}\left(\m{g}\right)\m{R}\m{h}\right)\circledast\m{x}_l,\;l=1,\cdots,L,
\end{aligned}
\end{equation}
where the second equality follows from the associativity of circulant convolution established in~\eqref{equ-associativity}. Therefore, choosing the preconditioning matrix as
\begin{equation}
\begin{aligned}
\widetilde{\m{R}}=\left[\mathcal{C}\left(\m{g}\right)^\top\mathcal{C}\left(\m{g}\right)\right]^{-\frac{1}{2}}
\end{aligned}
\end{equation}
and defining the orthogonal matrix
\begin{equation}
\begin{aligned}
\m{U}=\mathcal{C}\left(\m{g}\right)\left[\mathcal{C}\left(\m{g}\right)^\top\mathcal{C}\left(\m{g}\right)\right]^{-\frac{1}{2}}
\end{aligned}
\end{equation}
yield
\begin{equation}
\label{equ-compensate}
\begin{aligned}
\mathcal{C}\left(\m{g}\right)\widetilde{\m{R}}\m{U}^\top=\m{I}.
\end{aligned}
\end{equation}
Consequently, for $l=1,\cdots,L$,
\begin{equation}
\begin{aligned}
\left(\widetilde{\m{R}}\m{h}\right)\circledast\m{y}_l=&\left(\mathcal{C}\left(\m{g}\right)\left[\mathcal{C}\left(\m{g}\right)^\top\mathcal{C}\left(\m{g}\right)\right]^{-\frac{1}{2}}\m{U}^\top\m{q}\right)\circledast\m{x}_l=\m{q}\circledast\m{x}_l.
\end{aligned}
\end{equation}
This transformation converts the estimation of $\m{h}$ into that of $\m{q}$, where $\m{q}$ interacts directly with the input signals, with the effect of the unknown filter completely removed. Moreover, the preconditioning matrix $\m{R}$ defined in~\eqref{preconditioning-matrix} provides a non-asymptotic approximation to $\frac{1}{\sqrt{K}}\widetilde{\m{R}}$, as established later in Lemma~\ref{lem-tail-Delta}, thereby making it a natural candidate for preconditioning from the system perspective. The remainder of this subsection further justifies this choice from the optimization perspective and establishes the geometric properties of the resulting optimization problem.

Using the preconditioning matrix $\m{R}$ defined in~\eqref{preconditioning-matrix}, we reformulate problem~\eqref{opt-ini-general-0} as
\begin{equation}
\label{opt-ini-general}
\begin{aligned}
\min_{\m{h}\in\mathbb{S}^{N-1}}\;\sum_{n=1}^N\left(\frac{1}{L}\sum_{l=1}^L\left|\m{y}_l^\top\m{\Gamma}_n\left[\frac{1}{L}\sum_{l=1}^L\mathcal{C}\left(\m{y}_l\right)^\top\mathcal{C}\left(\m{y}_l\right)\right]^{-\frac{1}{2}}\m{h}\right|^2+\varepsilon\right)^{\frac{1}{2}}.
\end{aligned}
\end{equation}
For analytical convenience, we further consider the optimization problem
\begin{equation}
\label{opt-general-reparameterize}
\begin{aligned}
\min_{\m{h}\in\mathbb{S}^{N-1}}\;\psi_\varepsilon\left(\m{h}\right)=&\sum_{n=1}^N\left(\frac{1}{L}\sum_{l=1}^L\left|\m{y}_l^\top\m{\Gamma}_n\m{R}\m{U}^\top\m{h}\right|^2+\varepsilon\right)^{\frac{1}{2}}=\sum_{n=1}^N\left(\frac{1}{L}\sum_{l=1}^L\left|\m{x}_l^\top\mathcal{C}\left(\m{g}\right)^\top\m{\Gamma}_n\m{R}\m{U}^\top\m{h}\right|^2+\varepsilon\right)^{\frac{1}{2}}.
\end{aligned}
\end{equation}
Problems~\eqref{opt-ini-general} and~\eqref{opt-general-reparameterize} differ only by the orthogonal transformation $\m{U}$ applied to the optimization variable. Since orthogonal transformations preserve the geometry of optimization problems on the unit sphere, the two formulations possess identical geometric properties. Therefore, we focus on the geometric analysis of problem~\eqref{opt-general-reparameterize}, with particular emphasis on quantifying the discrepancy between problems~\eqref{opt-general-reparameterize} and~\eqref{opt-special} arising from the finite-sample approximation error of the preconditioning matrix. To proceed, we need the following result.
\begin{lem}
\label{lem-commutativity}
For every $n=1,\cdots,N$ and any vector $\m{u}$ of length $N$, it holds that
\begin{equation}
\label{symmetry}
\mathcal{C}\left(\m{u}\right)^\top\m{\Gamma}_n=\m{\Gamma}_n\mathcal{C}\left(\m{u}\right).
\end{equation}
\end{lem}
\begin{proof}
For any $n_1,n_2\in\left\{1,\cdots,N\right\}$, we have
\begin{equation}
\begin{aligned}
\m{e}_{n_1}^\top\mathcal{C}\left(\m{u}\right)^\top\m{\Gamma}_n\m{e}_{n_2}=&\left(\m{u}\circledast\m{e}_{n_1}\right)^\top\m{e}_{\left(\left(n-n_2\right)\bmod N\right)+1}\\
=&\begin{bmatrix}\m{u}^\top\m{\Gamma}_1\m{e}_{n_1}&\cdots&\m{u}^\top\m{\Gamma}_N\m{e}_{n_1}\end{bmatrix}\m{e}_{\left(\left(n-n_2\right)\bmod N\right)+1}\\
=&\m{u}^\top\m{\Gamma}_{\left(\left(n-n_2\right)\bmod N\right)+1}\m{e}_{n_1}\\
=&\m{u}^\top\m{e}_{\left(\left(\left(\left(n-n_2\right)\bmod N\right)+1-n_1\right)\bmod N\right)+1},
\end{aligned}
\end{equation}
where the last equality follows from
\begin{equation}
\begin{aligned}
\m{\Gamma}_n\m{e}_m=\m{e}_{\left(\left(n-m\right)\bmod N\right)+1},\;n=1,\cdots,N.
\end{aligned}
\end{equation}
On the other hand,
\begin{equation}
\begin{aligned}
\m{e}_{n_1}^\top\m{\Gamma}_n\mathcal{C}\left(\m{u}\right)\m{e}_{n_2}=&\m{e}_{\left(\left(n-n_1\right)\bmod N\right)+1}^\top\left(\m{u}\circledast\m{e}_{n_2}\right)\\
=&\left(\m{u}\circledast\m{e}_{n_2}\right)^\top\m{e}_{\left(\left(n-n_1\right)\bmod N\right)+1}\\
=&\m{u}^\top\m{e}_{\left(\left(\left(\left(n-n_1\right)\bmod N\right)+1-n_2\right)\bmod N\right)+1}.
\end{aligned}
\end{equation}
Since
\begin{equation}
\begin{aligned}
{\left(\left(\left(n-n_2\right)\bmod N\right)+1-n_1\right)\bmod N}={\left(\left(\left(n-n_1\right)\bmod N\right)+1-n_2\right)\bmod N},
\end{aligned}
\end{equation}
the two expressions above are identical. Hence, for all $n_1,n_2$,
\begin{equation}
\begin{aligned}
\m{e}_{n_1}^\top\mathcal{C}\left(\m{u}\right)^\top\m{\Gamma}_n\m{e}_{n_2}=\m{e}_{n_1}^\top\m{\Gamma}_n\mathcal{C}\left(\m{u}\right)\m{e}_{n_2},
\end{aligned}
\end{equation}
which completes the proof.
\end{proof}

Lemma~\ref{lem-commutativity} reveals a reversal-symmetry property of circulant matrices, namely,
\begin{equation}
\mathcal{C}\left(\m{u}\right)^\top=\m{\Gamma}_n\mathcal{C}\left(\m{u}\right)\m{\Gamma}_n^\top,
\end{equation}
which allows the coupled terms introduced by the joint-sparsity-promoting objective to be manipulated in a manner analogous to that used for entry-wise sparsity-promoting objectives. Specifically, using Lemma~\ref{lem-commutativity}, we obtain
\begin{equation}
\begin{aligned}
\mathcal{C}\left(\m{g}\right)^\top\m{\Gamma}_n\m{R}\m{U}^\top=\m{\Gamma}_n\mathcal{C}\left(\m{g}\right)\m{R}\m{U}^\top=\m{\Gamma}_n\left(\frac{1}{\sqrt{K}}\m{I}+\mathcal{C}\left(\m{g}\right)\m{R}\m{U}^\top-\frac{1}{\sqrt{K}}\m{I}\right).
\end{aligned}
\end{equation}
Substituting this identity into~\eqref{opt-general-reparameterize} yields
\begin{equation}
\label{objective-general}
\begin{aligned}
\psi_\varepsilon\left(\m{h}\right)=&\sum_{n=1}^N\left(\frac{1}{L}\sum_{l=1}^L\left|\m{x}_l^\top\m{\Gamma}_n\left(\frac{1}{\sqrt{K}}\m{I}+\mathcal{C}\left(\m{g}\right)\m{R}\m{U}^\top-\frac{1}{\sqrt{K}}\m{I}\right)\m{h}\right|^2+\varepsilon\right)^{\frac{1}{2}}\\
=&\sum_{n=1}^N\left(\frac{1}{L}\sum_{l=1}^L\left|\m{x}_l^\top\m{\Gamma}_n\left(\frac{1}{\sqrt{K}}\m{I}+\m{\Delta}\right)\m{h}\right|^2+\varepsilon\right)^{\frac{1}{2}}\\
=&\sum_{n=1}^N\left[\m{h}^\top\left(\frac{1}{\sqrt{K}}\m{I}+\m{\Delta}\right)^\top\m{\Gamma}_n^\top\left(\m{\Sigma}_{\Omega}+\m{E}_{\Omega}\right)\m{\Gamma}_n\left(\frac{1}{\sqrt{K}}\m{I}+\m{\Delta}\right)\m{h}+\varepsilon\right]^{\frac{1}{2}},
\end{aligned}
\end{equation}
where 
\begin{equation}
\begin{aligned}
\m{\Delta}=&\mathcal{C}\left(\m{g}\right)\m{R}\m{U}^\top-\frac{1}{\sqrt{K}}\m{I}.
\end{aligned}
\end{equation}
It is noteworthy that if $\m{\Delta}=\m{0}$, then problem~\eqref{opt-general-reparameterize} reduces to
\begin{equation}
\label{opt-delta=0}
\begin{aligned}
\min_{\m{h}\in\mathbb{S}^{N-1}}\;\sum_{n=1}^N\left[\frac{1}{K}\m{h}^\top\m{\Gamma}_n^\top\left(\m{\Sigma}_{\Omega}+\m{E}_{\Omega}\right)\m{\Gamma}_n\m{h}+\varepsilon\right]^{\frac{1}{2}}=\frac{1}{\sqrt{K}}\sum_{n=1}^N\left[\m{h}^\top\m{\Gamma}_n^\top\left(\m{\Sigma}_{\Omega}+\m{E}_{\Omega}\right)\m{\Gamma}_n\m{h}+K\varepsilon\right]^{\frac{1}{2}}.
\end{aligned}
\end{equation}
It follows immediately that solving problem~\eqref{opt-delta=0} is equivalent to solving problem~\eqref{opt-special}, up to a rescaling of the smoothing parameter $\varepsilon$ by a factor of $\frac{1}{K}$. Therefore, problem~\eqref{opt-general-reparameterize} reduces to the special-case formulation when $\m{\Delta}=\m{0}$, thereby justifying the choice of the preconditioning matrix. We next establish the following lemma for the subsequent analysis of the optimization problem, which provides a bound on $\left\|\m{\Delta}\right\|_2$ as a function of the sample size $L$. This result will serve as a key ingredient in quantifying the deviation of problem~\eqref{opt-general-reparameterize} from its ideal counterpart.
\begin{lem}
\label{lem-tail-Delta}
There exist positive constants $c_1$ and $c_2$ such that for any $t\in\left(0,\frac{\kappa^{-2}}{2}\right)$, the following bounds hold with probability at least $1-2N\exp\left\{-\frac{Lt^2}{c_1\log^22N+\left(c_2\log2N\right)t}\right\}$:
\begin{equation}
\begin{aligned}
\left\|\m{\Delta}\right\|_2\leq \frac{2\kappa^4t}{\sqrt{K}}.
\end{aligned}
\end{equation}
and
\begin{equation}
\begin{aligned}
\left\|\m{R}-\frac{1}{\sqrt{K}}\left[\mathcal{C}\left(\m{g}\right)^\top\mathcal{C}\left(\m{g}\right)\right]^{-\frac{1}{2}}\right\|_2\leq \frac{2\kappa^4t}{\sqrt{K}\sigma_{\text{min}}\left(\mathcal{C}\left(\m{g}\right)\right)},
\end{aligned}
\end{equation}
\end{lem}
\begin{proof}
The proof is inspired by that of \cite[Lemma~2]{shi2021manifold}. The key distinction is that our analysis explicitly exploits the joint sparsity structure of the input signals, which leads to a sharper bound on $\left\|\m{\Delta}\right\|_2$. The main idea is to first establish a relationship between $\left\|\m{\Delta}\right\|_2$ and $\left\|\frac{1}{KL}\sum_{l=1}^L\mathcal{C}\left(\m{x}_l\right)^\top\mathcal{C}\left(\m{x}_l\right)-\m{I}\right\|_2$ and then characterize the latter quantity using the DFT diagonalization property of circulant matrices \eqref{DFT diagonalization} established in Lemma~\ref{lemma-circulant convolution}. The detailed proof is provided in Section~\ref{pf-lem-tail-Delta} of Supplementary Material.
\end{proof}

Based on Lemma~\ref{lem-tail-Delta}, we investigate in the remainder of this subsection the geometric landscape of the objective function in \eqref{opt-general-reparameterize} over the unit sphere in the finite-sample regime. Similar to the unit-impulse-filter case, we show that all approximate second-order stationary points of problem~\eqref{opt-general-reparameterize} lie sufficiently close to the target solution when the sample size is sufficiently large. The Riemannian gradient of $\psi_\varepsilon$ is given by
\begin{equation}
\label{grad-general}
\begin{aligned}
\text{grad}\;\psi_\varepsilon\left(\m{h}\right)=&\sum_{n=1}^N\left[\m{h}^\top\left(\frac{1}{\sqrt{K}}\m{I}+\m{\Delta}\right)^\top\m{\Gamma}_n^\top\left(\m{\Sigma}_{\Omega}+\m{E}_{\Omega}\right)\m{\Gamma}_n\left(\frac{1}{\sqrt{K}}\m{I}+\m{\Delta}\right)\m{h}+\varepsilon\right]^{-\frac{1}{2}}\\
&\cdot\left(\m{I}-\m{h}\m{h}^\top\right)\left(\frac{1}{\sqrt{K}}\m{I}+\m{\Delta}\right)^\top\m{\Gamma}_n^\top\left(\m{\Sigma}_{\Omega}+\m{E}_{\Omega}\right)\m{\Gamma}_n\left(\frac{1}{\sqrt{K}}\m{I}+\m{\Delta}\right)\m{h}.
\end{aligned}
\end{equation}
Consider a vector $\m{h}\in\mathbb{S}^{N-1}$ satisfying \eqref{prop-max}, and suppose that $\m{h}$ is an approximate first-order stationary point, i.e.,
\begin{equation}
\label{prop-first-order stationary-general}
\begin{aligned}
\sqrt{g_{\m{h}}\left(\text{grad}\;\psi_\varepsilon\left(\m{h}\right),\text{grad}\;\psi_\varepsilon\left(\m{h}\right)\right)}=\left\|\text{grad}\;\psi_\varepsilon\left(\m{h}\right)\right\|_2\leq\xi.
\end{aligned}
\end{equation}
Let $\m{e}_o$ denote the canonical basis vector closest to $\m{h}$, as in Section~\ref{subsec:special}. Furthermore, let $\m{d}$, defined in \eqref{direction}, represent the direction of the Riemannian logarithmic direction from $\m{h}$ toward $\m{e}_o$. We characterize the second-order behavior of $\psi_\varepsilon$ along the direction $\m{d}$ through the quantity $\frac{\m{d}^\top\text{Hess}\;\psi_\varepsilon\left(\m{h}\right)\m{d}}{\left\|\m{d}\right\|_2^2}$. The following lemma establishes an upper bound on $\frac{\m{d}^\top\text{Hess}\;\psi_\varepsilon\left(\m{h}\right)\m{d}}{\left\|\m{d}\right\|_2^2}$ in terms of $\left\|\m{E}\right\|_\Omega$, $\left\|\m{\Delta}\right\|_2$, $\varepsilon$, and $\xi$.
\begin{lem}
\label{lem-2st-general}
Let $\m{h}$ be an approximate first-order stationary point satisfying \eqref{prop-max} and \eqref{prop-first-order stationary-general}. Suppose that $\Omega$ is non-uniform. Then there exist positive constants $c_1$, $c_2$, $c_3$, and $c_4$ such that, whenever $\left\|\m{E}_\Omega\right\|_2\leq c_1\min\left\{K^{-\frac{7}{2}}N^{-1},\varepsilon^{\frac{1}{2}}\right\}$, $\left\|\m{\Delta}\right\|_2\leq c_2\min\left\{K^{-4}N^{-1},K^{-3}N^{-\frac{3}{2}},\varepsilon^{\frac{1}{2}}K^{\frac{1}{2}}N^{-\frac{1}{2}}\right\}$, $\varepsilon\leq c_3K^{-7}N^{-2}$, and $\xi\leq c_4K^{-\frac{7}{2}}N^{-\frac{1}{2}}$, it holds that
\begin{equation}
\begin{aligned}
\frac{\m{d}^\top\text{Hess}\;\psi_\varepsilon\left(\m{h}\right)\m{d}}{\left\|\m{d}\right\|_2^2}
\leq&K^{-\frac{1}{2}}\left(1+2c_1+3c_2+c_3\right)^{-\frac{3}{2}}\frac{h_o^{2}}{\left\|\m{d}\right\|_2^2}\left(\min\left\{Kh_o^2,1-\frac{1-h_o^2}{K}\right\}\right)^{-\frac{3}{2}}\\
&\cdot\min\left\{Kh_o^2-1,-\frac{1-h_o^2}{K}\right\}\\
&+K^{-\frac{1}{2}}\frac{h_o^{2}}{\left\|\m{d}\right\|_2^2}\left(\min\left\{Kh_o^2,1-\frac{1-h_o^2}{K}\right\}\right)^{-\frac{3}{2}}\left(3K^{\frac{1}{2}}\left\|\m{\Delta}\right\|_2+2\left\|\m{E}_\Omega\right\|_2+\varepsilon K\right)\\
&+K^{-\frac{1}{2}}\frac{1}{\left\|\m{d}\right\|_2^2}\left(6c_1+3c_2+2c_3+c_4+16c_1c_2+4c_1c_3+c_2^2+2c_1c_2^2\right)K^{-3}\\
&+K^{-\frac{1}{2}}\frac{1}{\left\|\m{d}\right\|_2^2}\left(8c_1+27c_2+20c_1c_2\right)h_o^{-1}K^{-\frac{5}{2}}N^{-\frac{1}{2}},
\end{aligned}
\end{equation}
\end{lem}
\begin{proof}
This proof generalizes that of Lemma~\ref{lem-2st} to the general filter setting. The main challenge arises from the additional perturbation introduced by $\m{\Delta}$, which significantly alters the structure of the objective function. The key step is to carefully analyze the spectral properties of $\left(\frac{1}{\sqrt{K}}\m{I}+\m{\Delta}\right)^\top\m{\Gamma}_n^\top\m{\Sigma}_{\Omega}\m{\Gamma}_n\left(\frac{1}{\sqrt{K}}\m{I}+\m{\Delta}\right)$ under the influence of $\m{\Delta}$. The detailed proof is provided in Section~\ref{pf-lem-2st-general} of Supplementary Material.
\end{proof}

Based on Lemma~\ref{lem-2st-general}, the following proposition establishes that an approximate first-order stationary point that is sufficiently far from every canonical basis vector cannot be a second-order stationary point.
\begin{prop}
\label{prop-2st-general}
Let $\m{h}$ be an approximate first-order stationary point satisfying \eqref{prop-max} and \eqref{prop-first-order stationary-general}. Suppose that $\Omega$ is non-uniform and $h_o\in\left(\sqrt{\frac{1}{N}},\sqrt{1-\frac{1}{4K^2}}\right]$. Then there exist positive constants $c_1$, $c_2$, $c_3$, $c_4$, and $c_5$ such that, whenever $\left\|\m{E}_\Omega\right\|_2\leq c_1\min\left\{K^{-\frac{7}{2}}N^{-1},\varepsilon^{\frac{1}{2}}\right\}$, $\left\|\m{\Delta}\right\|_2\leq c_2\min\left\{K^{-4}N^{-1},K^{-3}N^{-\frac{3}{2}},\varepsilon^{\frac{1}{2}}K^{-\frac{1}{2}},\varepsilon^{\frac{1}{2}}K^{\frac{1}{2}}N^{-\frac{1}{2}}\right\}$, $\varepsilon\leq c_3K^{-7}N^{-2}$, and $\xi\leq c_4K^{-\frac{7}{2}}N^{-\frac{1}{2}}$, it holds that
\begin{equation}
\begin{aligned}
\frac{\m{d}^\top\text{Hess}\;\psi_\varepsilon\left(\m{h}\right)\m{d}}{\left\|\m{d}\right\|_2^2}\leq-c_5K^{-\frac{7}{2}}.
\end{aligned}
\end{equation}
\end{prop}
\begin{proof}
The proof follows from those of Propositions~\ref{thm-case1}, \ref{thm-case2-1}, and \ref{thm-case2-2}, together with Corollary~\ref{cor-case1,2-1,2-2}, by observing that
 $\left\|\m{E}_\Omega\right\|_2\leq c_1K^{-\frac{5}{2}}\max^{-1}\left\{N,KN^{\frac{1}{2}},K^2\right\}$, $\left\|\m{\Delta}\right\|_2\leq K^{-\frac{1}{2}}\left\|\m{E}_\Omega\right\|_2$, $\varepsilon\leq c_2K^{-6}N^{-2}$, and $\xi\leq c_3K^{-3}N^{-\frac{1}{2}}$.
\end{proof}

Proposition~\ref{prop-2st-general} indicates that no approximate second-order stationary points exist unless $h_o>\sqrt{1-\frac{1}{4K^2}}$. We next investigate the distance between an approximate first-order stationary point and a canonical basis vector in the regime $h_o>\sqrt{1-\frac{1}{4K^2}}$. The following proposition characterizes the result.
\begin{prop}
\label{thm-case3-general}
Let $\m{h}$ be an approximate first-order stationary point satisfying \eqref{prop-max} and \eqref{prop-first-order stationary-general}. Suppose that $\Omega$ is non-uniform and $h_o\in\left(\sqrt{1-\frac{1}{4K^2}},1\right]$. There exist positive constants $c_1$, $c_2$, $c_3$, and $c_4$ such that, whenever $\left\|\m{E}_\Omega\right\|_2\leq c_1K^{-2}$, $\left\|\m{\Delta}\right\|_2\leq c_2K^{-\frac{5}{2}}$, and $\varepsilon\leq c_3K^{-3}$, it holds that
\begin{equation}
\begin{aligned}
\left\|\m{h}-\m{e}_o\right\|_2\leq&c_4K^{-\frac{3}{2}}\left[\xi+2N\left(\left\|\m{E}_\Omega\right\|_2+\left\|\m{\Delta}\right\|_2\right)+\varepsilon^{-\frac{1}{2}}K^{-\frac{1}{2}}N\left\|\m{\Delta}\right\|_2\right].
\end{aligned}
\end{equation}
\end{prop}
\begin{proof}
The proof follows the same line of argument as that of Proposition~\ref{thm-case3}, with additional analysis to account for the perturbation induced by $\m{\Delta}$. The detailed proof is provided in Section~\ref{pf-thm-case3-general} of Supplementary Material.
\end{proof}

Combining Propositions~\ref{prop-2st-general} and~\ref{thm-case3-general} yields the following result.
\begin{thm}
\label{thm-err-general}
Suppose that $\Omega$ is non-uniform and $\m{h}$ is an approximate first-order stationary point satisfying \eqref{prop-max} and \eqref{prop-first-order stationary-general}. There exist positive constants $c_1$, $c_2$, $c_3$, $c_4$, $c_5$, and $c_6$ such that, whenever $\left\|\m{E}_\Omega\right\|_2\leq c_1\min\left\{K^{-\frac{7}{2}}N^{-1},\varepsilon^{\frac{1}{2}}\right\}$, $\left\|\m{\Delta}\right\|_2\leq c_2\min\left\{K^{-4}N^{-1},K^{-3}N^{-\frac{3}{2}},\varepsilon^{\frac{1}{2}}K^{-\frac{1}{2}},\varepsilon^{\frac{1}{2}}K^{\frac{1}{2}}N^{-\frac{1}{2}}\right\}$, $\varepsilon\leq c_3K^{-7}N^{-2}$, $\xi\leq c_4K^{-\frac{7}{2}}N^{-\frac{1}{2}}$, and $\text{Hess}\;\psi_\varepsilon\left(\m{h}\right)\succeq -c_5K^{-\frac{7}{2}}\m{I}$, it holds that
\begin{equation}
\begin{aligned}
\left\|\m{h}-\m{e}_o\right\|_2\leq&c_6K^{-\frac{3}{2}}\left[\xi+2N\left(\left\|\m{E}_\Omega\right\|_2+\left\|\m{\Delta}\right\|_2\right)+\varepsilon^{-\frac{1}{2}}K^{-\frac{1}{2}}N\left\|\m{\Delta}\right\|_2\right].
\end{aligned}
\end{equation}
\end{thm}
\begin{proof}
Since $\Omega$ is non-uniform and $\m{h}$ is an approximate first-order stationary point satisfying \eqref{prop-max} and \eqref{prop-first-order stationary-general}, Proposition~\ref{prop-2st-general} implies that $h_o\in\left(\sqrt{1-\frac{1}{4K^2}},1\right]$ provided $\text{Hess}\;\psi_\varepsilon\left(\m{h}\right)\succeq -c_5K^{-\frac{7}{2}}\m{I}$. Applying Proposition~\ref{thm-case3-general} then yields the desired result.
\end{proof}

Theorem~\ref{thm-err-general} establishes an explicit relationship between the recovery error and the perturbation terms $\left\|\m{E}_\Omega\right\|_2$ and $\left\|\m{\Delta}\right\|_2$. By combining Lemmas~\ref{lem-tail} and~\ref{lem-tail-Delta} with Theorem~\ref{thm-err-general}, we further characterize the dependence of the recovery error on the sample size $L$, as stated in the following theorem.
\begin{thm}
\label{thm-err-bound-general}
Suppose that $\Omega$ is non-uniform and $\m{h}$ is an approximate first-order stationary point satisfying \eqref{prop-max} and \eqref{prop-first-order stationary-general}. There exist positive constants $c_1$, $c_2$, $c_3$, $c_4$, and $c_5$ such that, whenever $\varepsilon\leq c_2K^{-7}N^{-2}$, $\xi\leq c_3K^{-\frac{7}{2}}N^{-\frac{1}{2}}$, $\text{Hess}\;\psi_\varepsilon\left(\m{h}\right)\succeq -c_4K^{-\frac{7}{2}}\m{I}$, and
\begin{equation}
\label{general-cond-L}
\begin{aligned}
L\geq c_1\kappa^8\log^2N\left(\sqrt{K}+\sqrt{\log L}\right)^2\max\left\{K^{7}N^2,K^{5}N^{3},\varepsilon^{-1},\varepsilon^{-1}K^{-2}N\right\},
\end{aligned}
\end{equation}
it holds with probability at least $1-4L^{-8}$ that
\begin{equation}
\label{err-bound-general}
\begin{aligned}
\left\|\m{h}-\m{e}_o\right\|_2\leq&c_5K^{-\frac{3}{2}}\left(\xi+\frac{\kappa^4\varepsilon^{-\frac{1}{2}}N\log N}{K}\sqrt{\frac{\log L}{L}}\right).
\end{aligned}
\end{equation}
\end{thm}
\begin{proof}
As in the derivation of \eqref{inequ-E}, setting $\delta=\sqrt{\frac{16\log L}{L}}$ and applying Lemma~\ref{lem-tail} imply that \eqref{inequ-E} holds with probability at least $1-2\exp\left\{-\frac{16\log L}{2}\right\}=1-2L^{-8}$. Moreover, there exists positive constant $c_1$ such that, if
\begin{equation}
\begin{aligned}
L\geq 16c_1^{-2}\left(\sqrt{K}+4\sqrt{\log L}\right)^2\left(\max\left\{K^{\frac{7}{2}}N,\varepsilon^{-\frac{1}{2}}\right\}\right)^2,
\end{aligned}
\end{equation}
then
\begin{equation}
\begin{aligned}
\left\|\m{E}_\Omega\right\|_2\leq4\left(\sqrt{\frac{K}{L}}+\sqrt{\frac{16\log L}{L}}\right)\leq c_1\min\left\{K^{-\frac{7}{2}}N^{-1},\varepsilon^{\frac{1}{2}}\right\}.
\end{aligned}
\end{equation}

On the other hand, Lemma~\ref{lem-tail-Delta} implies that there exist positive constants $c_2$ and $c_3$ such that, for any $t\in\left(0,\frac{1}{2}\kappa^{-2}\right)$, 
\begin{equation}
\begin{aligned}
\left\|\m{\Delta}\right\|_2\leq \frac{2\kappa^4t}{\sqrt{K}}
\end{aligned}
\end{equation}
holds with probability at least $1-2N\exp\left\{-\frac{Lt^2}{c_2\log^22N+\left(c_3\log2N\right)t}\right\}$. 
Choosing $t=\sqrt{\frac{2c_2\log^22N\log\left(2NL^8\right)}{L}}$, we observe that the condition $t\leq\min\left\{\frac{c_2}{c_3}\log2N,\frac{1}{2}\kappa^{-2}\right\}$ is satisfied whenever
\begin{equation}
\begin{aligned}
L\geq&\max\left\{\frac{c_3^2}{c_2^2\log^22N}2c_2\log^22N\log\left(2NL^8\right),8c_2\kappa^4\log^22N\log\left(2NL^8\right)\right\}\\
=&\max\left\{\frac{2c_3^2}{c_2}\log\left(2NL^8\right),8c_2\kappa^4\log^22N\log\left(2NL^8\right)\right\}.
\end{aligned}
\end{equation}
Under this condition, with probability at least $1-2N\exp\left\{-\frac{2c_2\log^22N\log\left(2NL^8\right)}{2c_2\log^22N}\right\}=1-2L^{-8}$, it holds that
\begin{equation}
\label{Delta-order}
\begin{aligned}
\left\|\m{\Delta}\right\|_2\leq \frac{2\kappa^4}{\sqrt{K}}\sqrt{\frac{2c_2\log^22N\log\left(2NL^8\right)}{L}}.
\end{aligned}
\end{equation}
Furthermore, there exists a positive constant $c_4$ such that, if
\begin{equation}
\begin{aligned}
L\geq 8c_2c_4^{-2}\kappa^8\log^22N\log\left(2NL^8\right)\left(\max\left\{K^{\frac{7}{2}}N,K^{\frac{5}{2}}N^{\frac{3}{2}},\varepsilon^{-\frac{1}{2}},\varepsilon^{-\frac{1}{2}}K^{-1}N^{\frac{1}{2}}\right\}\right)^2,
\end{aligned}
\end{equation}
then
\begin{equation}
\begin{aligned}
\left\|\m{\Delta}\right\|_2\leq c_4\min\left\{K^{-4}N^{-1},K^{-3}N^{-\frac{3}{2}},\varepsilon^{\frac{1}{2}}K^{-\frac{1}{2}},\varepsilon^{\frac{1}{2}}K^{\frac{1}{2}}N^{-\frac{1}{2}}\right\}.
\end{aligned}
\end{equation}
Applying Theorem~\ref{thm-err-general}, we conclude that there exist positive constants $c_5$, $c_6$, $c_7$, and $c_8$ such that, if $\varepsilon\leq c_5K^{-7}N^{-2}$, $\xi\leq c_6K^{-\frac{7}{2}}N^{-\frac{1}{2}}$, $\text{Hess}\;\psi_\varepsilon\left(\m{h}\right)\succeq -c_7K^{-\frac{7}{2}}\m{I}$, and
\begin{equation}
\label{L-3terms}
\begin{aligned}
L\geq \Bigg\{&16c_1^{-2}\left(\sqrt{K}+4\sqrt{\log L}\right)^2\left(\max\left\{K^{\frac{7}{2}}N,\varepsilon^{-\frac{1}{2}}\right\}\right)^2,\frac{2c_3^2}{c_2}\log\left(2NL^8\right),\\
&8c_2c_4^{-2}\kappa^8\log^22N\log\left(2NL^8\right)\left(\max\left\{K^{\frac{7}{2}}N,K^{\frac{5}{2}}N^{\frac{3}{2}},\varepsilon^{-\frac{1}{2}},\varepsilon^{-\frac{1}{2}}K^{-1}N^{\frac{1}{2}}\right\}\right)^2\Bigg\},
\end{aligned}
\end{equation}
then
\begin{equation}
\label{error-4terms}
\begin{aligned}
\left\|\m{h}-\m{e}_o\right\|_2\leq&c_8K^{-\frac{3}{2}}\Bigg[\xi+8N\left(\sqrt{\frac{K}{L}}+\sqrt{\frac{16\log L}{L}}\right)\Bigg]\\
&+c_8K^{-\frac{3}{2}}\left(2N+\varepsilon^{-\frac{1}{2}}K^{-\frac{1}{2}}N\right)\frac{2\kappa^4\log2N}{\sqrt{K}}\sqrt{\frac{2c_2\log\left(2NL^8\right)}{L}}\\
\leq&c_8K^{-\frac{3}{2}}\Bigg[\xi+\frac{20\kappa^4\varepsilon^{-\frac{1}{2}}N\log2N}{K}\sqrt{\frac{2c_2\log\left(2NL^8\right)}{L}}\Bigg].
\end{aligned}
\end{equation}
Note that the right-hand side of \eqref{L-3terms} is of order
$$
\begin{aligned}
O\left(\kappa^8\log^2N\left(\sqrt{K}+\sqrt{\log L}\right)^2\max\left\{K^{7}N^2,K^{5}N^{3},\varepsilon^{-1},\varepsilon^{-1}K^{-2}N\right\}\right),
\end{aligned}
$$
whereas the right-hand side of \eqref{error-4terms} is of order
$$
\begin{aligned}
O\left(K^{-\frac{3}{2}}\left(\xi+\frac{\kappa^4\varepsilon^{-\frac{1}{2}}N\log N}{K}\sqrt{\frac{\log L}{L}}\right)\right),
\end{aligned}
$$
which completes the proof.
\end{proof}

Similar to Theorem~\ref{thm-err} for the unit impulse filter case, Theorem~\ref{thm-err-bound-general} explicitly characterizes the recovery error achieved by approximate second-order stationary points in the general filter setting. Ignoring logarithmic factors, the sample complexity requirement in \eqref{general-cond-L} is approximately of order $O\left(\kappa^8K^6N^3\right)$ when $N\geq K^2$ and $\varepsilon$ is chosen on the order of $O\left(K^{-7}N^{-2}\right)$. A related result established in~\cite[Theorem~3]{shi2021manifold} requires a sample complexity of approximately $O\left(\kappa^8\theta^{-4}\mu^{-2}N^3\right)$. In that work, $\theta\approx K/N$ denotes the sparsity ratio, as mentioned previously, and $\mu$ is an algorithmic hyperparameter whose order does not exceed that of $\theta$. Consequently, when $\mu$ is chosen on the order of $O\left(KN^{-1}\right)$, the corresponding sample complexity requirement becomes approximately $O\left(\kappa^8K^{-6}N^9\right)$. Therefore, when $K^2\leq N$, the sample complexity condition derived in this paper is substantially less restrictive than that of \cite{shi2021manifold}. This improvement again highlights the benefits of exploiting joint sparsity in blind deconvolution, as in Section~\ref{subsec:special}.

At this stage, we have established the optimization landscape of problem \eqref{opt-general-reparameterize}. Since \eqref{opt-general-reparameterize} and \eqref{opt-ini-general} share identical first- and second-order optimality properties, this geometric characterization can be directly transferred to problem \eqref{opt-ini-general}. In the next section, we develop an algorithm for solving \eqref{opt-ini-general} based on the established optimization landscape.

\section{Algorithm}
\label{Sec:Algorithm}

In this section, we develop an algorithm for jointly sparse blind deconvolution by solving the optimization problem \eqref{opt-ini-general}. According to the optimization landscape established in Theorem~\ref{thm-err-bound-general}, successful recovery only requires finding an approximate second-order stationary point of the optimization problem on the unit sphere defined in \eqref{opt-ini-general}. Motivated by this observation, we develop an efficient algorithm with provable convergence guarantees that exploits the specific structure of the objective function.

\subsection{Riemannian Gradient Descent Algorithm with Negative Curvature Search}
Let
\begin{equation}
\begin{aligned}
\rho\left(\m{h}\right)=\sum_{n=1}^N\left(\frac{1}{L}\sum_{l=1}^L\left|\m{y}_l^\top\m{\Gamma}_n\m{R}\m{h}\right|^2+\varepsilon\right)^{\frac{1}{2}}
\end{aligned}
\end{equation}
denote the objective function in \eqref{opt-ini-general}. A straightforward calculation shows that the Riemannian gradient of $\rho$ is given by
\begin{equation}
\begin{aligned}
\text{grad}\;\rho\left(\m{h}\right)=&\sum_{n=1}^N\left(\frac{1}{L}\sum_{l=1}^L\left|\m{y}_l^\top\m{\Gamma}_n\m{R}\m{h}\right|^2+\varepsilon\right)^{-\frac{1}{2}}\left(\m{I}-\m{h}\m{h}^\top\right)\m{R}^\top\m{\Gamma}_n^\top\left(\frac{1}{L}\sum_{l=1}^L\m{y}_l\m{y}_l^\top\right)\m{\Gamma}_n\m{R}\m{h},
\end{aligned}
\end{equation}
and the Riemannian Hessian of $\rho$ is
\begin{equation}
\begin{aligned}
\text{Hess}\;\rho\left(\m{h}\right)=&\left(\m{I}-\m{h}\m{h}^\top\right)\Bigg\{\sum_{n=1}^N\left(\frac{1}{L}\sum_{l=1}^L\left|\m{y}_l^\top\m{\Gamma}_n\m{R}\m{h}\right|^2+\varepsilon\right)^{-\frac{1}{2}}\m{R}^\top\m{\Gamma}_n^\top\left(\frac{1}{L}\sum_{l=1}^L\m{y}_l\m{y}_l^\top\right)\m{\Gamma}_n\m{R}\\
&-\sum_{n=1}^N\left(\frac{1}{L}\sum_{l=1}^L\left|\m{y}_l^\top\m{\Gamma}_n\m{R}\m{h}\right|^2+\varepsilon\right)^{-\frac{3}{2}}\\
&\cdot\m{R}^\top\m{\Gamma}_n^\top\left(\frac{1}{L}\sum_{l=1}^L\m{y}_l\m{y}_l^\top\right)\m{\Gamma}_n\m{R}\m{h}\m{h}^\top\m{R}^\top\m{\Gamma}_n^\top\left(\frac{1}{L}\sum_{l=1}^L\m{y}_l\m{y}_l^\top\right)\m{\Gamma}_n\m{R}\\
&-\sum_{n=1}^N\left(\frac{1}{L}\sum_{l=1}^L\left|\m{y}_l^\top\m{\Gamma}_n\m{R}\m{h}\right|^2+\varepsilon\right)^{-\frac{1}{2}}\left[\m{h}^\top\m{R}^\top\m{\Gamma}_n^\top\left(\frac{1}{L}\sum_{l=1}^L\m{y}_l\m{y}_l^\top\right)\m{\Gamma}_n\m{R}\m{h}\right]\m{I}\Bigg\}\left(\m{I}-\m{h}\m{h}^\top\right).
\end{aligned}
\end{equation}
The point $\m{h}$ is an $\left(\epsilon_g,\epsilon_H\right)$-approximate second-order stationary point, if 
\begin{equation}
\begin{aligned}
\left\|\text{grad}\;\rho\left(\m{h}\right)\right\|_2\leq\epsilon_g,
\end{aligned}
\end{equation}
and
\begin{equation}
\begin{aligned}
\lambda_{\min}\left(\text{Hess}\;\rho\left(\m{h}\right)\right)\geq-\epsilon_H,
\end{aligned}
\end{equation}
where $\lambda_{\min}\left(\text{Hess}\;\rho\left(\m{h}\right)\right)$ denotes the smallest eigenvalue of the Riemannian Hessian, viewed as a self-adjoint linear operator on the tangent space $\text{T}_{\m{h}}\mathcal{M}$, and $\varepsilon_g>0$ and $\varepsilon_H>0$ represent the first- and second-order stationarity residuals, respectively. A variety of algorithms have been developed for computing second-order stationary points on Riemannian manifolds, including first-order methods~\cite{criscitiello2019efficiently,sun2019escaping} and second-order methods~\cite{absil2007trust}. Since the Riemannian Hessian can be computed efficiently in the present setting, we employ RGD-NCS to compute an approximate second-order stationary point. This algorithm may be viewed as a specialization of perturbed Riemannian gradient descent~\cite{criscitiello2019efficiently,sun2019escaping}, in which a direction of strictly negative curvature replaces the random perturbation. Similar ideas have been extensively studied in nonconvex optimization over Euclidean spaces~\cite{liu2018adaptive,curtis2019exploiting}; here, we extend them to the Riemannian setting. Specifically, the algorithm performs a gradient descent step whenever the length of the Riemannian gradient exceeds a prescribed threshold, thereby reducing the first-order stationarity residual. Once the gradient norm falls below this threshold, the algorithm computes the eigendirection corresponding to the smallest eigenvalue of the Riemannian Hessian. If a sufficiently negative curvature direction exists, the iterate is updated along this direction to reduce the second-order stationarity residual. Otherwise, the current iterate is declared an approximate second-order stationary point. The complete procedure is summarized in Algorithm~1. Although our numerical experiments in Section~\ref{Sec:simulation} indicate that essentially the same recovery performance is achieved by Riemannian gradient descent alone, we retain the negative-curvature search step in the proposed algorithm to facilitate the subsequent theoretical analysis. In the next subsection, we establish convergence guarantees to approximate second-order stationary points together with the corresponding reconstruction error bounds.
\begin{algorithm}[htbp]
\label{alg-rgd-nc}
    \caption{Riemannian Gradient Descent Algorithm with Negative Curvature Search (RGD-NCS)} 
    {\bf Input:}
    Observations $\left\{\m{y}_{l}\right\}_{l=1}^{L}$, initial point $\m{h}_{0}$, first-order tolerance $\epsilon_g>0$, second-order tolerance $\epsilon_H>0$, gradient descent step size $\alpha_{g}>0$, negative-curvature step size $\alpha_{H}>0$.\\
    {\bf Output:} 
    Estimate of the inverse filter representation $\widetilde{\m{g}}_{\text{inv}}$.\\
    \hspace*{0.02in} 1: \textbf{for} $t=1,\cdots$ \textbf{do}\\
    \hspace*{0.02in} 2: \hspace*{0.06in} Compute the Riemannian gradient $\text{grad}\;\rho\left(\m{h}_{t-1}\right)$;\\
    \hspace*{0.02in} 3: \hspace*{0.06in} \textbf{if} $\left\|\text{grad}\;\rho\left(\m{h}_{t-1}\right)\right\|_2>\epsilon_g$ \textbf{then}\\
    \hspace*{0.02in} 4: \hspace*{0.12in} Update $\m{h}_{t}=\dfrac{\m{h}_{t-1}-\alpha_g\text{grad}\;\rho\left(\m{h}_{t-1}\right)}{\left\|\m{h}_{t-1}-\alpha_g\text{grad}\;\rho\left(\m{h}_{t-1}\right)\right\|_2}$;\\
    \hspace*{0.02in} 5: \hspace*{0.06in} \textbf{else}\\
    \hspace*{0.02in} 6: \hspace*{0.12in} Compute the smallest eigenvalue $\lambda_{\min}\left(\text{Hess}\;\rho\left(\m{h}_{t-1}\right)\right)$ of the Riemannian Hessian and its associated unit eigenvector $\m{v}_t\in\text{T}_{\m{h}_{t-1}}\mathbb{S}^{N-1}$;\\
    \hspace*{0.02in} 7: \hspace*{0.12in} \textbf{if} $\lambda_{\min}\left(\text{Hess}\;\rho\left(\m{h}_{t-1}\right)\right)<-\epsilon_H$ \textbf{then}\\
    \hspace*{0.02in} 8: \hspace*{0.18in} Choose $\eta_t\in\left\{1,-1\right\}$ such that $\eta_tg_{\m{h}_{t-1}}\left(\text{grad}\;\rho\left(\m{h}_{t-1}\right),\m{v}_t\right)\geq0$;\\
    \hspace*{0.02in} 9: \hspace*{0.18in} Update $\m{h}_{t}=\dfrac{\m{h}_{t-1}-\alpha_H\eta_t\m{v}_t}{\left\|\m{h}_{t-1}-\alpha_H\eta_t\m{v}_t\right\|_2}$;\\
    \hspace*{0.02in} 10: \hspace*{0.12in} \textbf{else}\\
    \hspace*{0.02in} 11: \hspace*{0.18in} Set $\widetilde{\m{h}}=\m{h}_{t-1}$ and \textbf{break};\\
    \hspace*{0.02in} 12: \hspace*{0.12in} \textbf{end if}\\
    \hspace*{0.02in} 13: \hspace*{0.06in} \textbf{end if}\\
    \hspace*{0.02in} 14: \textbf{end for}\\
    \hspace*{0.02in} 15: \textbf{return} $\widetilde{\m{g}}_{\text{inv}}=\m{R}\widetilde{\m{h}}$.
\end{algorithm}

\subsection{Convergence and Error Analysis}
We now establish the convergence of the proposed algorithm. Specifically, we show that the algorithm is guaranteed to generate an iterate whose first- and second-order stationarity residuals are below any prescribed tolerances. The result is summarized in the following proposition.
\begin{prop}
\label{prop-alg}
For any prescribed tolerances $\epsilon_g>0$ and $\epsilon_H>0$, there exist a gradient descent step size $\alpha_g>0$, a negative-curvature step size $\alpha_H>0$, and a finite integer $T$ of order $O\left(\max\left\{\epsilon_g^{-2},\epsilon_H^{-3}\right\}\right)$ such that RGD-NCS returns an $\left(\epsilon_g,\epsilon_H\right)$-approximate second-order stationary point of \eqref{opt-ini-general} within at most $T$ iterations.
\end{prop}
\begin{proof}
The proof relies on the smoothness of the objective function and the tractability of the Riemannian Hessian. These properties enable the algorithm to reduce the first-order stationarity residual by taking steps along the Riemannian gradient direction. Once the first-order residual becomes sufficiently small, the algorithm exploits directions of negative curvature identified from the Riemannian Hessian to decrease the second-order stationarity residual. The detailed proof is provided in Section~\ref{pf-prop-alg} of Supplementary Material.
\end{proof}

Proposition~\ref{prop-alg} establishes that an approximate second-order stationary point with arbitrarily small stationarity residuals can be obtained in a finite number of iterations. Combining this result with the geometric analysis developed in the previous subsections, we obtain the following theorem, which characterizes the filter estimation accuracy achieved by the proposed algorithm.
\begin{thm}
\label{thm-err-alg}
Suppose that $\Omega$ is non-uniform. There exist positive constants $c_1$, $c_2$, $c_3$, $c_4$, and $c_5$ such that, whenever $\varepsilon\leq c_2K^{-7}N^{-2}$ and
\begin{equation}
\label{general-cond-L}
\begin{aligned}
L\geq c_1\kappa^8\log^2N\left(\sqrt{K}+\sqrt{\log L}\right)^2\max\left\{K^{7}N^2,K^{5}N^{3},\varepsilon^{-1},\varepsilon^{-1}K^{-2}N\right\},
\end{aligned}
\end{equation}
the following statement holds with probability at least $1-4L^{-8}$: if RGD-NCS is executed with first-order tolerance $\epsilon_g\leq c_3K^{-\frac{7}{2}}N^{-\frac{1}{2}}$ and second-order tolerance $\epsilon_H\leq c_4K^{-\frac{7}{2}}$, then the resulting estimate of the inverse filter representation $\widetilde{\m{g}}_{\text{inv}}$ satisfies
\begin{equation}
\begin{aligned}
\min_{s=\pm1,n=1,\cdots,N}\left\|K^{-\frac{1}{2}}\mathcal{C}\left(\m{g}\right)^{-1}\m{e}_n-s\widetilde{\m{g}}_{\text{inv}}\right\|_2\leq \frac{c_5}{\sigma_{\min}\left(\mathcal{C}\left(\m{g}\right)\right)}\left[\kappa^4\varepsilon^{-\frac{1}{2}}K^{-\frac{5}{2}}N\log N\sqrt{\frac{\log L}{L}}+K^{-\frac{3}{2}}\epsilon_g\right].
\end{aligned}
\end{equation}
\end{thm}
\begin{proof}
It follows from Proposition~\ref{prop-alg} that RGD-NCS produces an $\left(\epsilon_g,\epsilon_H\right)$-approximate second-order stationary point $\widetilde{\m{h}}$. Combining this result with Theorem~\ref{thm-err-bound-general}, we conclude that there exists a positive constant $c_6$ such that
\begin{equation}
\begin{aligned}
\left\|\m{U}\widetilde{\m{h}}-s\m{e}_o\right\|_2=\left\|\widetilde{\m{h}}-s\m{U}^\top\m{e}_o\right\|_2\leq&c_6K^{-\frac{3}{2}}\left(\epsilon_g+\frac{\kappa^4\varepsilon^{-\frac{1}{2}}N\log N}{K}\sqrt{\frac{\log L}{L}}\right),
\end{aligned}
\end{equation}
where $\m{e}_o$ denotes the canonical basis vector closest to $\m{U}\widetilde{\m{h}}$, and $s$ denotes the sign of the $o$-th entry of $\m{U}\widetilde{\m{h}}$. Therefore,
\begin{equation}
\begin{aligned}
\left\|sK^{-\frac{1}{2}}\mathcal{C}\left(\m{g}\right)^{-1}\m{e}_o-\widetilde{\m{g}}_{\text{inv}}\right\|_2=&\left\|sK^{-\frac{1}{2}}\mathcal{C}\left(\m{g}\right)^{-1}\m{e}_o-\m{R}\widetilde{\m{h}}\right\|_2\\
\leq&\left\|\mathcal{C}\left(\m{g}\right)^{-1}\right\|_2\left\|sK^{-\frac{1}{2}}\m{e}_o-\mathcal{C}\left(\m{g}\right)\m{R}\m{U}^\top\m{U}\widetilde{\m{h}}\right\|_2\\
\leq&\left\|\mathcal{C}\left(\m{g}\right)^{-1}\right\|_2\left\|sK^{-\frac{1}{2}}\m{e}_o-s\mathcal{C}\left(\m{g}\right)\m{R}\m{U}^\top\m{e}_o\right\|_2\\
&+\left\|\mathcal{C}\left(\m{g}\right)^{-1}\right\|_2\left\|s\mathcal{C}\left(\m{g}\right)\m{R}\m{U}^\top\m{e}_o-\mathcal{C}\left(\m{g}\right)\m{R}\m{U}^\top\m{U}\widetilde{\m{h}}\right\|_2\\
\leq&\left\|\mathcal{C}\left(\m{g}\right)^{-1}\right\|_2\left\|K^{-\frac{1}{2}}\m{I}-\mathcal{C}\left(\m{g}\right)\m{R}\m{U}^\top\right\|_2\\
&+\left\|\mathcal{C}\left(\m{g}\right)^{-1}\right\|_2\left\|\mathcal{C}\left(\m{g}\right)\m{R}\m{U}^\top\right\|_2\left\|s\m{e}_o-\m{U}\widetilde{\m{h}}\right\|_2.
\end{aligned}
\end{equation}
Furthermore, Lemma~\ref{lem-tail-Delta} implies that
\begin{equation}
\begin{aligned}
\min_{s=\pm1,n=1,\cdots,N}\left\|K^{-\frac{1}{2}}\mathcal{C}\left(\m{g}\right)^{-1}\m{e}_n-s\widetilde{\m{g}}_{\text{inv}}\right\|_2\leq&\frac{2}{\sigma_{\min}\left(\mathcal{C}\left(\m{g}\right)\right)}\left\|s\m{e}_o-\m{U}\widetilde{\m{h}}\right\|_2+\frac{1}{\sigma_{\min}\left(\mathcal{C}\left(\m{g}\right)\right)}\left\|\m{\Delta}\right\|_2.
\end{aligned}
\end{equation}
Substituting the bound in \eqref{Delta-order} into the above inequality yields that there exists a positive constant $c_7$ such that
\begin{equation}
\begin{aligned}
\min_{s=\pm1,n=1,\cdots,N}\left\|K^{-\frac{1}{2}}\mathcal{C}\left(\m{g}\right)^{-1}\m{e}_n-s\widetilde{\m{g}}_{\text{inv}}\right\|_2\leq&\frac{2}{\sigma_{\min}\left(\mathcal{C}\left(\m{g}\right)\right)}c_6K^{-\frac{3}{2}}\left(\epsilon_g+\frac{\kappa^4\varepsilon^{-\frac{1}{2}}N\log N}{K}\sqrt{\frac{\log L}{L}}\right)\\
&+\frac{1}{\sigma_{\min}\left(\mathcal{C}\left(\m{g}\right)\right)}\frac{2\kappa^4}{\sqrt{K}}\sqrt{\frac{2c_7\log^22N\log\left(2NL^8\right)}{L}}.
\end{aligned}
\end{equation}
Under the condition $\varepsilon\leq c_2K^{-7}N^{-2}$, the second term is of lower order than the first and can therefore be absorbed into the leading term. The desired result then follows by redefining the constant $c_5$, which completes the proof.
\end{proof}

Theorem~\ref{thm-err-alg} shows that, provided the sample size is sufficiently large, the estimation error can be made arbitrarily small through an appropriate choice of the smoothing parameter $\varepsilon$ and the algorithmic tolerances $\epsilon_g$ and $\epsilon_H$.

\section{Extension to the Complex-valued Setting}
\label{Sec:extension}


Jointly sparse blind deconvolution in the complex domain arises in a wide range of applications in array and radar signal processing~\cite{267014,757212,7779126,10535469,eldar2020sensor,vargas2023dual,jacome2024multi,4799379}. In this setting, the signal model is the complex-valued counterpart of \eqref{circulant-convolution}, where the input signals $\left\{\m{x}_{l}\right\}_{l=1}^L$, the filter $\m{g}$, and the observations $\left\{\m{y}_{l}\right\}_{l=1}^L$ are all complex-valued. In this section, we show that the optimization framework and algorithm developed for the real-valued problem can be naturally extended to the complex-valued setting.



We first formulate the optimization problem for complex-valued jointly sparse blind deconvolution, following the optimization framework developed in Section~\ref{Sec:mean-result}. Similar to the framework~\eqref{ini-opt-0} in the real-valued setting, we propose the following formulation:
\begin{equation}
\label{ini-opt-complex}
\begin{aligned}
\min_{\m{h}\in\mathbb{S}_{\text{C}}^{N-1},\m{Z}\in\mathbb{C}^{N\times L}}\;&\sum_{n=1}^N\left(\frac{1}{L}\left\|\m{Z}_{n,:}\right\|_2^2+\varepsilon\right)^{\frac{1}{2}},\;\st\m{Z}_{:,l}=\left(\m{R}_{\text{C}}\m{h}\right)\circledast\m{y}_l,\;l=1,\cdots,L,
\end{aligned}
\end{equation}
where $\mathbb{S}_{\text{C}}^{N-1}=\left\{\m{h}\in\mathbb{C}^N:\;\left\|\m{h}\right\|_2=1\right\}$ is the complex-valued counterpart of $\mathbb{S}^{N-1}$, and the preconditioning matrix in the complex-valued setting is naturally defined as
\begin{equation}
\begin{aligned}
\m{R}_{\text{C}}=\left[\frac{1}{L}\sum_{l=1}^L\mathcal{C}\left(\m{y}_l\right)^H\mathcal{C}\left(\m{y}_l\right)\right]^{-\frac{1}{2}}.
\end{aligned}
\end{equation}
Furthermore, since the representation~\eqref{circulant-convolution-matrix} remains valid in the complex-valued setting, problem~\eqref{ini-opt-complex} is equivalent to
\begin{equation}
\label{opt-ini-complex}
\begin{aligned}
\min_{\m{h}\in\mathbb{S}_{\text{C}}^{N-1}}\;\rho_{\text{C}}\left(\m{h}\right)=\sum_{n=1}^N\left(\frac{1}{L}\sum_{l=1}^L\left|\m{y}_l^\top\m{\Gamma}_n\m{R}_{\text{C}}\m{h}\right|^2+\varepsilon\right)^{\frac{1}{2}},
\end{aligned}
\end{equation}
which is the complex-valued counterpart of the optimization problem~\eqref{opt-ini-general}.

Next, we justify the optimization problem~\eqref{opt-ini-complex}. As in the real-valued case, to facilitate the analysis of its geometric properties, we consider the rotated optimization problem
\begin{equation}
\label{opt-complex-reparameterize}
\begin{aligned}
\min_{\m{h}\in\mathbb{S}_{\text{C}}^{N-1}}\;&\sum_{n=1}^N\left(\frac{1}{L}\sum_{l=1}^L\left|\m{y}_l^\top\m{\Gamma}_n\m{R}_{\text{C}}\m{U}_{\text{C}}^H\m{h}\right|^2+\varepsilon\right)^{\frac{1}{2}}.
\end{aligned}
\end{equation}
where $\m{U}_{\text{C}}$ is the unitary matrix defined by
\begin{equation}
\begin{aligned}
\m{U}_{\text{C}}=\mathcal{C}\left(\m{g}\right)\left[\mathcal{C}\left(\m{g}\right)^H\mathcal{C}\left(\m{g}\right)\right]^{-\frac{1}{2}}.
\end{aligned}
\end{equation}
Problems~\eqref{opt-ini-complex} and~\eqref{opt-complex-reparameterize} differ only by the unitary transformation $\m{U}_{\mathrm{C}}$ applied to the optimization variable and therefore possess identical geometric properties. As in the real-valued case, we make the following assumptions: 1) the nonzero components $\mathcal{P}_{\Omega}\left(\m{x}_l\right)$, for $l=1,\cdots,L$, are i.i.d. circularly symmetric Gaussian random vectors with zero mean and covariance matrix $\m{I}$; and 2) $\Omega$ is non-uniform. Under the first assumption, Lemma~\ref{lem-tail} can be extended directly to the complex-valued setting~\cite{tropp2012user}. Moreover, by following arguments analogous to those in the proof of Lemma~\ref{lem-tail-Delta}, Lemma~\ref{lem-tail-Delta} can likewise be extended to the complex-valued setting. Consequently, as the sample size $L$ tends to infinity, the following quantities converge to zero:
\begin{equation}
\begin{aligned}
\m{E}_{\text{C}}=\frac{1}{L}\sum_{l=1}^L\m{x}_l\m{x}_l^H-\m{\Sigma}_{\Omega},
\end{aligned}
\end{equation}
\begin{equation}
\begin{aligned}
\m{R}_{\text{C}}-\frac{1}{\sqrt{K}}\left[\mathcal{C}\left(\m{g}\right)^H\mathcal{C}\left(\m{g}\right)\right]^{-\frac{1}{2}},
\end{aligned}
\end{equation}
and
\begin{equation}
\begin{aligned}
\m{\Delta}_{\text{C}}=\mathcal{C}\left(\m{g}\right)\m{R}_{\text{C}}\m{U}_{\text{C}}^H-\frac{1}{\sqrt{K}}\m{I}.
\end{aligned}
\end{equation}
Moreover, it is straightforward to verify that
\begin{equation}
\label{objective-complex}
\begin{aligned}
&\sum_{n=1}^N\left(\frac{1}{L}\sum_{l=1}^L\left|\m{y}_l^\top\m{\Gamma}_n\m{R}_{\text{C}}\m{U}_{\text{C}}^H\m{h}\right|^2+\varepsilon\right)^{\frac{1}{2}}\\
=&\sum_{n=1}^N\left(\frac{1}{L}\sum_{l=1}^L\left|\m{x}_l^\top\m{\Gamma}_n\left(\frac{1}{\sqrt{K}}\m{I}+\mathcal{C}\left(\m{g}\right)\m{R}_{\text{C}}\m{U}_{\text{C}}^H-\frac{1}{\sqrt{K}}\m{I}\right)\m{h}\right|^2+\varepsilon\right)^{\frac{1}{2}}\\
=&\sum_{n=1}^N\left[\m{h}^H\left(\frac{1}{\sqrt{K}}\m{I}+\m{\Delta}_{\text{C}}\right)^H\m{\Gamma}_n^H\left(\m{\Sigma}_{\Omega}+\overline{\m{E}_{\text{C}}}\right)\m{\Gamma}_n\left(\frac{1}{\sqrt{K}}\m{I}+\m{\Delta}_{\text{C}}\right)\m{h}+\varepsilon\right]^{\frac{1}{2}}.
\end{aligned}
\end{equation}
Therefore, when $\m{E}_{\text{C}}=\m{0}$ and $\m{\Delta}_{\text{C}}=\m{0}$, the asymptotic optimization problem reduces to
\begin{equation}
\label{opt-complex-asymptotic}
\begin{aligned}
\min_{\m{h}\in\mathbb{S}_{\text{C}}^{N-1}}\;&\sum_{n=1}^N\left(\frac{1}{K}\m{h}^H\m{\Gamma}_n^\top\m{\Sigma}_{\Omega}\m{\Gamma}_n\m{h}+\varepsilon\right)^{\frac{1}{2}}=&\sum_{n=1}^N\left[\sum_{m=1}^N\left|h_m\right|^2\left(\frac{1}{K}\m{e}_m^\top\m{\Gamma}_n^\top\m{\Sigma}_{\Omega}\m{\Gamma}_n\m{e}_m+\varepsilon\right)\right]^{\frac{1}{2}}.
\end{aligned}
\end{equation}
Following arguments similar to those used in the proof of Proposition~\ref{lem-support}, it can be shown that the global minimizers of problem~\eqref{opt-complex-asymptotic} are precisely the complex signed canonical basis vectors, namely $\exp\left\{j\phi\right\}\m{e}_n$ for $\phi\in\left[0,2\pi\right)$ and $n=1,\cdots,N$, if and only if $\Omega$ is non-uniform. Consequently, problem~\eqref{opt-ini-complex}, which serves as a non-asymptotic counterpart of problem~\eqref{opt-complex-asymptotic}, is expected to recover the desired solution when the sample size is sufficiently large. Furthermore, theoretical guarantees analogous to those established in Theorem~\ref{thm-err-bound-general} can be derived through a similar line of analysis. However, a complete proof requires substantial technical derivations to characterize the geometric landscape rigorously, including the distance between the current iterate and the target solution, as well as the local second-order properties of the objective function. For the sake of brevity, these technical details are omitted. 

Finally, we turn to the algorithm for solving problem~\eqref{opt-ini-complex}. Observe that $\mathbb{S}_{\text{C}}^{N-1}$ is a Riemannian manifold equipped with the Riemannian metric
\begin{equation}
\label{metric-complex}
\begin{aligned}
g_{\m{h}}^{\text{C}}\left(\m{x},\m{y}\right)=\Re\left\{\m{x}^H\m{y}\right\},
\end{aligned}
\end{equation}
where $\m{x}$ and $\m{y}$ are tangent vectors in the tangent space at $\m{h}\in\mathbb{S}_{\text{C}}^{N-1}$, and $\Re\left\{\cdot\right\}$ denotes the real-part operator. Accordingly, the proposed algorithm in Section~\ref{Sec:Algorithm} can be extended naturally to the complex-valued setting by replacing the expressions for the Riemannian gradient and the Riemannian Hessian with their complex-valued counterparts. It follows from~\eqref{metric-complex} that the Riemannian gradient of $\rho_{\mathrm{C}}$ is given by
\begin{equation}
\begin{aligned}
\text{grad}\;\rho_{\text{C}}\left(\m{h}\right)=&\sum_{n=1}^N\left(\frac{1}{L}\sum_{l=1}^L\left|\m{y}_l^\top\m{\Gamma}_n\m{R}_{\text{C}}\m{h}\right|^2+\varepsilon\right)^{-\frac{1}{2}}\left(\m{I}-\m{h}\m{h}^H\right)\m{R}_{\text{C}}^H\m{\Gamma}_n^\top\left(\frac{1}{L}\sum_{l=1}^L\overline{\m{y}_l}\m{y}_l^\top\right)\m{\Gamma}_n\m{R}_{\text{C}}\m{h},
\end{aligned}
\end{equation}
and the Riemannian Hessian of $\rho_{\text{C}}$ is
\begin{equation}
\begin{aligned}
\text{Hess}\;\rho_{\text{C}}\left(\m{h}\right)\left[\m{\eta}\right]=&\sum_{n=1}^N\left(\frac{1}{L}\sum_{l=1}^L\left|\m{y}_l^\top\m{\Gamma}_n\m{R}_{\text{C}}\m{h}\right|^2+\varepsilon\right)^{-\frac{1}{2}}\m{R}_{\text{C}}^H\m{\Gamma}_n^\top\left(\frac{1}{L}\sum_{l=1}^L\overline{\m{y}_l}\m{y}_l^\top\right)\m{\Gamma}_n\m{R}_{\text{C}}\m{\eta}\\
&-\sum_{n=1}^N\left(\frac{1}{L}\sum_{l=1}^L\left|\m{y}_l^\top\m{\Gamma}_n\m{R}_{\text{C}}\m{h}\right|^2+\varepsilon\right)^{-\frac{1}{2}}\Re\left\{\m{h}^H\m{R}_{\text{C}}^H\m{\Gamma}_n^\top\left(\frac{1}{L}\sum_{l=1}^L\overline{\m{y}_l}\m{y}_l^\top\right)\m{\Gamma}_n\m{R}_{\text{C}}\m{\eta}\right\}\m{h}\\
&-\sum_{n=1}^N\left(\frac{1}{L}\sum_{l=1}^L\left|\m{y}_l^\top\m{\Gamma}_n\m{R}_{\text{C}}\m{h}\right|^2+\varepsilon\right)^{-\frac{3}{2}}\Re\left\{\m{h}^H\m{R}_{\text{C}}^H\m{\Gamma}_n^\top\left(\frac{1}{L}\sum_{l=1}^L\overline{\m{y}_l}\m{y}_l^\top\right)\m{\Gamma}_n\m{R}_{\text{C}}\m{\eta}\right\}\\
&\cdot\m{R}_{\text{C}}^H\m{\Gamma}_n^\top\left(\frac{1}{L}\sum_{l=1}^L\overline{\m{y}_l}\m{y}_l^\top\right)\m{\Gamma}_n\m{R}_{\text{C}}\m{h}\\
&+\sum_{n=1}^N\left(\frac{1}{L}\sum_{l=1}^L\left|\m{y}_l^\top\m{\Gamma}_n\m{R}_{\text{C}}\m{h}\right|^2+\varepsilon\right)^{-\frac{3}{2}}\Re\left\{\m{h}^H\m{R}_{\text{C}}^H\m{\Gamma}_n^\top\left(\frac{1}{L}\sum_{l=1}^L\overline{\m{y}_l}\m{y}_l^\top\right)\m{\Gamma}_n\m{R}_{\text{C}}\m{\eta}\right\}\\
&\cdot\Re\left\{\m{h}^H\m{R}_{\text{C}}^H\m{\Gamma}_n^\top\left(\frac{1}{L}\sum_{l=1}^L\overline{\m{y}_l}\m{y}_l^\top\right)\m{\Gamma}_n\m{R}_{\text{C}}\m{h}\right\}\m{h}\\
&-\sum_{n=1}^N\left(\frac{1}{L}\sum_{l=1}^L\left|\m{y}_l^\top\m{\Gamma}_n\m{R}_{\text{C}}\m{h}\right|^2+\varepsilon\right)^{-\frac{1}{2}}\left[\m{h}^H\m{R}_{\text{C}}^H\m{\Gamma}_n^\top\left(\frac{1}{L}\sum_{l=1}^L\overline{\m{y}_l}\m{y}_l^\top\right)\m{\Gamma}_n\m{R}_{\text{C}}\m{h}\right]\m{\eta},
\end{aligned}
\end{equation}
where $\m{\eta}$ is a tangent vector in the tangent space at $\m{h}\in\mathbb{S}_{\text{C}}^{N-1}$. Consequently, the algorithmic framework of Algorithm~1 remains unchanged, with only the Riemannian gradient and the Riemannian Hessian replaced by their complex-valued counterparts.

\section{Numerical Experiments}
\label{Sec:simulation}

In this section, we evaluate the performance of the proposed method. The proposed algorithm is implemented with a maximum of $2\times10^3$ iterations. The smoothing parameter is set to $\varepsilon=\min\left\{10^5N^{-5},10^{-4}\right\}$, the first-order stationarity tolerance is chosen as $\epsilon_g=10^{-2}$, and the second-order stationarity tolerance is set to $\epsilon_H=10^{-6}N^{-3}$. For computational efficiency, the initial step sizes for the gradient descent and the negative-curvature search are set to $10^{-2}$ and $10^{-4}$, respectively. These step sizes are subsequently refined using a backtracking line search strategy~\cite{hosseini2018line}.

\subsection{Random Signals}
In this subsection, we compare the proposed method with the existing approach developed in~\cite{shi2021manifold}. 
The method in~\cite{shi2021manifold} addresses the general multichannel sparse blind deconvolution problem using a general sparsity-promoting objective function, specifically the log-cosh objective $\sum_{l=1}^L\mu\log\cosh\left(\frac{1}{\mu}\mathcal{C}\left(\m{y}_l\right)\m{R}\m{h}\right)$, where $\log\cosh$ is applied entry-wise. The resulting optimization problem is solved using Riemannian gradient descent with multiple random initializations to seek a desirable solution. For the implementation of the method in~\cite{shi2021manifold}, we adopt the parameter settings recommended by the authors. Specifically, the maximum number of iterations is set to $2\times10^2$, the smoothing parameter is chosen as $\mu=\min\left\{10N^{-\frac{5}{4}},0.05\right\}$, and the initial step size is set to $10^{-1}$ and subsequently refined using a backtracking line search strategy. In addition, the number of random initializations is set to $10$.


In Experiment~1, we investigate the optimization landscapes induced by different sparsity-promoting objective functions for jointly sparse blind deconvolution. Specifically, we compare the exact $\ell_1$-based objective function, $\sum_{l=1}^L\left\|\mathcal{C}\left(\m{y}_l\right)\m{R}\m{h}\right\|_1$, the exact $\ell_{2,1}$-based objective function, $\sum_{n=1}^N\left(\frac{1}{L}\sum_{l=1}^L\left|\m{y}_l^\top\m{\Gamma}_n\m{R}\m{h}\right|^2\right)^{\frac{1}{2}}$, and the smooth objective function proposed in \eqref{opt-ini-general}, $\sum_{n=1}^N\left(\frac{1}{L}\sum_{l=1}^L\left|\m{y}_l^\top\m{\Gamma}_n\m{R}\m{h}\right|^2+\varepsilon\right)^{\frac{1}{2}}$, where $\varepsilon=10^{-4}$. We set $N=3$, $K=2$, and $\kappa=16$. The objective landscapes corresponding to $L=20$ and $L=2$ are shown in Fig.~\ref{fig.landscape}. It can be observed that, when the sample size is sufficiently large, all three objective functions exhibit favorable optimization landscapes, as their global minima coincide with the target solution. However, when the sample size is small, the global minimum of the $\ell_1$-based objective function no longer coincides with the target solution, suggesting that the corresponding optimization algorithm may fail to recover the true filter. In contrast, the global minimum of the exact $\ell_{2,1}$-based objective function remains aligned with the target solution. Moreover, the landscape induced by the proposed smooth objective function closely resembles that of the exact $\ell_{2,1}$ objective. These observations demonstrate that exploiting joint sparsity can substantially reduce the sample size required for successful recovery, while the proposed smooth objective function provides an accurate approximation to the exact $\ell_{2,1}$ objective.
\begin{figure}
\centering
\begin{minipage}[b]{.3\linewidth}
\includegraphics[width=2.0in]{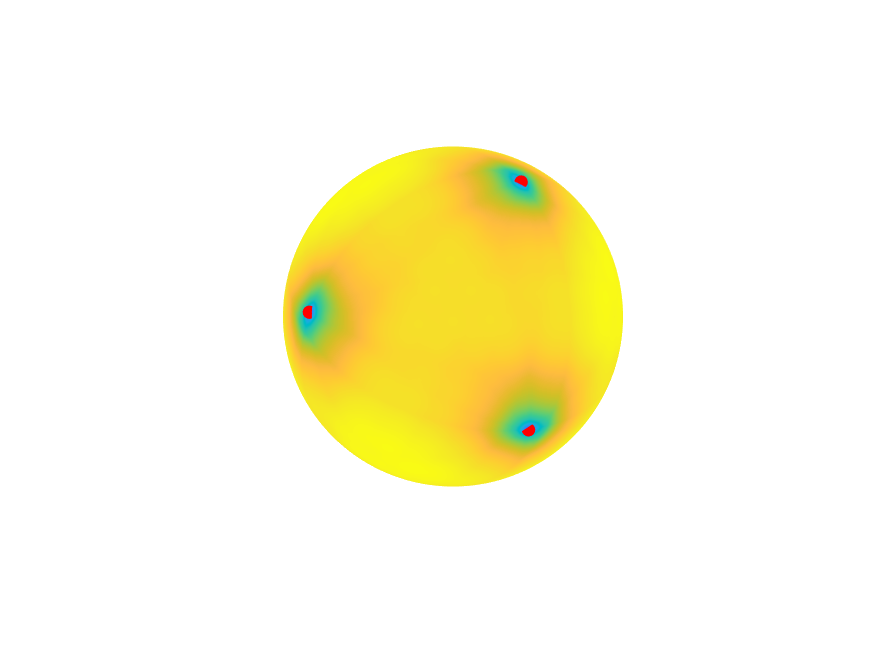}
\centering
\text{(a) $\ell_1$ objective, $L=20$}
\includegraphics[width=2.0in]{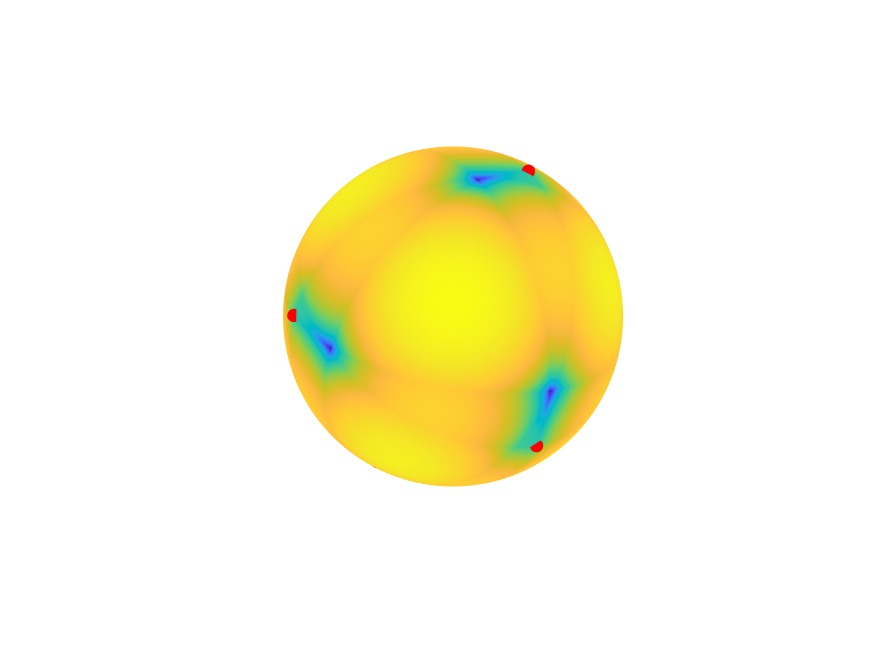}
\centering
\text{(b) $\ell_1$ objective, $L=2$}
\end{minipage}
\begin{minipage}[b]{.3\linewidth}
\includegraphics[width=2.0in]{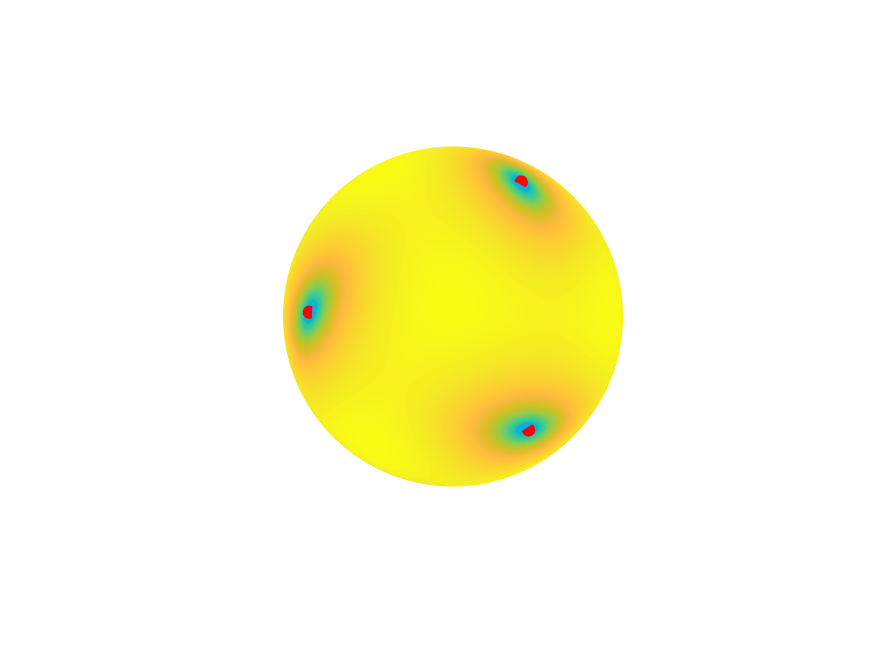}
\centering
\text{(c) $\ell_{2,1}$ objective, $L=20$}
\includegraphics[width=2.0in]{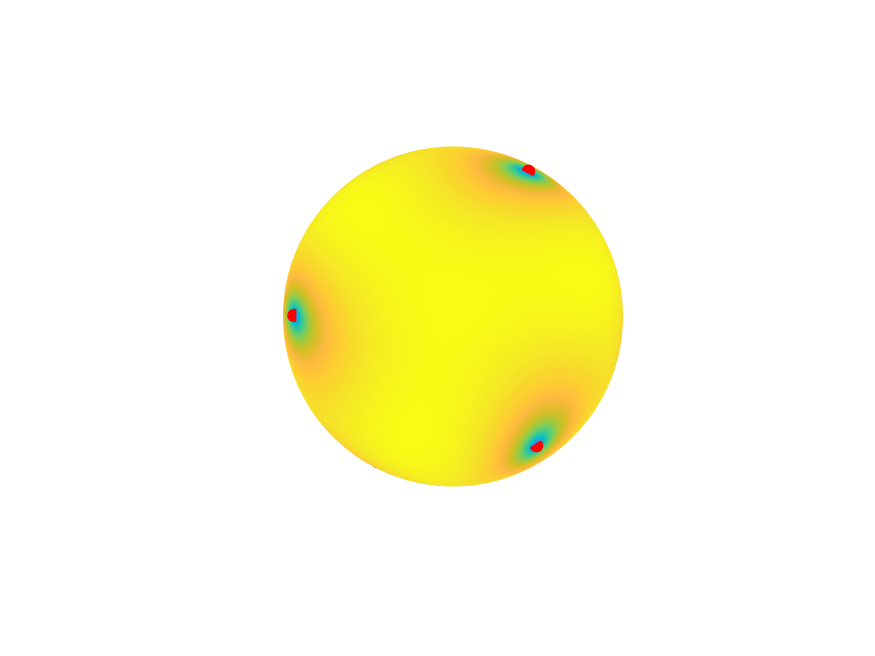}
\centering
\text{(d) $\ell_{2,1}$ objective, $L=2$}
\end{minipage}
\begin{minipage}[b]{.3\linewidth}
\includegraphics[width=2.0in]{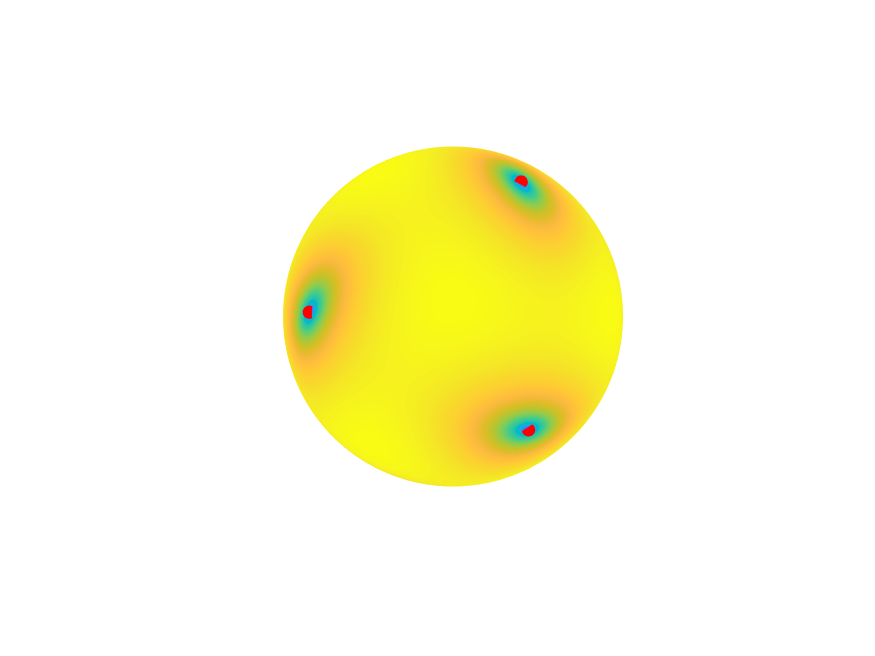}
\centering
\text{(e) Proposed objective, $L=20$}
\includegraphics[width=2.0in]{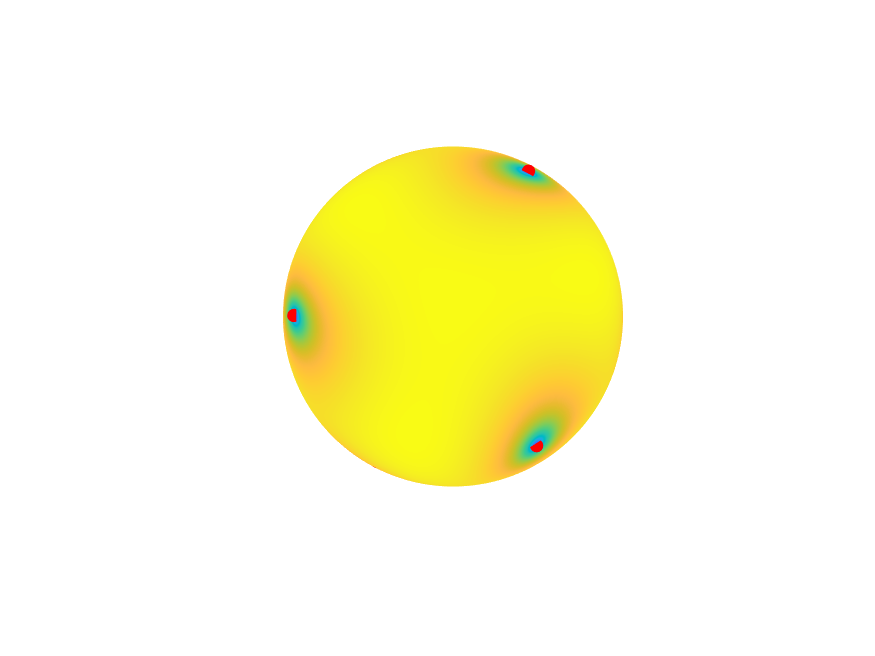}
\centering
\text{(f) Proposed objective, $L=2$}
\end{minipage}
\caption{Optimization landscapes corresponding to different sparsity-promoting objective functions. Darker colors indicate smaller objective values, whereas lighter colors indicate larger objective values. The red dots mark the target solutions.}
\label{fig.landscape}
\end{figure}

In Experiment~2, we evaluate the recovery success rate under different parameter settings. We adopt a simulation setup similar to that in~\cite{shi2021manifold}. The sparse input signals share a common support, where the support locations are generated at random and the nonzero entries are drawn independently from a standard Gaussian distribution. The filter $\m{g}$ is synthesized such that the circulant matrix $\mathcal{C}\left(\m{g}\right)$ has an approximately prescribed condition number $\widetilde{\kappa}$. Specifically, the filter is generated as follows. First, the magnitudes of $\widehat{\m{g}}$ are sampled independently from a uniform distribution over $\left[1,\widetilde{\kappa}\right]$, while their phases are sampled independently from a uniform distribution over $\left[0,2\pi\right]$. Second, Hermitian symmetry is imposed on $\widehat{\m{g}}$ to ensure that $\m{g}$ is real-valued, namely, $\widehat{{g}}_n=\overline{\widehat{{g}}_{\left(N+2-n\right)\bmod N}}$ for $n=1,\cdots,N$, where $\overline{\cdot}$ denotes complex conjugation. Recovery is declared successful if the estimated inverse representation of the filter, denoted by $\widetilde{\m{g}}_{\text{inv}}$, satisfies $\left\|\m{g}\circledast\widetilde{\m{g}}_{\text{inv}}\right\|_\infty/\left\|\m{g}\circledast\widetilde{\m{g}}_{\text{inv}}\right\|_2>0.99$. For each parameter setting, $100$ independent Monte Carlo trials are conducted to estimate the success rate. First, we fix $N=64$ and $\widetilde{\kappa}=16$, and vary $K\in\left\{1,3,\ldots,59\right\}$ and $L\in\left\{1,2,\ldots,30\right\}$. The corresponding results are shown in Fig.~\ref{fig.PT}(a) and Fig.~\ref{fig.PT}(b). Next, we fix $K=27$ and $\widetilde{\kappa}=16$, and vary $N\in\left\{32,34,\ldots,64\right\}$ and $L\in\left\{1,2,\ldots,30\right\}$. The results are presented in Fig.~\ref{fig.PT}(c) and Fig.~\ref{fig.PT}(d). Finally, we fix $N=64$ and $K=27$, and vary $\widetilde{\kappa}\in\left\{2,4,\ldots,2^{16}\right\}$ and $L\in\left\{1,2,\ldots,20\right\}$. The corresponding results are shown in Fig.~\ref{fig.PT}(e) and Fig.~\ref{fig.PT}(f). It can be observed that the proposed method consistently exhibits a substantially larger success region than the method of~\cite{shi2021manifold} and achieves successful recovery with significantly lower sample size across all parameter settings. These observations are consistent with the theoretical analysis and validate that exploiting joint sparsity reduces the sample complexity required for successful recovery. 
\begin{figure}
\centering
\begin{minipage}[b]{.3\linewidth}
\includegraphics[width=2.0in]{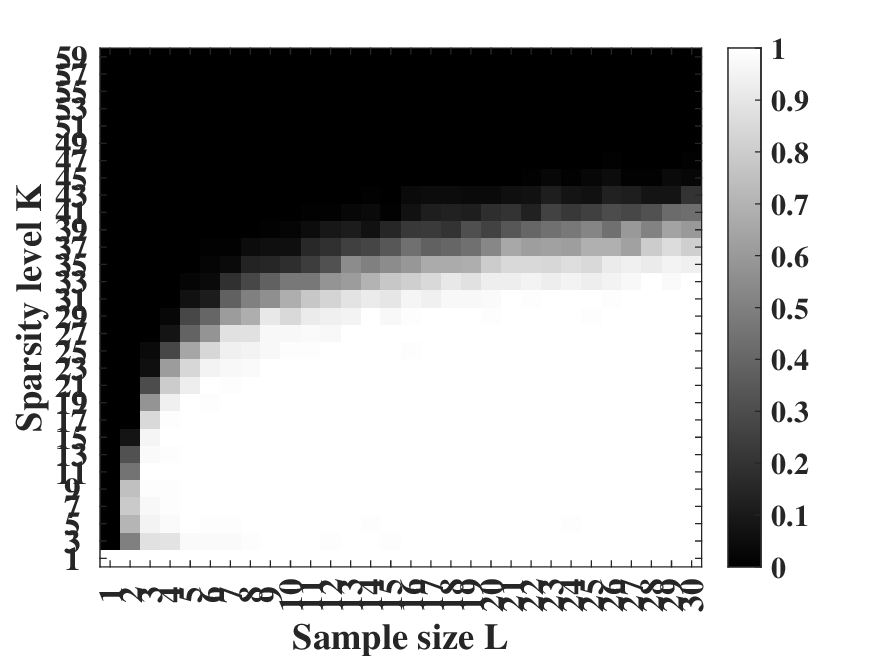}
\text{(a) \cite{shi2021manifold}, $K$ versus $L$}
\includegraphics[width=2.0in]{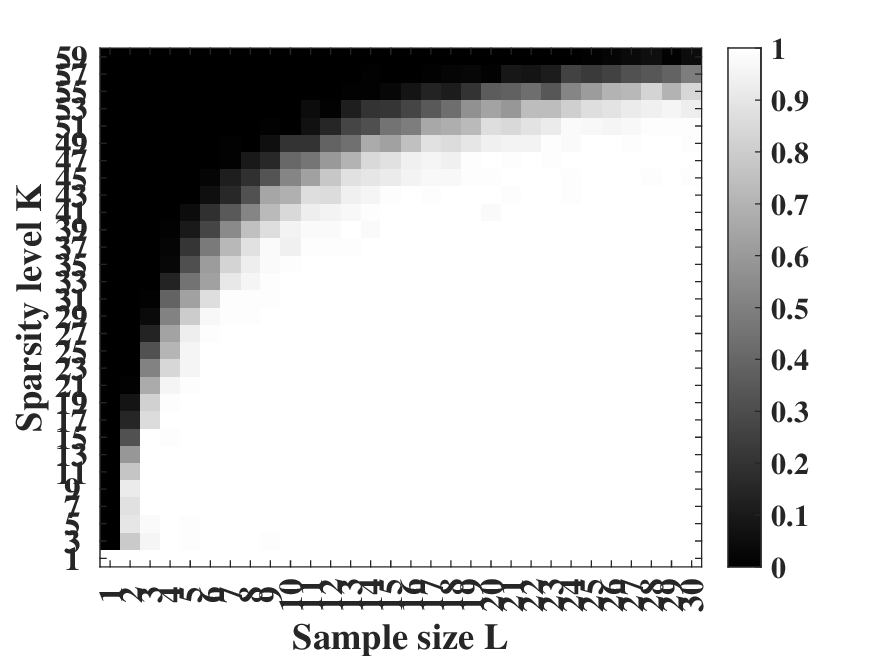}
\centering
\text{(b) Proposed, $K$ versus $L$}
\end{minipage}
\begin{minipage}[b]{.3\linewidth}
\includegraphics[width=2.0in]{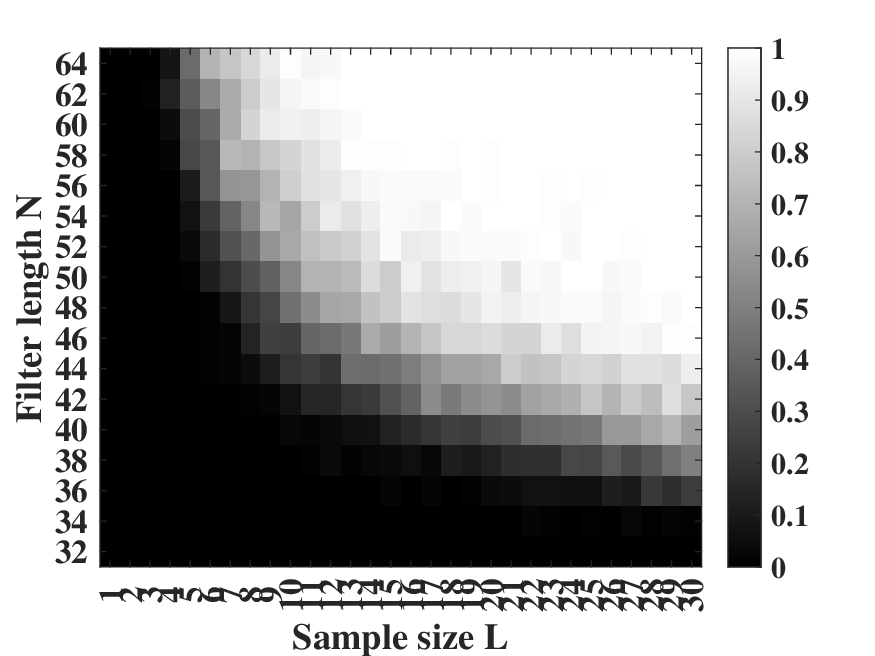}
\text{(c) \cite{shi2021manifold}, $N$ versus $L$}
\includegraphics[width=2.0in]{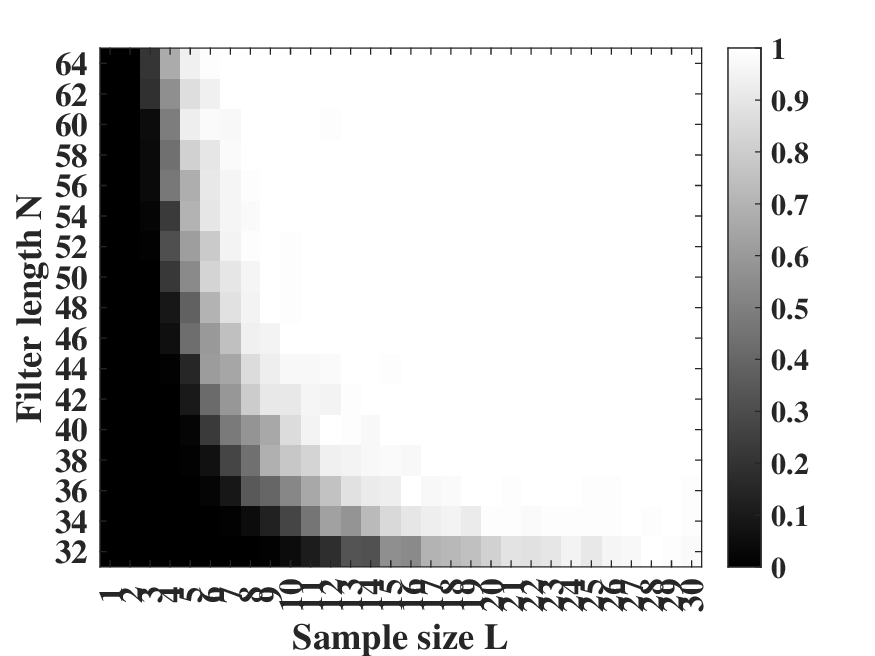}
\centering
\text{(d) Proposed, $N$ versus $L$}
\end{minipage}
\begin{minipage}[b]{.3\linewidth}
\includegraphics[width=2.0in]{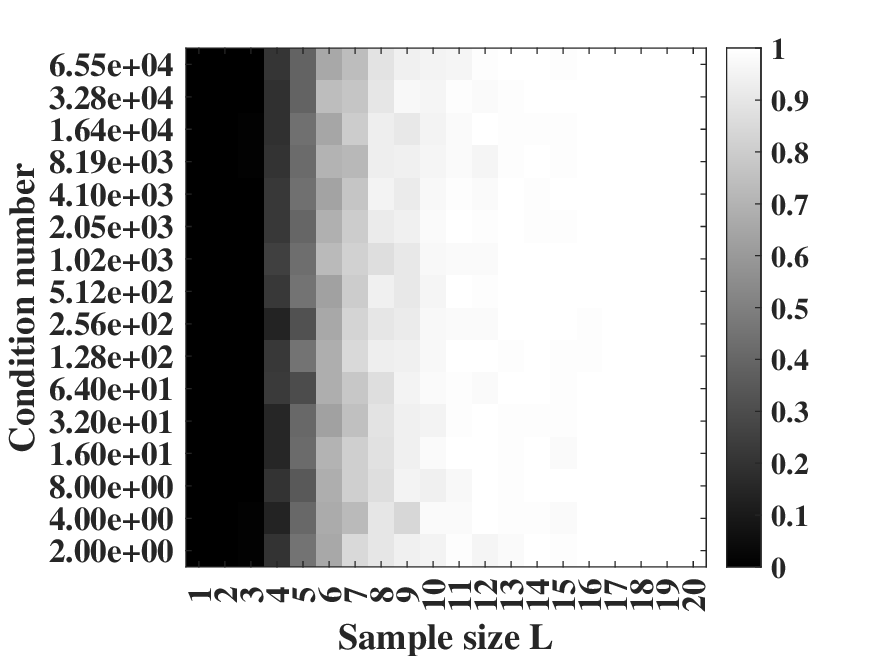}
\text{(e) \cite{shi2021manifold}, $\widetilde{\kappa}$ versus $L$}
\includegraphics[width=2.0in]{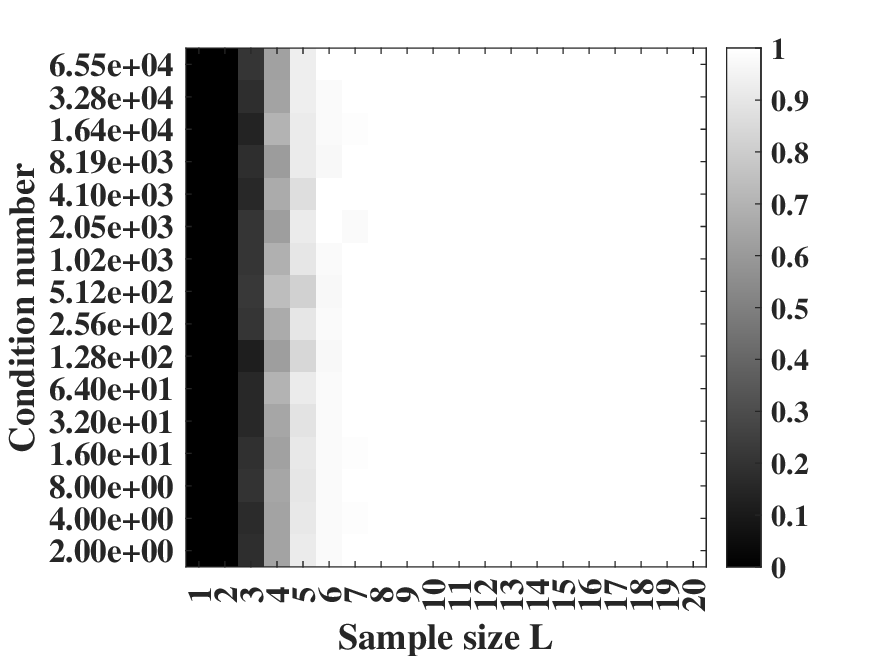}
\centering
\text{(f) Proposed, $\widetilde{\kappa}$ versus $L$}
\end{minipage}
\caption{Phase-transition behavior of jointly sparse blind deconvolution under different parameter settings. White indicates a success rate of $100\%$, whereas black indicates a success rate of $0\%$.}
\label{fig.PT}
\end{figure}

\subsection{Array Signal Processing Application}
In this subsection, we apply the proposed algorithm to a array signal processing application: multipath time-difference-of-arrival (TDOA) estimation with unknown transmitted waveforms~\cite{griffiths2022introduction}, which plays an important role in passive localization~\cite{fabrizio2013multipath}. We consider a reference-free passive radar receiver equipped with a uniform linear array consisting of $L$ antennas with half-wavelength inter-element spacing. The transmitted waveform propagates through a multipath environment with $K$ paths. Let $s$ denote the unknown transmitted waveform. The delay, complex path gain, and direction-of-arrival (DOA) associated with the $k$-th path are denoted by $\tau_k$, $\beta_k$, and $\omega_k$, respectively. WLOG, we assume that $\tau_1\leq\tau_2\leq\cdots\leq\tau_K$. The received signal at time $t$ at the $l$-th antenna is modeled as
\begin{equation}
\begin{aligned}
y_l\left(t\right)=\sum_{k=1}^K \beta_k\exp\left\{j2\pi\frac{\left(l-1\right)\cos\omega_k}{2}\right\}s\left(t-\tau_k\right)+n\left(t\right),\;t\in\left[T_{\text{p}},T_{\text{PRI}}\right],
\end{aligned}
\end{equation}
where $T_{\text{p}}$ denotes the pulse duration, i.e., $s\left(t\right)\neq0$ only for $t\in\left[0,T_{\text{p}}\right)$, $T_{\text{PRI}}$ denotes the pulse repetition interval (PRI), and $n\left(t\right)$ denotes additive noise. Sampling the received signal at a frequency $f_{\text{s}}$ yields the discrete-time observation matrix $\m{Y}\in\mathbb{C}^{N\times L}$, where $N=f_{\text{s}}\left(T_{\text{PRI}}-T_{\text{p}}\right)$. Suppose that $f_{\text{s}}\tau_k$ is an integer for each $k=1,\cdots,K$, which can be ensured when the sampling rate is sufficiently high. Then the discrete-time observations satisfy
\begin{equation}
\label{TDOA}
\begin{aligned}
Y_{n,l}=\sum_{k=1}^K \beta_k\exp\left\{j2\pi\frac{\left(l-1\right)\cos\omega_k}{2}\right\}s\left[n-f_{\text{s}}{\tau_k}\right]+N_{n,l},
\end{aligned}
\end{equation}
where $\m{N}\in\mathbb{C}^{N\times L}$ denotes the noise matrix and
\begin{equation}
\begin{aligned}
s\left[n-f_{\text{s}}{\tau_k}\right]=s\left(T_{\text{p}}+\frac{n-f_{\text{s}}{\tau_k}}{f_{\text{s}}}\right)=s\left(\frac{n}{f_{\text{s}}}+T_{\text{p}}-\tau_k\right).
\end{aligned}
\end{equation}
Furthermore, assume that $\tau_k\in\left[T_{\text{p}},T_{\text{PRI}}-T_{\text{p}}\right)$, which ensures that the received signals from a single pulse remain within a single PRI and do not overlap with the transmission interval. Define the filter $\m{g}\in\mathbb{C}^N$ and the input matrix $\m{X}\in\mathbb{C}^{N\times L}$ as
\begin{equation}
\begin{aligned}
g_n=&s\left[n\right],\\
X_{n,l}=&\beta_k\exp\left\{j2\pi\frac{\left(l-1\right)\cos\omega_k}{2}\right\}\mathbb{I}\left(n,f_{\text{s}}{\tau_k}\right),
\end{aligned}
\end{equation}
where $\mathbb{I}\left(n_1,n_2\right)=1$ if $n_1=n_2$, and $\mathbb{I}\left(n_1,n_2\right)=0$ otherwise. It follows that
\begin{equation}
\begin{aligned}
\m{Y}=\mathcal{C}\left(\m{g}\right)\m{X}+\m{N},
\end{aligned}
\end{equation}
where $\m{X}$ contains only $K$ nonzero rows. Therefore, the TDOA estimation problem can be formulated as a jointly sparse blind deconvolution problem. In practice, however, two important challenges must be taken into account: 1) the delays generally do not coincide with the sampling grid; and 2) the observations are polluted by noise. In the following experiments, we evaluate the TDOA estimation performance of the proposed algorithm under these two non-ideal factors through the estimated delay spectrum. The true power associated with the $k$-th path is given by $\left|\beta_k\right|^2$. The estimated delay spectrum is computed from the average row-wise energy of the recovered input signal matrix $\widetilde{\m{X}}$. Specifically, let $\widetilde{\m{p}}$ denote the estimated delay spectrum, whose $n$-th entry is defined as $\widetilde{p}_n=\frac{1}{L}\sum_{l=1}^L\left|\widetilde{\m{X}}_{n,l}\right|^2$. Note that blind deconvolution-based methods can recover only the relative delays, namely $\left\{\tau_k-\tau_1\right\}_{k=2}^K$. Therefore, when plotting the delay spectrum, we assume that the delay of the first path is known exactly and use it as the reference. For ease of comparison, both the true and estimated delay spectra are normalized so that their maximum values are equal to $1$. The pulse duration and PRI are set to $T_{\text{p}}=10~\mu\mathrm{s}$ and $T_{\text{PRI}}=40~\mu\mathrm{s}$, respectively, the sampling rate is $f_{\text{s}}=100$~MHz, yielding $N=3000$, and the number of antennas is $L=32$. The unknown transmitted waveform $s$ is chosen as a chirp signal with a bandwidth of $20$~MHz. We consider a multipath propagation environment with $K=3$ propagation paths, where the DOAs, delays, and complex path gains are generated randomly.

In Experiment~3, we investigate the performance of the proposed algorithm under different signal-to-noise ratio (SNR) levels, defined as $\frac{\left\|\mathcal{C}\left(\m{g}\right)\m{X}\right\|_2^2}{\left\|\m{N}\right\|_2^2}$, with values of $20~\mathrm{dB}$, $10~\mathrm{dB}$, and $0~\mathrm{dB}$. Note that it is highly nontrivial to extend the algorithm in~\cite{shi2021manifold} to the complex-valued setting, no comparison is provided. We consider two randomly generated sets of DOAs and propagation delays. The first set corresponds to a larger delay separation, whereas the second set represents a more challenging case with closely spaced delays: 1) DOAs are $\left\{142.3612,59.3150,114.7363\right\}^\circ$ and delays are $\left\{15.8764,17.2490,18.5540\right\}~\mu\mathrm{s}$; and 2) DOAs are $\left\{125.6770,40.8769,53.2631\right\}^\circ$ and delays are $\left\{17.3662,26.8032,26.9704\right\}~\mu\mathrm{s}$. Note that the delays in both cases do not lie on the sampling grid. The results are shown in Fig.~\ref{fig.TDOA}. The proposed algorithm provides stable TDOA estimates across all SNR levels, despite the combined effects of off-grid delays and additive noise.
\begin{figure}
\centering
\begin{minipage}[b]{.3\linewidth}
\includegraphics[width=2.0in]{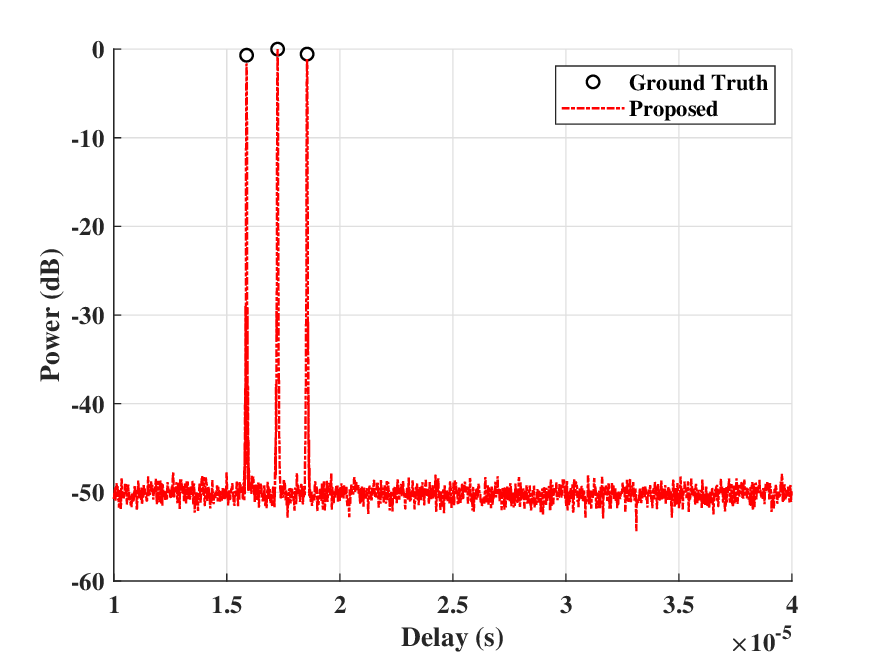}
\text{(a) Set 1, $\mathrm{SNR}=20~\mathrm{dB}$}
\includegraphics[width=2.0in]{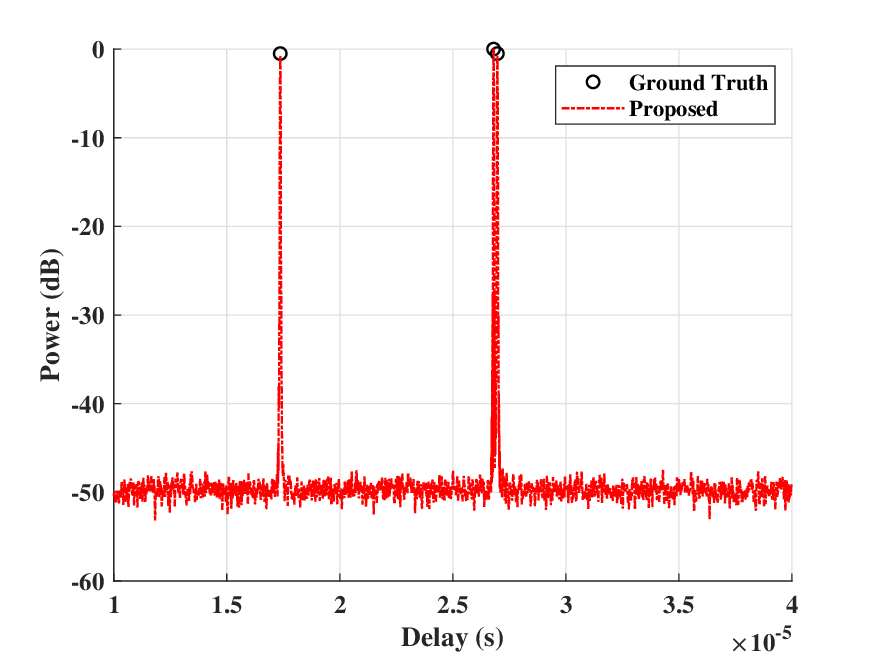}
\centering
\text{(d) Set 2, $\mathrm{SNR}=20~\mathrm{dB}$}
\end{minipage}
\begin{minipage}[b]{.3\linewidth}
\includegraphics[width=2.0in]{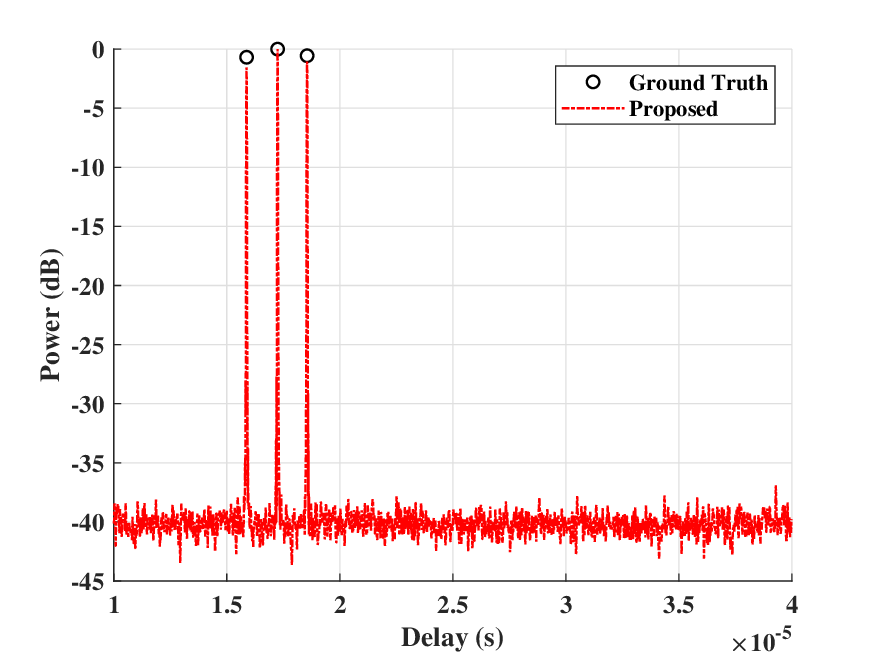}
\text{(b) Set 1, $\mathrm{SNR}=10~\mathrm{dB}$}
\includegraphics[width=2.0in]{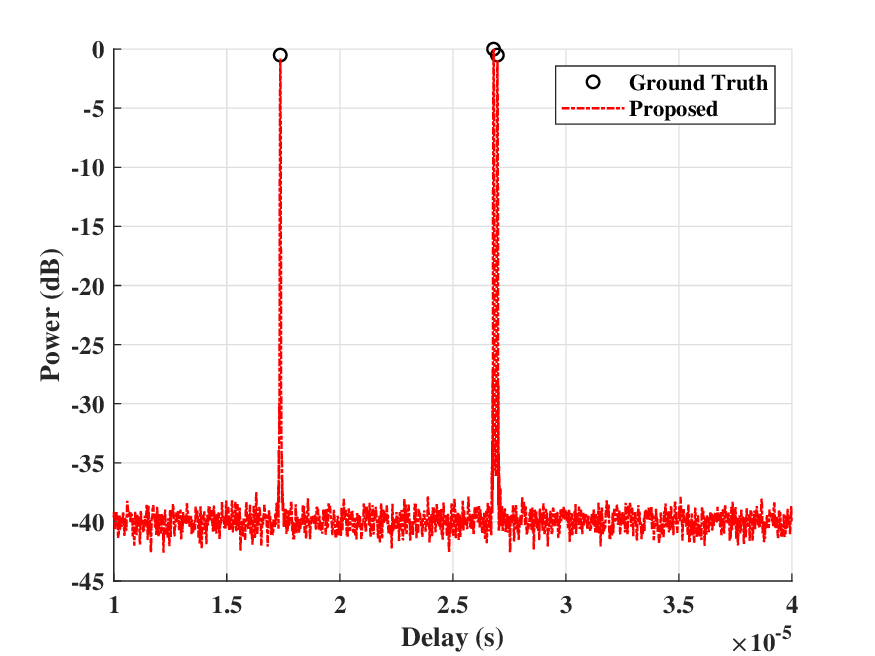}
\centering
\text{(e) Set 2, $\mathrm{SNR}=10~\mathrm{dB}$}
\end{minipage}
\begin{minipage}[b]{.3\linewidth}
\includegraphics[width=2.0in]{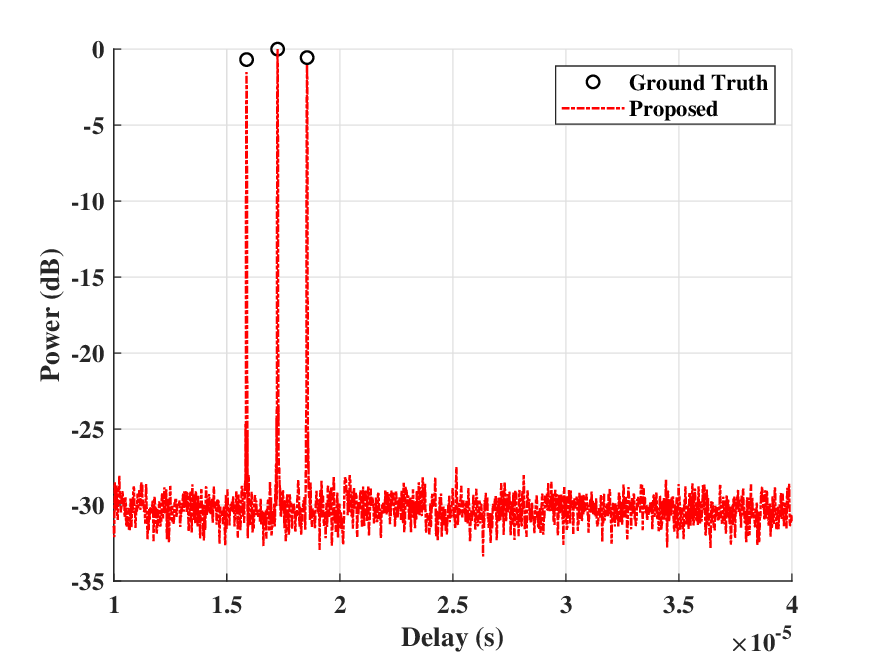}
\text{(c) Set 1, $\mathrm{SNR}=0~\mathrm{dB}$}
\includegraphics[width=2.0in]{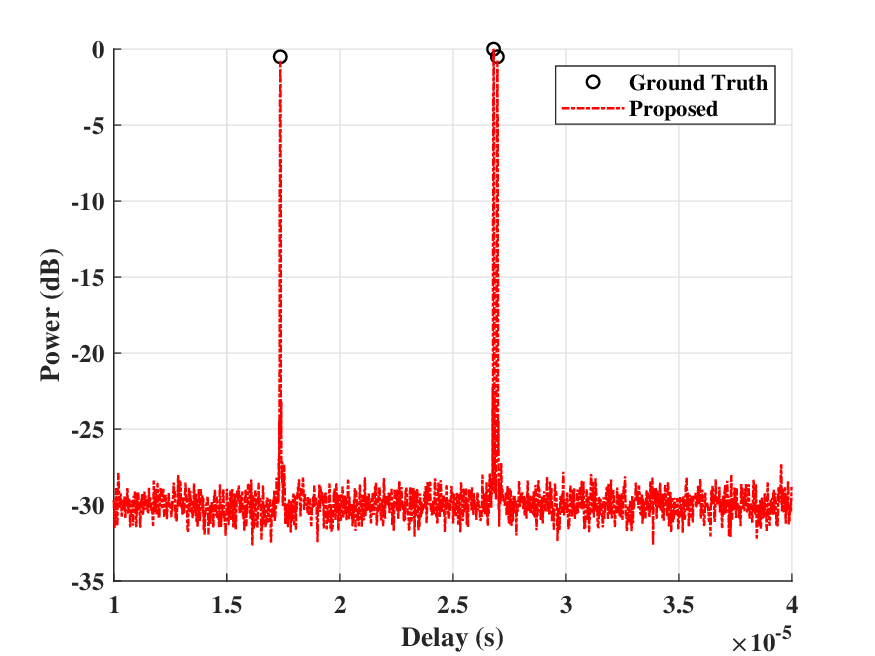}
\centering
\text{(f) Set 2, $\mathrm{SNR}=0~\mathrm{dB}$}
\end{minipage}
\caption{Estimated delay spectra under different SNR levels for two parameter sets.}
\label{fig.TDOA}
\end{figure}

In summary, the simulation results demonstrate the performance gains achieved by exploiting joint sparsity in blind deconvolution and validate the effectiveness of the proposed algorithm.

\section{Conclusion}
\label{Sec:conclusion}

This paper develops an optimization framework and establishes theoretical guarantees for jointly sparse blind deconvolution. Specifically, we formulate the jointly sparse blind deconvolution problem as a Riemannian optimization problem on the unit sphere and propose an algorithm that is guaranteed to compute an approximate second-order stationary point within a finite number of iterations. Furthermore, we establish a non-asymptotic relationship between the estimation error and the sample size, thereby characterizing the recovery performance of the proposed method under finite-sample conditions. Numerical experiments validate both the theoretical analysis and the effectiveness of the proposed algorithm, including its application to array signal processing.

While our analysis demonstrates that exploiting joint sparsity can substantially reduce the sample complexity required for successful recovery, several important directions remain for future research. In particular, it is of interest to extend the proposed framework to more realistic noisy settings and to develop corresponding theoretical guarantees. This includes the design of robust algorithms and the derivation of non-asymptotic error bounds that explicitly characterize the impact of measurement noise on recovery performance.


%
%
%

\suppproof{The proofs of several auxiliary lemmas and technical propositions are provided in the Supplementary Material. These results primarily involve technical estimates and intermediate arguments used in the analysis of the main theorems.}


\putbib[strings]

\end{bibunit}

\newpage

\begin{bibunit}[IEEEtran]

\title{Supplementary Material for\\Jointly Sparse Blind Deconvolution via Riemannian Optimization}

\author{Wenlong Wang$^1$\orcid{0009-0001-5855-8723}, Baiyang Guo$^2$\orcid{0009-0003-2399-0879}, Zai Yang$^{1,*}$\orcid{0000-0002-9502-5176}, Shixiang Chen$^3$\orcid{0000-0002-3261-0714} and Junpeng Shi$^2$\orcid{0000-0002-9910-0663}}

\affil{$^1$School of Mathematics and Statistics, Xi’an Jiaotong University, Xi’an, China}

\affil{$^2$College of Electronic Engineering, National University of Defense Technology, Hefei, China}

\affil{$^3$School of Mathematical Sciences,University of Science and Technology of China, Hefei, China}

\affil{$^*$Author to whom any correspondence should be addressed.}

\email{wang1813857265@stu.xjtu.edu.cn, guobaiyang@nudt.edu.cn, yangzai@xjtu.edu.cn, shxchen@ustc.edu.cn and shijunpeng20@nudt.edu.cn}

This supplementary material provides complete proofs of the theoretical results presented in the main manuscript. In particular, it contains the proofs of several auxiliary lemmas and technical propositions stated in the main text, whose proofs are deferred here to improve the readability of the manuscript.

\section{Proof of Lemma~\ref{lem-diff}}
\label{pf-lem-diff}
It follows from \eqref{prop-max} that
\begin{equation}
\begin{aligned}
\min_{n\in\chi_o}\m{h}^\top\m{\Gamma}_n^\top\m{\Sigma}_{\Omega}\m{\Gamma}_n\m{h}\leq Kh_o^2.
\end{aligned}
\end{equation}
It remains to show that
\begin{equation}
\label{lem-diff-2}
\begin{aligned}
\min_{n\in\chi_o}\m{h}^\top\m{\Gamma}_n^\top\m{\Sigma}_{\Omega}\m{\Gamma}_n\m{h}\leq1-\frac{1-h_o^2}{K}.
\end{aligned}
\end{equation}
We establish \eqref{lem-diff-2} by contradiction. Suppose that, for all $n\in\chi_o$,
\begin{equation}
\label{contradiction}
\begin{aligned}
\m{h}^\top\m{\Gamma}_{n}^\top\m{\Sigma}_{\Omega}\m{\Gamma}_{n}\m{h}>1-\frac{1-h_o^2}{K}.
\end{aligned}
\end{equation}
Then
\begin{equation}
\begin{aligned}
\m{h}^\top\m{\Gamma}_{n}^\top\m{\Sigma}_{\Omega}\m{\Gamma}_{n}\m{h}-h_o^2>\frac{\left(K-1\right)\left(1-h_o^2\right)}{K}.
\end{aligned}
\end{equation}
Consequently, there exists an index $\widetilde{o}$ such that
\begin{equation}
\begin{aligned}
h_{\widetilde{o}}^2>\frac{1-h_o^2}{K}.
\end{aligned}
\end{equation}
Since $\Omega$ is non-uniform, there exists $\widetilde{n}\in\chi_o$ such that $\widetilde{n}\notin\chi_{\widetilde{o}}$, namely,
\begin{equation}
\begin{aligned}
\m{h}^\top\m{\Gamma}_{\widetilde{n}}^\top\m{\Sigma}_{\Omega}\m{\Gamma}_{\widetilde{n}}\m{h}\leq1-h_{\widetilde{o}}^2<1-\frac{1-h_o^2}{K},
\end{aligned}
\end{equation}
which contradicts \eqref{contradiction}. Therefore, \eqref{lem-diff-2} holds, completing the proof.

\section{Proof of Lemma~\ref{lem-2st}}
\label{pf-lem-2st}
The Riemannian Hessian of $\psi_\varepsilon^0$ is
\begin{equation}
\begin{aligned}
\text{Hess}\;\psi_\varepsilon^0\left(\m{h}\right)=&\left(\m{I}-\m{h}\m{h}^\top\right)\Bigg\{\sum_{n=1}^N\left[\m{h}^\top\m{\Gamma}_n^\top\left(\m{\Sigma}_{\Omega}+\m{E}_{\Omega}\right)\m{\Gamma}_n\m{h}+\varepsilon\right]^{-\frac{1}{2}}\m{\Gamma}_n^\top\left(\m{\Sigma}_{\Omega}+\m{E}_{\Omega}\right)\m{\Gamma}_n\\
&-\sum_{n=1}^N\left[\m{h}^\top\m{\Gamma}_n^\top\left(\m{\Sigma}_{\Omega}+\m{E}_{\Omega}\right)\m{\Gamma}_n\m{h}+\varepsilon\right]^{-\frac{3}{2}}\m{\Gamma}_n^\top\left(\m{\Sigma}_{\Omega}+\m{E}_{\Omega}\right)\m{\Gamma}_n\m{h}\m{h}^\top\m{\Gamma}_n^\top\left(\m{\Sigma}_{\Omega}+\m{E}_{\Omega}\right)\m{\Gamma}_n\\
&-\sum_{n=1}^N\left[\m{h}^\top\m{\Gamma}_n^\top\left(\m{\Sigma}_{\Omega}+\m{E}_{\Omega}\right)\m{\Gamma}_n\m{h}+\varepsilon\right]^{-\frac{1}{2}}\left[\m{h}^\top\m{\Gamma}_n^\top\left(\m{\Sigma}_{\Omega}+\m{E}_{\Omega}\right)\m{\Gamma}_n\m{h}\right]\m{I}\Bigg\}\left(\m{I}-\m{h}\m{h}^\top\right)\\
=&\left(\m{I}-\m{h}\m{h}^\top\right)\Bigg\{\sum_{n=1}^N\left[\m{h}^\top\m{\Gamma}_n^\top\left(\m{\Sigma}_{\Omega}+\m{E}_{\Omega}\right)\m{\Gamma}_n\m{h}+\varepsilon\right]^{-\frac{3}{2}}\m{H}_n\Bigg\}\left(\m{I}-\m{h}\m{h}^\top\right),
\end{aligned}
\end{equation}
where
\begin{equation}
\begin{aligned}
\m{H}_n=&\left[\m{h}^\top\m{\Gamma}_n^\top\left(\m{\Sigma}_{\Omega}+\m{E}_{\Omega}\right)\m{\Gamma}_n\m{h}+\varepsilon\right]\left\{\m{\Gamma}_n^\top\left(\m{\Sigma}_{\Omega}+\m{E}_{\Omega}\right)\m{\Gamma}_n-\left[\m{h}^\top\m{\Gamma}_n^\top\left(\m{\Sigma}_{\Omega}+\m{E}_{\Omega}\right)\m{\Gamma}_n\m{h}\right]\m{I}\right\}\\
&-\m{\Gamma}_n^\top\left(\m{\Sigma}_{\Omega}+\m{E}_{\Omega}\right)\m{\Gamma}_n\m{h}\m{h}^\top\m{\Gamma}_n^\top\left(\m{\Sigma}_{\Omega}+\m{E}_{\Omega}\right)\m{\Gamma}_n.
\end{aligned}
\end{equation}
It is straightforward to verify that
\begin{equation}
\begin{aligned}
\m{d}^\top\m{H}_n\m{d}=&\m{d}^\top\Bigg\{\left[\m{h}^\top\m{\Gamma}_n^\top\left(\m{\Sigma}_{\Omega}+\m{E}_{\Omega}\right)\m{\Gamma}_n\m{h}+\varepsilon\right]\left\{\m{\Gamma}_n^\top\left(\m{\Sigma}_{\Omega}+\m{E}_{\Omega}\right)\m{\Gamma}_n-\left[\m{h}^\top\m{\Gamma}_n^\top\left(\m{\Sigma}_{\Omega}+\m{E}_{\Omega}\right)\m{\Gamma}_n\m{h}\right]\m{I}\right\}\\
&-\m{\Gamma}_n^\top\left(\m{\Sigma}_{\Omega}+\m{E}_{\Omega}\right)\m{\Gamma}_n\m{h}\m{h}^\top\m{\Gamma}_n^\top\left(\m{\Sigma}_{\Omega}+\m{E}_{\Omega}\right)\m{\Gamma}_n\Bigg\}\m{d}\\
=&\left[\m{h}^\top\m{\Gamma}_n^\top\left(\m{\Sigma}_{\Omega}+\m{E}_{\Omega}\right)\m{\Gamma}_n\m{h}+\varepsilon\right]\left\{\m{d}^\top\m{\Gamma}_n^\top\left(\m{\Sigma}_{\Omega}+\m{E}_{\Omega}\right)\m{\Gamma}_n\m{d}-\left[\m{h}^\top\m{\Gamma}_n^\top\left(\m{\Sigma}_{\Omega}+\m{E}_{\Omega}\right)\m{\Gamma}_n\m{h}\right]\left\|\m{d}\right\|_2^2\right\}\\
&-\m{d}^\top\m{\Gamma}_n^\top\left(\m{\Sigma}_{\Omega}+\m{E}_{\Omega}\right)\m{\Gamma}_n\m{h}\m{h}^\top\m{\Gamma}_n^\top\left(\m{\Sigma}_{\Omega}+\m{E}_{\Omega}\right)\m{\Gamma}_n\m{d},
\end{aligned}
\end{equation}
where
\begin{equation}
\begin{aligned}
\m{d}^\top\m{\Gamma}_n^\top\left(\m{\Sigma}_{\Omega}+\m{E}_{\Omega}\right)\m{\Gamma}_n\m{d}=&\m{e}_o^\top\m{\Gamma}_n^\top\left(\m{\Sigma}_{\Omega}+\m{E}_{\Omega}\right)\m{\Gamma}_n\m{e}_o-2h_o\m{h}^\top\m{\Gamma}_n^\top\left(\m{\Sigma}_{\Omega}+\m{E}_{\Omega}\right)\m{\Gamma}_n\m{e}_o\\
&+h_o^2\m{h}^\top\m{\Gamma}_n^\top\left(\m{\Sigma}_{\Omega}+\m{E}_{\Omega}\right)\m{\Gamma}_n\m{h},
\end{aligned}
\end{equation}
and
\begin{equation}
\begin{aligned}
&\m{d}^\top\m{\Gamma}_n^\top\left(\m{\Sigma}_{\Omega}+\m{E}_{\Omega}\right)\m{\Gamma}_n\m{h}\m{h}^\top\m{\Gamma}_n^\top\left(\m{\Sigma}_{\Omega}+\m{E}_{\Omega}\right)\m{\Gamma}_n\m{d}\\
=&\left[\m{e}_o^\top\m{\Gamma}_n^\top\left(\m{\Sigma}_{\Omega}+\m{E}_{\Omega}\right)\m{\Gamma}_n\m{h}\right]^2-2h_o\m{e}_o^\top\m{\Gamma}_n^\top\left(\m{\Sigma}_{\Omega}+\m{E}_{\Omega}\right)\m{\Gamma}_n\m{h}\m{h}^\top\m{\Gamma}_n^\top\left(\m{\Sigma}_{\Omega}+\m{E}_{\Omega}\right)\m{\Gamma}_n\m{h}\\
&+h_o^2\left[\m{h}^\top\m{\Gamma}_n^\top\left(\m{\Sigma}_{\Omega}+\m{E}_{\Omega}\right)\m{\Gamma}_n\m{h}\right]^2.
\end{aligned}
\end{equation}
Consequently,
\begin{equation}
\begin{aligned}
\m{d}^\top\m{H}_n\m{d}
=&-\left[\m{h}^\top\m{\Gamma}_n^\top\left(\m{\Sigma}_{\Omega}+\m{E}_{\Omega}\right)\m{\Gamma}_n\m{h}+\varepsilon\right]\left[\m{h}^\top\m{\Gamma}_n^\top\left(\m{\Sigma}_{\Omega}+\m{E}_{\Omega}\right)\m{\Gamma}_n\m{h}\right]\left\|\m{d}\right\|_2^2\\
&+\left[\m{h}^\top\m{\Gamma}_n^\top\left(\m{\Sigma}_{\Omega}+\m{E}_{\Omega}\right)\m{\Gamma}_n\m{h}+\varepsilon\right]\m{e}_o^\top\m{\Gamma}_n^\top\m{\Sigma}_{\Omega}\m{\Gamma}_n\m{e}_o\\
&+\left[\m{h}^\top\m{\Gamma}_n^\top\left(\m{\Sigma}_{\Omega}+\m{E}_{\Omega}\right)\m{\Gamma}_n\m{h}+\varepsilon\right]\m{e}_o^\top\m{\Gamma}_n^\top\m{E}_{\Omega}\m{\Gamma}_n\m{e}_o\\
&+\left[\m{h}^\top\m{\Gamma}_n^\top\left(\m{\Sigma}_{\Omega}+\m{E}_{\Omega}\right)\m{\Gamma}_n\m{h}+\varepsilon\right]\left[-2h_o\m{h}^\top\m{\Gamma}_n^\top\left(\m{\Sigma}_{\Omega}+\m{E}_{\Omega}\right)\m{\Gamma}_n\m{e}_o\right]\\
&+\left[\m{h}^\top\m{\Gamma}_n^\top\left(\m{\Sigma}_{\Omega}+\m{E}_{\Omega}\right)\m{\Gamma}_n\m{h}+\varepsilon\right]\left[h_o^2\m{h}^\top\m{\Gamma}_n^\top\left(\m{\Sigma}_{\Omega}+\m{E}_{\Omega}\right)\m{\Gamma}_n\m{h}\right]\\
&-\left[\m{e}_o^\top\m{\Gamma}_n^\top\left(\m{\Sigma}_{\Omega}+\m{E}_{\Omega}\right)\m{\Gamma}_n\m{h}\right]^2-h_o^2\left[\m{h}^\top\m{\Gamma}_n^\top\left(\m{\Sigma}_{\Omega}+\m{E}_{\Omega}\right)\m{\Gamma}_n\m{h}\right]^2\\
&+2h_o\left[\m{h}^\top\m{\Gamma}_n^\top\left(\m{\Sigma}_{\Omega}+\m{E}_{\Omega}\right)\m{\Gamma}_n\m{h}\right]\m{e}_o^\top\m{\Gamma}_n^\top\left(\m{\Sigma}_{\Omega}+\m{E}_{\Omega}\right)\m{\Gamma}_n\m{h}.
\end{aligned}
\end{equation}
Furthermore, we have
\begin{equation}
\begin{aligned}
\m{d}^\top\m{H}_n\m{d}
=&-\left[\m{h}^\top\m{\Gamma}_n^\top\left(\m{\Sigma}_{\Omega}+\m{E}_{\Omega}\right)\m{\Gamma}_n\m{h}+\varepsilon\right]\left[\m{h}^\top\m{\Gamma}_n^\top\left(\m{\Sigma}_{\Omega}+\m{E}_{\Omega}\right)\m{\Gamma}_n\m{h}\right]\left\|\m{d}\right\|_2^2\\
&+\left[\m{h}^\top\m{\Gamma}_n^\top\left(\m{\Sigma}_{\Omega}+\m{E}_{\Omega}\right)\m{\Gamma}_n\m{h}+\varepsilon\right]\m{e}_o^\top\m{\Gamma}_n^\top\m{\Sigma}_{\Omega}\m{\Gamma}_n\m{e}_o\\
&-\left[\m{e}_o^\top\m{\Gamma}_n^\top\left(\m{\Sigma}_{\Omega}+\m{E}_{\Omega}\right)\m{\Gamma}_n\m{h}\right]^2\\
&-2h_o\varepsilon\left[\m{h}^\top\m{\Gamma}_n^\top\left(\m{\Sigma}_{\Omega}+\m{E}_{\Omega}\right)\m{\Gamma}_n\m{e}_o\right]\\
&+\left[\m{h}^\top\m{\Gamma}_n^\top\left(\m{\Sigma}_{\Omega}+\m{E}_{\Omega}\right)\m{\Gamma}_n\m{h}+\varepsilon\right]\m{e}_o^\top\m{\Gamma}_n^\top\m{E}_{\Omega}\m{\Gamma}_n\m{e}_o\\
&+\varepsilon h_o^2\m{h}^\top\m{\Gamma}_n^\top\left(\m{\Sigma}_{\Omega}+\m{E}_{\Omega}\right)\m{\Gamma}_n\m{h}.
\end{aligned}
\end{equation}
Summing over $n$ yields
\begin{equation}
\label{res-main-term}
\begin{aligned}
\frac{\m{d}^\top\text{Hess}\;\psi_\varepsilon^0\left(\m{h}\right)\m{d}}{\left\|\m{d}\right\|_2^2}=&\frac{1}{\left\|\m{d}\right\|_2^2}\sum_{n=1}^N\left[\m{h}^\top\m{\Gamma}_n^\top\left(\m{\Sigma}_{\Omega}+\m{E}_{\Omega}\right)\m{\Gamma}_n\m{h}+\varepsilon\right]^{-\frac{3}{2}}\m{d}^\top\m{H}_n\m{d}=\text{main}+\text{res},
\end{aligned}
\end{equation}
where
\begin{equation}
\begin{aligned}
\text{main}=&-\sum_{n=1}^N\left[\m{h}^\top\m{\Gamma}_n^\top\left(\m{\Sigma}_{\Omega}+\m{E}_{\Omega}\right)\m{\Gamma}_n\m{h}+\varepsilon\right]^{-\frac{1}{2}}\left(\m{h}^\top\m{\Gamma}_n^\top\m{\Sigma}_{\Omega}\m{\Gamma}_n\m{h}\right)\\
&+\frac{1}{\left\|\m{d}\right\|_2^2}\sum_{n\in\chi_o}\left[\m{h}^\top\m{\Gamma}_n^\top\left(\m{\Sigma}_{\Omega}+\m{E}_{\Omega}\right)\m{\Gamma}_n\m{h}+\varepsilon\right]^{-\frac{1}{2}}\\
&-\frac{h_o^2}{\left\|\m{d}\right\|_2^2}\sum_{n\in\chi_o}\left[\m{h}^\top\m{\Gamma}_n^\top\left(\m{\Sigma}_{\Omega}+\m{E}_{\Omega}\right)\m{\Gamma}_n\m{h}+\varepsilon\right]^{-\frac{3}{2}},
\end{aligned}
\end{equation}
and
\begin{equation}
\begin{aligned}
\text{res}=&-\sum_{n=1}^N\left[\m{h}^\top\m{\Gamma}_n^\top\left(\m{\Sigma}_{\Omega}+\m{E}_{\Omega}\right)\m{\Gamma}_n\m{h}+\varepsilon\right]^{-\frac{1}{2}}\left(\m{h}^\top\m{\Gamma}_n^\top\m{E}_{\Omega}\m{\Gamma}_n\m{h}\right)\\
&-\frac{2h_o}{\left\|\m{d}\right\|_2^2}\sum_{n=1}^N\left[\m{h}^\top\m{\Gamma}_n^\top\left(\m{\Sigma}_{\Omega}+\m{E}_{\Omega}\right)\m{\Gamma}_n\m{h}+\varepsilon\right]^{-\frac{3}{2}}\left(\m{h}^\top\m{\Gamma}_n^\top\m{E}_{\Omega}\m{\Gamma}_n\m{e}_o\right)\\
&-\frac{1}{\left\|\m{d}\right\|_2^2}\sum_{n=1}^N\left[\m{h}^\top\m{\Gamma}_n^\top\left(\m{\Sigma}_{\Omega}+\m{E}_{\Omega}\right)\m{\Gamma}_n\m{h}+\varepsilon\right]^{-\frac{3}{2}}\left[\m{e}_o^\top\m{\Gamma}_n^\top\m{E}_{\Omega}\m{\Gamma}_n\m{h}\right]^2\\
&-\frac{2h_o^2\varepsilon}{\left\|\m{d}\right\|_2^2}\sum_{n\in\chi_o}\left[\m{h}^\top\m{\Gamma}_n^\top\left(\m{\Sigma}_{\Omega}+\m{E}_{\Omega}\right)\m{\Gamma}_n\m{h}+\varepsilon\right]^{-\frac{3}{2}}\\
&-\frac{2h_o\varepsilon}{\left\|\m{d}\right\|_2^2}\sum_{n=1}^N\left[\m{h}^\top\m{\Gamma}_n^\top\left(\m{\Sigma}_{\Omega}+\m{E}_{\Omega}\right)\m{\Gamma}_n\m{h}+\varepsilon\right]^{-\frac{3}{2}}\left[\m{h}^\top\m{\Gamma}_n^\top\m{E}_{\Omega}\m{\Gamma}_n\m{e}_o\right]\\
&+\frac{1}{\left\|\m{d}\right\|_2^2}\sum_{n=1}^N\left[\m{h}^\top\m{\Gamma}_n^\top\left(\m{\Sigma}_{\Omega}+\m{E}_{\Omega}\right)\m{\Gamma}_n\m{h}+\varepsilon\right]^{-\frac{1}{2}}\m{e}_o^\top\m{\Gamma}_n^\top\m{E}_{\Omega}\m{\Gamma}_n\m{e}_o\\
&+\frac{\varepsilon h_o^2}{\left\|\m{d}\right\|_2^2}\sum_{n=1}^N\left[\m{h}^\top\m{\Gamma}_n^\top\left(\m{\Sigma}_{\Omega}+\m{E}_{\Omega}\right)\m{\Gamma}_n\m{h}+\varepsilon\right]^{-\frac{3}{2}}\m{h}^\top\m{\Gamma}_n^\top\left(\m{\Sigma}_{\Omega}+\m{E}_{\Omega}\right)\m{\Gamma}_n\m{h}.
\end{aligned}
\end{equation}
By discarding nonpositive terms, the residual term can be further bounded as
\begin{equation}
\begin{aligned}
\text{res}
\leq&-\sum_{n=1}^N\left[\m{h}^\top\m{\Gamma}_n^\top\left(\m{\Sigma}_{\Omega}+\m{E}_{\Omega}\right)\m{\Gamma}_n\m{h}+\varepsilon\right]^{-\frac{1}{2}}\left(\m{h}^\top\m{\Gamma}_n^\top\m{E}_{\Omega}\m{\Gamma}_n\m{h}\right)\\
&-\frac{2h_o\left(1+\varepsilon\right)}{\left\|\m{d}\right\|_2^2}\sum_{n=1}^N\left[\m{h}^\top\m{\Gamma}_n^\top\left(\m{\Sigma}_{\Omega}+\m{E}_{\Omega}\right)\m{\Gamma}_n\m{h}+\varepsilon\right]^{-\frac{3}{2}}\left(\m{h}^\top\m{\Gamma}_n^\top\m{E}_{\Omega}\m{\Gamma}_n\m{e}_o\right)\\
&+\frac{1}{\left\|\m{d}\right\|_2^2}\sum_{n=1}^N\left[\m{h}^\top\m{\Gamma}_n^\top\left(\m{\Sigma}_{\Omega}+\m{E}_{\Omega}\right)\m{\Gamma}_n\m{h}+\varepsilon\right]^{-\frac{1}{2}}\m{e}_o^\top\m{\Gamma}_n^\top\m{E}_{\Omega}\m{\Gamma}_n\m{e}_o\\
&+\frac{\varepsilon h_o^2}{\left\|\m{d}\right\|_2^2}\sum_{n=1}^N\left[\m{h}^\top\m{\Gamma}_n^\top\left(\m{\Sigma}_{\Omega}+\m{E}_{\Omega}\right)\m{\Gamma}_n\m{h}+\varepsilon\right]^{-\frac{3}{2}}\m{h}^\top\m{\Gamma}_n^\top\left(\m{\Sigma}_{\Omega}+\m{E}_{\Omega}\right)\m{\Gamma}_n\m{h}\\
\leq&-\sum_{n=1}^N\left[\m{h}^\top\m{\Gamma}_n^\top\left(\m{\Sigma}_{\Omega}+\m{E}_{\Omega}\right)\m{\Gamma}_n\m{h}+\varepsilon\right]^{-\frac{1}{2}}\left(\m{h}^\top\m{\Gamma}_n^\top\m{E}_{\Omega}\m{\Gamma}_n\m{h}\right)\\
&+\frac{2h_o\left(1+\varepsilon\right)}{\left\|\m{d}\right\|_2^2}\left(h_o^2-\left\|\m{E}_\Omega\right\|_2+\varepsilon\right)^{-\frac{3}{2}}K\left\|\m{E}_\Omega\right\|_2\\
&+\frac{1}{\left\|\m{d}\right\|_2^2}\left(h_o^2-\left\|\m{E}_\Omega\right\|_2+\varepsilon\right)^{-\frac{1}{2}}K\left\|\m{E}_\Omega\right\|_2+\frac{h_o^2}{\left\|\m{d}\right\|_2^2}\varepsilon^{\frac{1}{2}}N\\
\leq&-\sum_{n=1}^N\left[\m{h}^\top\m{\Gamma}_n^\top\left(\m{\Sigma}_{\Omega}+\m{E}_{\Omega}\right)\m{\Gamma}_n\m{h}+\varepsilon\right]^{-\frac{1}{2}}\left(\m{h}^\top\m{\Gamma}_n^\top\m{E}_{\Omega}\m{\Gamma}_n\m{h}\right)\\
&+\frac{8\left(1+\varepsilon\right)}{\left\|\m{d}\right\|_2^2}h_o^{-2}K\left\|\m{E}_\Omega\right\|_2+\frac{2}{\left\|\m{d}\right\|_2^2}Kh_o^{-1}\left\|\m{E}_\Omega\right\|_2+\frac{h_o^2}{\left\|\m{d}\right\|_2^2}\varepsilon^{\frac{1}{2}}N.
\end{aligned}
\end{equation}
Note that
\begin{equation}
\label{temp-term-2}
\begin{aligned}
&\left|\sum_{n=1}^N\left[\m{h}^\top\m{\Gamma}_n^\top\left(\m{\Sigma}_{\Omega}+\m{E}_{\Omega}\right)\m{\Gamma}_n\m{h}+\varepsilon\right]^{-\frac{1}{2}}\left(\m{h}^\top\m{\Gamma}_n^\top\m{E}_{\Omega}\m{\Gamma}_n\m{h}\right)\right|\\
\leq&\sum_{n\in\chi_{\m{h}}}\frac{\left[\left(\m{h}^\top\m{\Gamma}_n^\top\m{E}_{\Omega}\m{\Gamma}_n\m{h}\right)/\left\|\mathcal{P}_{\Omega}\left(\m{\Gamma}_n\m{h}\right)\right\|_2^2\right]\left\|\mathcal{P}_{\Omega}\left(\m{\Gamma}_n\m{h}\right)\right\|_2}{\left[\m{h}^\top\m{\Gamma}_n^\top\left(\m{\Sigma}_{\Omega}+\m{E}_{\Omega}\right)\m{\Gamma}_n\m{h}/\left\|\mathcal{P}_{\Omega}\left(\m{\Gamma}_n\m{h}\right)\right\|_2^2+\varepsilon/\left\|\mathcal{P}_{\Omega}\left(\m{\Gamma}_n\m{h}\right)\right\|_2^2\right]^{\frac{1}{2}}}\\
\leq&\frac{\left\|\m{E}_\Omega\right\|_2}{\left(1-\left\|\m{E}_\Omega\right\|_2+\varepsilon\right)^{\frac{1}{2}}}\sum_{n\in\chi_{\m{h}}}\left\|\mathcal{P}_{\Omega}\left(\m{\Gamma}_n\m{h}\right)\right\|_2.
\end{aligned}
\end{equation}
Substituting \eqref{temp-term-2} into the preceding bound yields
\begin{equation}
\label{res-term}
\begin{aligned}
\text{res}
\leq&2\left\|\m{E}_\Omega\right\|_2\sum_{n\in\chi_{\m{h}}}\left\|\mathcal{P}_{\Omega}\left(\m{\Gamma}_n\m{h}\right)\right\|_2+\frac{8\left(1+\varepsilon\right)}{\left\|\m{d}\right\|_2^2}h_o^{-2}K\left\|\m{E}_\Omega\right\|_2+\frac{2}{\left\|\m{d}\right\|_2^2}Kh_o^{-1}\left\|\m{E}_\Omega\right\|_2+\frac{h_o^2}{\left\|\m{d}\right\|_2^2}\varepsilon^{\frac{1}{2}}N.
\end{aligned}
\end{equation}
To provide an upper bound for the main term, we establish the following lemma.
\begin{lem}
\label{lem-1st}
Let $\m{h}$ be an approximate first-order stationary point satisfying \eqref{prop-first-order stationary} and \eqref{prop-max}. Then, for $\left\|\m{E}_\Omega\right\|_2\leq\frac{1}{2}$, it holds that
\begin{equation}
\label{first-order}
\begin{aligned}
&\left|\sum_{n\in\chi_o}\frac{1}{\left[\m{h}^\top\m{\Gamma}_n^\top\left(\m{\Sigma}_{\Omega}+\m{E}_{\Omega}\right)\m{\Gamma}_n\m{h}+\varepsilon\right]^{\frac{1}{2}}}-\sum_{n=1}^N\frac{\m{h}^\top\m{\Gamma}_n^\top\m{\Sigma}_{\Omega}\m{\Gamma}_n\m{h}}{\left[\m{h}^\top\m{\Gamma}_n^\top\left(\m{\Sigma}_{\Omega}+\m{E}_{\Omega}\right)\m{\Gamma}_n\m{h}+\varepsilon\right]^{\frac{1}{2}}}\right|\\
\leq&\frac{1}{h_o}\left(\xi+2K\left\|\m{E}_\Omega\right\|_2\right)+2\left\|\m{E}_\Omega\right\|_2\sum_{n\in\chi_{\m{h}}}\left\|\mathcal{P}_{\Omega}\left(\m{\Gamma}_n\m{h}\right)\right\|_2.
\end{aligned}
\end{equation}
\end{lem}
\begin{proof}
It follows from \eqref{grad} that the $o$-th entry of $\text{grad}\;\psi_\varepsilon^0\left(\m{h}\right)$ is given by
\begin{equation}
\begin{aligned}
\mathcal{P}_{o}\left(\text{grad}\;\psi_\varepsilon^0\left(\m{h}\right)\right)=&\sum_{n\in\chi_o}\left[\m{h}^\top\m{\Gamma}_n^\top\left(\m{\Sigma}_{\Omega}+\m{E}_{\Omega}\right)\m{\Gamma}_n\m{h}+\varepsilon\right]^{-\frac{1}{2}}h_o\\
&-\sum_{n=1}^N\left[\m{h}^\top\m{\Gamma}_n^\top\left(\m{\Sigma}_{\Omega}+\m{E}_{\Omega}\right)\m{\Gamma}_n\m{h}+\varepsilon\right]^{-\frac{1}{2}}\left(\m{h}^\top\m{\Gamma}_n^\top\m{\Sigma}_{\Omega}\m{\Gamma}_n\m{h}\right)h_o\\
&+\mathcal{P}_{o}\left(\sum_{n=1}^N\left[\m{h}^\top\m{\Gamma}_n^\top\left(\m{\Sigma}_{\Omega}+\m{E}_{\Omega}\right)\m{\Gamma}_n\m{h}+\varepsilon\right]^{-\frac{1}{2}}\m{\Gamma}_n^\top\m{E}_{\Omega}\m{\Gamma}_n\m{h}\right)\\
&-\mathcal{P}_{o}\left(\sum_{n=1}^N\left[\m{h}^\top\m{\Gamma}_n^\top\left(\m{\Sigma}_{\Omega}+\m{E}_{\Omega}\right)\m{\Gamma}_n\m{h}+\varepsilon\right]^{-\frac{1}{2}}\left(\m{h}^\top\m{\Gamma}_n^\top\m{E}_{\Omega}\m{\Gamma}_n\m{h}\right)\m{h}\right).
\end{aligned}
\end{equation}
By the first-order stationarity condition \eqref{prop-first-order stationary},
\begin{equation}
\begin{aligned}
\left|\mathcal{P}_{o}\left(\text{grad}\;\psi_\varepsilon^0\left(\m{h}\right)\right)\right|\leq\xi,
\end{aligned}
\end{equation}
which implies
\begin{equation}
\label{temp-term-0}
\begin{aligned}
&h_o\left|\sum_{n\in\chi_o}\left[\m{h}^\top\m{\Gamma}_n^\top\left(\m{\Sigma}_{\Omega}+\m{E}_{\Omega}\right)\m{\Gamma}_n\m{h}+\varepsilon\right]^{-\frac{1}{2}}-\sum_{n=1}^N\left[\m{h}^\top\m{\Gamma}_n^\top\left(\m{\Sigma}_{\Omega}+\m{E}_{\Omega}\right)\m{\Gamma}_n\m{h}+\varepsilon\right]^{-\frac{1}{2}}\left(\m{h}^\top\m{\Gamma}_n^\top\m{\Sigma}_{\Omega}\m{\Gamma}_n\m{h}\right)\right|\\
\leq&\xi+\left|\mathcal{P}_{o}\left(\sum_{n=1}^N\left[\m{h}^\top\m{\Gamma}_n^\top\left(\m{\Sigma}_{\Omega}+\m{E}_{\Omega}\right)\m{\Gamma}_n\m{h}+\varepsilon\right]^{-\frac{1}{2}}\m{\Gamma}_n^\top\m{E}_{\Omega}\m{\Gamma}_n\m{h}\right)\right|\\
&+\left|\mathcal{P}_{o}\left(\sum_{n=1}^N\left[\m{h}^\top\m{\Gamma}_n^\top\left(\m{\Sigma}_{\Omega}+\m{E}_{\Omega}\right)\m{\Gamma}_n\m{h}+\varepsilon\right]^{-\frac{1}{2}}\left(\m{h}^\top\m{\Gamma}_n^\top\m{E}_{\Omega}\m{\Gamma}_n\m{h}\right)\m{h}\right)\right|\\
\leq&\xi+\left|\mathcal{P}_{o}\left(\sum_{n\in\chi_o}\left[\m{h}^\top\m{\Gamma}_n^\top\left(\m{\Sigma}_{\Omega}+\m{E}_{\Omega}\right)\m{\Gamma}_n\m{h}+\varepsilon\right]^{-\frac{1}{2}}\m{\Gamma}_n^\top\m{E}_{\Omega}\m{\Gamma}_n\m{h}\right)\right|\\
&+h_o\left|\sum_{n=1}^N\left[\m{h}^\top\m{\Gamma}_n^\top\left(\m{\Sigma}_{\Omega}+\m{E}_{\Omega}\right)\m{\Gamma}_n\m{h}+\varepsilon\right]^{-\frac{1}{2}}\left(\m{h}^\top\m{\Gamma}_n^\top\m{E}_{\Omega}\m{\Gamma}_n\m{h}\right)\right|.
\end{aligned}
\end{equation}
Moreover,
\begin{equation}
\label{temp-term-1}
\begin{aligned}
&\left|\mathcal{P}_{o}\left(\sum_{n\in\chi_o}\left[\m{h}^\top\m{\Gamma}_n^\top\left(\m{\Sigma}_{\Omega}+\m{E}_{\Omega}\right)\m{\Gamma}_n\m{h}+\varepsilon\right]^{-\frac{1}{2}}\m{\Gamma}_n^\top\m{E}_{\Omega}\m{\Gamma}_n\m{h}\right)\right|\\
\leq&\sum_{n\in\chi_{o}}\left[\frac{\left(\m{h}^\top\m{\Gamma}_n^\top\m{E}_{\Omega}^\top\m{E}_{\Omega}\m{\Gamma}_n\m{h}\right)/\left\|\mathcal{P}_{\Omega}\left(\m{\Gamma}_n\m{h}\right)\right\|_2^2}{\m{h}^\top\m{\Gamma}_n^\top\left(\m{\Sigma}_{\Omega}+\m{E}_{\Omega}\right)\m{\Gamma}_n\m{h}/\left\|\mathcal{P}_{\Omega}\left(\m{\Gamma}_n\m{h}\right)\right\|_2^2+\varepsilon/\left\|\mathcal{P}_{\Omega}\left(\m{\Gamma}_n\m{h}\right)\right\|_2^2}\right]^{\frac{1}{2}}\\
\leq&K\frac{\left\|\m{E}_\Omega\right\|_2}{\left(1-\left\|\m{E}_\Omega\right\|_2+\varepsilon\right)^{\frac{1}{2}}}.
\end{aligned}
\end{equation}
Substituting \eqref{temp-term-1} and \eqref{temp-term-2} into the inequality \eqref{temp-term-0} yields
\begin{equation}
\label{first-order}
\begin{aligned}
&h_o\left|\sum_{n\in\chi_o}\left[\m{h}^\top\m{\Gamma}_n^\top\left(\m{\Sigma}_{\Omega}+\m{E}_{\Omega}\right)\m{\Gamma}_n\m{h}+\varepsilon\right]^{-\frac{1}{2}}-\sum_{n=1}^N\left[\m{h}^\top\m{\Gamma}_n^\top\left(\m{\Sigma}_{\Omega}+\m{E}_{\Omega}\right)\m{\Gamma}_n\m{h}+\varepsilon\right]^{-\frac{1}{2}}\left(\m{h}^\top\m{\Gamma}_n^\top\m{\Sigma}_{\Omega}\m{\Gamma}_n\m{h}\right)\right|\\
\leq&\xi+K\frac{\left\|\m{E}_\Omega\right\|_2}{\left(1-\left\|\m{E}_\Omega\right\|_2+\varepsilon\right)^{\frac{1}{2}}}+h_o\frac{\left\|\m{E}_\Omega\right\|_2}{\left(1-\left\|\m{E}_\Omega\right\|_2+\varepsilon\right)^{\frac{1}{2}}}\sum_{n\in\chi_{\m{h}}}\left\|\mathcal{P}_{\Omega}\left(\m{\Gamma}_n\m{h}\right)\right\|_2\\
\leq&\xi+2K\left\|\m{E}_\Omega\right\|_2+2h_o\left\|\m{E}_\Omega\right\|_2\sum_{n\in\chi_{\m{h}}}\left\|\mathcal{P}_{\Omega}\left(\m{\Gamma}_n\m{h}\right)\right\|_2.
\end{aligned}
\end{equation}
Dividing both sides by $h_o$ completes the proof.
\end{proof}

It follows from Lemma~\ref{lem-1st} that
\begin{equation}
\begin{aligned}
\text{main}\leq&\left(\frac{1}{\left\|\m{d}\right\|_2^2}-1\right)\sum_{n\in\chi_o}\left[\m{h}^\top\m{\Gamma}_n^\top\left(\m{\Sigma}_{\Omega}+\m{E}_{\Omega}\right)\m{\Gamma}_n\m{h}+\varepsilon\right]^{-\frac{1}{2}}\\
&-\frac{h_o^2}{\left\|\m{d}\right\|_2^2}\sum_{n\in\chi_o}\left[\m{h}^\top\m{\Gamma}_n^\top\left(\m{\Sigma}_{\Omega}+\m{E}_{\Omega}\right)\m{\Gamma}_n\m{h}+\varepsilon\right]^{-\frac{3}{2}}\\
&+\frac{1}{h_o}\left(\xi+2K\left\|\m{E}_\Omega\right\|_2\right)+2\left\|\m{E}_\Omega\right\|_2\sum_{n\in\chi_{\m{h}}}\left\|\mathcal{P}_{\Omega}\left(\m{\Gamma}_n\m{h}\right)\right\|_2\\
=&\frac{h_o^2}{\left\|\m{d}\right\|_2^2}\sum_{n\in\chi_o}\left\{\left[\m{h}^\top\m{\Gamma}_n^\top\left(\m{\Sigma}_{\Omega}+\m{E}_{\Omega}\right)\m{\Gamma}_n\m{h}+\varepsilon\right]^{-\frac{1}{2}}-\left[\m{h}^\top\m{\Gamma}_n^\top\left(\m{\Sigma}_{\Omega}+\m{E}_{\Omega}\right)\m{\Gamma}_n\m{h}+\varepsilon\right]^{-\frac{3}{2}}\right\}\\
&+\frac{1}{h_o}\left(\xi+2K\left\|\m{E}_\Omega\right\|_2\right)+2\left\|\m{E}_\Omega\right\|_2\sum_{n\in\chi_{\m{h}}}\left\|\mathcal{P}_{\Omega}\left(\m{\Gamma}_n\m{h}\right)\right\|_2.
\end{aligned}
\end{equation}
Applying Lemma~\ref{lem-diff} yields
\begin{equation}
\begin{aligned}
&\sum_{n\in\chi_o}\left\{\left[\m{h}^\top\m{\Gamma}_n^\top\left(\m{\Sigma}_{\Omega}+\m{E}_{\Omega}\right)\m{\Gamma}_n\m{h}+\varepsilon\right]^{-\frac{1}{2}}-\left[\m{h}^\top\m{\Gamma}_n^\top\left(\m{\Sigma}_{\Omega}+\m{E}_{\Omega}\right)\m{\Gamma}_n\m{h}+\varepsilon\right]^{-\frac{3}{2}}\right\}\\
\leq&\left(\min\left\{Kh_o^2,1-\frac{1-h_o^2}{K}\right\}+\left\|\m{E}_\Omega\right\|_2+\varepsilon\right)^{-\frac{1}{2}}-\left(\min\left\{Kh_o^2,1-\frac{1-h_o^2}{K}\right\}+\left\|\m{E}_\Omega\right\|_2+\varepsilon\right)^{-\frac{3}{2}}\\
&+\left(K-1\right)\left[\left(1+\left\|\m{E}_\Omega\right\|_2+\varepsilon\right)^{-\frac{1}{2}}-\left(1+\left\|\m{E}_\Omega\right\|_2+\varepsilon\right)^{-\frac{3}{2}}\right]\\
\leq&\left(\min\left\{Kh_o^2,1-\frac{1-h_o^2}{K}\right\}+\left\|\m{E}_\Omega\right\|_2+\varepsilon\right)^{-\frac{3}{2}}\left(\min\left\{Kh_o^2-1,-\frac{1-h_o^2}{K}\right\}+\left\|\m{E}_\Omega\right\|_2+\varepsilon\right)\\
&+\left(K-1\right)\left(1+\left\|\m{E}_\Omega\right\|_2+\varepsilon\right)^{-\frac{3}{2}}\left(\left\|\m{E}_\Omega\right\|_2+\varepsilon\right).
\end{aligned}
\end{equation}
Consequently, we obtain that
\begin{equation}
\label{main-term}
\begin{aligned}
\text{main}\leq&\frac{h_o^2}{\left\|\m{d}\right\|_2^2}\left(\min\left\{1-\frac{1-h_o^2}{K},Kh_o^2\right\}+\left\|\m{E}_\Omega\right\|_2+\varepsilon\right)^{-\frac{3}{2}}\left(\min\left\{-\frac{1-h_o^2}{K},Kh_o^2-1\right\}+\left\|\m{E}_\Omega\right\|_2+\varepsilon\right)\\
&+\frac{h_o^2}{\left\|\m{d}\right\|_2^2}\left(K-1\right)\left(1+\left\|\m{E}_\Omega\right\|_2+\varepsilon\right)^{-\frac{3}{2}}\left(\left\|\m{E}_\Omega\right\|_2+\varepsilon\right)\\
&+\frac{1}{h_o}\left(\xi+2K\left\|\m{E}_\Omega\right\|_2\right)+2\left\|\m{E}_\Omega\right\|_2\sum_{n\in\chi_{\m{h}}}\left\|\mathcal{P}_{\Omega}\left(\m{\Gamma}_n\m{h}\right)\right\|_2\\
\leq&\frac{h_o^2}{\left\|\m{d}\right\|_2^2}\left(\min\left\{1-\frac{1-h_o^2}{K},Kh_o^2\right\}+\left\|\m{E}_\Omega\right\|_2+\varepsilon\right)^{-\frac{3}{2}}\left(\min\left\{-\frac{1-h_o^2}{K},Kh_o^2-1\right\}+\left\|\m{E}_\Omega\right\|_2+\varepsilon\right)\\
&+\frac{h_o^2}{\left\|\m{d}\right\|_2^2}K\left(\left\|\m{E}_\Omega\right\|_2+\varepsilon\right)+\frac{1}{h_o}\left(\xi+2K\left\|\m{E}_\Omega\right\|_2\right)+2\left\|\m{E}_\Omega\right\|_2\sum_{n\in\chi_{\m{h}}}\left\|\mathcal{P}_{\Omega}\left(\m{\Gamma}_n\m{h}\right)\right\|_2.
\end{aligned}
\end{equation}
Substituting \eqref{res-term} and \eqref{main-term} into \eqref{res-main-term} yields
\begin{equation}
\begin{aligned}
\frac{\m{d}^\top\text{Hess}\;\psi_\varepsilon^0\left(\m{h}\right)\m{d}}{\left\|\m{d}\right\|_2^2}\leq&\frac{h_o^2}{\left\|\m{d}\right\|_2^2}\left(\min\left\{1-\frac{1-h_o^2}{K},Kh_o^2\right\}+\left\|\m{E}_\Omega\right\|_2+\varepsilon\right)^{-\frac{3}{2}}\\
&\cdot\left(\min\left\{-\frac{1-h_o^2}{K},Kh_o^2-1\right\}+\left\|\m{E}_\Omega\right\|_2+\varepsilon\right)\\
&+\frac{1}{\left\|\m{d}\right\|_2^2}h_o^2K\left\|\m{E}_\Omega\right\|_2+4\left\|\m{E}_\Omega\right\|_2\sum_{n\in\chi_{\m{h}}}\left\|\mathcal{P}_{\Omega}\left(\m{\Gamma}_n\m{h}\right)\right\|_2+\frac{4}{\left\|\m{d}\right\|_2^2}h_o^{-1}K\left\|\m{E}_\Omega\right\|_2\\
&+\frac{8}{\left\|\m{d}\right\|_2^2}h_o^{-2}\left(1+\varepsilon\right)K\left\|\m{E}_\Omega\right\|_2+\frac{1}{\left\|\m{d}\right\|_2^2}h_o^2\varepsilon K+\frac{2}{\left\|\m{d}\right\|_2^2}h_o^2\varepsilon^{\frac{1}{2}}N+h_o^{-1}\xi,
\end{aligned}
\end{equation}
Therefore, there exist positive constants $c_1\leq\frac{1}{2}$, $c_2\leq1$, and $c_3$ such that, whenever $\left\|\m{E}_\Omega\right\|_2\leq c_1K^{-2}\max^{-1}\left\{K^{2},N\right\}$, $\varepsilon\leq c_2K^{-6}N^{-2}$, and $\xi\leq c_3 K^{-3}N^{-\frac{1}{2}}$, it holds that
\begin{equation}
\begin{aligned}
\frac{\m{d}^\top\text{Hess}\;\psi_\varepsilon^0\left(\m{h}\right)\m{d}}{\left\|\m{d}\right\|_2^2}\leq&\left(1+c_1+c_2\right)^{-\frac{3}{2}}\frac{h_o^2}{\left\|\m{d}\right\|_2^2}\left(\min\left\{1-\frac{1-h_o^2}{K},Kh_o^2\right\}\right)^{-\frac{3}{2}}\min\left\{-\frac{1-h_o^2}{K},Kh_o^2-1\right\}\\
&+\frac{h_o^2}{\left\|\m{d}\right\|_2^2}\left(\min\left\{1-\frac{1-h_o^2}{K},Kh_o^2\right\}\right)^{-\frac{3}{2}}\left(\left\|\m{E}_\Omega\right\|_2+\varepsilon\right)\\
&+\frac{c_1}{\left\|\m{d}\right\|_2^2}K^{-3}+4\left\|\m{E}_\Omega\right\|_2\sum_{n\in\chi_{\m{h}}}\left\|\mathcal{P}_{\Omega}\left(\m{\Gamma}_n\m{h}\right)\right\|_2\\
&+\frac{20}{\left\|\m{d}\right\|_2^2}h_o^{-2}K\left\|\m{E}_\Omega\right\|_2+\frac{c_2}{\left\|\m{d}\right\|_2^2}K^{-3}+\frac{2c_2^{\frac{1}{2}}}{\left\|\m{d}\right\|_2^2}K^{-3}+c_3K^{-3},
\end{aligned}
\end{equation}
which completes the proof.

\section{Proof of Proposition~\ref{thm-case1}}
\label{pf-thm-case1}
By Lemma~\ref{lem-2st} and the assumption $h_o\in\left[\sqrt{\frac{1}{N}},\sqrt{\frac{1}{K+1}}\right]$, there exist positive constants $c_1$, $c_2$, and $c_3$ such that, whenever $\left\|\m{E}_\Omega\right\|_2\leq c_1K^{-2}\max^{-1}\left\{K^{2},N\right\}$, $\varepsilon\leq c_2K^{-6}N^{-2}$, and $\xi\leq c_3K^{-3}N^{-\frac{1}{2}}$, it holds that
\begin{equation}
\begin{aligned}
\frac{\m{d}^\top\text{Hess}\;\psi_\varepsilon^0\left(\m{h}\right)\m{d}}{\left\|\m{d}\right\|_2^2}
\leq&\left(1+c_1+c_2\right)^{-\frac{3}{2}}\frac{h_o^{-1}}{\left\|\m{d}\right\|_2^2}K^{-\frac{3}{2}}\left(Kh_o^2-1\right)+\frac{1}{\left\|\m{d}\right\|_2^2}\left(c_1+c_2\right)K^{-2}\\
&+4c_1K^{-2}+\frac{20}{\left\|\m{d}\right\|_2^2}h_o^{-2}K\left\|\m{E}_\Omega\right\|_2+\frac{c_1+c_2+2c_2^{\frac{1}{2}}}{\left\|\m{d}\right\|_2^2}K^{-3}+c_3K^{-3}.
\end{aligned}
\end{equation}
Furthermore, if $\left\|\m{E}_\Omega\right\|_2\leq c_1K^{-\frac{7}{2}}N^{-\frac{1}{2}}$, then
\begin{equation}
\begin{aligned}
\frac{\m{d}^\top\text{Hess}\;\psi_\varepsilon^0\left(\m{h}\right)\m{d}}{\left\|\m{d}\right\|_2^2}
\leq&-\left(1+c_1+c_2\right)^{-\frac{3}{2}}\frac{h_o^{-1}}{\left\|\m{d}\right\|_2^2}\frac{K^{-\frac{5}{2}}}{2}+\left(c_1+c_2\right)\frac{h_o^{-1}}{\left\|\m{d}\right\|_2^2}K^{-\frac{5}{2}}\\
&+4c_1\frac{h_o^{-1}}{\left\|\m{d}\right\|_2^2}K^{-\frac{5}{2}}+20c_1\frac{h_o^{-1}}{\left\|\m{d}\right\|_2^2}K^{-\frac{5}{2}}+\frac{c_1+c_2+2c_2^{\frac{1}{2}}+c_3}{\left\|\m{d}\right\|_2^2}K^{-3}\\
\leq&-\frac{h_o^{-1}}{\left\|\m{d}\right\|_2^2}\left[\frac{\left(1+c_1+c_2\right)^{-\frac{3}{2}}}{2}-\left(26c_1+2c_2+2c_2^{\frac{1}{2}}+c_3\right)\right]K^{-\frac{5}{2}}\\
\leq&-\frac{1}{\left\|\m{d}\right\|_2^2}\left[\frac{\left(1+c_1+c_2\right)^{-\frac{3}{2}}}{2}-\left(26c_1+2c_2+2c_2^{\frac{1}{2}}+c_3\right)\right]K^{-2},
\end{aligned}
\end{equation}
which completes the proof.

\section{Proof of Proposition~\ref{thm-case2-1}}
\label{pf-thm-case2-1}
By Lemma~\ref{lem-2st} and the assumption $h_o\in\left(\sqrt{\frac{1}{K+1}},\sqrt{1-\frac{1}{4K}}\right]$, there exist positive constants $c_1$, $c_2$, and $c_3$ such that, whenever $\left\|\m{E}_\Omega\right\|_2\leq c_1K^{-2}\max^{-1}\left\{K^{2},N\right\}$, $\varepsilon\leq c_2K^{-6}N^{-2}$, and $\xi\leq c_3K^{-3}N^{-\frac{1}{2}}$, it holds that
\begin{equation}
\begin{aligned}
\frac{\m{d}^\top\text{Hess}\;\psi_\varepsilon^0\left(\m{h}\right)\m{d}}{\left\|\m{d}\right\|_2^2}
\leq&\left(1+c_1+c_2\right)^{-\frac{3}{2}}\frac{h_o^2}{\left\|\m{d}\right\|_2^2}\left(1-\frac{1-h_o^2}{K}\right)^{-\frac{3}{2}}\left(-\frac{1-h_o^2}{K}\right)+\frac{4\left(c_1+c_2\right)}{\left\|\m{d}\right\|_2^2}K^{-4}\\
&+4N\left\|\m{E}_\Omega\right\|_2+\frac{20}{\left\|\m{d}\right\|_2^2}h_o^{-2}K\left\|\m{E}_\Omega\right\|_2+\frac{c_1+c_2+2c_2^{\frac{1}{2}}+c_3}{\left\|\m{d}\right\|_2^2}K^{-3}.
\end{aligned}
\end{equation}
Furthermore, if $\left\|\m{E}_\Omega\right\|_2\leq c_1K^{-\frac{5}{2}}\max^{-1}\left\{N,K^2\right\}$, then
\begin{equation}
\begin{aligned}
\frac{\m{d}^\top\text{Hess}\;\psi_\varepsilon^0\left(\m{h}\right)\m{d}}{\left\|\m{d}\right\|_2^2}
\leq&-\frac{\left(1+c_1+c_2\right)^{-\frac{3}{2}}}{8}\frac{1}{\left\|\m{d}\right\|_2^2}K^{-2}+\frac{4\left(c_1+c_2\right)}{\left\|\m{d}\right\|_2^2}K^{-4}\\
&+4c_1K^{-\frac{5}{2}}+40c_1\frac{1}{\left\|\m{d}\right\|_2^2}K^{-\frac{5}{2}}+\frac{c_1+c_2+2c_2^{\frac{1}{2}}+c_3}{\left\|\m{d}\right\|_2^2}K^{-3}\\
\leq&-\frac{1}{\left\|\m{d}\right\|_2^2}\left[\frac{\left(1+c_1+c_2\right)^{-\frac{3}{2}}}{8}-\left(49c_1+5c_2+2c_2^{\frac{1}{2}}+c_3\right)\right]K^{-2},
\end{aligned}
\end{equation}
which completes the proof.

\section{Proof of Proposition~\ref{thm-case2-2}}
\label{pf-thm-case2-2}
By Lemma~\ref{lem-2st} and the assumption $h_o\in\left(\sqrt{1-\frac{1}{4K}},\sqrt{1-\frac{1}{4K^2}}\right]$, there exist positive constants $c_1$, $c_2$, and $c_3$ such that, whenever $\left\|\m{E}_\Omega\right\|_2\leq c_1K^{-2}\max^{-1}\left\{K^{2},N\right\}$, $\varepsilon\leq c_2K^{-6}N^{-2}$, and $\xi\leq c_3K^{-3}N^{-\frac{1}{2}}$, it holds that
\begin{equation}
\begin{aligned}
\frac{\m{d}^\top\text{Hess}\;\psi_\varepsilon^0\left(\m{h}\right)\m{d}}{\left\|\m{d}\right\|_2^2}
\leq&-\left(1+c_1+c_2\right)^{-\frac{3}{2}}\frac{h_o^2}{\left\|\m{d}\right\|_2^2}\left(1-\frac{1-h_o^2}{K}\right)^{-\frac{3}{2}}\frac{1-h_o^2}{K}+\frac{4\left(c_1+c_2\right)}{\left\|\m{d}\right\|_2^2}K^{-4}\\
&+4\left\|\m{E}_\Omega\right\|_2\sum_{n\in\chi_{\m{h}}}\left\|\mathcal{P}_{\Omega}\left(\m{\Gamma}_n\m{h}\right)\right\|_2+\frac{80c_1}{\left\|\m{d}\right\|_2^2}K^{-3}+\frac{c_1+c_2+2c_2^{\frac{1}{2}}+c_3}{\left\|\m{d}\right\|_2^2}K^{-3}.
\end{aligned}
\end{equation}
Next, observe that
\begin{equation}
\begin{aligned}
\sum_{n\in\chi_{\m{h}}}\left\|\mathcal{P}_{\Omega}\left(\m{\Gamma}_n\m{h}\right)\right\|_2=&\sum_{n\in\chi_{\m{h}}\cap\chi_o}\left\|\mathcal{P}_{\Omega}\left(\m{\Gamma}_n\m{h}\right)\right\|_2+\sum_{n\in\chi_{\m{h}}\setminus\chi_o}\left\|\mathcal{P}_{\Omega}\left(\m{\Gamma}_n\m{h}\right)\right\|_2\leq K+\frac{N-K}{2\sqrt{K}}.
\end{aligned}
\end{equation}
Therefore, if $\left\|\m{E}_\Omega\right\|_2\leq c_1K^{-\frac{5}{2}}\max^{-1}\left\{N,K^2\right\}$, it holds that
\begin{equation}
\begin{aligned}
\frac{\m{d}^\top\text{Hess}\;\psi_\varepsilon^0\left(\m{h}\right)\m{d}}{\left\|\m{d}\right\|_2^2}
\leq&-\frac{\left(1+c_1+c_2\right)^{-\frac{3}{2}}}{8}\frac{1}{\left\|\m{d}\right\|_2^2}K^{-3}+\frac{4\left(c_1+c_2\right)}{\left\|\m{d}\right\|_2^2}K^{-4}\\
&+\left(8c_1+2c_1\right)K^{-3}+\frac{80c_1}{\left\|\m{d}\right\|_2^2}K^{-3}+\frac{c_1+c_2+2c_2^{\frac{1}{2}}+c_3}{\left\|\m{d}\right\|_2^2}K^{-3}\\
\leq&-\frac{1}{\left\|\m{d}\right\|_2^2}\left[\frac{\left(1+c_1+c_2\right)^{-\frac{3}{2}}}{8}-\left(95c_1+5c_2+2c_2^{\frac{1}{2}}+c_3\right)\right]K^{-3},
\end{aligned}
\end{equation}
which completes the proof.

\section{Proof of Proposition~\ref{thm-case3}}
\label{pf-thm-case3}
From \eqref{grad} and \eqref{direction}, we obtain
\begin{equation}
\begin{aligned}
\m{d}^\top\text{grad}\;\psi_\varepsilon^0\left(\m{h}\right)=&\sum_{n=1}^N\left[\m{h}^\top\m{\Gamma}_n^\top\left(\m{\Sigma}_{\Omega}+\m{E}_{\Omega}\right)\m{\Gamma}_n\m{h}+\varepsilon\right]^{-\frac{1}{2}}\m{d}^\top\m{\Gamma}_n^\top\left(\m{\Sigma}_{\Omega}+\m{E}_{\Omega}\right)\m{\Gamma}_n\m{h}\\
=&\sum_{n=1}^N\frac{\m{e}_o^\top\m{\Gamma}_n^\top\m{\Sigma}_{\Omega}\m{\Gamma}_n\m{h}-h_o\m{h}^\top\m{\Gamma}_n^\top\m{\Sigma}_{\Omega}\m{\Gamma}_n\m{h}}{\left[\m{h}^\top\m{\Gamma}_n^\top\left(\m{\Sigma}_{\Omega}+\m{E}_{\Omega}\right)\m{\Gamma}_n\m{h}+\varepsilon\right]^{\frac{1}{2}}}+\sum_{n=1}^N\frac{\m{e}_o^\top\m{\Gamma}_n^\top\m{E}_{\Omega}\m{\Gamma}_n\m{h}-h_o\m{h}^\top\m{\Gamma}_n^\top\m{E}_{\Omega}\m{\Gamma}_n\m{h}}{\left[\m{h}^\top\m{\Gamma}_n^\top\left(\m{\Sigma}_{\Omega}+\m{E}_{\Omega}\right)\m{\Gamma}_n\m{h}+\varepsilon\right]^{\frac{1}{2}}}\\
=&\sum_{n\in\chi_o}\frac{h_o\left(1-\m{h}^\top\m{\Gamma}_n^\top\m{\Sigma}_{\Omega}\m{\Gamma}_n\m{h}\right)}{\left[\m{h}^\top\m{\Gamma}_n^\top\left(\m{\Sigma}_{\Omega}+\m{E}_{\Omega}\right)\m{\Gamma}_n\m{h}+\varepsilon\right]^{\frac{1}{2}}}-\sum_{n\in\chi_o^\text{C}}\frac{h_o\m{h}^\top\m{\Gamma}_n^\top\m{\Sigma}_{\Omega}\m{\Gamma}_n\m{h}}{\left[\m{h}^\top\m{\Gamma}_n^\top\left(\m{\Sigma}_{\Omega}+\m{E}_{\Omega}\right)\m{\Gamma}_n\m{h}+\varepsilon\right]^{\frac{1}{2}}}\\
&+\sum_{n=1}^N\frac{\m{d}^\top\m{\Gamma}_n^\top\m{E}_{\Omega}\m{\Gamma}_n\m{h}}{\left[\m{h}^\top\m{\Gamma}_n^\top\left(\m{\Sigma}_{\Omega}+\m{E}_{\Omega}\right)\m{\Gamma}_n\m{h}+\varepsilon\right]^{\frac{1}{2}}},
\end{aligned}
\end{equation}
where $\chi_o^\text{C}$ denotes the complement of $\chi_o$ in $\left\{1,2,\cdots,N\right\}$. Next, observe that
\begin{equation}
\begin{aligned}
&\sum_{n\in\chi_o}\frac{h_o\left(1-\m{h}^\top\m{\Gamma}_n^\top\m{\Sigma}_{\Omega}\m{\Gamma}_n\m{h}\right)}{\left[\m{h}^\top\m{\Gamma}_n^\top\left(\m{\Sigma}_{\Omega}+\m{E}_{\Omega}\right)\m{\Gamma}_n\m{h}+\varepsilon\right]^{\frac{1}{2}}}-\sum_{n\in\chi_o^\text{C}}\frac{h_o\m{h}^\top\m{\Gamma}_n^\top\m{\Sigma}_{\Omega}\m{\Gamma}_n\m{h}}{\left[\m{h}^\top\m{\Gamma}_n^\top\left(\m{\Sigma}_{\Omega}+\m{E}_{\Omega}\right)\m{\Gamma}_n\m{h}+\varepsilon\right]^{\frac{1}{2}}}\\
\leq&\frac{h_o}{\left[h_o^2-\left\|\m{E}_\Omega\right\|_2+\varepsilon\right]^{\frac{1}{2}}}\sum_{n\in\chi_o}\left(1-\m{h}^\top\m{\Gamma}_n^\top\m{\Sigma}_{\Omega}\m{\Gamma}_n\m{h}\right)-\frac{h_o}{\left[1-h_o^2+\left\|\m{E}_\Omega\right\|_2+\varepsilon\right]^{\frac{1}{2}}}\sum_{n\in\chi_o^\text{C}}\m{h}^\top\m{\Gamma}_n^\top\m{\Sigma}_{\Omega}\m{\Gamma}_n\m{h}\\
\leq&\frac{Kh_o\left(1-h_o^2\right)}{\left(h_o^2-\left\|\m{E}_\Omega\right\|_2+\varepsilon\right)^{\frac{1}{2}}}-\frac{h_o\left(1-h_o^2\right)}{\left(1-h_o^2+\left\|\m{E}_\Omega\right\|_2+\varepsilon\right)^{\frac{1}{2}}}\\
=&h_o\left(1-h_o^2\right)\left[\frac{K}{\left(h_o^2-\left\|\m{E}_\Omega\right\|_2+\varepsilon\right)^{\frac{1}{2}}}-\frac{1}{\left(1-h_o^2+\left\|\m{E}_\Omega\right\|_2+\varepsilon\right)^{\frac{1}{2}}}\right],
\end{aligned}
\end{equation}
where the second inequality follows from the nonuniformity of $\Omega$. Specifically, for every $m\neq o$, there exists $\m{\Gamma}_n$ such that $\m{e}_m^\top\m{\Gamma}_n^\top\m{\Sigma}_{\Omega}\m{\Gamma}_n\m{e}_m=1$ and $\m{e}_o^\top\m{\Gamma}_n^\top\m{\Sigma}_{\Omega}\m{\Gamma}_n\m{e}_o=0$. Consequently,
\begin{equation}
\label{out-concentrate}
\begin{aligned}
\sum_{n\in\chi_o^\text{C}}\m{h}^\top\m{\Gamma}_n^\top\m{\Sigma}_{\Omega}\m{\Gamma}_n\m{h}\geq1-h_o^2.
\end{aligned}
\end{equation}
Moreover,
\begin{equation}
\begin{aligned}
&\left|\sum_{n=1}^N\frac{\m{d}^\top\m{\Gamma}_n^\top\m{E}_{\Omega}\m{\Gamma}_n\m{h}}{\left[\m{h}^\top\m{\Gamma}_n^\top\left(\m{\Sigma}_{\Omega}+\m{E}_{\Omega}\right)\m{\Gamma}_n\m{h}+\varepsilon\right]^{\frac{1}{2}}}\right|\\
\leq&\left\|\m{d}\right\|_2\sum_{n=1}^N\frac{\left\|\m{E}_{\Omega}\m{\Gamma}_n\m{h}\right\|_2}{\left[\m{h}^\top\m{\Gamma}_n^\top\left(\m{\Sigma}_{\Omega}+\m{E}_{\Omega}\right)\m{\Gamma}_n\m{h}+\varepsilon\right]^{\frac{1}{2}}}\\
\leq&\left\|\m{d}\right\|_2\sum_{n\in\chi_{\m{h}}}\left[\frac{\left(\m{h}^\top\m{\Gamma}_n^\top\m{E}_{\Omega}^\top\m{E}_{\Omega}\m{\Gamma}_n\m{h}\right)/\left\|\mathcal{P}_{\Omega}\left(\m{\Gamma}_n\m{h}\right)\right\|_2^2}{\m{h}^\top\m{\Gamma}_n^\top\left(\m{\Sigma}_{\Omega}+\m{E}_{\Omega}\right)\m{\Gamma}_n\m{h}/\left\|\mathcal{P}_{\Omega}\left(\m{\Gamma}_n\m{h}\right)\right\|_2^2+\varepsilon/\left\|\mathcal{P}_{\Omega}\left(\m{\Gamma}_n\m{h}\right)\right\|_2^2}\right]^{\frac{1}{2}}\\
\leq&N\left\|\m{d}\right\|_2\frac{\left\|\m{E}_\Omega\right\|_2}{\left(1-\left\|\m{E}_\Omega\right\|_2+\varepsilon\right)^{\frac{1}{2}}}.
\end{aligned}
\end{equation}
Since $\m{h}$ satisfies \eqref{prop-first-order stationary},
\begin{equation}
\begin{aligned}
\frac{\left|\m{d}^\top\text{grad}\;\psi_\varepsilon^0\left(\m{h}\right)\right|}{\left\|\m{d}\right\|_2}\leq\left\|\text{grad}\;\psi_\varepsilon^0\left(\m{h}\right)\right\|_2\leq\xi.
\end{aligned}
\end{equation}
Therefore, if $\frac{K}{\left(h_o^2-\left\|\m{E}_\Omega\right\|_2+\varepsilon\right)^{\frac{1}{2}}}-\frac{1}{\left(1-h_o^2+\left\|\m{E}_\Omega\right\|_2+\varepsilon\right)^{\frac{1}{2}}}<0$, then
\begin{equation}
\begin{aligned}
&h_o\left(1-h_o^2\right)^{\frac{1}{2}}\left|\frac{1}{\left(1-h_o^2+\left\|\m{E}_\Omega\right\|_2+\varepsilon\right)^{\frac{1}{2}}}-\frac{K}{\left(h_o^2-\left\|\m{E}_\Omega\right\|_2+\varepsilon\right)^{\frac{1}{2}}}\right|\leq\xi+N\frac{\left\|\m{E}_\Omega\right\|_2}{\left(1-\left\|\m{E}_\Omega\right\|_2+\varepsilon\right)^{\frac{1}{2}}}.
\end{aligned}
\end{equation}
Now choose $c_1=\frac{1}{32}$, $c_2=\frac{1}{32}$. Under the conditions $\left\|\m{E}_\Omega\right\|_2\leq c_1K^{-2}$ and $\varepsilon\leq c_2K^{-2}$, we obtain
\begin{equation}
\begin{aligned}
&h_o\left(1-h_o^2\right)^{\frac{1}{2}}\left|\frac{4K}{\sqrt{5}}-\frac{\sqrt{2}K}{\sqrt{1-\frac{1}{4K^2}}}\right|\leq\xi+2N\left\|\m{E}_\Omega\right\|_2.
\end{aligned}
\end{equation}
Consequently,
\begin{equation}
\begin{aligned}
\left(1-h_o^2\right)^{\frac{1}{2}}\leq&2^{-\frac{1}{2}}\left(\sqrt{\frac{8}{5}}-\sqrt{\frac{4}{3}}\right)^{-1}\frac{1}{Kh_o}\left(\xi+2N\left\|\m{E}_\Omega\right\|_2\right)\leq\left(\sqrt{\frac{8}{5}}-\sqrt{\frac{4}{3}}\right)^{-1}K^{-1}\left(\xi+2N\left\|\m{E}_\Omega\right\|_2\right).
\end{aligned}
\end{equation}
Choosing $c_3=\sqrt{2}\left(\sqrt{\frac{8}{5}}-\sqrt{\frac{4}{3}}\right)^{-1}$, $c_4=2\sqrt{2}\left(\sqrt{\frac{8}{5}}-\sqrt{\frac{4}{3}}\right)^{-1}$ yields
\begin{equation}
\begin{aligned}
\left(1-h_o^2\right)^{\frac{1}{2}}\leq&\frac{1}{\sqrt{2}}\left(c_3K^{-1}\xi+c_4K^{-1}N\left\|\m{E}_\Omega\right\|_2\right).
\end{aligned}
\end{equation}
Finally, since
\begin{equation}
\label{error}
\begin{aligned}
\left\|\m{h}-\m{e}_o\right\|_2\leq\sqrt{2}\left(1-h_o^2\right)^{\frac{1}{2}},
\end{aligned}
\end{equation}
the proof is complete.

\section{Proof of Lemma~\ref{lem-tail-Delta}}
\label{pf-lem-tail-Delta}
Observe that
\begin{equation}
\label{Delta-R}
\begin{aligned}
\left\|\m{\Delta}\right\|_2=&\left\|\mathcal{C}\left(\m{g}\right)\m{R}\m{U}^\top-\frac{1}{\sqrt{K}}\m{I}\right\|_2\\
=&\left\|\mathcal{C}\left(\m{g}\right)\m{R}-\frac{1}{\sqrt{K}}\m{U}\right\|_2\\
=&\left\|\mathcal{C}\left(\m{g}\right)\m{R}-\frac{1}{\sqrt{K}}\mathcal{C}\left(\m{g}\right)\left[\mathcal{C}\left(\m{g}\right)^\top\mathcal{C}\left(\m{g}\right)\right]^{-\frac{1}{2}}\right\|_2\\
\leq&\left\|\mathcal{C}\left(\m{g}\right)\right\|_2\left\|\m{R}-\frac{1}{\sqrt{K}}\left[\mathcal{C}\left(\m{g}\right)^\top\mathcal{C}\left(\m{g}\right)\right]^{-\frac{1}{2}}\right\|_2.
\end{aligned}
\end{equation}
It follows from \cite[Theorem 6.2]{higham2008functions} that
\begin{equation}
\begin{aligned}
\left\|\m{R}-\frac{1}{\sqrt{K}}\left[\mathcal{C}\left(\m{g}\right)^\top\mathcal{C}\left(\m{g}\right)\right]^{-\frac{1}{2}}\right\|_2=&\left\|\left[\frac{1}{L}\sum_{l=1}^L\mathcal{C}\left(\m{y}_l\right)^\top\mathcal{C}\left(\m{y}_l\right)\right]^{-\frac{1}{2}}
-\frac{1}{\sqrt{K}}\left[\mathcal{C}\left(\m{g}\right)^\top\mathcal{C}\left(\m{g}\right)\right]^{-\frac{1}{2}}\right\|_2\\
\leq&\frac{\left\|\left[\frac{1}{L}\sum_{l=1}^L\mathcal{C}\left(\m{y}_l\right)^\top\mathcal{C}\left(\m{y}_l\right)\right]^{-1}
-\frac{1}{K}\left[\mathcal{C}\left(\m{g}\right)^\top\mathcal{C}\left(\m{g}\right)\right]^{-1}\right\|_2}{\lambda_{\text{min}}\left(\frac{1}{\sqrt{K}}\left[\mathcal{C}\left(\m{g}\right)^\top\mathcal{C}\left(\m{g}\right)\right]^{-\frac{1}{2}}\right)}.
\end{aligned}
\end{equation}
Hence,
\begin{equation}
\label{equ-delta}
\begin{aligned}
\left\|\m{\Delta}\right\|_2\leq&\left\|\mathcal{C}\left(\m{g}\right)\right\|_2\frac{\left\|\left[\mathcal{C}\left(\m{g}\right)^\top\m{Q}\mathcal{C}\left(\m{g}\right)\right]^{-1}
-\frac{1}{K}\left[\mathcal{C}\left(\m{g}\right)^\top\mathcal{C}\left(\m{g}\right)\right]^{-1}\right\|_2}{\lambda_{\text{min}}\left(\frac{1}{\sqrt{K}}\left[\mathcal{C}\left(\m{g}\right)^\top\mathcal{C}\left(\m{g}\right)\right]^{-\frac{1}{2}}\right)}\\
\leq&\left\|\mathcal{C}\left(\m{g}\right)\right\|_2\left\|\left[\mathcal{C}\left(\m{g}\right)^\top\mathcal{C}\left(\m{g}\right)\right]^{-1}\right\|_2\frac{\left\|\mathcal{C}\left(\m{g}\right)^\top\mathcal{C}\left(\m{g}\right)\left[\mathcal{C}\left(\m{g}\right)^\top\m{Q}\mathcal{C}\left(\m{g}\right)\right]^{-1}
-\frac{1}{K}\m{I}\right\|_2}{\lambda_{\text{min}}\left(\frac{1}{\sqrt{K}}\left[\mathcal{C}\left(\m{g}\right)^\top\mathcal{C}\left(\m{g}\right)\right]^{-\frac{1}{2}}\right)}\\
=&\sqrt{K}\left\|\mathcal{C}\left(\m{g}\right)\right\|_2^2\left\|\left[\mathcal{C}\left(\m{g}\right)^\top\mathcal{C}\left(\m{g}\right)\right]^{-1}\right\|_2\left\|\mathcal{C}\left(\m{g}\right)^\top\mathcal{C}\left(\m{g}\right)\left[\mathcal{C}\left(\m{g}\right)^\top\m{Q}\mathcal{C}\left(\m{g}\right)\right]^{-1}
-\frac{1}{K}\m{I}\right\|_2\\
=&\sqrt{K}{\kappa^2}\left\|\mathcal{C}\left(\m{g}\right)^\top\mathcal{C}\left(\m{g}\right)\left[\mathcal{C}\left(\m{g}\right)^\top\m{Q}\mathcal{C}\left(\m{g}\right)\right]^{-1}
-\frac{1}{K}\m{I}\right\|_2,
\end{aligned}
\end{equation}
where
\begin{equation}
\begin{aligned}
\m{Q}=\frac{1}{L}\sum_{l=1}^L\mathcal{C}\left(\m{x}_l\right)^\top\mathcal{C}\left(\m{x}_l\right).
\end{aligned}
\end{equation}
Furthermore,
\begin{equation}
\label{leq-cond-t}
\begin{aligned}
&\left\|\mathcal{C}\left(\m{g}\right)^\top\mathcal{C}\left(\m{g}\right)\left[\mathcal{C}\left(\m{g}\right)^\top\m{Q}\mathcal{C}\left(\m{g}\right)\right]^{-1}
-\frac{1}{K}\m{I}\right\|_2\\
=&\left\|\left(K\m{I}+\mathcal{C}\left(\m{g}\right)^\top\left(\m{Q}-K\m{I}\right)\mathcal{C}\left(\m{g}\right)\left[\mathcal{C}\left(\m{g}\right)^\top\mathcal{C}\left(\m{g}\right)\right]^{-1}\right)^{-1}-\frac{1}{K}\m{I}\right\|_2\\
=&\frac{1}{K}\left\|\left(\m{I}+\mathcal{C}\left(\m{g}\right)^\top\left(\frac{1}{K}\m{Q}-\m{I}\right)\mathcal{C}\left(\m{g}\right)\left[\mathcal{C}\left(\m{g}\right)^\top\mathcal{C}\left(\m{g}\right)\right]^{-1}\right)^{-1}-\m{I}\right\|_2\\
\leq&\frac{1}{K}\left\|\left(\m{I}+\mathcal{C}\left(\m{g}\right)^\top\left(\frac{1}{K}\m{Q}-\m{I}\right)\mathcal{C}\left(\m{g}\right)\left[\mathcal{C}\left(\m{g}\right)^\top\mathcal{C}\left(\m{g}\right)\right]^{-1}\right)^{-1}\right\|_2\\
&\cdot\left\|\mathcal{C}\left(\m{g}\right)^\top\left(\frac{1}{K}\m{Q}-\m{I}\right)\mathcal{C}\left(\m{g}\right)\left[\mathcal{C}\left(\m{g}\right)^\top\mathcal{C}\left(\m{g}\right)\right]^{-1}\right\|_2.
\end{aligned}
\end{equation}
We next derive a bound on $\left\|\frac{1}{K}\m{Q}-\m{I}\right\|_2$ using the following moment-controlled Bernstein inequality.
\begin{lem}[{\cite[Lemma 7]{shi2021manifold}}]
\label{moment-controlled}
Let $\left\{\m{Z}_l\right\}_{l=1}^L$ be a set of independent random matrices with dimension $N\times N$. Assume there exist $\sigma$ and $R$ such that for all $m\geq2$, $\mathbb{E}\left[\left\|\m{Z}_l\right\|_2^m\right]\leq\frac{m!}{2}\sigma^2R^{m-2}$. Denote $\m{S}=\frac{1}{L}\sum_{l=1}^L\m{Z}_l$, then we have for any $t>0$, the following bound holds with probability at least $1-2N\exp\left\{-\frac{Lt^2}{2\sigma^2+2Rt}\right\}$:
\begin{equation}
\begin{aligned}
\left\|\m{S}-\mathbb{E}\left[\m{S}\right]\right\|_2\leq t.
\end{aligned}
\end{equation}
\end{lem}

Following Lemma~\ref{moment-controlled}, define $\m{Z}_l=\frac{1}{K}\mathcal{C}\left(\m{x}_l\right)^\top\mathcal{C}\left(\m{x}_l\right)$ for $l=1,\cdots,L$. It follows immediately that $\m{S}=\frac{1}{K}\m{Q}$. Furthermore, $\mathbb{E}\left[\m{S}\right]=\m{I}$, since the $\left(n_1,n_2\right)$-th entry of $\mathbb{E}\!\left[\mathcal{C}\left(\m{x}_l\right)^\top\mathcal{C}\left(\m{x}_l\right)\right]$ satisfies
\begin{equation}
\mathbb{E}\left[\mathcal{C}\left(\m{x}_l\right)^\top\mathcal{C}\left(\m{x}_l\right)\right]=
\begin{cases}
\mathbb{E}\left[\left\|\m{x}_l\right\|_2^2\right]=K, & n_1=n_2,\\
0, & n_1\neq n_2,
\end{cases}
\end{equation}
since $\left\{\mathcal{P}_{\Omega}\left(\m{x}_l\right)\right\}_{l=1}^L$ are i.i.d. Gaussian random vectors with zero mean and covariance matrix $\m{I}$. Therefore, if there exist constants $\sigma$ and $R$ such that $\mathbb{E}\left[\left\|\m{Z}_l\right\|_2^m\right]\leq\frac{m!}{2}\sigma^2R^{m-2}$ for $m\geq2$, then a bound on $\left\|\m{S}-\mathbb{E}\left[\m{S}\right]\right\|_2=\left\|\frac{1}{K}\m{Q}-\m{I}\right\|_2$ follows directly from Lemma~\ref{moment-controlled}. It follows from Lemma~\ref{lemma-circulant convolution} that the circulant matrix $\mathcal{C}\left(\m{x}_l\right)$ admits the diagonalization
\begin{equation}
\begin{aligned}
\mathcal{C}\left(\m{x}_l\right)={\m{F}^{\left[N\right]}}\diag\left(\widehat{\m{x}_l}\right){\m{F}^{\left[N\right]}}^H.
\end{aligned}
\end{equation}
Note that the $n$-th entry of $\widehat{\m{x}_l}$, given by $\mathcal{P}_n\left(\widehat{\m{x}}_l\right)$, is a complex Gaussian random variable with zero mean and variance 
\begin{equation}
\begin{aligned}
\mathbb{E}\left[\left|\mathcal{P}_n\left(\widehat{\m{x}_l}\right)\right|^2\right]=&\mathbb{E}\left[N{\m{F}^{\left[N\right]}}_{:,n}^H\m{x}_l\m{x}_l^\top{\m{F}^{\left[N\right]}}_{:,n}\right]=N{\m{F}^{\left[N\right]}}_{:,n}^H\m{\Sigma}_\Omega{\m{F}^{\left[N\right]}}_{:,n}=K.
\end{aligned}
\end{equation}
Hence, by \cite[Proposition 2.5.8]{vershynin2018high}, $\mathcal{P}_n\left(\widehat{\m{x}_l}\right)$ is a sub-Gaussian random variable. Consequently, there exists a positive constant $c_1$ such that the sub-Gaussian norm of $\mathcal{P}_n\left(\widehat{\m{x}_l}\right)$ satisfies
\begin{equation}
\begin{aligned}
\left\|\mathcal{P}_n\left(\widehat{\m{x}_l}\right)\right\|_{\text{sub-Gaussian}}\leq c_1\sqrt{K}.
\end{aligned}
\end{equation}
Consequently,
\begin{equation}
\label{union bound}
\begin{aligned}
\mathbb{P}\left[\left\|\mathcal{C}\left(\m{x}_l\right)\right\|_2\geq t\right]
=&\mathbb{P}\left[\left\|{\m{F}^{\left[N\right]}}^H\diag\left(\widehat{\m{x}_l}\right){\m{F}^{\left[N\right]}}\right\|_2\geq t\right]\\
=&\mathbb{P}\left[\max_{n=1,\cdots,N}\left|\mathcal{P}_n\left(\widehat{\m{x}_l}\right)\right|\geq t\right]\\
\leq&\sum_{n=1}^N\mathbb{P}\left[\left|\mathcal{P}_n\left(\widehat{\m{x}_l}\right)\right|\geq t\right].
\end{aligned}
\end{equation}
Applying the union bound in \eqref{union bound} together with the tail bound for sub-Gaussian random variables in \cite[Eq. (2.14)]{vershynin2018high}, we conclude that there exists a positive constant $c_2$ such that
\begin{equation}
\begin{aligned}
\mathbb{P}\left(\left\|\mathcal{C}\left(\m{x}_l\right)\right\|_2\geq t\right)
\leq2N\exp\left\{-\frac{c_2t^2}{c_1K}\right\}.
\end{aligned}
\end{equation}
Furthermore,
\begin{equation}
\begin{aligned}
\mathbb{E}\left[\left\|\mathcal{C}\left(\m{x}_l\right)\right\|_2^{2m}\right]=&\int_0^\infty\mathbb{P}\left(\left\|\mathcal{C}\left(\m{x}_l\right)\right\|_2^{2m}\geq u\right)du=\int_0^\infty\mathbb{P}\left(\left\|\mathcal{C}\left(\m{x}_l\right)\right\|_2\geq t\right)2mt^{2m-1}dt.
\end{aligned}
\end{equation}
Note that if $t\geq\sqrt{\frac{2c_1K}{c_2}\log 2N}$, then
\begin{equation}
\begin{aligned}
2N\exp\left\{-\frac{c_2t^2}{c_1K}\right\}\leq\exp\left\{-\frac{c_2t^2}{2c_1K}\right\}.
\end{aligned}
\end{equation}
Therefore,
\begin{equation}
\begin{aligned}
\mathbb{E}\left[\left\|\mathcal{C}\left(\m{x}_l\right)\right\|_2^{2m}\right]=&\int_0^\infty\mathbb{P}\left(\left\|\mathcal{C}\left(\m{x}_l\right)\right\|_2^{2m}\geq u\right)du\\
=&\int_0^\infty\mathbb{P}\left(\left\|\mathcal{C}\left(\m{x}_l\right)\right\|_2\geq t\right)2mt^{2m-1}dt\\
\leq&\int_0^{\sqrt{\frac{2c_1K}{c_2}\log 2N}}2mt^{2m-1}dt+\int_0^\infty\exp\left\{-\frac{c_2t^2}{2c_1K}\right\}2mt^{2m-1}dt\\
\leq&\int_0^{\sqrt{\frac{2c_1K}{c_2}\log 2N}}2mt^{2m-1}dt+\int_0^\infty\exp\left\{-\frac{c_2t^2}{2c_1K}\right\}2mt^{2m-1}dt\\
=&\left(\sqrt{\frac{2c_1K}{c_2}\log 2N}\right)^{2m}+m\int_0^\infty\exp\left\{-v\right\}\left(\frac{2c_1Kv}{c_2}\right)^{\frac{2m-1}{2}}\left(\frac{2c_1K}{c_2}\right)^{\frac{1}{2}}v^{-\frac{1}{2}}dv\\
=&\left(\frac{2c_1K}{c_2}\log2N\right)^m+m\left(\frac{2c_1K}{c_2}\right)^{m}\int_0^\infty\exp\left\{-v\right\}v^{m-1}dv\\
=&\left(\frac{2c_1K}{c_2}\log2N\right)^m+\left(\frac{2c_1K}{c_2}\right)^{m}m!,
\end{aligned}
\end{equation}
where the last equality follows from the definition of the Gamma function \cite{sebah2002introduction}. Moreover,
\begin{equation}
\begin{aligned}
\mathbb{E}\left[\left\|\mathcal{C}\left(\m{x}_l\right)\right\|_2^{2m}\right]\leq&\left(\frac{2c_1K}{c_2}\log2N\right)^m+\left(\frac{2c_1K}{c_2}\right)^{m}m!\leq2m!\left(\frac{2c_1K}{c_2}\log2N\right)^m.
\end{aligned}
\end{equation}
It follows immediately that
\begin{equation}
\begin{aligned}
\mathbb{E}\left[\left\|\frac{1}{K}\mathcal{C}\left(\m{x}_l\right)^\top\mathcal{C}\left(\m{x}_l\right)\right\|_2^m\right]=&\frac{1}{K^m}\mathbb{E}\left[\left\|\mathcal{C}\left(\m{x}_l\right)\right\|_2^{2m}\right]\leq2m!\left(\frac{2c_1}{c_2}\log2N\right)^m.
\end{aligned}
\end{equation}
Therefore, we may choose $\sigma^2=4\left(\frac{2c_1}{c_2}\log2N\right)^2$ and $R=\frac{2c_1}{c_2}\log2N$ such that $\mathbb{E}\left[\left\|\m{Z}_l\right\|_2^m\right]\leq\frac{m!}{2}\sigma^2R^{m-2}$. Applying Lemma~\ref{moment-controlled} then yields that, for any $t>0$, the following bound holds with probability at least $1-2N\exp\left\{-\frac{Lt^2}{2\sigma^2+2Rt}\right\}$:
\begin{equation}
\begin{aligned}
\left\|\frac{1}{K}\m{Q}-\m{I}\right\|_2\leq t.
\end{aligned}
\end{equation}
Furthermore, for any $t\in\left(0,\frac{1}{2}\kappa^{-2}\right)$, it follows that
\begin{equation}
\begin{aligned}
\left\|\mathcal{C}\left(\m{g}\right)^\top\left(\frac{1}{K}\m{Q}-\m{I}\right)\mathcal{C}\left(\m{g}\right)\left[\mathcal{C}\left(\m{g}\right)^\top\mathcal{C}\left(\m{g}\right)\right]^{-1}\right\|_2\leq\kappa^2t\leq\frac{1}{2}.
\end{aligned}
\end{equation}
Consequently, by \eqref{leq-cond-t},
\begin{equation}
\begin{aligned}
&\left\|\mathcal{C}\left(\m{g}\right)^\top\mathcal{C}\left(\m{g}\right)\left[\mathcal{C}\left(\m{g}\right)^\top\m{Q}\mathcal{C}\left(\m{g}\right)\right]^{-1}
-\frac{1}{K}\m{I}\right\|_2\\
\leq&\frac{1}{K}\frac{\left\|\mathcal{C}\left(\m{g}\right)^\top\left(\frac{1}{K}\m{Q}-\m{I}\right)\mathcal{C}\left(\m{g}\right)\left[\mathcal{C}\left(\m{g}\right)^\top\mathcal{C}\left(\m{g}\right)\right]^{-1}\right\|_2}{1-\left\|\mathcal{C}\left(\m{g}\right)^\top\left(\frac{1}{K}\m{Q}-\m{I}\right)\mathcal{C}\left(\m{g}\right)\left[\mathcal{C}\left(\m{g}\right)^\top\mathcal{C}\left(\m{g}\right)\right]^{-1}\right\|_2}\\
\leq&\frac{2}{K}\left\|\mathcal{C}\left(\m{g}\right)^\top\left(\frac{1}{K}\m{Q}-\m{I}\right)\mathcal{C}\left(\m{g}\right)\left[\mathcal{C}\left(\m{g}\right)^\top\mathcal{C}\left(\m{g}\right)\right]^{-1}\right\|_2\\
\leq&\frac{2\kappa^2}{K}\left\|\frac{1}{K}\m{Q}-\m{I}\right\|_2.
\end{aligned}
\end{equation}
Therefore, by \eqref{equ-delta},
\begin{equation}
\begin{aligned}
\left\|\m{\Delta}\right\|_2\leq \frac{2\kappa^4t}{\sqrt{K}}.
\end{aligned}
\end{equation}
Substituting this bound into \eqref{Delta-R} further gives
\begin{equation}
\begin{aligned}
\left\|\m{R}-\frac{1}{\sqrt{K}}\left[\mathcal{C}\left(\m{g}\right)^\top\mathcal{C}\left(\m{g}\right)\right]^{-\frac{1}{2}}\right\|_2\leq&\frac{1}{\sigma_{\text{min}}\left(\mathcal{C}\left(\m{g}\right)\right)}\left\|\m{\Delta}\right\|_2\\
\leq&\frac{2\kappa^4t}{\sqrt{K}\sigma_{\text{min}}\left(\mathcal{C}\left(\m{g}\right)\right)}.
\end{aligned}
\end{equation}
This completes the proof.



\section{Proof of Lemma~\ref{lem-2st-general}}
\label{pf-lem-2st-general}
The Riemannian Hessian of $\psi_\varepsilon$ is given by
\begin{equation}
\begin{aligned}
\text{Hess}\;\psi_\varepsilon\left(\m{h}\right)=&\left(\m{I}-\m{h}\m{h}^\top\right)\Bigg\{\sum_{n=1}^N\left[\m{h}^\top\left(\frac{1}{\sqrt{K}}\m{I}+\m{\Delta}\right)^\top\m{\Gamma}_n^\top\left(\m{\Sigma}_{\Omega}+\m{E}_{\Omega}\right)\m{\Gamma}_n\left(\frac{1}{\sqrt{K}}\m{I}+\m{\Delta}\right)\m{h}+\varepsilon\right]^{-\frac{1}{2}}\\
&\cdot\left(\frac{1}{\sqrt{K}}\m{I}+\m{\Delta}\right)^\top\m{\Gamma}_n^\top\left(\m{\Sigma}_{\Omega}+\m{E}_{\Omega}\right)\m{\Gamma}_n\left(\frac{1}{\sqrt{K}}\m{I}+\m{\Delta}\right)\\
&-\sum_{n=1}^N\left[\m{h}^\top\left(\frac{1}{\sqrt{K}}\m{I}+\m{\Delta}\right)^\top\m{\Gamma}_n^\top\left(\m{\Sigma}_{\Omega}+\m{E}_{\Omega}\right)\m{\Gamma}_n\left(\frac{1}{\sqrt{K}}\m{I}+\m{\Delta}\right)\m{h}+\varepsilon\right]^{-\frac{3}{2}}\\
&\cdot\left(\frac{1}{\sqrt{K}}\m{I}+\m{\Delta}\right)^\top\m{\Gamma}_n^\top\left(\m{\Sigma}_{\Omega}+\m{E}_{\Omega}\right)\m{\Gamma}_n\left(\frac{1}{\sqrt{K}}\m{I}+\m{\Delta}\right)\m{h}\m{h}^\top\\
&\cdot\left(\frac{1}{\sqrt{K}}\m{I}+\m{\Delta}\right)^\top\m{\Gamma}_n^\top\left(\m{\Sigma}_{\Omega}+\m{E}_{\Omega}\right)\m{\Gamma}_n\left(\frac{1}{\sqrt{K}}\m{I}+\m{\Delta}\right)\\
&-\sum_{n=1}^N\left[\m{h}^\top\left(\frac{1}{\sqrt{K}}\m{I}+\m{\Delta}\right)^\top\m{\Gamma}_n^\top\left(\m{\Sigma}_{\Omega}+\m{E}_{\Omega}\right)\m{\Gamma}_n\left(\frac{1}{\sqrt{K}}\m{I}+\m{\Delta}\right)\m{h}+\varepsilon\right]^{-\frac{1}{2}}\\
&\cdot\left[\m{h}^\top\left(\frac{1}{\sqrt{K}}\m{I}+\m{\Delta}\right)^\top\m{\Gamma}_n^\top\left(\m{\Sigma}_{\Omega}+\m{E}_{\Omega}\right)\m{\Gamma}_n\left(\frac{1}{\sqrt{K}}\m{I}+\m{\Delta}\right)\m{h}\right]\m{I}\Bigg\}\left(\m{I}-\m{h}\m{h}^\top\right)\\
=&\sum_{n=1}^N\frac{\left(\m{I}-\m{h}\m{h}^\top\right)\widetilde{\m{H}}_n\left(\m{I}-\m{h}\m{h}^\top\right)}{\left[\m{h}^\top\left(\frac{1}{\sqrt{K}}\m{I}+\m{\Delta}\right)^\top\m{\Gamma}_n^\top\left(\m{\Sigma}_{\Omega}+\m{E}_{\Omega}\right)\m{\Gamma}_n\left(\frac{1}{\sqrt{K}}\m{I}+\m{\Delta}\right)\m{h}+\varepsilon\right]^{\frac{3}{2}}},
\end{aligned}
\end{equation}
where
\begin{equation}
\begin{aligned}
\widetilde{\m{H}}_n=&\left[\m{h}^\top\left(\frac{1}{\sqrt{K}}\m{I}+\m{\Delta}\right)^\top\m{\Gamma}_n^\top\left(\m{\Sigma}_{\Omega}+\m{E}_{\Omega}\right)\m{\Gamma}_n\left(\frac{1}{\sqrt{K}}\m{I}+\m{\Delta}\right)\m{h}+\varepsilon\right]\\
&\cdot\left(\frac{1}{\sqrt{K}}\m{I}+\m{\Delta}\right)^\top\m{\Gamma}_n^\top\left(\m{\Sigma}_{\Omega}+\m{E}_{\Omega}\right)\m{\Gamma}_n\left(\frac{1}{\sqrt{K}}\m{I}+\m{\Delta}\right)\\
&-\left(\frac{1}{\sqrt{K}}\m{I}+\m{\Delta}\right)^\top\m{\Gamma}_n^\top\left(\m{\Sigma}_{\Omega}+\m{E}_{\Omega}\right)\m{\Gamma}_n\left(\frac{1}{\sqrt{K}}\m{I}+\m{\Delta}\right)\m{h}\m{h}^\top\\
&\cdot\left(\frac{1}{\sqrt{K}}\m{I}+\m{\Delta}\right)^\top\m{\Gamma}_n^\top\left(\m{\Sigma}_{\Omega}+\m{E}_{\Omega}\right)\m{\Gamma}_n\left(\frac{1}{\sqrt{K}}\m{I}+\m{\Delta}\right)\\
&-\left[\m{h}^\top\left(\frac{1}{\sqrt{K}}\m{I}+\m{\Delta}\right)^\top\m{\Gamma}_n^\top\left(\m{\Sigma}_{\Omega}+\m{E}_{\Omega}\right)\m{\Gamma}_n\left(\frac{1}{\sqrt{K}}\m{I}+\m{\Delta}\right)\m{h}+\varepsilon\right]\\
&\cdot\left[\m{h}^\top\left(\frac{1}{\sqrt{K}}\m{I}+\m{\Delta}\right)^\top\m{\Gamma}_n^\top\left(\m{\Sigma}_{\Omega}+\m{E}_{\Omega}\right)\m{\Gamma}_n\left(\frac{1}{\sqrt{K}}\m{I}+\m{\Delta}\right)\m{h}\right]\m{I}.
\end{aligned}
\end{equation}
We obtain
\begin{equation}
\begin{aligned}
\m{d}^\top\widetilde{\m{H}}_n\m{d}
=&\left[\m{h}^\top\left(\frac{1}{\sqrt{K}}\m{I}+\m{\Delta}\right)^\top\m{\Gamma}_n^\top\left(\m{\Sigma}_{\Omega}+\m{E}_{\Omega}\right)\m{\Gamma}_n\left(\frac{1}{\sqrt{K}}\m{I}+\m{\Delta}\right)\m{h}+\varepsilon\right]\\
&\cdot\m{d}^\top\left(\frac{1}{\sqrt{K}}\m{I}+\m{\Delta}\right)^\top\m{\Gamma}_n^\top\left(\m{\Sigma}_{\Omega}+\m{E}_{\Omega}\right)\m{\Gamma}_n\left(\frac{1}{\sqrt{K}}\m{I}+\m{\Delta}\right)\m{d}\\
&-\m{d}^\top\left(\frac{1}{\sqrt{K}}\m{I}+\m{\Delta}\right)^\top\m{\Gamma}_n^\top\left(\m{\Sigma}_{\Omega}+\m{E}_{\Omega}\right)\m{\Gamma}_n\left(\frac{1}{\sqrt{K}}\m{I}+\m{\Delta}\right)\m{h}\\
&\cdot\m{h}^\top\left(\frac{1}{\sqrt{K}}\m{I}+\m{\Delta}\right)^\top\m{\Gamma}_n^\top\left(\m{\Sigma}_{\Omega}+\m{E}_{\Omega}\right)\m{\Gamma}_n\left(\frac{1}{\sqrt{K}}\m{I}+\m{\Delta}\right)\m{d}\\
&-\left[\m{h}^\top\left(\frac{1}{\sqrt{K}}\m{I}+\m{\Delta}\right)^\top\m{\Gamma}_n^\top\left(\m{\Sigma}_{\Omega}+\m{E}_{\Omega}\right)\m{\Gamma}_n\left(\frac{1}{\sqrt{K}}\m{I}+\m{\Delta}\right)\m{h}+\varepsilon\right]\\
&\cdot\left[\m{h}^\top\left(\frac{1}{\sqrt{K}}\m{I}+\m{\Delta}\right)^\top\m{\Gamma}_n^\top\left(\m{\Sigma}_{\Omega}+\m{E}_{\Omega}\right)\m{\Gamma}_n\left(\frac{1}{\sqrt{K}}\m{I}+\m{\Delta}\right)\m{h}\right]\left\|\m{d}\right\|_2^2.
\end{aligned}
\end{equation}
Furthermore,
\begin{equation}
\begin{aligned}
\m{d}^\top\widetilde{\m{H}}_n\m{d}
=&\left[\m{h}^\top\left(\frac{1}{\sqrt{K}}\m{I}+\m{\Delta}\right)^\top\m{\Gamma}_n^\top\left(\m{\Sigma}_{\Omega}+\m{E}_{\Omega}\right)\m{\Gamma}_n\left(\frac{1}{\sqrt{K}}\m{I}+\m{\Delta}\right)\m{h}+\varepsilon\right]\\
&\cdot\m{e}_o^\top\left(\frac{1}{\sqrt{K}}\m{I}+\m{\Delta}\right)^\top\m{\Gamma}_n^\top\left(\m{\Sigma}_{\Omega}+\m{E}_{\Omega}\right)\m{\Gamma}_n\left(\frac{1}{\sqrt{K}}\m{I}+\m{\Delta}\right)\m{e}_o\\
&+\varepsilon h_o^2\m{h}^\top\left(\frac{1}{\sqrt{K}}\m{I}+\m{\Delta}\right)^\top\m{\Gamma}_n^\top\left(\m{\Sigma}_{\Omega}+\m{E}_{\Omega}\right)\m{\Gamma}_n\left(\frac{1}{\sqrt{K}}\m{I}+\m{\Delta}\right)\m{h}\\
&-2\varepsilon h_o\m{h}^\top\left(\frac{1}{\sqrt{K}}\m{I}+\m{\Delta}\right)^\top\m{\Gamma}_n^\top\left(\m{\Sigma}_{\Omega}+\m{E}_{\Omega}\right)\m{\Gamma}_n\left(\frac{1}{\sqrt{K}}\m{I}+\m{\Delta}\right)\m{e}_o\\
&-\left[\m{e}_o^\top\left(\frac{1}{\sqrt{K}}\m{I}+\m{\Delta}\right)^\top\m{\Gamma}_n^\top\left(\m{\Sigma}_{\Omega}+\m{E}_{\Omega}\right)\m{\Gamma}_n\left(\frac{1}{\sqrt{K}}\m{I}+\m{\Delta}\right)\m{h}\right]^2\\
&-\left[\m{h}^\top\left(\frac{1}{\sqrt{K}}\m{I}+\m{\Delta}\right)^\top\m{\Gamma}_n^\top\left(\m{\Sigma}_{\Omega}+\m{E}_{\Omega}\right)\m{\Gamma}_n\left(\frac{1}{\sqrt{K}}\m{I}+\m{\Delta}\right)\m{h}+\varepsilon\right]\\
&\cdot\left[\m{h}^\top\left(\frac{1}{\sqrt{K}}\m{I}+\m{\Delta}\right)^\top\m{\Gamma}_n^\top\left(\m{\Sigma}_{\Omega}+\m{E}_{\Omega}\right)\m{\Gamma}_n\left(\frac{1}{\sqrt{K}}\m{I}+\m{\Delta}\right)\m{h}\right]\left\|\m{d}\right\|_2^2.
\end{aligned}
\end{equation}
Consequently, by defining
\begin{equation}
\begin{aligned}
\chi_{\m{\Delta},\m{h}}=&\left\{n:\;\m{h}^\top\left(\frac{1}{\sqrt{K}}\m{I}+\m{\Delta}\right)^\top\m{\Gamma}_n^\top\m{\Sigma}_{\Omega}\m{\Gamma}_n\left(\frac{1}{\sqrt{K}}\m{I}+\m{\Delta}\right)\m{h}>0\right\},
\end{aligned}
\end{equation}
and following derivations similar to those in the proof of Lemma~\ref{lem-2st}, we obtain
\begin{equation}
\begin{aligned}
\frac{\m{d}^\top\text{Hess}\;\psi_\varepsilon\left(\m{h}\right)\m{d}}{\left\|\m{d}\right\|_2^2}
=&\frac{1}{\left\|\m{d}\right\|_2^2}\sum_{n=1}^N\frac{\m{d}^\top\widetilde{\m{H}}_n\m{d}}{\left[\m{h}^\top\left(\frac{1}{\sqrt{K}}\m{I}+\m{\Delta}\right)^\top\m{\Gamma}_n^\top\left(\m{\Sigma}_{\Omega}+\m{E}_{\Omega}\right)\m{\Gamma}_n\left(\frac{1}{\sqrt{K}}\m{I}+\m{\Delta}\right)\m{h}+\varepsilon\right]^{\frac{3}{2}}}\\
=&\widetilde{\text{main}}+\widetilde{\text{res}},
\end{aligned}
\end{equation}
where
\begin{equation}
\begin{aligned}
\widetilde{\text{main}}=&\frac{1}{\left\|\m{d}\right\|_2^2}\sum_{n=1}^N\frac{\m{e}_o^\top\left(\frac{1}{\sqrt{K}}\m{I}+\m{\Delta}\right)^\top\m{\Gamma}_n^\top\m{\Sigma}_{\Omega}\m{\Gamma}_n\left(\frac{1}{\sqrt{K}}\m{I}+\m{\Delta}\right)\m{e}_o}{\left[\m{h}^\top\left(\frac{1}{\sqrt{K}}\m{I}+\m{\Delta}\right)^\top\m{\Gamma}_n^\top\left(\m{\Sigma}_{\Omega}+\m{E}_{\Omega}\right)\m{\Gamma}_n\left(\frac{1}{\sqrt{K}}\m{I}+\m{\Delta}\right)\m{h}+\varepsilon\right]^{\frac{1}{2}}}\\
&-\sum_{n\in\chi_{\m{\Delta},\m{h}}}\frac{\m{h}^\top\left(\frac{1}{\sqrt{K}}\m{I}+\m{\Delta}\right)^\top\m{\Gamma}_n^\top\m{\Sigma}_{\Omega}\m{\Gamma}_n\left(\frac{1}{\sqrt{K}}\m{I}+\m{\Delta}\right)\m{h}}{\left[\m{h}^\top\left(\frac{1}{\sqrt{K}}\m{I}+\m{\Delta}\right)^\top\m{\Gamma}_n^\top\left(\m{\Sigma}_{\Omega}+\m{E}_{\Omega}\right)\m{\Gamma}_n\left(\frac{1}{\sqrt{K}}\m{I}+\m{\Delta}\right)\m{h}+\varepsilon\right]^{\frac{1}{2}}}\\
&-\frac{1}{\left\|\m{d}\right\|_2^2}\sum_{n\in\chi_{\m{\Delta},\m{h}}}\frac{\left[\m{e}_o^\top\left(\frac{1}{\sqrt{K}}\m{I}+\m{\Delta}\right)^\top\m{\Gamma}_n^\top\m{\Sigma}_{\Omega}\m{\Gamma}_n\left(\frac{1}{\sqrt{K}}\m{I}+\m{\Delta}\right)\m{h}\right]^2}{\left[\m{h}^\top\left(\frac{1}{\sqrt{K}}\m{I}+\m{\Delta}\right)^\top\m{\Gamma}_n^\top\left(\m{\Sigma}_{\Omega}+\m{E}_{\Omega}\right)\m{\Gamma}_n\left(\frac{1}{\sqrt{K}}\m{I}+\m{\Delta}\right)\m{h}+\varepsilon\right]^{\frac{3}{2}}},
\end{aligned}
\end{equation}
and
\begin{equation}
\begin{aligned}
\widetilde{\text{res}}\leq&\frac{1}{\left\|\m{d}\right\|_2^2}\sum_{n\in\chi_{\m{\Delta},\m{h}}}\frac{\varepsilon h_o^2\m{h}^\top\left(\frac{1}{\sqrt{K}}\m{I}+\m{\Delta}\right)^\top\m{\Gamma}_n^\top\left(\m{\Sigma}_{\Omega}+\m{E}_{\Omega}\right)\m{\Gamma}_n\left(\frac{1}{\sqrt{K}}\m{I}+\m{\Delta}\right)\m{h}}{\left[\m{h}^\top\left(\frac{1}{\sqrt{K}}\m{I}+\m{\Delta}\right)^\top\m{\Gamma}_n^\top\left(\m{\Sigma}_{\Omega}+\m{E}_{\Omega}\right)\m{\Gamma}_n\left(\frac{1}{\sqrt{K}}\m{I}+\m{\Delta}\right)\m{h}+\varepsilon\right]^{\frac{3}{2}}}\\
&-\frac{1}{\left\|\m{d}\right\|_2^2}\sum_{n\in\chi_{\m{\Delta},\m{h}}}\frac{2\varepsilon h_o\m{h}^\top\left(\frac{1}{\sqrt{K}}\m{I}+\m{\Delta}\right)^\top\m{\Gamma}_n^\top\left(\m{\Sigma}_{\Omega}+\m{E}_{\Omega}\right)\m{\Gamma}_n\left(\frac{1}{\sqrt{K}}\m{I}+\m{\Delta}\right)\m{e}_o}{\left[\m{h}^\top\left(\frac{1}{\sqrt{K}}\m{I}+\m{\Delta}\right)^\top\m{\Gamma}_n^\top\left(\m{\Sigma}_{\Omega}+\m{E}_{\Omega}\right)\m{\Gamma}_n\left(\frac{1}{\sqrt{K}}\m{I}+\m{\Delta}\right)\m{h}+\varepsilon\right]^{\frac{3}{2}}}\\
&+\frac{1}{\left\|\m{d}\right\|_2^2}\sum_{n=1}^N\frac{\m{e}_o^\top\left(\frac{1}{\sqrt{K}}\m{I}+\m{\Delta}\right)^\top\m{\Gamma}_n^\top\m{E}_{\Omega}\m{\Gamma}_n\left(\frac{1}{\sqrt{K}}\m{I}+\m{\Delta}\right)\m{e}_o}{\left[\m{h}^\top\left(\frac{1}{\sqrt{K}}\m{I}+\m{\Delta}\right)^\top\m{\Gamma}_n^\top\left(\m{\Sigma}_{\Omega}+\m{E}_{\Omega}\right)\m{\Gamma}_n\left(\frac{1}{\sqrt{K}}\m{I}+\m{\Delta}\right)\m{h}+\varepsilon\right]^{\frac{1}{2}}}\\
&-\sum_{n\in\chi_{\m{\Delta},\m{h}}}\frac{\m{h}^\top\left(\frac{1}{\sqrt{K}}\m{I}+\m{\Delta}\right)^\top\m{\Gamma}_n^\top\m{E}_{\Omega}\m{\Gamma}_n\left(\frac{1}{\sqrt{K}}\m{I}+\m{\Delta}\right)\m{h}}{\left[\m{h}^\top\left(\frac{1}{\sqrt{K}}\m{I}+\m{\Delta}\right)^\top\m{\Gamma}_n^\top\left(\m{\Sigma}_{\Omega}+\m{E}_{\Omega}\right)\m{\Gamma}_n\left(\frac{1}{\sqrt{K}}\m{I}+\m{\Delta}\right)\m{h}+\varepsilon\right]^{\frac{1}{2}}}\\
&-\frac{1}{\left\|\m{d}\right\|_2^2}\sum_{n\in\chi_{\m{\Delta},\m{h}}}\frac{\left[\m{e}_o^\top\left(\frac{1}{\sqrt{K}}\m{I}+\m{\Delta}\right)^\top\m{\Gamma}_n^\top\m{E}_{\Omega}\m{\Gamma}_n\left(\frac{1}{\sqrt{K}}\m{I}+\m{\Delta}\right)\m{h}\right]^2}{\left[\m{h}^\top\left(\frac{1}{\sqrt{K}}\m{I}+\m{\Delta}\right)^\top\m{\Gamma}_n^\top\left(\m{\Sigma}_{\Omega}+\m{E}_{\Omega}\right)\m{\Gamma}_n\left(\frac{1}{\sqrt{K}}\m{I}+\m{\Delta}\right)\m{h}+\varepsilon\right]^{\frac{3}{2}}}\\
&-\frac{2}{\left\|\m{d}\right\|_2^2}\sum_{n\in\chi_{\m{\Delta},\m{h}}}\frac{\m{e}_o^\top\left(\frac{1}{\sqrt{K}}\m{I}+\m{\Delta}\right)^\top\m{\Gamma}_n^\top\m{\Sigma}_{\Omega}\m{\Gamma}_n\left(\frac{1}{\sqrt{K}}\m{I}+\m{\Delta}\right)\m{h}}{\left[\m{h}^\top\left(\frac{1}{\sqrt{K}}\m{I}+\m{\Delta}\right)^\top\m{\Gamma}_n^\top\left(\m{\Sigma}_{\Omega}+\m{E}_{\Omega}\right)\m{\Gamma}_n\left(\frac{1}{\sqrt{K}}\m{I}+\m{\Delta}\right)\m{h}+\varepsilon\right]^{\frac{3}{2}}}\\
&\cdot\m{e}_o^\top\left(\frac{1}{\sqrt{K}}\m{I}+\m{\Delta}\right)^\top\m{\Gamma}_n^\top\m{E}_{\Omega}\m{\Gamma}_n\left(\frac{1}{\sqrt{K}}\m{I}+\m{\Delta}\right)\m{h}.
\end{aligned}
\end{equation}

First, we derive an upper bound for the term $\widetilde{\text{res}}$. By discarding all nonpositive terms, $\widetilde{\text{res}}$ can be further bounded as
\begin{equation}
\begin{aligned}
\widetilde{\text{res}}\leq&-\sum_{n\in\chi_{\m{\Delta},\m{h}}}\frac{\m{h}^\top\left(\frac{1}{\sqrt{K}}\m{I}+\m{\Delta}\right)^\top\m{\Gamma}_n^\top\m{E}_{\Omega}\m{\Gamma}_n\left(\frac{1}{\sqrt{K}}\m{I}+\m{\Delta}\right)\m{h}}{\left[\m{h}^\top\left(\frac{1}{\sqrt{K}}\m{I}+\m{\Delta}\right)^\top\m{\Gamma}_n^\top\left(\m{\Sigma}_{\Omega}+\m{E}_{\Omega}\right)\m{\Gamma}_n\left(\frac{1}{\sqrt{K}}\m{I}+\m{\Delta}\right)\m{h}+\varepsilon\right]^{\frac{1}{2}}}\\
&+\frac{1}{\left\|\m{d}\right\|_2^2}\sum_{n\in\chi_{\m{\Delta},\m{h}}}\frac{\varepsilon h_o^2\m{h}^\top\left(\frac{1}{\sqrt{K}}\m{I}+\m{\Delta}\right)^\top\m{\Gamma}_n^\top\left(\m{\Sigma}_{\Omega}+\m{E}_{\Omega}\right)\m{\Gamma}_n\left(\frac{1}{\sqrt{K}}\m{I}+\m{\Delta}\right)\m{h}}{\left[\m{h}^\top\left(\frac{1}{\sqrt{K}}\m{I}+\m{\Delta}\right)^\top\m{\Gamma}_n^\top\left(\m{\Sigma}_{\Omega}+\m{E}_{\Omega}\right)\m{\Gamma}_n\left(\frac{1}{\sqrt{K}}\m{I}+\m{\Delta}\right)\m{h}+\varepsilon\right]^{\frac{3}{2}}}\\
&+\frac{1}{\left\|\m{d}\right\|_2^2}\sum_{n=1}^N\frac{\m{e}_o^\top\left(\frac{1}{\sqrt{K}}\m{I}+\m{\Delta}\right)^\top\m{\Gamma}_n^\top\m{E}_{\Omega}\m{\Gamma}_n\left(\frac{1}{\sqrt{K}}\m{I}+\m{\Delta}\right)\m{e}_o}{\left[\m{h}^\top\left(\frac{1}{\sqrt{K}}\m{I}+\m{\Delta}\right)^\top\m{\Gamma}_n^\top\left(\m{\Sigma}_{\Omega}+\m{E}_{\Omega}\right)\m{\Gamma}_n\left(\frac{1}{\sqrt{K}}\m{I}+\m{\Delta}\right)\m{h}+\varepsilon\right]^{\frac{1}{2}}}\\
&-\frac{2\varepsilon h_o}{\left\|\m{d}\right\|_2^2}\sum_{n\in\chi_{\m{\Delta},\m{h}}}\frac{\m{h}^\top\left(\frac{1}{\sqrt{K}}\m{I}+\m{\Delta}\right)^\top\m{\Gamma}_n^\top\m{E}_{\Omega}\m{\Gamma}_n\left(\frac{1}{\sqrt{K}}\m{I}+\m{\Delta}\right)\m{e}_o}{\left[\m{h}^\top\left(\frac{1}{\sqrt{K}}\m{I}+\m{\Delta}\right)^\top\m{\Gamma}_n^\top\left(\m{\Sigma}_{\Omega}+\m{E}_{\Omega}\right)\m{\Gamma}_n\left(\frac{1}{\sqrt{K}}\m{I}+\m{\Delta}\right)\m{h}+\varepsilon\right]^{\frac{3}{2}}}\\
&-\frac{2}{\left\|\m{d}\right\|_2^2}\sum_{n\in\chi_{\m{\Delta},\m{h}}}\frac{\m{e}_o^\top\left(\frac{1}{\sqrt{K}}\m{I}+\m{\Delta}\right)^\top\m{\Gamma}_n^\top\m{\Sigma}_{\Omega}\m{\Gamma}_n\left(\frac{1}{\sqrt{K}}\m{I}+\m{\Delta}\right)\m{h}}{\left[\m{h}^\top\left(\frac{1}{\sqrt{K}}\m{I}+\m{\Delta}\right)^\top\m{\Gamma}_n^\top\left(\m{\Sigma}_{\Omega}+\m{E}_{\Omega}\right)\m{\Gamma}_n\left(\frac{1}{\sqrt{K}}\m{I}+\m{\Delta}\right)\m{h}+\varepsilon\right]^{\frac{3}{2}}}\\
&\cdot\m{e}_o^\top\left(\frac{1}{\sqrt{K}}\m{I}+\m{\Delta}\right)^\top\m{\Gamma}_n^\top\m{E}_{\Omega}\m{\Gamma}_n\left(\frac{1}{\sqrt{K}}\m{I}+\m{\Delta}\right)\m{h}\\
&-\frac{2\varepsilon h_o}{\left\|\m{d}\right\|_2^2}\sum_{n\in\chi_{\m{\Delta},\m{h}}}\frac{\frac{1}{\sqrt{K}}\m{h}^\top\m{\Gamma}_n^\top\m{\Sigma}_{\Omega}\m{\Gamma}_n\m{\Delta}\m{e}_o+\frac{1}{\sqrt{K}}\m{h}^\top\m{\Delta}^\top\m{\Gamma}_n^\top\m{\Sigma}_{\Omega}\m{\Gamma}_n\m{e}_o+\m{h}^\top\m{\Delta}^\top\m{\Gamma}_n^\top\m{\Sigma}_{\Omega}\m{\Gamma}_n\m{\Delta}\m{e}_o}{\left[\m{h}^\top\left(\frac{1}{\sqrt{K}}\m{I}+\m{\Delta}\right)^\top\m{\Gamma}_n^\top\left(\m{\Sigma}_{\Omega}+\m{E}_{\Omega}\right)\m{\Gamma}_n\left(\frac{1}{\sqrt{K}}\m{I}+\m{\Delta}\right)\m{h}+\varepsilon\right]^{\frac{3}{2}}}.
\end{aligned}
\end{equation}
Next, observe that
\begin{equation}
\begin{aligned}
&\left|\sum_{n\in\chi_{\m{\Delta},\m{h}}}\frac{\m{h}^\top\left(\frac{1}{\sqrt{K}}\m{I}+\m{\Delta}\right)^\top\m{\Gamma}_n^\top\m{E}_{\Omega}\m{\Gamma}_n\left(\frac{1}{\sqrt{K}}\m{I}+\m{\Delta}\right)\m{h}}{\left[\m{h}^\top\left(\frac{1}{\sqrt{K}}\m{I}+\m{\Delta}\right)^\top\m{\Gamma}_n^\top\left(\m{\Sigma}_{\Omega}+\m{E}_{\Omega}\right)\m{\Gamma}_n\left(\frac{1}{\sqrt{K}}\m{I}+\m{\Delta}\right)\m{h}+\varepsilon\right]^{\frac{1}{2}}}\right|\\
\leq&\frac{\left\|\m{E}_\Omega\right\|_2}{\left(1-\left\|\m{E}_\Omega\right\|_2\right)^{\frac{1}{2}}}\sum_{n\in\chi_{\m{\Delta},\m{h}}}\left\|\mathcal{P}_\Omega\left(\m{\Gamma}_n\left(\frac{1}{\sqrt{K}}\m{I}+\m{\Delta}\right)\m{h}\right)\right\|_2,
\end{aligned}
\end{equation}
and
\begin{equation}
\begin{aligned}
&\sum_{n\in\chi_{\m{\Delta},\m{h}}}\frac{\varepsilon h_o^2\m{h}^\top\left(\frac{1}{\sqrt{K}}\m{I}+\m{\Delta}\right)^\top\m{\Gamma}_n^\top\left(\m{\Sigma}_{\Omega}+\m{E}_{\Omega}\right)\m{\Gamma}_n\left(\frac{1}{\sqrt{K}}\m{I}+\m{\Delta}\right)\m{h}}{\left[\m{h}^\top\left(\frac{1}{\sqrt{K}}\m{I}+\m{\Delta}\right)^\top\m{\Gamma}_n^\top\left(\m{\Sigma}_{\Omega}+\m{E}_{\Omega}\right)\m{\Gamma}_n\left(\frac{1}{\sqrt{K}}\m{I}+\m{\Delta}\right)\m{h}+\varepsilon\right]^{\frac{3}{2}}}\leq h_o^2\varepsilon^{\frac{1}{2}}N.
\end{aligned}
\end{equation}
Moreover, we obtain
\begin{equation}
\begin{aligned}
&\sum_{n=1}^N\frac{\left|\m{e}_o^\top\left(\frac{1}{\sqrt{K}}\m{I}+\m{\Delta}\right)^\top\m{\Gamma}_n^\top\m{E}_{\Omega}\m{\Gamma}_n\left(\frac{1}{\sqrt{K}}\m{I}+\m{\Delta}\right)\m{e}_o\right|}{\left[\m{h}^\top\left(\frac{1}{\sqrt{K}}\m{I}+\m{\Delta}\right)^\top\m{\Gamma}_n^\top\left(\m{\Sigma}_{\Omega}+\m{E}_{\Omega}\right)\m{\Gamma}_n\left(\frac{1}{\sqrt{K}}\m{I}+\m{\Delta}\right)\m{h}+\varepsilon\right]^{\frac{1}{2}}}\\
\leq&\sum_{n=1}^K\frac{\frac{1}{K}\left|\m{e}_o^\top\m{\Gamma}_n^\top\m{E}_{\Omega}\m{\Gamma}_n\m{e}_o\right|}{\left[\m{h}^\top\left(\frac{1}{\sqrt{K}}\m{I}+\m{\Delta}\right)^\top\m{\Gamma}_n^\top\left(\m{\Sigma}_{\Omega}+\m{E}_{\Omega}\right)\m{\Gamma}_n\left(\frac{1}{\sqrt{K}}\m{I}+\m{\Delta}\right)\m{h}+\varepsilon\right]^{\frac{1}{2}}}\\
&+\sum_{n=1}^N\frac{\frac{1}{\sqrt{K}}\left|\m{e}_o^\top\m{\Delta}^\top\m{\Gamma}_n^\top\m{E}_{\Omega}\m{\Gamma}_n\m{e}_o\right|+\frac{1}{\sqrt{K}}\left|\m{e}_o^\top\m{\Gamma}_n^\top\m{E}_{\Omega}\m{\Gamma}_n\m{\Delta}\m{e}_o\right|+\left|\m{e}_o^\top\m{\Delta}^\top\m{\Gamma}_n^\top\m{E}_{\Omega}\m{\Gamma}_n\m{\Delta}\m{e}_o\right|}{\left[\m{h}^\top\left(\frac{1}{\sqrt{K}}\m{I}+\m{\Delta}\right)^\top\m{\Gamma}_n^\top\left(\m{\Sigma}_{\Omega}+\m{E}_{\Omega}\right)\m{\Gamma}_n\left(\frac{1}{\sqrt{K}}\m{I}+\m{\Delta}\right)\m{h}+\varepsilon\right]^{\frac{1}{2}}}\\
\leq&\frac{\left\|\m{E}_\Omega\right\|_2}{\left(\frac{h_o^2}{K}-\frac{2\left\|\m{\Delta}\right\|_2}{\sqrt{K}}-\left\|\m{\Delta}\right\|_2^2-\left(\frac{1}{\sqrt{K}}+\left\|\m{\Delta}\right\|_2\right)^2\left\|\m{E}_\Omega\right\|_2\right)^{\frac{1}{2}}}+\varepsilon^{-\frac{1}{2}}N\left[\frac{2\left\|\m{E}_\Omega\right\|_2\left\|\m{\Delta}\right\|_2}{\sqrt{K}}+\left\|\m{E}_\Omega\right\|_2\left\|\m{\Delta}\right\|_2^2\right].
\end{aligned}
\end{equation}
Similarly, we have
\begin{equation}
\begin{aligned}
&\sum_{n\in\chi_{\m{\Delta},\m{h}}}\frac{\left|\m{e}_o^\top\left(\frac{1}{\sqrt{K}}\m{I}+\m{\Delta}\right)^\top\m{\Gamma}_n^\top\m{E}_{\Omega}\m{\Gamma}_n\left(\frac{1}{\sqrt{K}}\m{I}+\m{\Delta}\right)\m{h}\right|}{\left[\m{h}^\top\left(\frac{1}{\sqrt{K}}\m{I}+\m{\Delta}\right)^\top\m{\Gamma}_n^\top\left(\m{\Sigma}_{\Omega}+\m{E}_{\Omega}\right)\m{\Gamma}_n\left(\frac{1}{\sqrt{K}}\m{I}+\m{\Delta}\right)\m{h}+\varepsilon\right]^{\frac{3}{2}}}\\
\leq&\frac{\left\|\m{E}_\Omega\right\|_2}{\left(\frac{h_o^2}{K}-\frac{2\left\|\m{\Delta}\right\|_2}{\sqrt{K}}-\left\|\m{\Delta}\right\|_2^2-\left(\frac{1}{\sqrt{K}}+\left\|\m{\Delta}\right\|_2\right)^2\left\|\m{E}_\Omega\right\|_2\right)^{\frac{3}{2}}}+\varepsilon^{-\frac{3}{2}}N\left[\frac{2\left\|\m{E}_\Omega\right\|_2\left\|\m{\Delta}\right\|_2}{\sqrt{K}}+\left\|\m{E}_\Omega\right\|_2\left\|\m{\Delta}\right\|_2^2\right].
\end{aligned}
\end{equation}
Furthermore, we obtain
\begin{equation}
\begin{aligned}
&\sum_{n\in\chi_{\m{\Delta},\m{h}}}\frac{\left|\m{e}_o^\top\left(\frac{1}{\sqrt{K}}\m{I}+\m{\Delta}\right)^\top\m{\Gamma}_n^\top\m{\Sigma}_{\Omega}\m{\Gamma}_n\left(\frac{1}{\sqrt{K}}\m{I}+\m{\Delta}\right)\m{h}\right|}{\left[\m{h}^\top\left(\frac{1}{\sqrt{K}}\m{I}+\m{\Delta}\right)^\top\m{\Gamma}_n^\top\left(\m{\Sigma}_{\Omega}+\m{E}_{\Omega}\right)\m{\Gamma}_n\left(\frac{1}{\sqrt{K}}\m{I}+\m{\Delta}\right)\m{h}+\varepsilon\right]^{\frac{3}{2}}}\\
&\cdot\left|\m{e}_o^\top\left(\frac{1}{\sqrt{K}}\m{I}+\m{\Delta}\right)^\top\m{\Gamma}_n^\top\m{E}_{\Omega}\m{\Gamma}_n\left(\frac{1}{\sqrt{K}}\m{I}+\m{\Delta}\right)\m{h}\right|\\
\leq&\sum_{n\in\chi_{\m{\Delta},\m{h}}\cap\chi_o}\frac{\frac{h_o}{K}\left|\m{e}_o^\top\left(\frac{1}{\sqrt{K}}\m{I}+\m{\Delta}\right)^\top\m{\Gamma}_n^\top\m{E}_{\Omega}\m{\Gamma}_n\left(\frac{1}{\sqrt{K}}\m{I}+\m{\Delta}\right)\m{h}\right|}{\left[\m{h}^\top\left(\frac{1}{\sqrt{K}}\m{I}+\m{\Delta}\right)^\top\m{\Gamma}_n^\top\left(\m{\Sigma}_{\Omega}+\m{E}_{\Omega}\right)\m{\Gamma}_n\left(\frac{1}{\sqrt{K}}\m{I}+\m{\Delta}\right)\m{h}+\varepsilon\right]^{\frac{3}{2}}}\\
&+\sum_{n\in\chi_{\m{\Delta},\m{h}}\cap\chi_o}\frac{\frac{\left\|\m{E}_\Omega\right\|_2}{K}\left(\frac{2\left\|\m{\Delta}\right\|_2}{\sqrt{K}}+\left\|\m{\Delta}\right\|_2^2\right)}{\left[\m{h}^\top\left(\frac{1}{\sqrt{K}}\m{I}+\m{\Delta}\right)^\top\m{\Gamma}_n^\top\left(\m{\Sigma}_{\Omega}+\m{E}_{\Omega}\right)\m{\Gamma}_n\left(\frac{1}{\sqrt{K}}\m{I}+\m{\Delta}\right)\m{h}+\varepsilon\right]^{\frac{3}{2}}}\\
&+\sum_{n\in\chi_{\m{\Delta},\m{h}}\cap\chi_o}\frac{\frac{1}{\sqrt{K}}\left|\m{e}_o^\top\m{\Gamma}_n^\top\m{\Sigma}_\Omega\m{\Gamma}_n\m{\Delta}\m{h}\right|\left(\frac{2\left\|\m{\Delta}\right\|_2}{\sqrt{K}}+\left\|\m{\Delta}\right\|_2^2\right)\left\|\m{E}_\Omega\right\|_2}{\left[\m{h}^\top\left(\frac{1}{\sqrt{K}}\m{I}+\m{\Delta}\right)^\top\m{\Gamma}_n^\top\left(\m{\Sigma}_{\Omega}+\m{E}_{\Omega}\right)\m{\Gamma}_n\left(\frac{1}{\sqrt{K}}\m{I}+\m{\Delta}\right)\m{h}+\varepsilon\right]^{\frac{3}{2}}}\\
&+\sum_{n\in\chi_{\m{\Delta},\m{h}}\cap\chi_o}\frac{\frac{1}{\sqrt{K}}\left|\m{e}_o^\top\m{\Gamma}_n^\top\m{E}_\Omega\m{\Gamma}_n\m{\Delta}\m{h}\right|\left(\frac{\left\|\m{\Delta}\right\|_2}{\sqrt{K}}+\left\|\m{\Delta}\right\|_2^2\right)}{\left[\m{h}^\top\left(\frac{1}{\sqrt{K}}\m{I}+\m{\Delta}\right)^\top\m{\Gamma}_n^\top\left(\m{\Sigma}_{\Omega}+\m{E}_{\Omega}\right)\m{\Gamma}_n\left(\frac{1}{\sqrt{K}}\m{I}+\m{\Delta}\right)\m{h}+\varepsilon\right]^{\frac{3}{2}}}\\
&+\sum_{n\in\chi_{\m{\Delta},\m{h}}}\frac{\left|\m{e}_o^\top\m{\Delta}^\top\m{\Gamma}_n^\top\m{\Sigma}_\Omega\m{\Gamma}_n\left(\frac{1}{\sqrt{K}}\m{I}+\m{\Delta}\right)\m{h}\right|\left|\m{e}_o^\top\m{\Delta}^\top\m{\Gamma}_n^\top\m{E}_\Omega\m{\Gamma}_n\left(\frac{1}{\sqrt{K}}\m{I}+\m{\Delta}\right)\m{h}\right|}{\left[\m{h}^\top\left(\frac{1}{\sqrt{K}}\m{I}+\m{\Delta}\right)^\top\m{\Gamma}_n^\top\left(\m{\Sigma}_{\Omega}+\m{E}_{\Omega}\right)\m{\Gamma}_n\left(\frac{1}{\sqrt{K}}\m{I}+\m{\Delta}\right)\m{h}+\varepsilon\right]^{\frac{3}{2}}}\\
\leq&\frac{h_o}{K}\frac{K\left(\frac{1}{\sqrt{K}}+\left\|\m{\Delta}\right\|_2\right)^2\left\|\m{E}_\Omega\right\|_2}{\left(\frac{h_o^2}{K}-\frac{2\left\|\m{\Delta}\right\|_2}{\sqrt{K}}-\left\|\m{\Delta}\right\|_2^2-\left(\frac{1}{\sqrt{K}}+\left\|\m{\Delta}\right\|_2\right)^2\left\|\m{E}_\Omega\right\|_2\right)^{\frac{3}{2}}}\\
&+\left(\frac{2\left\|\m{\Delta}\right\|_2}{\sqrt{K}}+\left\|\m{\Delta}\right\|_2^2\right)\frac{\left\|\m{E}_\Omega\right\|_2}{\left(\frac{h_o^2}{K}-\frac{2\left\|\m{\Delta}\right\|_2}{\sqrt{K}}-\left\|\m{\Delta}\right\|_2^2-\left(\frac{1}{\sqrt{K}}+\left\|\m{\Delta}\right\|_2\right)^2\left\|\m{E}_\Omega\right\|_2\right)^{\frac{3}{2}}}\\
&+\left(\frac{2\left\|\m{\Delta}\right\|_2}{\sqrt{K}}+\left\|\m{\Delta}\right\|_2^2\right)\frac{\sqrt{K}\left\|\m{\Delta}\right\|_2\left\|\m{E}_\Omega\right\|_2}{\left(\frac{h_o^2}{K}-\frac{2\left\|\m{\Delta}\right\|_2}{\sqrt{K}}-\left\|\m{\Delta}\right\|_2^2-\left(\frac{1}{\sqrt{K}}+\left\|\m{\Delta}\right\|_2\right)^2\left\|\m{E}_\Omega\right\|_2\right)^{\frac{3}{2}}}\\
&+\left(\frac{\left\|\m{\Delta}\right\|_2}{\sqrt{K}}+\left\|\m{\Delta}\right\|_2^2\right)\frac{\sqrt{K}\left\|\m{\Delta}\right\|_2\left\|\m{E}_\Omega\right\|_2}{\left(\frac{h_o^2}{K}-\frac{2\left\|\m{\Delta}\right\|_2}{\sqrt{K}}-\left\|\m{\Delta}\right\|_2^2-\left(\frac{1}{\sqrt{K}}+\left\|\m{\Delta}\right\|_2\right)^2\left\|\m{E}_\Omega\right\|_2\right)^{\frac{3}{2}}}\\
&+\varepsilon^{-\frac{1}{2}}N\frac{\left\|\m{\Delta}\right\|_2^2\left\|\m{E}_\Omega\right\|_2}{1-\left\|\m{E}_\Omega\right\|_2},
\end{aligned}
\end{equation}
and
\begin{equation}
\begin{aligned}
&\varepsilon\sum_{n\in\chi_{\m{\Delta},\m{h}}}\frac{\frac{1}{\sqrt{K}}\m{h}^\top\m{\Gamma}_n^\top\m{\Sigma}_{\Omega}\m{\Gamma}_n\m{\Delta}\m{e}_o+\frac{1}{\sqrt{K}}\m{h}^\top\m{\Delta}^\top\m{\Gamma}_n^\top\m{\Sigma}_{\Omega}\m{\Gamma}_n\m{e}_o+\m{h}^\top\m{\Delta}^\top\m{\Gamma}_n^\top\m{\Sigma}_{\Omega}\m{\Gamma}_n\m{\Delta}\m{e}_o}{\left[\m{h}^\top\left(\frac{1}{\sqrt{K}}\m{I}+\m{\Delta}\right)^\top\m{\Gamma}_n^\top\left(\m{\Sigma}_{\Omega}+\m{E}_{\Omega}\right)\m{\Gamma}_n\left(\frac{1}{\sqrt{K}}\m{I}+\m{\Delta}\right)\m{h}+\varepsilon\right]^{\frac{3}{2}}}\\
\leq&\sum_{n\in\chi_{\m{\Delta},\m{h}}}\frac{\m{h}^\top\left(\frac{1}{\sqrt{K}}\m{I}+\m{\Delta}\right)^\top\m{\Gamma}_n^\top\m{\Sigma}_{\Omega}\m{\Gamma}_n\m{\Delta}\m{e}_o+\frac{1}{\sqrt{K}}\m{h}^\top\m{\Delta}^\top\m{\Gamma}_n^\top\m{\Sigma}_{\Omega}\m{\Gamma}_n\m{e}_o}{\left[\m{h}^\top\left(\frac{1}{\sqrt{K}}\m{I}+\m{\Delta}\right)^\top\m{\Gamma}_n^\top\left(\m{\Sigma}_{\Omega}+\m{E}_{\Omega}\right)\m{\Gamma}_n\left(\frac{1}{\sqrt{K}}\m{I}+\m{\Delta}\right)\m{h}+\varepsilon\right]^{\frac{1}{2}}}\\
\leq&\frac{1}{\left(1-\left\|\m{E}_\Omega\right\|_2\right)^{\frac{1}{2}}}N\left\|\m{\Delta}\right\|_2+\frac{\sqrt{K}\left\|\m{\Delta}\right\|_2}{\left(\frac{h_o^2}{K}-\frac{2\left\|\m{\Delta}\right\|_2}{\sqrt{K}}-\left\|\m{\Delta}\right\|_2^2-\left(\frac{1}{\sqrt{K}}+\left\|\m{\Delta}\right\|_2\right)^2\left\|\m{E}_\Omega\right\|_2\right)^{\frac{1}{2}}}.
\end{aligned}
\end{equation}
Therefore, for $\left\|\m{E}_\Omega\right\|_2\leq\frac{1}{32}N^{-1}$ and $\left\|\m{\Delta}\right\|_2\leq\frac{1}{32}K^{-\frac{1}{2}}N^{-1}$, we obtain
\begin{equation}
\begin{aligned}
\widetilde{\text{res}}\leq&2\left\|\m{E}_\Omega\right\|_2\sum_{n\in\chi_{\m{\Delta},\m{h}}}\left\|\mathcal{P}_\Omega\left(\m{\Gamma}_n\left(\frac{1}{\sqrt{K}}\m{I}+\m{\Delta}\right)\m{h}\right)\right\|_2+\frac{1}{\left\|\m{d}\right\|_2^2}h_o^2\varepsilon^{\frac{1}{2}}N\\
&+\frac{1}{\left\|\m{d}\right\|_2^2}2h_o^{-1}K^{\frac{1}{2}}\left\|\m{E}_\Omega\right\|_2+\frac{1}{\left\|\m{d}\right\|_2^2}4\varepsilon^{-\frac{1}{2}}K^{-\frac{1}{2}}N\left\|\m{E}_\Omega\right\|_2\left\|\m{\Delta}\right\|_2\\
&+\frac{1}{\left\|\m{d}\right\|_2^2}4h_o^{-2}\varepsilon K^{\frac{3}{2}}\left\|\m{E}_\Omega\right\|_2+\frac{1}{\left\|\m{d}\right\|_2^2}8h_o\varepsilon^{-\frac{1}{2}}K^{-\frac{1}{2}}N\left\|\m{E}_\Omega\right\|_2\left\|\m{\Delta}\right\|_2\\
&+\frac{1}{\left\|\m{d}\right\|_2^2}4h_o^{-2}K^{\frac{1}{2}}\left\|\m{E}_\Omega\right\|_2+\frac{1}{\left\|\m{d}\right\|_2^2}18h_o^{-3}K\left\|\m{E}_\Omega\right\|_2\left\|\m{\Delta}\right\|_2+\frac{1}{\left\|\m{d}\right\|_2^2}2\varepsilon^{-\frac{1}{2}}N\left\|\m{\Delta}\right\|_2^2\left\|\m{E}_\Omega\right\|_2\\
&+\frac{1}{\left\|\m{d}\right\|_2^2}2h_o\left(2N+2h_o^{-1}K\right)\left\|\m{\Delta}\right\|_2.
\end{aligned}
\end{equation}

We next derive an upper bound for the term $\widetilde{\text{main}}$. To this end, we first establish the following result.
\begin{lem}
\label{lem-1st-general}
Let $\m{h}$ be an approximate first-order stationary point satisfying \eqref{prop-max} and \eqref{prop-first-order stationary-general}. Then, for $\left\|\m{E}_\Omega\right\|_2\leq\frac{1}{2}$, the following inequality holds:
\begin{equation}
\begin{aligned}
&\left|\sum_{n\in\chi_{\m{\Delta},\m{h}}}\frac{\left(\m{e}_o-h_o\m{h}\right)^\top\left(\frac{1}{\sqrt{K}}\m{I}+\m{\Delta}\right)^\top\m{\Gamma}_n^\top\m{\Sigma}_{\Omega}\m{\Gamma}_n\left(\frac{1}{\sqrt{K}}\m{I}+\m{\Delta}\right)\m{h}}{\left[\m{h}^\top\left(\frac{1}{\sqrt{K}}\m{I}+\m{\Delta}\right)^\top\m{\Gamma}_n^\top\left(\m{\Sigma}_{\Omega}+\m{E}_{\Omega}\right)\m{\Gamma}_n\left(\frac{1}{\sqrt{K}}\m{I}+\m{\Delta}\right)\m{h}+\varepsilon\right]^{\frac{1}{2}}}\right|\\
\leq&\xi+2K^{\frac{1}{2}}\left\|\m{E}_\Omega\right\|_2+2N\left\|\m{\Delta}\right\|_2\left\|\m{E}_\Omega\right\|_2+2h_o\left\|\m{E}_\Omega\right\|_2\sum_{n\in\chi_{\m{\Delta},\m{h}}}\left\|\mathcal{P}_{\Omega}\left(\m{\Gamma}_n\left(\frac{1}{\sqrt{K}}\m{I}+\m{\Delta}\right)\m{h}\right)\right\|_2.
\end{aligned}
\end{equation}
\end{lem}
\begin{proof}
From \eqref{grad-general}, the $o$-th entry of $\text{grad}\;\psi_\varepsilon\left(\m{h}\right)$ can be expressed as
\begin{equation}
\begin{aligned}
\mathcal{P}_{o}\left(\text{grad}\;\psi_\varepsilon\left(\m{h}\right)\right)=&\sum_{n\in\chi_{\m{\Delta},\m{h}}}\frac{\m{d}^\top\left(\frac{1}{\sqrt{K}}\m{I}+\m{\Delta}\right)^\top\m{\Gamma}_n^\top\m{\Sigma}_{\Omega}\m{\Gamma}_n\left(\frac{1}{\sqrt{K}}\m{I}+\m{\Delta}\right)\m{h}}{\left[\m{h}^\top\left(\frac{1}{\sqrt{K}}\m{I}+\m{\Delta}\right)^\top\m{\Gamma}_n^\top\left(\m{\Sigma}_{\Omega}+\m{E}_{\Omega}\right)\m{\Gamma}_n\left(\frac{1}{\sqrt{K}}\m{I}+\m{\Delta}\right)\m{h}+\varepsilon\right]^{\frac{1}{2}}}\\
&+\sum_{n\in\chi_{\m{\Delta},\m{h}}}\frac{\m{e}_o^\top\left(\frac{1}{\sqrt{K}}\m{I}+\m{\Delta}\right)^\top\m{\Gamma}_n^\top\m{E}_{\Omega}\m{\Gamma}_n\left(\frac{1}{\sqrt{K}}\m{I}+\m{\Delta}\right)\m{h}}{\left[\m{h}^\top\left(\frac{1}{\sqrt{K}}\m{I}+\m{\Delta}\right)^\top\m{\Gamma}_n^\top\left(\m{\Sigma}_{\Omega}+\m{E}_{\Omega}\right)\m{\Gamma}_n\left(\frac{1}{\sqrt{K}}\m{I}+\m{\Delta}\right)\m{h}+\varepsilon\right]^{\frac{1}{2}}}\\
&-h_o\sum_{n\in\chi_{\m{\Delta},\m{h}}}\frac{\m{h}^\top\left(\frac{1}{\sqrt{K}}\m{I}+\m{\Delta}\right)^\top\m{\Gamma}_n^\top\m{E}_{\Omega}\m{\Gamma}_n\left(\frac{1}{\sqrt{K}}\m{I}+\m{\Delta}\right)\m{h}}{\left[\m{h}^\top\left(\frac{1}{\sqrt{K}}\m{I}+\m{\Delta}\right)^\top\m{\Gamma}_n^\top\left(\m{\Sigma}_{\Omega}+\m{E}_{\Omega}\right)\m{\Gamma}_n\left(\frac{1}{\sqrt{K}}\m{I}+\m{\Delta}\right)\m{h}+\varepsilon\right]^{\frac{1}{2}}}.
\end{aligned}
\end{equation}
By the first-order stationarity condition \eqref{prop-first-order stationary-general},
\begin{equation}
\begin{aligned}
\left|\mathcal{P}_{o}\left(\text{grad}\;\psi_\varepsilon\left(\m{h}\right)\right)\right|\leq\xi,
\end{aligned}
\end{equation}
which implies
\begin{equation}
\label{lem-1st-general-1}
\begin{aligned}
&\left|\sum_{n\in\chi_{\m{\Delta},\m{h}}}\frac{\m{d}^\top\left(\frac{1}{\sqrt{K}}\m{I}+\m{\Delta}\right)^\top\m{\Gamma}_n^\top\m{\Sigma}_{\Omega}\m{\Gamma}_n\left(\frac{1}{\sqrt{K}}\m{I}+\m{\Delta}\right)\m{h}}{\left[\m{h}^\top\left(\frac{1}{\sqrt{K}}\m{I}+\m{\Delta}\right)^\top\m{\Gamma}_n^\top\left(\m{\Sigma}_{\Omega}+\m{E}_{\Omega}\right)\m{\Gamma}_n\left(\frac{1}{\sqrt{K}}\m{I}+\m{\Delta}\right)\m{h}+\varepsilon\right]^{\frac{1}{2}}}\right|\\
\leq&\xi+\left|\sum_{n\in\chi_{\m{\Delta},\m{h}}}\frac{\m{e}_o^\top\left(\frac{1}{\sqrt{K}}\m{I}+\m{\Delta}\right)^\top\m{\Gamma}_n^\top\m{E}_{\Omega}\m{\Gamma}_n\left(\frac{1}{\sqrt{K}}\m{I}+\m{\Delta}\right)\m{h}}{\left[\m{h}^\top\left(\frac{1}{\sqrt{K}}\m{I}+\m{\Delta}\right)^\top\m{\Gamma}_n^\top\left(\m{\Sigma}_{\Omega}+\m{E}_{\Omega}\right)\m{\Gamma}_n\left(\frac{1}{\sqrt{K}}\m{I}+\m{\Delta}\right)\m{h}+\varepsilon\right]^{\frac{1}{2}}}\right|\\
&+\left|h_o\sum_{n\in\chi_{\m{\Delta},\m{h}}}\frac{\m{h}^\top\left(\frac{1}{\sqrt{K}}\m{I}+\m{\Delta}\right)^\top\m{\Gamma}_n^\top\m{E}_{\Omega}\m{\Gamma}_n\left(\frac{1}{\sqrt{K}}\m{I}+\m{\Delta}\right)\m{h}}{\left[\m{h}^\top\left(\frac{1}{\sqrt{K}}\m{I}+\m{\Delta}\right)^\top\m{\Gamma}_n^\top\left(\m{\Sigma}_{\Omega}+\m{E}_{\Omega}\right)\m{\Gamma}_n\left(\frac{1}{\sqrt{K}}\m{I}+\m{\Delta}\right)\m{h}+\varepsilon\right]^{\frac{1}{2}}}\right|.
\end{aligned}
\end{equation}
Moreover, we have
\begin{equation}
\label{}
\begin{aligned}
&\left|\sum_{n\in\chi_{\m{\Delta},\m{h}}}\frac{\m{e}_o^\top\left(\frac{1}{\sqrt{K}}\m{I}+\m{\Delta}\right)^\top\m{\Gamma}_n^\top\m{E}_{\Omega}\m{\Gamma}_n\left(\frac{1}{\sqrt{K}}\m{I}+\m{\Delta}\right)\m{h}}{\left[\m{h}^\top\left(\frac{1}{\sqrt{K}}\m{I}+\m{\Delta}\right)^\top\m{\Gamma}_n^\top\left(\m{\Sigma}_{\Omega}+\m{E}_{\Omega}\right)\m{\Gamma}_n\left(\frac{1}{\sqrt{K}}\m{I}+\m{\Delta}\right)\m{h}+\varepsilon\right]^{\frac{1}{2}}}\right|\\
\leq&\frac{1}{\sqrt{K}}\left|\sum_{n\in\chi_{\m{\Delta},\m{h}}}\frac{\m{e}_o^\top\m{\Gamma}_n^\top\m{E}_{\Omega}\m{\Gamma}_n\left(\frac{1}{\sqrt{K}}\m{I}+\m{\Delta}\right)\m{h}}{\left[\m{h}^\top\left(\frac{1}{\sqrt{K}}\m{I}+\m{\Delta}\right)^\top\m{\Gamma}_n^\top\left(\m{\Sigma}_{\Omega}+\m{E}_{\Omega}\right)\m{\Gamma}_n\left(\frac{1}{\sqrt{K}}\m{I}+\m{\Delta}\right)\m{h}+\varepsilon\right]^{\frac{1}{2}}}\right|\\
&+\left|\sum_{n\in\chi_{\m{\Delta},\m{h}}}\frac{\m{e}_o^\top\m{\Delta}^\top\m{\Gamma}_n^\top\m{E}_{\Omega}\m{\Gamma}_n\left(\frac{1}{\sqrt{K}}\m{I}+\m{\Delta}\right)\m{h}}{\left[\m{h}^\top\left(\frac{1}{\sqrt{K}}\m{I}+\m{\Delta}\right)^\top\m{\Gamma}_n^\top\left(\m{\Sigma}_{\Omega}+\m{E}_{\Omega}\right)\m{\Gamma}_n\left(\frac{1}{\sqrt{K}}\m{I}+\m{\Delta}\right)\m{h}+\varepsilon\right]^{\frac{1}{2}}}\right|.
\end{aligned}
\end{equation}
The two terms on the right-hand side can be bounded as
\begin{equation}
\begin{aligned}
&\left|\sum_{n\in\chi_{\m{\Delta},\m{h}}}\frac{\m{e}_o^\top\m{\Gamma}_n^\top\m{E}_{\Omega}\m{\Gamma}_n\left(\frac{1}{\sqrt{K}}\m{I}+\m{\Delta}\right)\m{h}}{\left[\m{h}^\top\left(\frac{1}{\sqrt{K}}\m{I}+\m{\Delta}\right)^\top\m{\Gamma}_n^\top\left(\m{\Sigma}_{\Omega}+\m{E}_{\Omega}\right)\m{\Gamma}_n\left(\frac{1}{\sqrt{K}}\m{I}+\m{\Delta}\right)\m{h}+\varepsilon\right]^{\frac{1}{2}}}\right|\\
\leq&\sum_{n\in\chi_{o}\cap\chi_{\m{\Delta},\m{h}}}\left[\frac{\left(\m{h}^\top\left(\frac{1}{\sqrt{K}}\m{I}+\m{\Delta}\right)^\top\m{\Gamma}_n^\top\m{E}_{\Omega}^\top\m{E}_{\Omega}\m{\Gamma}_n\left(\frac{1}{\sqrt{K}}\m{I}+\m{\Delta}\right)\m{h}\right)/\left\|\mathcal{P}_{\Omega}\left(\m{\Gamma}_n\left(\frac{1}{\sqrt{K}}\m{I}+\m{\Delta}\right)\m{h}\right)\right\|_2^2}{\m{h}^\top\left(\frac{1}{\sqrt{K}}\m{I}+\m{\Delta}\right)^\top\m{\Gamma}_n^\top\left(\m{\Sigma}_{\Omega}+\m{E}_{\Omega}\right)\m{\Gamma}_n\left(\frac{1}{\sqrt{K}}\m{I}+\m{\Delta}\right)\m{h}/\left\|\mathcal{P}_{\Omega}\left(\m{\Gamma}_n\left(\frac{1}{\sqrt{K}}\m{I}+\m{\Delta}\right)\m{h}\right)\right\|_2^2}\right]^{\frac{1}{2}}\\
\leq&K\frac{\left\|\m{E}_\Omega\right\|_2}{\left(1-\left\|\m{E}_\Omega\right\|_2\right)^{\frac{1}{2}}},
\end{aligned}
\end{equation}
and
\begin{equation}
\begin{aligned}
&\left|\sum_{n\in\chi_{\m{\Delta},\m{h}}}\frac{\m{e}_o^\top\m{\Delta}^\top\m{\Gamma}_n^\top\m{E}_{\Omega}\m{\Gamma}_n\left(\frac{1}{\sqrt{K}}\m{I}+\m{\Delta}\right)\m{h}}{\left[\m{h}^\top\left(\frac{1}{\sqrt{K}}\m{I}+\m{\Delta}\right)^\top\m{\Gamma}_n^\top\left(\m{\Sigma}_{\Omega}+\m{E}_{\Omega}\right)\m{\Gamma}_n\left(\frac{1}{\sqrt{K}}\m{I}+\m{\Delta}\right)\m{h}+\varepsilon\right]^{\frac{1}{2}}}\right|\\
\leq&\left\|\m{\Delta}\right\|_2\sum_{n\in\chi_{\m{\Delta},\m{h}}}\left[\frac{\left(\m{h}^\top\left(\frac{1}{\sqrt{K}}\m{I}+\m{\Delta}\right)^\top\m{\Gamma}_n^\top\m{E}_{\Omega}^\top\m{E}_{\Omega}\m{\Gamma}_n\left(\frac{1}{\sqrt{K}}\m{I}+\m{\Delta}\right)\m{h}\right)/\left\|\mathcal{P}_{\Omega}\left(\m{\Gamma}_n\left(\frac{1}{\sqrt{K}}\m{I}+\m{\Delta}\right)\m{h}\right)\right\|_2^2}{\m{h}^\top\left(\frac{1}{\sqrt{K}}\m{I}+\m{\Delta}\right)^\top\m{\Gamma}_n^\top\left(\m{\Sigma}_{\Omega}+\m{E}_{\Omega}\right)\m{\Gamma}_n\left(\frac{1}{\sqrt{K}}\m{I}+\m{\Delta}\right)\m{h}/\left\|\mathcal{P}_{\Omega}\left(\m{\Gamma}_n\left(\frac{1}{\sqrt{K}}\m{I}+\m{\Delta}\right)\m{h}\right)\right\|_2^2}\right]^{\frac{1}{2}}\\
\leq&N\left\|\m{\Delta}\right\|_2\frac{\left\|\m{E}_\Omega\right\|_2}{\left(1-\left\|\m{E}_\Omega\right\|_2\right)^{\frac{1}{2}}}.
\end{aligned}
\end{equation}
Consequently,
\begin{equation}
\label{lem-1st-general-2}
\begin{aligned}
&\left|\sum_{n\in\chi_{\m{\Delta},\m{h}}}\frac{\m{e}_o^\top\left(\frac{1}{\sqrt{K}}\m{I}+\m{\Delta}\right)^\top\m{\Gamma}_n^\top\m{E}_{\Omega}\m{\Gamma}_n\left(\frac{1}{\sqrt{K}}\m{I}+\m{\Delta}\right)\m{h}}{\left[\m{h}^\top\left(\frac{1}{\sqrt{K}}\m{I}+\m{\Delta}\right)^\top\m{\Gamma}_n^\top\left(\m{\Sigma}_{\Omega}+\m{E}_{\Omega}\right)\m{\Gamma}_n\left(\frac{1}{\sqrt{K}}\m{I}+\m{\Delta}\right)\m{h}+\varepsilon\right]^{\frac{1}{2}}}\right|\\
\leq&K^{\frac{1}{2}}\frac{\left\|\m{E}_\Omega\right\|_2}{\left(1-\left\|\m{E}_\Omega\right\|_2\right)^{\frac{1}{2}}}+N\left\|\m{\Delta}\right\|_2\frac{\left\|\m{E}_\Omega\right\|_2}{\left(1-\left\|\m{E}_\Omega\right\|_2\right)^{\frac{1}{2}}}
\end{aligned}
\end{equation}
Following an argument similar to that used in deriving \eqref{temp-term-2}, we obtain
\begin{equation}
\label{lem-1st-general-3}
\begin{aligned}
&\left|\sum_{n\in\chi_{\m{\Delta},\m{h}}}\frac{\m{h}^\top\left(\frac{1}{\sqrt{K}}\m{I}+\m{\Delta}\right)^\top\m{\Gamma}_n^\top\m{E}_{\Omega}\m{\Gamma}_n\left(\frac{1}{\sqrt{K}}\m{I}+\m{\Delta}\right)\m{h}}{\left[\m{h}^\top\left(\frac{1}{\sqrt{K}}\m{I}+\m{\Delta}\right)^\top\m{\Gamma}_n^\top\left(\m{\Sigma}_{\Omega}+\m{E}_{\Omega}\right)\m{\Gamma}_n\left(\frac{1}{\sqrt{K}}\m{I}+\m{\Delta}\right)\m{h}+\varepsilon\right]^{\frac{1}{2}}}\right|\\
\leq&\frac{\left\|\m{E}_\Omega\right\|_2}{\left(1-\left\|\m{E}_\Omega\right\|_2\right)^{\frac{1}{2}}}\sum_{n\in\chi_{\m{\Delta},\m{h}}}\left\|\mathcal{P}_{\Omega}\left(\m{\Gamma}_n\left(\frac{1}{\sqrt{K}}\m{I}+\m{\Delta}\right)\m{h}\right)\right\|_2.
\end{aligned}
\end{equation}
Substituting \eqref{lem-1st-general-2} and \eqref{lem-1st-general-3} into \eqref{lem-1st-general-1} yields
\begin{equation}
\begin{aligned}
&\left|\sum_{n\in\chi_{\m{\Delta},\m{h}}}\frac{\left(\m{e}_o-h_o\m{h}\right)^\top\left(\frac{1}{\sqrt{K}}\m{I}+\m{\Delta}\right)^\top\m{\Gamma}_n^\top\m{\Sigma}_{\Omega}\m{\Gamma}_n\left(\frac{1}{\sqrt{K}}\m{I}+\m{\Delta}\right)\m{h}}{\left[\m{h}^\top\left(\frac{1}{\sqrt{K}}\m{I}+\m{\Delta}\right)^\top\m{\Gamma}_n^\top\left(\m{\Sigma}_{\Omega}+\m{E}_{\Omega}\right)\m{\Gamma}_n\left(\frac{1}{\sqrt{K}}\m{I}+\m{\Delta}\right)\m{h}+\varepsilon\right]^{\frac{1}{2}}}\right|\\
\leq&\xi+\frac{K^{\frac{1}{2}}\left\|\m{E}_\Omega\right\|_2}{\left(1-\left\|\m{E}_\Omega\right\|_2\right)^{\frac{1}{2}}}+\frac{N\left\|\m{\Delta}\right\|_2\left\|\m{E}_\Omega\right\|_2}{\left(1-\left\|\m{E}_\Omega\right\|_2\right)^{\frac{1}{2}}}+\frac{h_o\left\|\m{E}_\Omega\right\|_2}{\left(1-\left\|\m{E}_\Omega\right\|_2\right)^{\frac{1}{2}}}\sum_{n\in\chi_{\m{\Delta},\m{h}}}\left\|\mathcal{P}_{\Omega}\left(\m{\Gamma}_n\left(\frac{1}{\sqrt{K}}\m{I}+\m{\Delta}\right)\m{h}\right)\right\|_2.
\end{aligned}
\end{equation}
Finally, recalling that $\left\|\m{E}_\Omega\right\|_2\leq\frac{1}{2}$ completes the proof.
\end{proof}

It follows from Lemma~\ref{lem-1st-general} that
\begin{equation}
\begin{aligned}
\widetilde{\text{main}}\leq&\frac{1}{\left\|\m{d}\right\|_2^2}\sum_{n=1}^N\frac{\m{e}_o^\top\left(\frac{1}{\sqrt{K}}\m{I}+\m{\Delta}\right)^\top\m{\Gamma}_n^\top\m{\Sigma}_{\Omega}\m{\Gamma}_n\left(\frac{1}{\sqrt{K}}\m{I}+\m{\Delta}\right)\m{e}_o}{\left[\m{h}^\top\left(\frac{1}{\sqrt{K}}\m{I}+\m{\Delta}\right)^\top\m{\Gamma}_n^\top\left(\m{\Sigma}_{\Omega}+\m{E}_{\Omega}\right)\m{\Gamma}_n\left(\frac{1}{\sqrt{K}}\m{I}+\m{\Delta}\right)\m{h}+\varepsilon\right]^{\frac{1}{2}}}\\
&-h_o^{-1}\sum_{n\in\chi_{\m{\Delta},\m{h}}}\frac{\m{e}_o^\top\left(\frac{1}{\sqrt{K}}\m{I}+\m{\Delta}\right)^\top\m{\Gamma}_n^\top\m{\Sigma}_{\Omega}\m{\Gamma}_n\left(\frac{1}{\sqrt{K}}\m{I}+\m{\Delta}\right)\m{h}}{\left[\m{h}^\top\left(\frac{1}{\sqrt{K}}\m{I}+\m{\Delta}\right)^\top\m{\Gamma}_n^\top\left(\m{\Sigma}_{\Omega}+\m{E}_{\Omega}\right)\m{\Gamma}_n\left(\frac{1}{\sqrt{K}}\m{I}+\m{\Delta}\right)\m{h}+\varepsilon\right]^{\frac{1}{2}}}\\
&-\frac{1}{\left\|\m{d}\right\|_2^2}\sum_{n\in\chi_{\m{\Delta},\m{h}}}\frac{\left|\m{e}_o^\top\left(\frac{1}{\sqrt{K}}\m{I}+\m{\Delta}\right)^\top\m{\Gamma}_n^\top\m{\Sigma}_{\Omega}\m{\Gamma}_n\left(\frac{1}{\sqrt{K}}\m{I}+\m{\Delta}\right)\m{h}\right|^2}{\left[\m{h}^\top\left(\frac{1}{\sqrt{K}}\m{I}+\m{\Delta}\right)^\top\m{\Gamma}_n^\top\left(\m{\Sigma}_{\Omega}+\m{E}_{\Omega}\right)\m{\Gamma}_n\left(\frac{1}{\sqrt{K}}\m{I}+\m{\Delta}\right)\m{h}+\varepsilon\right]^{\frac{3}{2}}}\\
&+h_o^{-1}\left(\xi+2K^{\frac{1}{2}}\left\|\m{E}_\Omega\right\|_2+2N\left\|\m{\Delta}\right\|_2\left\|\m{E}_\Omega\right\|_2\right)\\
&+2\left\|\m{E}_\Omega\right\|_2\sum_{n\in\chi_{\m{\Delta},\m{h}}}\left\|\mathcal{P}_{\Omega}\left(\m{\Gamma}_n\left(\frac{1}{\sqrt{K}}\m{I}+\m{\Delta}\right)\m{h}\right)\right\|_2.
\end{aligned}
\end{equation}
For $\left\|\m{E}_\Omega\right\|_2\leq\frac{1}{32}N^{-1}$ and $\left\|\m{\Delta}\right\|_2\leq\frac{1}{32}K^{-\frac{1}{2}}N^{-1}$, it holds that
\begin{equation}
\begin{aligned}
&\sum_{n=1}^N\frac{\m{e}_o^\top\left(\frac{1}{\sqrt{K}}\m{I}+\m{\Delta}\right)^\top\m{\Gamma}_n^\top\m{\Sigma}_{\Omega}\m{\Gamma}_n\left(\frac{1}{\sqrt{K}}\m{I}+\m{\Delta}\right)\m{e}_o}{\left[\m{h}^\top\left(\frac{1}{\sqrt{K}}\m{I}+\m{\Delta}\right)^\top\m{\Gamma}_n^\top\left(\m{\Sigma}_{\Omega}+\m{E}_{\Omega}\right)\m{\Gamma}_n\left(\frac{1}{\sqrt{K}}\m{I}+\m{\Delta}\right)\m{h}+\varepsilon\right]^{\frac{1}{2}}}\\
\leq&\frac{1}{K}\sum_{n\in\chi_o}\frac{1}{\left[\m{h}^\top\left(\frac{1}{\sqrt{K}}\m{I}+\m{\Delta}\right)^\top\m{\Gamma}_n^\top\left(\m{\Sigma}_{\Omega}+\m{E}_{\Omega}\right)\m{\Gamma}_n\left(\frac{1}{\sqrt{K}}\m{I}+\m{\Delta}\right)\m{h}+\varepsilon\right]^{\frac{1}{2}}}\\
&+\sum_{n=1}^N\frac{\left|\frac{2}{\sqrt{K}}\m{e}_o^\top\m{\Gamma}_n^\top\m{\Sigma}_{\Omega}\m{\Gamma}_n\m{\Delta}\m{e}_o\right|+\m{e}_o^\top\m{\Delta}^\top\m{\Gamma}_n^\top\m{\Sigma}_{\Omega}\m{\Gamma}_n\m{\Delta}\m{e}_o}{\left[\m{h}^\top\left(\frac{1}{\sqrt{K}}\m{I}+\m{\Delta}\right)^\top\m{\Gamma}_n^\top\left(\m{\Sigma}_{\Omega}+\m{E}_{\Omega}\right)\m{\Gamma}_n\left(\frac{1}{\sqrt{K}}\m{I}+\m{\Delta}\right)\m{h}+\varepsilon\right]^{\frac{1}{2}}}\\
\leq&\frac{1}{K}\sum_{n\in\chi_o}\frac{1}{\left[\m{h}^\top\left(\frac{1}{\sqrt{K}}\m{I}+\m{\Delta}\right)^\top\m{\Gamma}_n^\top\left(\m{\Sigma}_{\Omega}+\m{E}_{\Omega}\right)\m{\Gamma}_n\left(\frac{1}{\sqrt{K}}\m{I}+\m{\Delta}\right)\m{h}+\varepsilon\right]^{\frac{1}{2}}}\\
&+4h_o^{-1}K\left\|\m{\Delta}\right\|_2+\varepsilon^{-\frac{1}{2}}N\left\|\m{\Delta}\right\|_2^2.
\end{aligned}
\end{equation}
Similarly,
\begin{equation}
\begin{aligned}
&h_o^{-1}\sum_{n\in\chi_{\m{\Delta},\m{h}}}\frac{\m{e}_o^\top\left(\frac{1}{\sqrt{K}}\m{I}+\m{\Delta}\right)^\top\m{\Gamma}_n^\top\m{\Sigma}_{\Omega}\m{\Gamma}_n\left(\frac{1}{\sqrt{K}}\m{I}+\m{\Delta}\right)\m{h}}{\left[\m{h}^\top\left(\frac{1}{\sqrt{K}}\m{I}+\m{\Delta}\right)^\top\m{\Gamma}_n^\top\left(\m{\Sigma}_{\Omega}+\m{E}_{\Omega}\right)\m{\Gamma}_n\left(\frac{1}{\sqrt{K}}\m{I}+\m{\Delta}\right)\m{h}+\varepsilon\right]^{\frac{1}{2}}}\\
\geq&\frac{1}{K}\sum_{n\in\chi_o}\frac{1}{\left[\m{h}^\top\left(\frac{1}{\sqrt{K}}\m{I}+\m{\Delta}\right)^\top\m{\Gamma}_n^\top\left(\m{\Sigma}_{\Omega}+\m{E}_{\Omega}\right)\m{\Gamma}_n\left(\frac{1}{\sqrt{K}}\m{I}+\m{\Delta}\right)\m{h}+\varepsilon\right]^{\frac{1}{2}}}\\
&-h_o^{-1}\left(2N+2h_o^{-1}K\right)\left\|\m{\Delta}\right\|_2,
\end{aligned}
\end{equation}
and
\begin{equation}
\begin{aligned}
&\sum_{n\in\chi_{\m{\Delta},\m{h}}}\frac{\left|\m{e}_o^\top\left(\frac{1}{\sqrt{K}}\m{I}+\m{\Delta}\right)^\top\m{\Gamma}_n^\top\m{\Sigma}_{\Omega}\m{\Gamma}_n\left(\frac{1}{\sqrt{K}}\m{I}+\m{\Delta}\right)\m{h}\right|^2}{\left[\m{h}^\top\left(\frac{1}{\sqrt{K}}\m{I}+\m{\Delta}\right)^\top\m{\Gamma}_n^\top\left(\m{\Sigma}_{\Omega}+\m{E}_{\Omega}\right)\m{\Gamma}_n\left(\frac{1}{\sqrt{K}}\m{I}+\m{\Delta}\right)\m{h}+\varepsilon\right]^{\frac{3}{2}}}\\
\geq&\frac{1}{K}\sum_{n\in\chi_o}\frac{h_o^2K^{-1}}{\left[\m{h}^\top\left(\frac{1}{\sqrt{K}}\m{I}+\m{\Delta}\right)^\top\m{\Gamma}_n^\top\left(\m{\Sigma}_{\Omega}+\m{E}_{\Omega}\right)\m{\Gamma}_n\left(\frac{1}{\sqrt{K}}\m{I}+\m{\Delta}\right)\m{h}+\varepsilon\right]^{\frac{3}{2}}}\\
&-\sum_{n\in\chi_o}\frac{2h_oK^{-1}\left|\frac{1}{\sqrt{K}}\m{e}_o^\top\m{\Gamma}_n^\top\m{\Sigma}_{\Omega}\m{\Gamma}_n\m{\Delta}\m{h}\right|}{\left[\m{h}^\top\left(\frac{1}{\sqrt{K}}\m{I}+\m{\Delta}\right)^\top\m{\Gamma}_n^\top\left(\m{\Sigma}_{\Omega}+\m{E}_{\Omega}\right)\m{\Gamma}_n\left(\frac{1}{\sqrt{K}}\m{I}+\m{\Delta}\right)\m{h}+\varepsilon\right]^{\frac{3}{2}}}\\
&-\sum_{n\in\chi_o}\frac{2h_oK^{-1}\left|\m{e}_o^\top\m{\Delta}^\top\m{\Gamma}_n^\top\m{\Sigma}_{\Omega}\m{\Gamma}_n\left(\frac{1}{\sqrt{K}}\m{I}+\m{\Delta}\right)\m{h}\right|}{\left[\m{h}^\top\left(\frac{1}{\sqrt{K}}\m{I}+\m{\Delta}\right)^\top\m{\Gamma}_n^\top\left(\m{\Sigma}_{\Omega}+\m{E}_{\Omega}\right)\m{\Gamma}_n\left(\frac{1}{\sqrt{K}}\m{I}+\m{\Delta}\right)\m{h}+\varepsilon\right]^{\frac{3}{2}}}.
\end{aligned}
\end{equation}
Therefore,
\begin{equation}
\begin{aligned}
\widetilde{\text{main}}\leq&\frac{1}{K\left\|\m{d}\right\|_2^2}\sum_{n\in\chi_o}\frac{1}{\left[\m{h}^\top\left(\frac{1}{\sqrt{K}}\m{I}+\m{\Delta}\right)^\top\m{\Gamma}_n^\top\left(\m{\Sigma}_{\Omega}+\m{E}_{\Omega}\right)\m{\Gamma}_n\left(\frac{1}{\sqrt{K}}\m{I}+\m{\Delta}\right)\m{h}+\varepsilon\right]^{\frac{1}{2}}}\\
&-\frac{h_o^{2}}{K\left\|\m{d}\right\|_2^2}\sum_{n\in\chi_o}\frac{K^{-1}}{\left[\m{h}^\top\left(\frac{1}{\sqrt{K}}\m{I}+\m{\Delta}\right)^\top\m{\Gamma}_n^\top\left(\m{\Sigma}_{\Omega}+\m{E}_{\Omega}\right)\m{\Gamma}_n\left(\frac{1}{\sqrt{K}}\m{I}+\m{\Delta}\right)\m{h}+\varepsilon\right]^{\frac{3}{2}}}\\
&-\frac{1}{K}\sum_{n\in\chi_o}\frac{1}{\left[\m{h}^\top\left(\frac{1}{\sqrt{K}}\m{I}+\m{\Delta}\right)^\top\m{\Gamma}_n^\top\left(\m{\Sigma}_{\Omega}+\m{E}_{\Omega}\right)\m{\Gamma}_n\left(\frac{1}{\sqrt{K}}\m{I}+\m{\Delta}\right)\m{h}+\varepsilon\right]^{\frac{1}{2}}}\\
&+\frac{1}{\left\|\m{d}\right\|_2^2}\left(4h_o^{-1}K\left\|\m{\Delta}\right\|_2+\varepsilon^{-\frac{1}{2}}N\left\|\m{\Delta}\right\|_2^2+12h_o^{-2}K\left\|\m{\Delta}\right\|_2\right)+h_o^{-1}\left(2N+2h_o^{-1}K\right)\left\|\m{\Delta}\right\|_2\\
&+h_o^{-1}\left(\xi+2K^{\frac{1}{2}}\left\|\m{E}_\Omega\right\|_2+2N\left\|\m{\Delta}\right\|_2\left\|\m{E}_\Omega\right\|_2\right)\\
&+2\left\|\m{E}_\Omega\right\|_2\sum_{n\in\chi_{\m{\Delta},\m{h}}}\left\|\mathcal{P}_{\Omega}\left(\m{\Gamma}_n\left(\frac{1}{\sqrt{K}}\m{I}+\m{\Delta}\right)\m{h}\right)\right\|_2\\
\leq&\frac{h_o^{2}}{K\left\|\m{d}\right\|_2^2}\sum_{n\in\chi_o}\frac{1}{\left[\m{h}^\top\left(\frac{1}{\sqrt{K}}\m{I}+\m{\Delta}\right)^\top\m{\Gamma}_n^\top\left(\m{\Sigma}_{\Omega}+\m{E}_{\Omega}\right)\m{\Gamma}_n\left(\frac{1}{\sqrt{K}}\m{I}+\m{\Delta}\right)\m{h}+\varepsilon\right]^{\frac{1}{2}}}\\
&-\frac{h_o^{2}}{K\left\|\m{d}\right\|_2^2}\sum_{n\in\chi_o}\frac{K^{-1}}{\left[\m{h}^\top\left(\frac{1}{\sqrt{K}}\m{I}+\m{\Delta}\right)^\top\m{\Gamma}_n^\top\left(\m{\Sigma}_{\Omega}+\m{E}_{\Omega}\right)\m{\Gamma}_n\left(\frac{1}{\sqrt{K}}\m{I}+\m{\Delta}\right)\m{h}+\varepsilon\right]^{\frac{3}{2}}}\\
&+\frac{1}{\left\|\m{d}\right\|_2^2}\left(4h_o^{-1}K\left\|\m{\Delta}\right\|_2+\varepsilon^{-\frac{1}{2}}N\left\|\m{\Delta}\right\|_2^2+12h_o^{-2}K\left\|\m{\Delta}\right\|_2\right)+h_o^{-1}\left(2N+2h_o^{-1}K\right)\left\|\m{\Delta}\right\|_2\\
&+h_o^{-1}\left(\xi+2K^{\frac{1}{2}}\left\|\m{E}_\Omega\right\|_2+2N\left\|\m{\Delta}\right\|_2\left\|\m{E}_\Omega\right\|_2\right)\\
&+2\left\|\m{E}_\Omega\right\|_2\sum_{n\in\chi_{\m{\Delta},\m{h}}}\left\|\mathcal{P}_{\Omega}\left(\m{\Gamma}_n\left(\frac{1}{\sqrt{K}}\m{I}+\m{\Delta}\right)\m{h}\right)\right\|_2.
\end{aligned}
\end{equation}
It follows from Lemma~\ref{lem-diff} that
\begin{equation}
\begin{aligned}
&\sum_{n\in\chi_o}\frac{1}{\left[\m{h}^\top\left(\frac{1}{\sqrt{K}}\m{I}+\m{\Delta}\right)^\top\m{\Gamma}_n^\top\left(\m{\Sigma}_{\Omega}+\m{E}_{\Omega}\right)\m{\Gamma}_n\left(\frac{1}{\sqrt{K}}\m{I}+\m{\Delta}\right)\m{h}+\varepsilon\right]^{\frac{1}{2}}}\\
&-\sum_{n\in\chi_o}\frac{K^{-1}}{\left[\m{h}^\top\left(\frac{1}{\sqrt{K}}\m{I}+\m{\Delta}\right)^\top\m{\Gamma}_n^\top\left(\m{\Sigma}_{\Omega}+\m{E}_{\Omega}\right)\m{\Gamma}_n\left(\frac{1}{\sqrt{K}}\m{I}+\m{\Delta}\right)\m{h}+\varepsilon\right]^{\frac{3}{2}}}\\
\leq&\left[\frac{1}{K}\min\left\{Kh_o^2,1-\frac{1-h_o^2}{K}\right\}+\frac{2\left\|\m{\Delta}\right\|_2}{\sqrt{K}}+\left\|\m{\Delta}\right\|_2^2+\left(\frac{1}{\sqrt{K}}+\left\|\m{\Delta}\right\|_2\right)^2\left\|\m{E}_\Omega\right\|_2+\varepsilon\right]^{-\frac{1}{2}}\\
&-K^{-1}\left[\frac{1}{K}\min\left\{Kh_o^2,1-\frac{1-h_o^2}{K}\right\}+\frac{2\left\|\m{\Delta}\right\|_2}{\sqrt{K}}+\left\|\m{\Delta}\right\|_2^2+\left(\frac{1}{\sqrt{K}}+\left\|\m{\Delta}\right\|_2\right)^2\left\|\m{E}_\Omega\right\|_2+\varepsilon\right]^{-\frac{3}{2}}\\
&+\left(K-1\right)\left[\frac{1}{K}+\frac{2\left\|\m{\Delta}\right\|_2}{\sqrt{K}}+\left\|\m{\Delta}\right\|_2^2+\left(\frac{1}{\sqrt{K}}+\left\|\m{\Delta}\right\|_2\right)^2\left\|\m{E}_\Omega\right\|_2+\varepsilon\right]^{-\frac{1}{2}}\\
&-\left(K-1\right)K^{-1}\left[\frac{1}{K}+\frac{2\left\|\m{\Delta}\right\|_2}{\sqrt{K}}+\left\|\m{\Delta}\right\|_2^2+\left(\frac{1}{\sqrt{K}}+\left\|\m{\Delta}\right\|_2\right)^2\left\|\m{E}_\Omega\right\|_2+\varepsilon\right]^{-\frac{3}{2}}\\
\leq&\left(\frac{1}{K}\min\left\{Kh_o^2,1-\frac{1-h_o^2}{K}\right\}+\frac{3}{\sqrt{K}}\left\|\m{\Delta}\right\|_2+\frac{2}{K}\left\|\m{E}_\Omega\right\|_2+\varepsilon\right)^{-\frac{3}{2}}\\
&\cdot\left(\frac{1}{K}\min\left\{Kh_o^2,1-\frac{1-h_o^2}{K}\right\}+\frac{3}{\sqrt{K}}\left\|\m{\Delta}\right\|_2+\frac{2}{K}\left\|\m{E}_\Omega\right\|_2-\frac{1}{K}+\varepsilon\right)\\
&+K^{\frac{5}{2}}\left(\frac{3}{\sqrt{K}}\left\|\m{\Delta}\right\|_2+\frac{2}{K}\left\|\m{E}_\Omega\right\|_2+\varepsilon\right).
\end{aligned}
\end{equation}
Thus, 
\begin{equation}
\begin{aligned}
\widetilde{\text{main}}\leq&\frac{h_o^{2}}{K\left\|\m{d}\right\|_2^2}\left(\frac{1}{K}\min\left\{Kh_o^2,1-\frac{1-h_o^2}{K}\right\}+\frac{3}{\sqrt{K}}\left\|\m{\Delta}\right\|_2+\frac{2}{K}\left\|\m{E}_\Omega\right\|_2+\varepsilon\right)^{-\frac{3}{2}}\\
&\cdot\left(\frac{1}{K}\min\left\{Kh_o^2,1-\frac{1-h_o^2}{K}\right\}+\frac{3}{\sqrt{K}}\left\|\m{\Delta}\right\|_2+\frac{2}{K}\left\|\m{E}_\Omega\right\|_2-\frac{1}{K}+\varepsilon\right)\\
&+\frac{h_o^{2}}{K\left\|\m{d}\right\|_2^2}K^{\frac{5}{2}}\left(\frac{3}{\sqrt{K}}\left\|\m{\Delta}\right\|_2+\frac{2}{K}\left\|\m{E}_\Omega\right\|_2+\varepsilon\right)\\
&+\frac{1}{\left\|\m{d}\right\|_2^2}\left(4h_o^{-1}K\left\|\m{\Delta}\right\|_2+\varepsilon^{-\frac{1}{2}}N\left\|\m{\Delta}\right\|_2^2+12h_o^{-2}K\left\|\m{\Delta}\right\|_2\right)+h_o^{-1}\left(2N+2h_o^{-1}K\right)\left\|\m{\Delta}\right\|_2\\
&+h_o^{-1}\left(\xi+2K^{\frac{1}{2}}\left\|\m{E}_\Omega\right\|_2+2N\left\|\m{\Delta}\right\|_2\left\|\m{E}_\Omega\right\|_2\right)\\
&+2\left\|\m{E}_\Omega\right\|_2\sum_{n\in\chi_{\m{\Delta},\m{h}}}\left\|\mathcal{P}_{\Omega}\left(\m{\Gamma}_n\left(\frac{1}{\sqrt{K}}\m{I}+\m{\Delta}\right)\m{h}\right)\right\|_2.
\end{aligned}
\end{equation}
Consequently,
\begin{equation}
\begin{aligned}
\frac{\m{d}^\top\text{Hess}\;\psi_\varepsilon\left(\m{h}\right)\m{d}}{\left\|\m{d}\right\|_2^2}
\leq&\frac{h_o^{2}}{K\left\|\m{d}\right\|_2^2}\left(\frac{1}{K}\min\left\{Kh_o^2,1-\frac{1-h_o^2}{K}\right\}+\frac{3}{\sqrt{K}}\left\|\m{\Delta}\right\|_2+\frac{2}{K}\left\|\m{E}_\Omega\right\|_2+\varepsilon\right)^{-\frac{3}{2}}\\
&\cdot\left(\frac{1}{K}\min\left\{Kh_o^2,1-\frac{1-h_o^2}{K}\right\}+\frac{3}{\sqrt{K}}\left\|\m{\Delta}\right\|_2+\frac{2}{K}\left\|\m{E}_\Omega\right\|_2-\frac{1}{K}+\varepsilon\right)\\
&+\frac{h_o^{2}}{K\left\|\m{d}\right\|_2^2}K^{\frac{5}{2}}\left(\frac{3}{\sqrt{K}}\left\|\m{\Delta}\right\|_2+\frac{2}{K}\left\|\m{E}_\Omega\right\|_2+\varepsilon\right)\\
&+\frac{1}{\left\|\m{d}\right\|_2^2}\left(4h_o^{-1}K\left\|\m{\Delta}\right\|_2+\varepsilon^{-\frac{1}{2}}N\left\|\m{\Delta}\right\|_2^2+12h_o^{-2}K\left\|\m{\Delta}\right\|_2\right)\\
&+h_o^{-1}\left(2N+2h_o^{-1}K\right)\left\|\m{\Delta}\right\|_2\\
&+h_o^{-1}\left(\xi+2K^{\frac{1}{2}}\left\|\m{E}_\Omega\right\|_2+2N\left\|\m{\Delta}\right\|_2\left\|\m{E}_\Omega\right\|_2\right)\\
&+2\left\|\m{E}_\Omega\right\|_2\sum_{n\in\chi_{\m{\Delta},\m{h}}}\left\|\mathcal{P}_{\Omega}\left(\m{\Gamma}_n\left(\frac{1}{\sqrt{K}}\m{I}+\m{\Delta}\right)\m{h}\right)\right\|_2\\
&+2\left\|\m{E}_\Omega\right\|_2\sum_{n\in\chi_{\m{\Delta},\m{h}}}\left\|\mathcal{P}_\Omega\left(\m{\Gamma}_n\left(\frac{1}{\sqrt{K}}\m{I}+\m{\Delta}\right)\m{h}\right)\right\|_2+\frac{1}{\left\|\m{d}\right\|_2^2}h_o^2\varepsilon^{\frac{1}{2}}N\\
&+\frac{1}{\left\|\m{d}\right\|_2^2}2h_o^{-1}K^{\frac{1}{2}}\left\|\m{E}_\Omega\right\|_2+\frac{1}{\left\|\m{d}\right\|_2^2}4\varepsilon^{-\frac{1}{2}}K^{-\frac{1}{2}}N\left\|\m{E}_\Omega\right\|_2\left\|\m{\Delta}\right\|_2\\
&+\frac{1}{\left\|\m{d}\right\|_2^2}4h_o^{-2}\varepsilon K^{\frac{3}{2}}\left\|\m{E}_\Omega\right\|_2+\frac{1}{\left\|\m{d}\right\|_2^2}8h_o\varepsilon^{-\frac{1}{2}}K^{-\frac{1}{2}}N\left\|\m{E}_\Omega\right\|_2\left\|\m{\Delta}\right\|_2\\
&+\frac{1}{\left\|\m{d}\right\|_2^2}4h_o^{-2}K^{\frac{1}{2}}\left\|\m{E}_\Omega\right\|_2+\frac{1}{\left\|\m{d}\right\|_2^2}18h_o^{-3}K\left\|\m{E}_\Omega\right\|_2\left\|\m{\Delta}\right\|_2\\
&+\frac{1}{\left\|\m{d}\right\|_2^2}2\varepsilon^{-\frac{1}{2}}N\left\|\m{\Delta}\right\|_2^2\left\|\m{E}_\Omega\right\|_2+\frac{1}{\left\|\m{d}\right\|_2^2}2h_o\left(2N+2h_o^{-1}K\right)\left\|\m{\Delta}\right\|_2.
\end{aligned}
\end{equation}
Therefore, there exist positive constants $c_1\leq\frac{1}{32}$ and $c_2\leq\frac{1}{32}$ such that, for $\left\|\m{E}_\Omega\right\|_2\leq c_1\min\left\{K^{-\frac{7}{2}}N^{-1},\varepsilon^{\frac{1}{2}}\right\}$ and $\left\|\m{\Delta}\right\|_2\leq c_2\min\left\{K^{-4}N^{-1},K^{-3}N^{-\frac{3}{2}},\varepsilon^{\frac{1}{2}}K^{\frac{1}{2}}N^{-\frac{1}{2}}\right\}$, it holds that
\begin{equation}
\begin{aligned}
\frac{\m{d}^\top\text{Hess}\;\psi_\varepsilon\left(\m{h}\right)\m{d}}{\left\|\m{d}\right\|_2^2}
\leq&\frac{h_o^{2}}{K\left\|\m{d}\right\|_2^2}\left(\frac{1}{K}\min\left\{Kh_o^2,1-\frac{1-h_o^2}{K}\right\}+3c_2K^{-\frac{9}{2}}N^{-1}+2c_1K^{-\frac{9}{2}}N^{-1}+\varepsilon\right)^{-\frac{3}{2}}\\
&\cdot\left(\frac{1}{K}\min\left\{Kh_o^2,1-\frac{1-h_o^2}{K}\right\}-\frac{1}{K}+\frac{3}{\sqrt{K}}\left\|\m{\Delta}\right\|_2+\frac{2}{K}\left\|\m{E}_\Omega\right\|_2+\varepsilon\right)\\
&+\frac{h_o^{2}}{K\left\|\m{d}\right\|_2^2}\left(3c_2K^{-2}N^{-1}+2c_1K^{-2}N^{-1}+\varepsilon\right)\\
&+\frac{1}{K\left\|\m{d}\right\|_2^2}\left(4c_2h_o^{-1}K^{-2}N^{-1}+c_2^2K^{-\frac{5}{2}}N^{-\frac{1}{2}}+12c_2h_o^{-2}K^{-2}N^{-1}\right)\\
&+\frac{1}{K\left\|\m{d}\right\|_2^2}\left(6c_2h_o^{-1}K^{-2}N^{-\frac{1}{2}}+6c_2h_o^{-2}K^{-2}N^{-1}\right)\\
&+\frac{1}{K\left\|\m{d}\right\|_2^2}h_o^{-1}\left(2c_1K^{-2}N^{-1}+2c_1c_2K^{-2}N^{-1}\right)\\
&+\frac{1}{K\left\|\m{d}\right\|_2^2}\left(4c_1K^{-3}+4c_1c_2K^{-3}\right)+\frac{1}{K\left\|\m{d}\right\|_2^2}\varepsilon^{\frac{1}{2}}KN\\
&+\frac{1}{K\left\|\m{d}\right\|_2^2}2c_1h_o^{-1}K^{-2}N^{-1}+\frac{1}{K\left\|\m{d}\right\|_2^2}4c_1c_2K^{-\frac{5}{2}}N^{-\frac{1}{2}}\\
&+\frac{1}{K\left\|\m{d}\right\|_2^2}4c_1h_o^{-2}\varepsilon K^{-1}N^{-1}+\frac{1}{K\left\|\m{d}\right\|_2^2}8c_1c_2h_oK^{-\frac{5}{2}}N^{-\frac{1}{2}}\\
&+\frac{1}{K\left\|\m{d}\right\|_2^2}4c_1h_o^{-2}K^{-2}N^{-1}+\frac{1}{K\left\|\m{d}\right\|_2^2}18c_1c_2h_o^{-3}K^{-\frac{11}{2}}N^{-2}\\
&+\frac{1}{K\left\|\m{d}\right\|_2^2}2c_1c_2^2K^{-7}N^{-1}+h_o^{-1}\xi.
\end{aligned}
\end{equation}
Neglecting higher-order terms, we obtain that there exist positive constants $c_3\leq1$, and $c_4$ such that, for $\varepsilon\leq c_3K^{-7}N^{-2}$, and $\xi\leq c_4K^{-\frac{7}{2}}N^{-\frac{1}{2}}$, it holds that
\begin{equation}
\begin{aligned}
\frac{\m{d}^\top\text{Hess}\;\psi_\varepsilon\left(\m{h}\right)\m{d}}{\left\|\m{d}\right\|_2^2}
\leq&\frac{h_o^{2}}{K\left\|\m{d}\right\|_2^2}\left[\frac{1}{K}\min\left\{Kh_o^2,1-\frac{1-h_o^2}{K}\right\}+\left(2c_1+3c_2+c_3\right)K^{-\frac{9}{2}}N^{-1}\right]^{-\frac{3}{2}}\\
&\cdot\left[\frac{1}{K}\min\left\{Kh_o^2,1-\frac{1-h_o^2}{K}\right\}-\frac{1}{K}+\frac{3}{\sqrt{K}}\left\|\m{\Delta}\right\|_2+\frac{2}{K}\left\|\m{E}_\Omega\right\|_2+\varepsilon\right]\\
&+\frac{1}{K\left\|\m{d}\right\|_2^2}\left(6c_1+3c_2+2c_3+4c_1c_3+c_2^2+2c_1c_2^2+16c_1c_2\right)K^{-\frac{5}{2}}\\
&+\frac{1}{K\left\|\m{d}\right\|_2^2}\left(8c_1+27c_2+20c_1c_2\right)h_o^{-1}K^{-2}N^{-\frac{1}{2}}+\frac{1}{K\left\|\m{d}\right\|_2^2}c_4K^{-\frac{5}{2}}.
\end{aligned}
\end{equation}
Note that
\begin{equation}
\begin{aligned}
\frac{1}{K}\min\left\{Kh_o^2,1-\frac{1-h_o^2}{K}\right\}\geq\frac{1}{N}.
\end{aligned}
\end{equation}
Hence, 
\begin{equation}
\begin{aligned}
\frac{\m{d}^\top\text{Hess}\;\psi_\varepsilon\left(\m{h}\right)\m{d}}{\left\|\m{d}\right\|_2^2}
\leq&\left(1+2c_1+3c_2+c_3\right)^{-\frac{3}{2}}\frac{h_o^{2}}{\left\|\m{d}\right\|_2^2}K^{-\frac{1}{2}}\left(\min\left\{Kh_o^2,1-\frac{1-h_o^2}{K}\right\}\right)^{-\frac{3}{2}}\\
&\cdot\left(\min\left\{Kh_o^2,1-\frac{1-h_o^2}{K}\right\}-1\right)\\
&+\frac{h_o^{2}}{K\left\|\m{d}\right\|_2^2}\left(\frac{1}{K}\min\left\{Kh_o^2,1-\frac{1-h_o^2}{K}\right\}\right)^{-\frac{3}{2}}\left(\frac{3}{\sqrt{K}}\left\|\m{\Delta}\right\|_2+\frac{2}{K}\left\|\m{E}_\Omega\right\|_2+\varepsilon\right)\\
&+\frac{1}{K\left\|\m{d}\right\|_2^2}\left(6c_1+3c_2+2c_3+4c_1c_3+c_2^2+2c_1c_2^2+16c_1c_2+c_4\right)K^{-\frac{5}{2}}\\
&+\frac{1}{K\left\|\m{d}\right\|_2^2}\left(8c_1+27c_2+20c_1c_2\right)h_o^{-1}K^{-2}N^{-\frac{1}{2}},
\end{aligned}
\end{equation}
which completes the proof.


\section{Proof of Proposition~\ref{thm-case3-general}}
\label{pf-thm-case3-general}
From \eqref{direction} and \eqref{grad-general}, we obtain
\begin{equation}
\begin{aligned}
\m{d}^\top\text{grad}\;\psi_\varepsilon\left(\m{h}\right)=&\sum_{n=1}^N\left[\m{h}^\top\left(\frac{1}{\sqrt{K}}\m{I}+\m{\Delta}\right)^\top\m{\Gamma}_n^\top\left(\m{\Sigma}_{\Omega}+\m{E}_{\Omega}\right)\m{\Gamma}_n\left(\frac{1}{\sqrt{K}}\m{I}+\m{\Delta}\right)\m{h}+\varepsilon\right]^{-\frac{1}{2}}\\
&\cdot\m{d}^\top\left(\frac{1}{\sqrt{K}}\m{I}+\m{\Delta}\right)^\top\m{\Gamma}_n^\top\left(\m{\Sigma}_{\Omega}+\m{E}_{\Omega}\right)\m{\Gamma}_n\left(\frac{1}{\sqrt{K}}\m{I}+\m{\Delta}\right)\m{h}\\
=&\sum_{n=1}^N\frac{\m{d}^\top\left(\frac{1}{\sqrt{K}}\m{I}+\m{\Delta}\right)^\top\m{\Gamma}_n^\top\m{\Sigma}_{\Omega}\m{\Gamma}_n\left(\frac{1}{\sqrt{K}}\m{I}+\m{\Delta}\right)\m{h}}{\left[\m{h}^\top\left(\frac{1}{\sqrt{K}}\m{I}+\m{\Delta}\right)^\top\m{\Gamma}_n^\top\left(\m{\Sigma}_{\Omega}+\m{E}_{\Omega}\right)\m{\Gamma}_n\left(\frac{1}{\sqrt{K}}\m{I}+\m{\Delta}\right)\m{h}+\varepsilon\right]^{\frac{1}{2}}}\\
&+\sum_{n=1}^N\frac{\m{d}^\top\left(\frac{1}{\sqrt{K}}\m{I}+\m{\Delta}\right)^\top\m{\Gamma}_n^\top\m{E}_{\Omega}\m{\Gamma}_n\left(\frac{1}{\sqrt{K}}\m{I}+\m{\Delta}\right)\m{h}}{\left[\m{h}^\top\left(\frac{1}{\sqrt{K}}\m{I}+\m{\Delta}\right)^\top\m{\Gamma}_n^\top\left(\m{\Sigma}_{\Omega}+\m{E}_{\Omega}\right)\m{\Gamma}_n\left(\frac{1}{\sqrt{K}}\m{I}+\m{\Delta}\right)\m{h}+\varepsilon\right]^{\frac{1}{2}}}.
\end{aligned}
\end{equation}
Moreover,
\begin{equation}
\begin{aligned}
&\m{d}^\top\left(\frac{1}{\sqrt{K}}\m{I}+\m{\Delta}\right)^\top\m{\Gamma}_n^\top\m{\Sigma}_{\Omega}\m{\Gamma}_n\left(\frac{1}{\sqrt{K}}\m{I}+\m{\Delta}\right)\m{h}\\
=&\frac{1}{K}\m{d}^\top\m{\Gamma}_n^\top\m{\Sigma}_{\Omega}\m{\Gamma}_n\m{h}+\m{d}^\top\m{\Delta}^\top\m{\Gamma}_n^\top\m{\Sigma}_{\Omega}\m{\Gamma}_n\left(\frac{1}{\sqrt{K}}\m{I}+\m{\Delta}\right)\m{h}+\frac{1}{\sqrt{K}}\m{d}^\top\m{\Gamma}_n^\top\m{\Sigma}_{\Omega}\m{\Gamma}_n\m{\Delta}\m{h}.
\end{aligned}
\end{equation}
It follows from \eqref{out-concentrate} that
\begin{equation}
\begin{aligned}
&\sum_{n=1}^N\frac{\frac{1}{K}\m{d}^\top\m{\Gamma}_n^\top\m{\Sigma}_{\Omega}\m{\Gamma}_n\m{h}}{\left[\m{h}^\top\left(\frac{1}{\sqrt{K}}\m{I}+\m{\Delta}\right)^\top\m{\Gamma}_n^\top\left(\m{\Sigma}_{\Omega}+\m{E}_{\Omega}\right)\m{\Gamma}_n\left(\frac{1}{\sqrt{K}}\m{I}+\m{\Delta}\right)\m{h}+\varepsilon\right]^{\frac{1}{2}}}\\
=&\frac{1}{K}\sum_{n\in\chi_o}\frac{h_o\left(1-\m{h}^\top\m{\Gamma}_n^\top\m{\Sigma}_{\Omega}\m{\Gamma}_n\m{h}\right)}{\left[\m{h}^\top\left(\frac{1}{\sqrt{K}}\m{I}+\m{\Delta}\right)^\top\m{\Gamma}_n^\top\left(\m{\Sigma}_{\Omega}+\m{E}_{\Omega}\right)\m{\Gamma}_n\left(\frac{1}{\sqrt{K}}\m{I}+\m{\Delta}\right)\m{h}+\varepsilon\right]^{\frac{1}{2}}}\\
&-\frac{1}{K}\sum_{n\in\chi_o^\text{C}}\frac{h_o\m{h}^\top\m{\Gamma}_n^\top\m{\Sigma}_{\Omega}\m{\Gamma}_n\m{h}}{\left[\m{h}^\top\left(\frac{1}{\sqrt{K}}\m{I}+\m{\Delta}\right)^\top\m{\Gamma}_n^\top\left(\m{\Sigma}_{\Omega}+\m{E}_{\Omega}\right)\m{\Gamma}_n\left(\frac{1}{\sqrt{K}}\m{I}+\m{\Delta}\right)\m{h}+\varepsilon\right]^{\frac{1}{2}}}\\
\leq&\frac{1}{K}\frac{h_o}{\left[\frac{h_o^2}{K}-\frac{2\left\|\m{\Delta}\right\|_2}{\sqrt{K}}-\left\|\m{\Delta}\right\|_2^2-\left(\frac{1}{\sqrt{K}}+\left\|\m{\Delta}\right\|_2\right)^2\left\|\m{E}_\Omega\right\|_2+\varepsilon\right]^{\frac{1}{2}}}\sum_{n\in\chi_o}\left(1-\m{h}^\top\m{\Gamma}_n^\top\m{\Sigma}_{\Omega}\m{\Gamma}_n\m{h}\right)\\
&-\frac{1}{K}\frac{h_o}{\left[\frac{1-h_o^2}{K}+\frac{2\left\|\m{\Delta}\right\|_2}{\sqrt{K}}+\left\|\m{\Delta}\right\|_2^2+\left(\frac{1}{\sqrt{K}}+\left\|\m{\Delta}\right\|_2\right)^2\left\|\m{E}_\Omega\right\|_2+\varepsilon\right]^{\frac{1}{2}}}\sum_{n\in\chi_o^\text{C}}\m{h}^\top\m{\Gamma}_n^\top\m{\Sigma}_{\Omega}\m{\Gamma}_n\m{h}\\
\leq&\frac{1}{K}\frac{h_o}{\left[\frac{h_o^2}{K}-\frac{2\left\|\m{\Delta}\right\|_2}{\sqrt{K}}-\left\|\m{\Delta}\right\|_2^2-\left(\frac{1}{\sqrt{K}}+\left\|\m{\Delta}\right\|_2\right)^2\left\|\m{E}_\Omega\right\|_2+\varepsilon\right]^{\frac{1}{2}}}K\left(1-h_o^2\right)\\
&-\frac{1}{K}\frac{h_o}{\left[\frac{1-h_o^2}{K}+\frac{2\left\|\m{\Delta}\right\|_2}{\sqrt{K}}+\left\|\m{\Delta}\right\|_2^2+\left(\frac{1}{\sqrt{K}}+\left\|\m{\Delta}\right\|_2\right)^2\left\|\m{E}_\Omega\right\|_2+\varepsilon\right]^{\frac{1}{2}}}\left(1-h_o^2\right)\\
=&\frac{1}{K}h_o\left(1-h_o^2\right)\Bigg\{K\left[\frac{h_o^2}{K}-\frac{2\left\|\m{\Delta}\right\|_2}{\sqrt{K}}-\left\|\m{\Delta}\right\|_2^2-\left(\frac{1}{\sqrt{K}}+\left\|\m{\Delta}\right\|_2\right)^2\left\|\m{E}_\Omega\right\|_2+\varepsilon\right]^{-\frac{1}{2}}\\
&-\left[\frac{1-h_o^2}{K}+\frac{2\left\|\m{\Delta}\right\|_2}{\sqrt{K}}+\left\|\m{\Delta}\right\|_2^2+\left(\frac{1}{\sqrt{K}}+\left\|\m{\Delta}\right\|_2\right)^2\left\|\m{E}_\Omega\right\|_2+\varepsilon\right]^{-\frac{1}{2}}\Bigg\}.
\end{aligned}
\end{equation}
Furthermore,
\begin{equation}
\begin{aligned}
&\sum_{n=1}^N\frac{\m{d}^\top\m{\Delta}^\top\m{\Gamma}_n^\top\m{\Sigma}_{\Omega}\m{\Gamma}_n\left(\frac{1}{\sqrt{K}}\m{I}+\m{\Delta}\right)\m{h}+\frac{1}{\sqrt{K}}\m{d}^\top\m{\Gamma}_n^\top\m{\Sigma}_{\Omega}\m{\Gamma}_n\m{\Delta}\m{h}}{\left[\m{h}^\top\left(\frac{1}{\sqrt{K}}\m{I}+\m{\Delta}\right)^\top\m{\Gamma}_n^\top\left(\m{\Sigma}_{\Omega}+\m{E}_{\Omega}\right)\m{\Gamma}_n\left(\frac{1}{\sqrt{K}}\m{I}+\m{\Delta}\right)\m{h}+\varepsilon\right]^{\frac{1}{2}}}\\
\leq&N\left\|\m{d}\right\|_2\left\|\m{\Delta}\right\|_2\left(1-\left\|\m{E}_\Omega\right\|_2\right)^{-\frac{1}{2}}+\varepsilon^{-\frac{1}{2}}K^{-\frac{1}{2}}N\left\|\m{d}\right\|_2\left\|\m{\Delta}\right\|_2,
\end{aligned}
\end{equation}
and
\begin{equation}
\begin{aligned}
&\sum_{n=1}^N\frac{\m{d}^\top\left(\frac{1}{\sqrt{K}}\m{I}+\m{\Delta}\right)^\top\m{\Gamma}_n^\top\m{E}_{\Omega}\m{\Gamma}_n\left(\frac{1}{\sqrt{K}}\m{I}+\m{\Delta}\right)\m{h}}{\left[\m{h}^\top\left(\frac{1}{\sqrt{K}}\m{I}+\m{\Delta}\right)^\top\m{\Gamma}_n^\top\left(\m{\Sigma}_{\Omega}+\m{E}_{\Omega}\right)\m{\Gamma}_n\left(\frac{1}{\sqrt{K}}\m{I}+\m{\Delta}\right)\m{h}+\varepsilon\right]^{\frac{1}{2}}}\leq N\left\|\m{d}\right\|_2\frac{\left\|\m{E}_\Omega\right\|_2}{\left(1-\left\|\m{E}_\Omega\right\|_2\right)^{\frac{1}{2}}}.
\end{aligned}
\end{equation}
Since $\m{h}$ satisfies \eqref{prop-first-order stationary-general},
\begin{equation}
\begin{aligned}
\frac{\left|\m{d}^\top\text{grad}\;\psi_\varepsilon\left(\m{h}\right)\right|}{\left\|\m{d}\right\|_2}\leq\left\|\text{grad}\;\psi_\varepsilon\left(\m{h}\right)\right\|_2\leq\xi.
\end{aligned}
\end{equation}
Consequently, if
\begin{equation}
\begin{aligned}
&\Bigg\{K\left[\frac{h_o^2}{K}-\frac{2\left\|\m{\Delta}\right\|_2}{\sqrt{K}}-\left\|\m{\Delta}\right\|_2^2-\left(\frac{1}{\sqrt{K}}+\left\|\m{\Delta}\right\|_2\right)^2\left\|\m{E}_\Omega\right\|_2+\varepsilon\right]^{-\frac{1}{2}}\\
&-\left[\frac{1-h_o^2}{K}+\frac{2\left\|\m{\Delta}\right\|_2}{\sqrt{K}}+\left\|\m{\Delta}\right\|_2^2+\left(\frac{1}{\sqrt{K}}+\left\|\m{\Delta}\right\|_2\right)^2\left\|\m{E}_\Omega\right\|_2+\varepsilon\right]^{-\frac{1}{2}}\Bigg\}<0,
\end{aligned}
\end{equation}
then
\begin{equation}
\begin{aligned}
&\frac{1}{K}h_o\left(1-h_o^2\right)^{\frac{1}{2}}\Bigg\{\left[\frac{1-h_o^2}{K}+\frac{2\left\|\m{\Delta}\right\|_2}{\sqrt{K}}+\left\|\m{\Delta}\right\|_2^2+\left(\frac{1}{\sqrt{K}}+\left\|\m{\Delta}\right\|_2\right)^2\left\|\m{E}_\Omega\right\|_2+\varepsilon\right]^{-\frac{1}{2}}\\
&-K\left[\frac{h_o^2}{K}-\frac{2\left\|\m{\Delta}\right\|_2}{\sqrt{K}}-\left\|\m{\Delta}\right\|_2^2-\left(\frac{1}{\sqrt{K}}+\left\|\m{\Delta}\right\|_2\right)^2\left\|\m{E}_\Omega\right\|_2+\varepsilon\right]^{-\frac{1}{2}}\Bigg\}\\
\leq&\xi+N\left\|\m{\Delta}\right\|_2\left(1-\left\|\m{E}_\Omega\right\|_2\right)^{-\frac{1}{2}}+\varepsilon^{-\frac{1}{2}}K^{-\frac{1}{2}}N\left\|\m{\Delta}\right\|_2+N\frac{\left\|\m{E}_\Omega\right\|_2}{\left(1-\left\|\m{E}_\Omega\right\|_2\right)^{\frac{1}{2}}}.
\end{aligned}
\end{equation}
Therefore, there exist positive constants $c_1\leq\frac{1}{2}$, $c_2\leq1$, and $c_3\leq1$ such that, whenever $\left\|\m{E}_\Omega\right\|_2\leq c_1K^{-2}$, $\left\|\m{\Delta}\right\|_2\leq c_2K^{-\frac{5}{2}}$, and $\varepsilon\leq c_3K^{-3}$, the following inequality holds for all $h_o\in\left(\sqrt{1-\frac{1}{4K^2}},1\right]$:
\begin{equation}
\begin{aligned}
&\Bigg\{\left[\frac{1-h_o^2}{K}+\frac{2\left\|\m{\Delta}\right\|_2}{\sqrt{K}}+\left\|\m{\Delta}\right\|_2^2+\left(\frac{1}{\sqrt{K}}+\left\|\m{\Delta}\right\|_2\right)^2\left\|\m{E}_\Omega\right\|_2+\varepsilon\right]^{-\frac{1}{2}}\\
&-K\left[\frac{h_o^2}{K}-\frac{2\left\|\m{\Delta}\right\|_2}{\sqrt{K}}-\left\|\m{\Delta}\right\|_2^2-\left(\frac{1}{\sqrt{K}}+\left\|\m{\Delta}\right\|_2\right)^2\left\|\m{E}_\Omega\right\|_2+\varepsilon\right]^{-\frac{1}{2}}\Bigg\}\\
\geq&\left[2\left(1+4c_1+3c_2+c_3\right)^{-\frac{1}{2}}-\left(1-4c_1-3c_2\right)^{-\frac{1}{2}}\right]K^{\frac{3}{2}}.
\end{aligned}
\end{equation}
It follows that
\begin{equation}
\begin{aligned}
\left(1-h_o^2\right)^{\frac{1}{2}}\leq&\sqrt{2}\left[2\left(1+4c_1+3c_2+c_3\right)^{-\frac{1}{2}}-\left(1-4c_1-3c_2\right)^{-\frac{1}{2}}\right]^{-1}K^{-\frac{3}{2}}\\
&\cdot \left[\xi+2N\left(\left\|\m{E}_\Omega\right\|_2+\left\|\m{\Delta}\right\|_2\right)+\varepsilon^{-\frac{1}{2}}K^{-\frac{1}{2}}N\left\|\m{\Delta}\right\|_2\right].
\end{aligned}
\end{equation}
Finally, applying \eqref{error} yields
\begin{equation}
\begin{aligned}
\left\|\m{h}-\m{e}_o\right\|_2\leq&2\left[2\left(1+4c_1+3c_2+c_3\right)^{-\frac{1}{2}}-\left(1-4c_1-3c_2\right)^{-\frac{1}{2}}\right]^{-1}K^{-\frac{3}{2}}\\
&\cdot \left[\xi+2N\left(\left\|\m{E}_\Omega\right\|_2+\left\|\m{\Delta}\right\|_2\right)+\varepsilon^{-\frac{1}{2}}K^{-\frac{1}{2}}N\left\|\m{\Delta}\right\|_2\right],
\end{aligned}
\end{equation}
which completes the proof.

\section{Proof of Proposition~\ref{prop-alg}}
\label{pf-prop-alg}
The proof relies on the fact that the objective function in~\eqref{opt-ini-general} is three-times continuously differentiable and that
\begin{equation}
\label{retraction}
\begin{aligned}
\mathcal{R}_{\m{h}}\left(\m{v}\right)=\frac{\m{h}+\m{v}}{\left\|\m{h}+\m{v}\right\|_2}
\end{aligned}
\end{equation}
defines a second-order retraction on $\mathbb{S}^{N-1}$; see~\cite[Example~5.41]{boumal2023introduction}. Specifically, since $\mathbb{S}^{N-1}$ is compact, the Riemannian gradient $\text{grad}\;\rho\left(\m{h}\right)$ and $\text{Hess}\;\rho\left(\m{h}\right)$ are bounded and Lipschitz continuous on $\mathbb{S}^{N-1}$. Let $L_g$ denote the corresponding Lipschitz constant of the Riemannian gradient. We separately analyze the gradient descent steps and the negative-curvature search steps.

For $\left\|\text{grad}\;\rho\left(\m{h}_{t-1}\right)\right\|_2>\epsilon_g$, the algorithm performs a gradient descent step. By~\cite[Exercise~10.56]{boumal2023introduction},
\begin{equation}
\begin{aligned}
\rho\left(\m{h}_{t}\right)=&\rho\left(\mathcal{R}_{\m{h}_{t-1}}\left(-\alpha_g\text{grad}\;\rho\left(\m{h}_{t-1}\right)\right)\right)\\
\leq&\rho\left(\m{h}_{t-1}\right)+g_{\m{h}_{t-1}}\left(-\alpha_g\text{grad}\;\rho\left(\m{h}_{t-1}\right),\text{grad}\;\rho\left(\m{h}_{t-1}\right)\right)+\frac{L_g}{2}\left\|\alpha_g\text{grad}\;\rho\left(\m{h}_{t-1}\right)\right\|_2^2.
\end{aligned}
\end{equation}
Choosing $\alpha_g=L_g^{-1}$ yields
\begin{equation}
\begin{aligned}
\rho\left(\m{h}_{t}\right)-\rho\left(\m{h}_{t-1}\right)\leq&-\frac{1}{2L_g}\left\|\text{grad}\;\rho\left(\m{h}_{t-1}\right)\right\|_2^2\leq-\frac{\epsilon_g^2}{2L_g}.
\end{aligned}
\end{equation}

Next, suppose that $\left\|\text{grad}\;\rho\left(\m{h}_{t-1}\right)\right\|_2\leq\epsilon_g$ and $\lambda_{\min}\left(\text{Hess}\;\rho\left(\m{h}_{t-1}\right)\right)<-\epsilon_H$. In this case, the algorithm performs a negative-curvature search step. Define $\gamma\left(\alpha\right)=\mathcal{R}_{\m{h}_{t-1}}\left(-\alpha\eta_t\m{v}_t\right)$. Since $\rho\left(\gamma\left(\alpha\right)\right)$ is three-times continuously differentiable in a neighborhood of $\alpha=0$, Taylor theorem with Lagrange remainder implies that, for some $\nu$,
\begin{equation}
\begin{aligned}
\rho\left(\gamma\left(\alpha\right)\right)=&\rho\left(\m{h}_{t-1}\right)+\alpha g_{\m{h}_{t-1}}\left(\text{grad}\;\rho\left(\m{h}_{t-1}\right),-\eta_t\m{v}_t\right)+\frac{\alpha^2}{2}g_{\m{h}_{t-1}}\left(-\eta_t\m{v}_t,-\eta_t\text{Hess}\;\rho\left(\m{h}_{t-1}\right)\m{v}_t\right)\\
&+\frac{\alpha^3}{6}\left(\rho\circ\gamma\right)'''\left(\nu\right),
\end{aligned}
\end{equation}
where $\rho\circ\gamma$ denotes the composition of $\rho$ and $\gamma$. Since $\mathbb{S}^{N-1}$ is compact and $\gamma\left(\alpha\right)\in\mathbb{S}^{N-1}$ for all $\alpha$, there exists a constant $L_\gamma>0$ such that
\begin{equation}
\left|\left(\rho\circ\gamma\right)'''\left(\nu\right)\right|\leq L_\gamma
\end{equation}
Therefore,
\begin{equation}
\left|\frac{\alpha^3}{6}\left(\rho\circ\gamma\right)'''\left(\nu\right)\right|\leq\frac{\alpha^3L_\gamma}{6}.
\end{equation}
Substituting $\alpha=\alpha_H$ gives
\begin{equation}
\begin{aligned}
\rho\left(\m{h}_{t}\right)=&\rho\left(\mathcal{R}_{\m{h}_{t-1}}\left(-\alpha_H\eta_t\m{v}_t\right)\right)\\
\leq&\rho\left(\m{h}_{t-1}\right)+\alpha_Hg_{\m{h}_{t-1}}\left(\text{grad}\;\rho\left(\m{h}_{t-1}\right),-\eta_t\m{v}_t\right)+\frac{\alpha_H^2}{2}g_{\m{h}_{t-1}}\left(-\eta_t\m{v}_t,-\eta_t\text{Hess}\;\rho\left(\m{h}_{t-1}\right)\m{v}_t\right)\\
&+\frac{\alpha_H^3L_\gamma}{6}.
\end{aligned}
\end{equation}
Choosing $\alpha_H=\frac{3\epsilon_H}{2L_\gamma}$ yields
\begin{equation}
\begin{aligned}
\rho\left(\m{h}_{t}\right)-\rho\left(\m{h}_{t-1}\right)\leq-\frac{\alpha_H^2\epsilon_H}{2}+\frac{\alpha_H^3L_\gamma}{6}=-\frac{9\epsilon_H^3}{16L_\gamma^2}.
\end{aligned}
\end{equation}

Combining the two cases, every iteration decreases the objective function by at least
\begin{equation}
\begin{aligned}
\rho_{\Delta}=\min\left\{\frac{\epsilon_g^2}{2L_g},\frac{9\epsilon_H^3}{16L_\gamma^2}\right\}.
\end{aligned}
\end{equation}
Since $\rho\left(\m{h}\right)\geq0$, the algorithm can perform at most
\begin{equation}
\begin{aligned}
\frac{\rho\left(\m{h}_0\right)}{\rho_{\Delta}}=\frac{\rho\left(\m{h}_0\right)}{\min\left\{\frac{\epsilon_g^2}{2L_g},\frac{9\epsilon_H^3}{16L_\gamma^2}\right\}}
\end{aligned}
\end{equation}
iterations. Therefore, the algorithm terminates after a finite number of iterations, with iteration complexity
\begin{equation}
\begin{aligned}
T=O\left(\max\left\{\epsilon_g^{-2},\epsilon_H^{-3}\right\}\right).
\end{aligned}
\end{equation}
Upon termination, neither the gradient descent condition nor the negative-curvature condition is satisfied, implying that the returned iterate is an $\left(\epsilon_g,\epsilon_H\right)$-approximate second-order stationary point.

\putbib[strings]

\end{bibunit}

\end{document}